\documentclass[leqno,12pt]{article}

\usepackage[dvips]{graphics}
\usepackage{color}
\usepackage{ifthen}
\usepackage{soul}
\usepackage{latexsym}
\usepackage{amsthm,amssymb}
\usepackage{amsbsy,amsfonts,amsmath}
\usepackage{txfonts}
\usepackage{setspace}

\usepackage{a4}

\numberwithin{equation}{section}

\newtheorem{theorem}{Theorem}[section]
\newtheorem{proposition}[theorem]{Proposition}

\newtheorem{lemma}[theorem]{Lemma}

\theoremstyle{definition}
\newtheorem{definition}[theorem]{Definition}

\theoremstyle{remark}
\newtheorem{remark}[theorem]{Remark}

\newcommand{\al}{\alpha}
\newcommand{\dal}{{\dot{\al}}}

\newcommand{\fkJ}{J}
\newcommand{\fkI}{I}

\newcommand{\bN}{{\mathbb{N}}}
\newcommand{\bZ}{{\mathbb{Z}}}
\newcommand{\bZgeqo}{\bZ_{\geq 0}}
\newcommand{\bZleqo}{\bZ_{\leq 0}}
\newcommand{\bQ}{{\mathbb{Q}}}
\newcommand{\bR}{{\mathbb{R}}}

\newcommand{\bRgeqo}{\bR_{\geq 0}}
\newcommand{\bRgneqo}{\bR_{>0}}

\newcommand{\bC}{{\mathbb{C}}}
\newcommand{\bCt}{\bC^\times}

\newcommand{\bK}{{\mathbb{K}}}
\newcommand{\bKt}{\bK^\times}

\newcommand{\fca}{{\mathfrak{a}}}
\newcommand{\Pfca}{{\mathfrak{P}}(\fca)}
\newcommand{\Bfca}{{\mathfrak{B}}(\fca)}
\newcommand{\rR}{R}
\newcommand{\rmSpan}{{\mathrm{Span}}}

\newcommand{\bB}{{\mathbb{B}}}
\newcommand{\bBR}{\bB(\rR)}
\newcommand{\bBRsim}{{\bBR/\!\!\sim}}
\newcommand{\tausim}{\tau^\sim}
\newcommand{\tausimi}{{\tausim_i}}
\newcommand{\tausimj}{{\tausim_j}}

\newcommand{\frg}{{\mathfrak{g}}}

\newcommand{\hatC}{{\hat C}}
\newcommand{\hatc}{{\hat c}}

\newcommand{\FzeroI}{{{\mathcal{F}}_0(\fkI)}}
\newcommand{\FoneI}{{{\mathcal{F}}_1(\fkI)}}
\newcommand{\FtwoI}{{{\mathcal{F}}_2(\fkI)}}

\newcommand{\trxi}{\xi}

\newcommand{\mcA}{{\mathcal A}}
\newcommand{\mcC}{{\mathcal C}}
\newcommand{\mcR}{{\mathcal R}}
\newcommand{\hatfca}{{\hat{\fca}}}
\newcommand{\hatal}{{\hat{\al}}}
\newcommand{\hats}{{\hat{s}}}
\newcommand{\hatR}{{\hat{\rR}}}
\newcommand{\hatW}{{\hat{W}}}

\newcommand{\hattrxi}{{\hat{\trxi}}}

\newcommand{\hatm}{{\hat{m}}}

\newcommand{\rmid}{{\mathrm{id}}}

\newcommand{\hate}{{\hat{e}}}

\newcommand{\hatell}{{\hat{\ell}}}

\newcommand{\hacun}{{\acute{n}}}
\newcommand{\ddotsim}{{\,{\ddot{\sim}}\,}}

\newcommand{\wealsolet}{
\setlength{\unitlength}{1mm}
\begin{picture}(10,10)(1.5,1)
\put(0,8){\mbox{We also let}}
\put(22.5,8){\circle{2}}\put(23.5,8){\line(1,0){5.5}}
	\put(29.5,7.75){\dots}\put(34.5,8){\line(1,0){5.5}}\put(41,8){\circle{2}}
\put(29.5,3.5){\mbox{$(k)$}}
\put(45,8){\mbox{mean}}
\put(60,8){\circle{2}}
\put(65,8){\mbox{(if $k=1$), }}
\put(87.5,8){\circle{2}}\put(88.5,8){\line(1,0){5.5}}
	\put(94.5,7.75){\dots}\put(99.5,8){\line(1,0){5.5}}\put(106,8){\circle{2}}
\put(91.5,3.5){\mbox{$(k-1)$}}\put(113.5,8){\circle{2}}
\put(107,8){\line(1,0){5.5}}
\put(122,8){\mbox{(if $k\geq 2$).}}
\end{picture}}

\newcommand{\graydot}{
\setlength{\unitlength}{1mm}
\begin{picture}(10,10)(1.5,1)
\put(1.5,1){\circle{2}}\put(0.75,0.25){\rotatebox{45}{\line(1,0){2}}}\put(0.75,0.25){\rotatebox{45}{\line(0,1){2}}}
\end{picture}}

\newcommand{\DynkinAn}{
\setlength{\unitlength}{1mm}
\begin{picture}(40,20)(5,-5)
	\put(0,7){$A_n$}
      \put(0,3){$(n\geq 1)$}
	\put(30,8){\circle{2}}\put(31,8){\line(1,0){5.5}}
	\put(37,7.75){\dots}\put(37,3.5){\mbox{$(n)$}}
       \put(42,8){\line(1,0){5.5}}\put(48.5,8){\circle{2}}

\end{picture}}

\newcommand{\DynkinBn}{
\setlength{\unitlength}{1mm}
\begin{picture}(40,20)(5,-5)
	\put(0,7){$B_n$}
      \put(0,3){$(n\geq 3)$}
	\put(30,8){\circle{2}}\put(31,8){\line(1,0){5.5}}
	\put(37,7.75){\dots}\put(34,3.5){\mbox{$(n-1)$}}
       \put(42,8){\line(1,0){5.5}}\put(48.5,8){\circle{2}}
\put(49.25,7){$\Longrightarrow$}\put(56.25,8){\circle{2}}
	
\end{picture}}

\newcommand{\DynkinCn}{
\setlength{\unitlength}{1mm}
\begin{picture}(40,20)(5,-5)
	\put(0,7){$C_n$}
      \put(0,3){$(n\geq 2)$}
	\put(30,8){\circle{2}}\put(31,8){\line(1,0){5.5}}
	\put(37,7.75){\dots}\put(34,3.5){\mbox{$(n-1)$}}
       \put(42,8){\line(1,0){5.5}}\put(48.5,8){\circle{2}}
\put(49.25,7){$\Longleftarrow$}\put(56.25,8){\circle{2}}
	
\end{picture}}

\newcommand{\DynkinDn}{
\setlength{\unitlength}{1mm}
\begin{picture}(40,20)(5,-5)
	\put(0,7){$D_n$}
      \put(0,3){$(n\geq 4)$}
	\put(30,8){\circle{2}}\put(31,8){\line(1,0){5.5}}
	\put(37,7.75){\dots}\put(34,3.5){\mbox{$(n-2)$}}
       \put(42,8){\line(1,0){5.5}}\put(48.5,8){\circle{2}}
	\put(48.5,9){\line(0,1){5.5}}\put(48.5,15.5){\circle{2}}
      \put(49.5,8){\line(1,0){5.5}}
     \put(56,8){\circle{2}}
	
\end{picture}}

\newcommand{\DynkinEn}{
\setlength{\unitlength}{1mm}
\begin{picture}(40,20)(5,-5)
	\put(0,7){$E_n$}
      \put(0,3){$(n= 6,7,8)$}
	\put(30,8){\circle{2}}\put(31,8){\line(1,0){5.5}}
	\put(37,7.75){\dots}\put(34,3.5){\mbox{$(n-3)$}}
       \put(42,8){\line(1,0){5.5}}\put(48.5,8){\circle{2}}
	\put(48.5,9){\line(0,1){5.5}}\put(48.5,15.5){\circle{2}}
      \put(49.5,8){\line(1,0){5.5}}
     \put(56,8){\circle{2}}
\put(57,8){\line(1,0){5.5}}
     \put(63.5,8){\circle{2}}
	
\end{picture}}

\newcommand{\DynkinFfour}{
\setlength{\unitlength}{1mm}
\begin{picture}(40,20)(5,-5)
	\put(0,7){$\mbox{$F_4$}$}
	\put(30,8){\circle{2}}
\put(31,8){\line(1,0){5.5}}
\put(37.5,8){\circle{2}}
	\put(38.35,7){$\Longleftarrow$}
      \put(45.25,8){\circle{2}}
\put(46.25,8){\line(1,0){5.5}}
\put(52.75,8){\circle{2}}

\end{picture}}

\newcommand{\DynkinGtwo}{
\setlength{\unitlength}{1mm}
\begin{picture}(40,20)(5,-5)
	\put(0,7){$G_2$}

\put(30,8){\circle{2}}

\put(31,8){\rotatebox{45}{\line(1,0){3}}}
\put(31,6){\rotatebox{45}{\line(0,1){3}}}

\put(32,9){\line(1,0){5.3}}
\put(31,8){\line(1,0){6}}
\put(32,7){\line(1,0){5.3}}

\put(38,8){\circle{2}}

\end{picture}}

\newcommand{\Dynkinsloneone}{
\setlength{\unitlength}{1mm}
\begin{picture}(40,20)(5,-5)
\put(0,7){$sl(1,1)$}
\put(20,8){$\graydot$}

\put(50,7){${\rm{Heisenberg}}$}
\put(80,7){\circle{2}}
\put(79,7){\line(1,0){2}}
\end{picture}}

\newcommand{\DynkinAmzero}{
\setlength{\unitlength}{1mm}
\begin{picture}(40,20)(5,-5)
	\put(0,7){$A(m,0)$}
      \put(0,3){$(m\geq 1)$}
	\put(30,8){\circle{2}}\put(31,8){\line(1,0){5.5}}
	\put(37,7.75){\dots}\put(37,3.5){\mbox{$(m)$}}
       \put(42,8){\line(1,0){5.5}}\put(48.5,8){\circle{2}}
	\put(49.5,8){\line(1,0){5.5}}
       \put(56,8){$\graydot$}

\end{picture}}

\newcommand{\DynkinAmn}{
\setlength{\unitlength}{1mm}
\begin{picture}(40,20)(5,-5)
	\put(0,7){$A(m,n)$}
      \put(0,3){$(m\geq n\geq 1)$}
	\put(30,8){\circle{2}}\put(31,8){\line(1,0){5.5}}
	\put(37,7.75){\dots}\put(37,3.5){\mbox{$(m)$}}
       \put(42,8){\line(1,0){5.5}}\put(48.5,8){\circle{2}}
	\put(49.5,8){\line(1,0){5.5}}\put(51,9){$x$}
     \put(56,8){$\graydot$}
      \put(57,8){\line(1,0){5.5}}\put(57.5,9){$-x$}
     \put(63.5,8){\circle{2}}
	\put(64.5,8){\line(1,0){5.5}}
     \put(70.5,7.75){\dots}\put(70.5,3.5){\mbox{$(n)$}}
      \put(75.5,8){\line(1,0){5.5}}
	\put(82,8){\circle{2}}
     
\put(86,7){$(x\ne 0)$}

\end{picture}}

\newcommand{\DynkinBzeron}{
\setlength{\unitlength}{1mm}
\begin{picture}(40,20)(5,-5)
	\put(0,7){$B(0,n)$}
      \put(0,3){$(n\geq 2)$}
     \put(30,8){\circle{2}}\put(31,8){\line(1,0){5.5}}
	\put(37,7.75){\dots}\put(34,3.5){\mbox{$(n-1)$}}
       \put(42,8){\line(1,0){5.5}}\put(48.5,8){\circle{2}}
\put(49.25,7){$\Longrightarrow$}\put(56.25,8){\circle*{2}}

\put(95,7){$B(0,1)$}
\put(113.5,8){\circle*{2}}

\end{picture}}

\newcommand{\DynkinBonem}{
\setlength{\unitlength}{1mm}
\begin{picture}(40,20)(5,-5)
	\put(0,7){$B(1,n)$}
      \put(0,3){$(n\geq 2)$}
     \put(30,8){\circle{2}}\put(31,8){\line(1,0){5.5}}
	\put(37,7.75){\dots}\put(34,3.5){\mbox{$(n-1)$}}
       \put(42,8){\line(1,0){5.5}}\put(48.5,8){\circle{2}}
	\put(49.5,8){\line(1,0){5.5}}\put(50,9){$-x$}
      \put(56,8){$\graydot$}
      \put(56.75,7){$\Longrightarrow$}\put(58.5,9.5){$x$}
      \put(63.75,8){\circle{2}}

\put(68,7){$(x\ne 0)$}

\put(95,7){$B(1,1)$}

\put(113.5,8){$\graydot$}
\put(114.25,7){$\Longrightarrow$}
      \put(121.25,8){\circle{2}}

\end{picture}}

\newcommand{\DynkinBmone}{
\setlength{\unitlength}{1mm}
\begin{picture}(40,20)(5,-5)
	\put(0,7){$B(m,1)$}
      \put(0,3){$(m\geq 2)$}
	\put(30,8){$\graydot$}\put(31,8){\line(1,0){5.5}}
	\put(37.5,8){\circle{2}}\put(38.5,8){\line(1,0){5.5}}
	\put(44.5,7.75){\dots}\put(40.5,3.5){\mbox{$(m-1)$}}
     \put(49.5,8){\line(1,0){5.5}}\put(56,8){\circle{2}}
	\put(56.75,7){$\Longrightarrow$}
      \put(63.75,8){\circle{2}}

\end{picture}}

\newcommand{\DynkinBmn}{
\setlength{\unitlength}{1mm}
\begin{picture}(40,20)(5,-5)
	\put(0,7){$B(m,n)$}
      \put(0,3){$(m\geq 2, n\geq 2)$}
	\put(30,8){\circle{2}}\put(31,8){\line(1,0){5.5}}
	\put(37,7.75){\dots}\put(34,3.5){\mbox{$(n-1)$}}
       \put(42,8){\line(1,0){5.5}}\put(48.5,8){\circle{2}}
	\put(49.5,8){\line(1,0){5.5}}\put(50,9){$-x$}
     \put(56,8){$\graydot$}
      \put(57,8){\line(1,0){5.5}}\put(58.5,9){$x$}
     \put(63.5,8){\circle{2}}
	\put(64.5,8){\line(1,0){5.5}}
     \put(70.5,7.75){\dots}\put(66.5,3.5){\mbox{$(m-1)$}}
      \put(75.5,8){\line(1,0){5.5}}
	\put(82,8){\circle{2}}
\put(82.75,7){$\Longrightarrow$}
      \put(89.75,8){\circle{2}}

\put(94,7){$(x\ne 0)$}

\end{picture}}

\newcommand{\DynkinSCn}{
\setlength{\unitlength}{1mm}
\begin{picture}(40,20)(5,-5)
	\put(0,7){$C(n)$}
      \put(0,3){$(n\geq 3)$}
	\put(30,8){$\graydot$}\put(31,8){\line(1,0){5.5}}
	\put(37.5,8){\circle{2}}\put(38.5,8){\line(1,0){5.5}}
	\put(44.5,7.75){\dots}\put(41.5,3.5){\mbox{$(n-2)$}}
     \put(49.5,8){\line(1,0){5.5}}\put(56,8){\circle{2}}
	\put(56.75,7){$\Longleftarrow$}
      \put(63.75,8){\circle{2}}
	
\end{picture}}

\newcommand{\DynkinDtwon}{
\setlength{\unitlength}{1mm}
\begin{picture}(40,20)(5,-5)
	\put(0,7){$D(2,n)$}
      \put(0,3){$(n\geq 2)$}
	\put(30,8){\circle{2}}\put(31,8){\line(1,0){5.5}}
	\put(37,7.75){\dots}\put(34,3.5){\mbox{$(n-1)$}}
       \put(42,8){\line(1,0){5.5}}\put(48.5,8){\circle{2}}
	\put(49.5,8){\line(1,0){5.5}}\put(50,5){$-x$}
     \put(56,8){$\graydot$}\put(56,9){\line(0,1){5.5}}\put(56,15.5){\circle{2}}\put(57,11){$x$}
      \put(57,8){\line(1,0){5.5}}\put(58.5,5){$x$}
     \put(63.5,8){\circle{2}}
	
\put(68,7){$(x\ne 0)$}

\put(90,7){$D(2,1)$}

\put(108.5,8){\circle{2}}
	\put(109.5,8){\line(1,0){5.5}}\put(111,9){$x$}
     \put(116,8){$\graydot$}
      \put(117,8){\line(1,0){5.5}}\put(118.5,9){$x$}
     \put(123.5,8){\circle{2}}

\put(128,7){$(x\ne 0)$}

\end{picture}}

\newcommand{\DynkinDmone}{
\setlength{\unitlength}{1mm}
\begin{picture}(40,20)(5,-5)
	\put(0,7){$D(m,1)$}
      \put(0,3){$(m\geq 3)$}
	\put(30,8){$\graydot$}\put(31,8){\line(1,0){5.5}}
	\put(37.5,8){\circle{2}}\put(38.5,8){\line(1,0){5.5}}
	\put(44.5,7.75){\dots}\put(41,3.5){\mbox{$(m-2)$}}
     \put(49.5,8){\line(1,0){5.5}}\put(56,8){\circle{2}}
	\put(56,9){\line(0,1){5.5}}\put(56,15.5){\circle{2}}
      \put(57,8){\line(1,0){5.5}}
     \put(63.5,8){\circle{2}}
	
\end{picture}}

\newcommand{\DynkinDmn}{
\setlength{\unitlength}{1mm}
\begin{picture}(40,20)(5,-5)
	\put(0,7){$D(m,n)$}
      \put(0,3){$(m\geq 3, n\geq 2)$}
	\put(30,8){\circle{2}}\put(31,8){\line(1,0){5.5}}
	\put(37,7.75){\dots}\put(34,3.5){\mbox{$(n-1)$}}
       \put(42,8){\line(1,0){5.5}}\put(48.5,8){\circle{2}}
	\put(49.5,8){\line(1,0){5.5}}\put(50,9){$-x$}
     \put(56,8){$\graydot$}
      \put(57,8){\line(1,0){5.5}}\put(58.5,9){$x$}
     \put(63.5,8){\circle{2}}
	\put(64.5,8){\line(1,0){5.5}}
     \put(70.5,7.75){\dots}\put(66.5,3.5){\mbox{$(m-2)$}}
      \put(75.5,8){\line(1,0){5.5}}
	\put(82,8){\circle{2}}
\put(82,9){\line(0,1){5.5}}\put(82,15.5){\circle{2}}
      \put(83,8){\line(1,0){5.5}}
     \put(89.5,8){\circle{2}}

\put(94,7){$(x\ne 0)$}

\end{picture}}

\newcommand{\DynkinSFfour}{
\setlength{\unitlength}{1mm}
\begin{picture}(40,20)(5,-5)
	\put(0,7){$\mbox{$F(4)$}$}
	\put(30,8){$\graydot$}
\put(31,8){\line(1,0){5.5}}
\put(37.5,8){\circle{2}}
	\put(38.35,7){$\Longleftarrow$}
      \put(45.25,8){\circle{2}}
\put(46.25,8){\line(1,0){5.5}}
\put(52.75,8){\circle{2}}

\end{picture}}

\newcommand{\DynkinSGthree}{
\setlength{\unitlength}{1mm}
\begin{picture}(40,20)(5,-5)
	\put(0,7){$\mbox{$G(3)$}$}
	\put(30,8){$\graydot$}
\put(31,8){\line(1,0){5.5}}
\put(37.5,8){\circle{2}}
\put(38.5,8){\rotatebox{45}{\line(1,0){3}}}
\put(38.5,6){\rotatebox{45}{\line(0,1){3}}}

\put(39.5,9){\line(1,0){5.7}}
\put(38.5,8){\line(1,0){6}}
\put(39.5,7){\line(1,0){5.7}}
\put(45.5,8){\circle{2}}

\end{picture}}

\newcommand{\DynkinDtwooneal}{
\setlength{\unitlength}{1mm}
\begin{picture}(40,20)(5,-5)
	\put(0,7){$\mbox{$D(2,1;x)$}$}
	\put(30,8){\circle{2}}\put(33,9.5){$y$}
\put(31,8){\line(1,0){5.5}}
\put(37.5,8){$\graydot$}\put(39.7,9.5){$xy$}
\put(38.5,8){\line(1,0){5.5}}
\put(45,8){\circle{2}}

\put(50,7){$(x\notin\{0,-1\},\,y\ne 0)$}

\end{picture}}

\newcommand{\quantumDynkinDtwooneala}{
\setlength{\unitlength}{1mm}
\begin{picture}(110,30)(5,-5)
	\put(-5,7){$a:$}
\put(25,9){\oval(30,15)[t]}\put(24,18){${\hat z}$}
\put(10,8){\circle{2}}\put(17,9.5){${\hat x}$}
\put(11,8){\line(1,0){13}}
\put(25,8){\circle{2}}\put(32,9.5){${\hat y}$}
\put(26,8){\line(1,0){13}}
\put(40,8){\circle{2}}
\put(9,3){${\hat{\al}}^a_2$}\put(4,10){$-1$}
\put(24,3){${\hat{\al}}^a_1$}\put(23,10){$-1$}
\put(39,3){${\hat{\al}}^a_3$}\put(41,10){$-1$}
\end{picture}}

\newcommand{\quantumDynkinDtwoonealb}{
\setlength{\unitlength}{1mm}
\begin{picture}(110,30)(5,-5)
	\put(-5,7){$b:$}
\put(10,8){\circle{2}}\put(16,9.5){${\hat z}^{-1}$}
\put(11,8){\line(1,0){13}}
\put(25,8){\circle{2}}\put(31,9.5){${\hat y}^{-1}$}
\put(26,8){\line(1,0){13}}
\put(40,8){\circle{2}}
\put(9,3){${\hat{\al}}^b_3$}\put(9,11){${\hat z}$}
\put(24,3){${\hat{\al}}^b_2$}\put(23,11){$-1$}
\put(39,3){${\hat{\al}}^b_1$}\put(39,11){${\hat y}$}
\end{picture}}

\newcommand{\quantumDynkinDtwoonealc}{
\setlength{\unitlength}{1mm}
\begin{picture}(110,30)(5,-5)
	\put(-5,7){$c:$}
\put(10,8){\circle{2}}\put(16,9.5){${\hat x}^{-1}$}
\put(11,8){\line(1,0){13}}
\put(25,8){\circle{2}}\put(31,9.5){${\hat y}^{-1}$}
\put(26,8){\line(1,0){13}}
\put(40,8){\circle{2}}
\put(9,3){${\hat{\al}}^c_2$}\put(9,11){${\hat x}$}
\put(24,3){${\hat{\al}}^c_1$}\put(23,11){$-1$}
\put(39,3){${\hat{\al}}^c_3$}\put(39,11){${\hat y}$}
\end{picture}}

\newcommand{\quantumDynkinDtwooneald}{
\setlength{\unitlength}{1mm}
\begin{picture}(110,30)(5,-5)
	\put(-5,7){$d:$}
\put(10,8){\circle{2}}\put(16,9.5){${\hat y}^{-1}$}
\put(11,8){\line(1,0){13}}
\put(25,8){\circle{2}}\put(31,9.5){${\hat z}^{-1}$}
\put(26,8){\line(1,0){13}}
\put(40,8){\circle{2}}
\put(9,3){${\hat{\al}}^d_1$}\put(9,11){${\hat y}$}
\put(24,3){${\hat{\al}}^d_3$}\put(23,11){$-1$}
\put(39,3){${\hat{\al}}^d_2$}\put(39,11){${\hat z}$}
\end{picture}}

\newcommand{\quantumDynkinDtwooneal}{
\setlength{\unitlength}{1mm}
\begin{picture}(110,55)(0,-5)
\put(0,20){$\quantumDynkinDtwooneala$}
\put(70,20){$\quantumDynkinDtwoonealb$}

\put(0,00){$\quantumDynkinDtwoonealc$}
\put(70,00){$\quantumDynkinDtwooneald$}

\end{picture}}

\newcommand{\whitecircleNortha}{
\setlength{\unitlength}{1mm}
\begin{picture}(10,10)(2.74,0)
\put(0,0){\circle{2}}
\put(-0.9,1.5){$a$}
\end{picture}
}

\newcommand{\whitecircleSoutha}{
\setlength{\unitlength}{1mm}
\begin{picture}(10,10)(2.74,0)
\put(0,0){\circle{2}}
\put(-0.9,-3.4){$a$}
\end{picture}
}

\newcommand{\whitecircleEastb}{
\setlength{\unitlength}{1mm}
\begin{picture}(10,10)(2.74,0)
\put(0,0){\circle{2}}
\put(1.5,-0.7){$b$}
\end{picture}
}

\newcommand{\whitecircleWestb}{
\setlength{\unitlength}{1mm}
\begin{picture}(10,10)(2.74,0)
\put(0,0){\circle{2}}
\put(-3.8,-0.7){$b$}
\end{picture}
}

\newcommand{\whitecircleNorthb}{
\setlength{\unitlength}{1mm}
\begin{picture}(10,10)(2.74,0)
\put(0,0){\circle{2}}
\put(-0.9,1.5){$b$}
\end{picture}
}

\newcommand{\whitecircleSouthb}{
\setlength{\unitlength}{1mm}
\begin{picture}(10,10)(2.74,0)
\put(0,0){\circle{2}}
\put(-0.9,-4.8){$b$}
\end{picture}
}

\newcommand{\whitecircleNorthc}{
\setlength{\unitlength}{1mm}
\begin{picture}(10,10)(2.74,0)
\put(0,0){\circle{2}}
\put(-0.9,1.5){$c$}
\end{picture}
}

\newcommand{\whitecircleSouthc}{
\setlength{\unitlength}{1mm}
\begin{picture}(10,10)(2.74,0)
\put(0,0){\circle{2}}
\put(-0.9,-3.4){$c$}
\end{picture}
}

\newcommand{\whitecircleNorthd}{
\setlength{\unitlength}{1mm}
\begin{picture}(10,10)(2.74,0)
\put(0,0){\circle{2}}
\put(-0.9,1.5){$d$}
\end{picture}
}

\newcommand{\whitecircleSouthd}{
\setlength{\unitlength}{1mm}
\begin{picture}(10,10)(2.74,0)
\put(0,0){\circle{2}}
\put(-0.9,-4.8){$d$}
\end{picture}
}

\newcommand{\holizontallineAboveoneTHICK}{
\setlength{\unitlength}{1mm}
\begin{picture}(10,10)(2.74,0)
\put(1,0){\line(1,0){8}}\multiput(1,0)(0,0.05){4}{\line(1,0){8}}\multiput(1,0)(0,-0.05){4}{\line(1,0){8}}
\put(4.5,0.5){{\tiny{1}}}
\end{picture}
}

\newcommand{\EastUplineaboveoneTHICK}{
\setlength{\unitlength}{1mm}
\begin{picture}(10,10)(2.74,0)
\put(0.8,0.8){\line(1,1){8.4}}\multiput(0.8,0.8)(0.03,-0.03){4}{\line(1,1){8.4}}\multiput(0.8,0.8)(-0.03,0.03){4}{\line(1,1){8.4}}
\put(4,5.5){{\tiny{1}}}
\end{picture}
}

\newcommand{\EastDownlineaboveoneTHICK}{
\setlength{\unitlength}{1mm}
\begin{picture}(10,10)(2.74,0)
\put(0.8,-0.8){\line(1,-1){8.4}}\multiput(0.8,-0.8)(0.03,0.03){4}{\line(1,-1){8.4}}\multiput(0.8,-0.8)(-0.03,-0.03){4}{\line(1,-1){8.4}}
\put(5.5,-5){{\tiny{1}}}
\end{picture}
}

\newcommand{\holizontallineAbovetwo}{
\setlength{\unitlength}{1mm}
\begin{picture}(10,10)(2.74,0)
\put(1,0){\line(1,0){8}}
\put(4.5,0.5){{\tiny{2}}}
\end{picture}
}

\newcommand{\holizontallineAbovetwoTHICK}{
\setlength{\unitlength}{1mm}
\begin{picture}(10,10)(2.74,0)
\put(1,0){\line(1,0){8}}\multiput(1,0)(0,0.05){4}{\line(1,0){8}}\multiput(1,0)(0,-0.05){4}{\line(1,0){8}}
\put(4.5,0.5){{\tiny{2}}}
\end{picture}
}

\newcommand{\verticallineRighttwo}{
\setlength{\unitlength}{1mm}
\begin{picture}(10,10)(2.74,0)
\put(0,1){\line(0,1){8}}
\put(0.5,4.5){{\tiny{2}}}
\end{picture}
}

\newcommand{\verticallineRighttwoTHICK}{
\setlength{\unitlength}{1mm}
\begin{picture}(10,10)(2.74,0)
\put(0,1){\line(0,1){8}}\multiput(0,1)(0.05,0){4}{\line(0,1){8}}\multiput(0,1)(-0.05,0){4}{\line(0,1){8}}
\put(0.5,4.5){{\tiny{2}}}
\end{picture}
}

\newcommand{\EastUplineabovetwo}{
\setlength{\unitlength}{1mm}
\begin{picture}(10,10)(2.74,0)
\put(0.8,0.8){\line(1,1){8.4}}
\put(4,5.5){{\tiny{2}}}
\end{picture}
}

\newcommand{\EastUplineabovetwoTHICK}{
\setlength{\unitlength}{1mm}
\begin{picture}(10,10)(2.74,0)
\put(0.8,0.8){\line(1,1){8.4}}\multiput(0.8,0.8)(0.03,-0.03){4}{\line(1,1){8.4}}\multiput(0.8,0.8)(-0.03,0.03){4}{\line(1,1){8.4}}
\put(4,5.5){{\tiny{2}}}
\end{picture}
}

\newcommand{\EastDownlineabovetwo}{
\setlength{\unitlength}{1mm}
\begin{picture}(10,10)(2.74,0)
\put(0.8,-0.8){\line(1,-1){8.4}}
\put(5.5,-5){{\tiny{2}}}
\end{picture}
}

\newcommand{\EastDownlineabovetwoTHICK}{
\setlength{\unitlength}{1mm}
\begin{picture}(10,10)(2.74,0)
\put(0.8,-0.8){\line(1,-1){8.4}}\multiput(0.8,-0.8)(0.03,0.03){4}{\line(1,-1){8.4}}\multiput(0.8,-0.8)(-0.03,-0.03){4}{\line(1,-1){8.4}}
\put(5.5,-5){{\tiny{2}}}
\end{picture}
}

\newcommand{\verticallineRightthree}{
\setlength{\unitlength}{1mm}
\begin{picture}(10,10)(2.74,0)
\put(0,1){\line(0,1){8}}
\put(0.5,4.5){{\tiny{3}}}
\end{picture}
}

\newcommand{\verticallineRightthreeTHICK}{
\setlength{\unitlength}{1mm}
\begin{picture}(10,10)(2.74,0)
\put(0,9){\line(0,-1){8}}\multiput(0,9)(0.05,0){4}{\line(0,-1){8}}\multiput(0,9)(-0.05,0){4}{\line(0,-1){8}}
\put(0.75,4.5){{\tiny{3}}}
\end{picture}
}

\newcommand{\EastUplineabovethree}{
\setlength{\unitlength}{1mm}
\begin{picture}(10,10)(2.74,0)
\put(0.8,0.8){\line(1,1){8.4}}
\put(4,5.5){{\tiny{3}}}
\end{picture}
}

\newcommand{\EastUplineabovethreeTHICK}{
\setlength{\unitlength}{1mm}
\begin{picture}(10,10)(2.74,0)
\put(0.8,0.8){\line(1,1){8.4}}\multiput(0.8,0.8)(0.03,-0.03){4}{\line(1,1){8.4}}\multiput(0.8,0.8)(-0.03,0.03){4}{\line(1,1){8.4}}
\put(4,5.5){{\tiny{3}}}
\end{picture}
}

\newcommand{\EastDownlineabovethree}{
\setlength{\unitlength}{1mm}
\begin{picture}(10,10)(2.74,0)
\put(0.8,-0.8){\line(1,-1){8.4}}
\put(5.5,-5){{\tiny{3}}}
\end{picture}
}

\newcommand{\EastDownlineabovethreeTHICK}{
\setlength{\unitlength}{1mm}
\begin{picture}(10,10)(2.74,0)
\put(0.8,-0.8){\line(1,-1){8.4}}\multiput(0.8,-0.8)(0.03,0.03){4}{\line(1,-1){8.4}}\multiput(0.8,-0.8)(-0.03,-0.03){4}{\line(1,-1){8.4}}
\put(5.5,-5){{\tiny{3}}}
\end{picture}
}

\newcommand{\curveNorthWestTHICK}{
\setlength{\unitlength}{1mm}
\begin{picture}(10,10)(2.74,0)
\put(0,0){\oval(20.20,20.40)[tl]}
\put(0,0){\oval(20.15,20.30)[tl]}
\put(0,0){\oval(20.10,20.20)[tl]}
\put(0,0){\oval(20.05,20.10)[tl]}
\put(0,0){\oval(20,20)[tl]}
\put(0,0){\oval(19.90,19.90)[tl]}
\put(0,0){\oval(19.80,19.80)[tl]}
\put(0,0){\oval(19.70,19.70)[tl]}
\put(0,0){\oval(19.70,19.60)[tl]}
\end{picture}
}

\newcommand{\curveNorthEastTHICK}{
\setlength{\unitlength}{1mm}
\begin{picture}(10,10)(2.74,0)
\put(0,0){\oval(20.40,20.40)[tr]}
\put(0,0){\oval(20.30,20.30)[tr]}
\put(0,0){\oval(20.20,20.20)[tr]}
\put(0,0){\oval(20.10,20.10)[tr]}
\put(0,0){\oval(20,20)[tr]}
\put(0,0){\oval(19.90,19.90)[tr]}
\put(0,0){\oval(19.80,19.80)[tr]}
\put(0,0){\oval(19.70,19.70)[tr]}
\put(0,0){\oval(19.70,19.60)[tr]}
\end{picture}
}

\newcommand{\curveSouthWestTHICK}{
\setlength{\unitlength}{1mm}
\begin{picture}(10,10)(2.74,0)

\put(0,0){\oval(20.20,20.40)[bl]}
\put(0,0){\oval(20.15,20.30)[bl]}
\put(0,0){\oval(20.10,20.20)[bl]}
\put(0,0){\oval(20.05,20.10)[bl]}
\put(0,0){\oval(20,20)[bl]}
\put(0,0){\oval(19.90,19.90)[bl]}
\put(0,0){\oval(19.80,19.80)[bl]}
\put(0,0){\oval(19.70,19.70)[bl]}
\put(0,0){\oval(19.70,19.60)[bl]}
\end{picture}
}

\newcommand{\curveSouthEastTHICK}{
\setlength{\unitlength}{1mm}
\begin{picture}(10,10)(2.74,0)
\put(0,0){\oval(20.40,20.40)[br]}
\put(0,0){\oval(20.30,20.30)[br]}
\put(0,0){\oval(20.20,20.20)[br]}
\put(0,0){\oval(20.10,20.10)[br]}
\put(0,0){\oval(20,20)[br]}
\put(0,0){\oval(19.90,19.90)[br]}
\put(0,0){\oval(19.80,19.80)[br]}
\put(0,0){\oval(19.70,19.70)[br]}
\put(0,0){\oval(19.70,19.60)[br]}
\end{picture}
}

\newcommand{\CayleyNine}{
\setlength{\unitlength}{1mm}
\begin{picture}(60,100)(25,-5)

\put(34,64){\rotatebox{-45}{$\rightarrow$}}

\put(50,10){\whitecircleSouthc}
\put(50,10){\holizontallineAbovetwoTHICK}
\put(60,10){\whitecircleSouthc}
\put(50,10){\verticallineRightthree}
\put(60,10){\verticallineRightthree}

\put(50,20){\whitecircleNorthc}
\put(50,20){\holizontallineAbovetwoTHICK}
\put(60,20){\whitecircleNorthc}
\put(60,20){\EastUplineaboveoneTHICK}

\put(10,30){\whitecircleNortha}
\put(10,30){\EastUplineabovethree}
\put(40,30){\whitecircleSoutha}
\put(40,30){\EastUplineabovetwo}
\put(40,30){\EastDownlineaboveoneTHICK}
\put(70,30){\whitecircleSoutha}
\put(70,30){\EastUplineabovethree}
\put(100,30){\whitecircleNortha}
\put(100,30){\EastUplineabovetwo}

\put(0,40){\whitecircleWestb}
\put(0,40){\verticallineRightthree}
\put(0,40){\EastDownlineabovetwoTHICK}
\put(20,40){\whitecircleSouthd}
\put(20,40){\verticallineRighttwoTHICK}
\put(20,40){\holizontallineAboveoneTHICK}
\put(30,40){\whitecircleSouthd}
\put(30,40){\verticallineRighttwo}
\put(30,40){\EastDownlineabovethreeTHICK}
\put(50,40){\whitecircleSouthb}
\put(50,40){\verticallineRightthreeTHICK}
\put(50,40){\holizontallineAboveoneTHICK}
\put(60,40){\whitecircleSouthb}
\put(60,40){\verticallineRightthree}
\put(60,40){\EastDownlineabovetwoTHICK}
\put(80,40){\whitecircleSouthd}
\put(80,40){\verticallineRighttwoTHICK}
\put(80,40){\holizontallineAboveoneTHICK}
\put(90,40){\whitecircleSouthd}
\put(90,40){\verticallineRighttwo}
\put(90,40){\EastDownlineabovethreeTHICK}
\put(110,40){\whitecircleEastb}
\put(110,40){\verticallineRightthreeTHICK}

\put(0,50){\whitecircleWestb}
\put(0,50){\EastUplineabovetwoTHICK}
\put(20,50){\whitecircleNorthd}
\put(20,50){\holizontallineAboveoneTHICK}
\put(30,50){\whitecircleNorthd}
\put(30,50){\EastUplineabovethreeTHICK}
\put(50,50){\whitecircleNorthb}
\put(50,50){\holizontallineAboveoneTHICK}
\put(60,50){\whitecircleNorthb}
\put(60,50){\EastUplineabovetwoTHICK}
\put(80,50){\whitecircleNorthd}
\put(80,50){\holizontallineAboveoneTHICK}
\put(90,50){\whitecircleNorthd}
\put(90,50){\EastUplineabovethreeTHICK}
\put(110,50){\whitecircleEastb}

\put(10,60){\whitecircleSoutha}
\put(10,60){\EastDownlineabovethree}
\put(40,60){\whitecircleNortha}
\put(40,60){\EastUplineaboveoneTHICK}
\put(40,60){\EastDownlineabovetwo}
\put(70,60){\whitecircleSoutha}
\put(70,60){\EastDownlineabovethree}
\put(100,60){\whitecircleSoutha}
\put(100,60){\EastDownlineabovetwo}

\put(50,70){\whitecircleSouthc}
\put(50,70){\verticallineRightthreeTHICK}
\put(50,70){\holizontallineAbovetwo}
\put(60,70){\whitecircleSouthc}
\put(60,70){\verticallineRightthreeTHICK}
\put(60,70){\EastDownlineaboveoneTHICK}

\put(50,80){\whitecircleNorthc}
\put(50,80){\holizontallineAbovetwo}
\put(60,80){\whitecircleNorthc}

\put(9.40,0){\line(1,0){90}}\multiput(9.40,0)(0,0.05){4}{\line(1,0){90}}\multiput(9.40,0)(0,-0.05){4}{\line(1,0){90}}

\put(10,10){\curveSouthWestTHICK}
\put(-0.6,10){\line(0,1){29}}\multiput(-0.6,10)(0.05,0){4}{\line(0,1){29}}\multiput(-0.6,10)(-0.05,0){4}{\line(0,1){29}}
\put(100,10){\curveSouthEastTHICK}
\put(109.40,10){\line(0,1){29}}\multiput(109.40,10)(0.05,0){4}{\line(0,1){29}}\multiput(109.40,10)(-0.05,0){4}{\line(0,1){29}}
\put(19.40,10){\line(1,0){29}}\multiput(19.40,10)(0,0.05){4}{\line(1,0){29}}\multiput(19.40,10)(0,-0.05){4}{\line(1,0){29}}
\put(60.40,10){\line(1,0){29}}\multiput(60.40,10)(0,0.05){4}{\line(1,0){29}}\multiput(60.40,10)(0,-0.05){4}{\line(1,0){29}}

\put(20,20){\curveSouthWestTHICK}
\put(9.40,20){\line(0,1){9}}\multiput(9.40,20)(0.05,0){4}{\line(0,1){9}}\multiput(9.40,20)(-0.05,0){4}{\line(0,1){9}}
\put(90,20){\curveSouthEastTHICK}
\put(99.40,20){\line(0,1){9}}\multiput(99.40,20)(0.05,0){4}{\line(0,1){9}}\multiput(99.40,20)(-0.05,0){4}{\line(0,1){9}}

\put(-0.60,80){\line(0,-1){29}}\multiput(-0.60,80)(0.05,0){4}{\line(0,-1){29}}\multiput(-0.60,80)(-0.05,0){4}{\line(0,-1){29}}
\put(109.40,51){\line(0,1){29}}\multiput(109.40,51)(0.05,0){4}{\line(0,1){29}}\multiput(109.40,51)(-0.05,0){4}{\line(0,1){29}}

\put(9.40,61){\line(0,1){9}}\multiput(9.40,61)(0.05,0){4}{\line(0,1){9}}\multiput(9.40,61)(-0.05,0){4}{\line(0,1){9}}
\put(99.40,70){\line(0,-1){9}}\multiput(99.40,70)(0.05,0){4}{\line(0,-1){9}}\multiput(99.40,70)(-0.05,0){4}{\line(0,-1){9}}

\put(20,70){\curveNorthWestTHICK}
\put(90,70){\curveNorthEastTHICK}

\put(10,80){\curveNorthWestTHICK}
\put(19.40,80){\line(1,0){29}}\multiput(19.40,80)(0,0.05){4}{\line(1,0){29}}\multiput(19.40,80)(0,-0.05){4}{\line(1,0){29}}
\put(60.40,80){\line(1,0){29}}\multiput(60.40,80)(0,0,05){4}{\line(1,0){29}}\multiput(60.40,80)(0,-0,05){4}{\line(1,0){29}}
\put(100,80){\curveNorthEastTHICK}

\put(9.40,90){\line(1,0){90}}\multiput(9.40,90)(0,0,05){4}{\line(1,0){90}}\multiput(9.40,90)(0,-0,05){4}{\line(1,0){90}}

\put(1,70){{\tiny{$1$}}}
\put(11,70){{\tiny{$1$}}}
\put(101,70){{\tiny{$1$}}}

\put(1,15){{\tiny{$1$}}}
\put(11,15){{\tiny{$1$}}}
\put(101,15){{\tiny{$1$}}}

\end{picture}
}

\newcommand{\CayleyAltFourDash}{
\setlength{\unitlength}{1mm}
\begin{picture}(80,120)(-3,-20)

\put(39,82.5){$a_1$}
\put(-1,62.5){$a_{12}$}\put(17.5,62.5){$a_2$}\put(34.5,59){$a_5$}\put(62,59){$a_6$}

\put(22,39){$a_{10}$}\put(42,39){$a_9$}\put(62,40){$a_{11}$}
\put(19,16){$a_7$}\put(39,16){$a_4$}\put(59,16){$a_3$}
\put(82,19){$a_8$}

\put(40,80){\circle{2}}
\put(00,60){\circle{2}}\put(20,60){\circle{2}}\put(40,60){\circle{2}}\put(60,60){\circle{2}}
\put(20,40){\circle{2}}\put(40,40){\circle{2}}\put(60,40){\circle{2}}
\put(20,20){\circle{2}}\put(40,20){\circle{2}}\put(60,20){\circle{2}}\put(80,20){\circle{2}}

\multiput(40.7,79.3)(0.05,0.05){5}{\line(1,-1){18.6}}\multiput(40.7,79.3)(-0.05,-0.05){5}{\line(1,-1){18.6}}

\multiput(40,79)(0.05,0.00){7}{\line(0,-1){18}}\multiput(40,79)(-0.05,0.00){7}{\line(0,-1){18}}

\put(39.3,79.3){\line(-1,-1){18.6}}

\multiput(01,60)(0.00,0.05){7}{\line(1,0){18}}\multiput(01,60)(0.00,-0.05){7}{\line(1,0){18}}

\put(41,60){\line(1,0){18}}

\put(00.7,59.3){\line(1,-1){18.6}}

\multiput(20,59)(0.05,0.00){7}{\line(0,-1){18}}\multiput(20,59)(-0.05,0.00){7}{\line(0,-1){18}}
\multiput(40,59)(0.05,0.00){7}{\line(0,-1){18}}\multiput(40,59)(-0.05,0.00){7}{\line(0,-1){18}}
\multiput(60,59)(0.05,0.00){7}{\line(0,-1){18}}\multiput(60,59)(-0.05,0.00){7}{\line(0,-1){18}}

\multiput(20,39)(0.05,0.00){7}{\line(0,-1){18}}\multiput(20,39)(-0.05,0.00){7}{\line(0,-1){18}}
\multiput(40,39)(0.05,0.00){7}{\line(0,-1){18}}\multiput(40,39)(-0.05,0.00){7}{\line(0,-1){18}}
\multiput(60,39)(0.05,0.00){7}{\line(0,-1){18}}\multiput(60,39)(-0.05,0.00){7}{\line(0,-1){18}}

\put(39.3,39.3){\line(-1,-1){18.6}}\put(60.7,39.3){\line(1,-1){18.6}}

\multiput(21,20)(0.00,0.05){7}{\line(1,0){18}}\multiput(21,20)(0.00,-0.05){7}{\line(1,0){18}}
\multiput(61,20)(0.00,0.05){7}{\line(1,0){18}}\multiput(61,20)(0.00,-0.05){7}{\line(1,0){18}}
\put(41,20){\line(1,0){18}}

\multiput(20,00)(0.00,0.05){7}{\line(1,0){40}}\multiput(20,00)(0.00,-0.05){7}{\line(1,0){40}}

\multiput(00,59)(0.05,0.00){7}{\line(0,-1){39}}\multiput(00,59)(-0.05,0.00){7}{\line(0,-1){39}}

\put(60,19){\oval(39.4,37.4)[br]}
\put(60,19){\oval(39.5,37.5)[br]}
\put(60,19){\oval(39.6,37.6)[br]}
\put(60,19){\oval(39.7,37.7)[br]}
\put(60,19){\oval(39.8,37.8)[br]}
\put(60,19){\oval(39.9,37.9)[br]}
\put(60,19){\oval(40,38)[br]}
\put(60,19){\oval(40.1,38.1)[br]}
\put(60,19){\oval(40.2,38.2)[br]}
\put(60,19){\oval(40.3,38.3)[br]}
\put(60,19){\oval(40.4,38.4)[br]}
\put(60,19){\oval(40.5,38.5)[br]}
\put(60,19){\oval(40.6,38.6)[br]}

\put(20,20){\oval(39.4,39.4)[bl]}
\put(20,20){\oval(39.5,39.5)[bl]}
\put(20,20){\oval(39.6,39.6)[bl]}
\put(20,20){\oval(39.7,39.7)[bl]}
\put(20,20){\oval(39.8,39.8)[bl]}
\put(20,20){\oval(39.9,39.9)[bl]}
\put(20,20){\oval(40,40)[bl]}
\put(20,20){\oval(40.1,40.1)[bl]}
\put(20,20){\oval(40.2,40.2)[bl]}
\put(20,20){\oval(40.3,40.3)[bl]}
\put(20,20){\oval(40.4,40.4)[bl]}
\put(20,20){\oval(40.5,40.5)[bl]}
\put(20,20){\oval(40.6,40.6)[bl]}

\put(-10,-15){{\rm{Fig.~1. Cayley graph of $\Gamma({\mathrm{Alt}}(n))$ and its Hamiltonian cycle}}}

\end{picture}
}

\newcommand{\suai}[2]{s^{{\overline{#2}}}_{#1}}
\newcommand{\bsuai}[2]{{\mathbf{s^{{\overline{#2}}}_{#1}}}}

\begin{document}

\begingroup
\renewcommand{\arraystretch}{1.1}
\begin{center}
{\Huge{
\begin{tabular}{c}
Hamiltonian Cycles \\ for Finite Weyl Groupoids 
\end{tabular}}}
\end{center}
\endgroup

\begin{center}
{\large{Takato Inoue,\quad Hiroyuki Yamane}}
\end{center}

\vspace{1cm}

\begin{abstract} 
Let $\Gamma({\mathcal{W}})$ be the Cayley graph of a finite Weyl groupoid ${\mathcal{W}}$.
In this paper, we show an existence of a Hamitonian cycle of $\Gamma({\mathcal{W}})$ for any ${\mathcal{W}}$.
We exatctly draw a Hamiltonian cycle of $\Gamma({\mathcal{W}})$
for any (resp. some) irreducible ${\mathcal{W}}$ of 
rank three (resp. four). 
Moreover for the irreducible ${\mathcal{W}}$ of rank three, we give a second largest eigenvalue
of the adjacency matrix of $\Gamma({\mathcal{W}})$,   
and know if $\Gamma({\mathcal{W}})$ is a bipartite Ramanujan graph or not.
\end{abstract}

\section{Introduction} 
This paper is a continuation of \cite{Y22}.
See \cite{PR06}, \cite[Introduction]{Y22} for motivation of studying Hamiltonian cycles.
The Weyl groupoids appeared in study of Lie superalgebras 
\cite{LSS85}, \cite{S96},
and Nichols algebras \cite{Hec09}, \cite{HS20}.
The axiomatic definition of the Weyl groupoids
and a study of some of their basic properties 
such as Coxeter-like presentation and Matsumoto-lile theorem 
were achieved in \cite{HY08}.
The Weyl groupoids (see Definition~\ref{definition:CayGofW}~(1)) are defined for the generalized root systems
(see Definitions~\ref{definition:DefOfFGRS} and \ref{definition:DefCFGRS}~(3)).
Finite generalized root systems
were classified by \cite{CH15}
for study of the crystallographic hyperplane arrangements \cite{Cu11},
see also Theorem~\ref{theorem:EqvHpyAndFGRS}.
In this paper, we show in Theorem~\ref{theorem:Main} that
\begin{equation}\label{eqn:MainInt}
\begin{array}{l}
\mbox{there exists
a Hamitonian cycle of
the Cayley graph} \\
\mbox{of the finite Weyl groupoid defined for every finite generalized root system.}
\end{array}
\end{equation}
For all rank three cases and some four cases of \eqref{eqn:MainInt}, 
we exactly draw a Hamitonian cycle 
by using a computer program of \cite{Mathe23},
see Subsections~\ref{subsection:rankthree}
and \ref{subsection:rankfour}.

The Nichols algebras of diagonal-type with the finite generalized root systems
(NADFs for short)
were classified by \cite{Hec09} using finite Weyl groupoids. 
The same claim as \eqref{eqn:MainInt} 
for a generalized root system associated with an NADF
was proved by \cite[Theorem~6.3]{Y22}.
However the celebrated result of \cite{CH15} tells that there are so many 
finite generalized root systems being not associated with any NADF.

For $x$, $y\in\bR\cup\{\pm\infty\}$,
let $\fkJ_{x,y}:=\{z\in\bZ|x\leq z\leq y\}$.
Then $\bN=\fkJ_{1,\infty}$ and $\bZgeqo=\fkJ_{0,\infty}$.
For a set $X$, let $|X|$ (resp. ${\mathfrak{P}}(X)$) mean the cardinal number of $X$,
(resp. the power set of $X$). 
That is, ${\mathfrak{P}}(X)$ is the set of all the subsets of $X$.
For a set $X$ and $n\in\bZgeqo$,
let ${\mathfrak{P}}_n(X):=\{Y\in {\mathfrak{P}}(X)||Y|=n\}$.

\begin{definition}\label{definition:HamMap}
{\rm{(1)}} Let $V$ be a non-empty set.
Let $\Gamma(V,E)$ be the pair $(V,E)$ of $V$ and a subset $E$ of ${\mathfrak{P}}_2(V)$. 
We call such $\Gamma(V,E)$ a {\it{graph}}.
We call $V$ (resp. $E$) the  {\it{set of vertices}} (resp. {\it{edges}}) of 
$\Gamma(V,E)$.
If $|V|<\infty$, we say that $\Gamma(V,E)$ is a {\it{finite graph}}.
\newline
{\rm{(2)}} 
Let $\Gamma:=\Gamma(V,E)$ be a finite graph,
and let $n:=|V|$.
We say that $\Gamma$ is {\it{disconnected}}
if there exist $m\in\fkJ_{1,n-1}$ and a bijection
$f:\fkJ_{1,n}\to V$ such that $E\subset\{\{f(u_1),f(u_2)\}|u_1,u_2\in\fkJ_{1,m}\}
\cup\{\{f(v_1),f(v_2)\}|v_1,v_2\in\fkJ_{m+1,n}\}$.
We say that $\Gamma$ is {\it{connected}} if $\Gamma$ is not disconnected.
Notice that if $|V|=1$, then $\Gamma$ is connected.
\newline
{\rm{(3)}} 
Let $\Gamma(V,E)$ be a finite connected graph.
Let $n:=|V|\,(\in\bN)$. We say that a map $\varphi:\fkJ_{1,n+1}\to V$ 
is a {\it{Hamiltonian cycle}} if $\varphi(\fkJ_{1,n})=V$, $\varphi(n+1)=\varphi(1)$ 
and $\{\varphi(i),\varphi(i+1)\}\in E$
$(i\in\fkJ_{1,n})$.
\end{definition}

Although an efficient way for finding a Hamiltonian cycle of a finite graph
(if it exists) has not been known at the present,
the widely-known argument of the proof of the following theorem can be said to be one of
main ways for finding it for many Cayley graphs of finite groups.
See Definition~\ref{definition:defofCayley} for the definition of the Cayley graph of a finite group.

\begin{theorem}\label{theorem:one} {\rm{(Rapaport-Strasser~(1959)~\cite{SR59}, {\it{see also}} \cite[Lemma~1]{PR06})}} Let $G$ be a finite group
generated by three elements $a$, $b$, $c$ of $G\setminus\{e\}$ with
$a^2=b^2=c^2=abab=e$.
Let $Z:=\{\{x,xy\}|x\in G, y\in\{a,b,c\}\}$. 
Then there exists a Hamiltonian cycle for $\Gamma(G,Z)$.
\end{theorem} 
\noindent
The proof is given in Subsection~\ref{subsection:wellknownpf}.
\newline\par
Let $r\in\bN$ and let $\fkI:=\fkJ_{1,l}$. 
Let $m:\fkI\times \fkI\to\bN\cup\{\infty\}$ be a map satisfying the condition that
$m(i,i)=1$ $(i\in\fkI)$ and $m(j,k)=m(k,j)\geq 2$ $(j,k\in\fkI, j\ne k)$.
Let $W$ be the group presented by generators
$s_i$ $(i\in\fkI)$ and relations $(s_is_j)^{m(i,j)}=e$
$(i,j\in\fkI,m(i,j)\ne\infty)$, that is, $W$ is a Coxeter group. 
Let $S:=\{s_i|i\in\fkI\}$, that is, $(W,S)$ is a Coxeter system.
Let ${\mathcal{E}}_m:=\{\{w,ws\}|w\in W, s\in S\}$.
Note that $m(\fkI\times \fkI)\in\bN$ if $|W|<\infty$.
In a way similar to that for Theorem~\ref{theorem:one}, 
we can see the following, where see also \cite[Theorem~2.3]{Y22} for its proof.

\begin{theorem} \label{theorem:two}
{\rm{(J.H.~Conway, N.J.A.~Sloane, Allan~R.~Wilks~(1989) \cite{CSW89})}}
If $|W|<\infty$, then there exists a Hamiltonian cycle for $\Gamma(W,{\mathcal{E}}_m)$.
\end{theorem}
\noindent
A rough proof is given in Subsection~\ref{subsection:wellknownpf}. 
\newline\par
Let ${\mathfrak{g}}$ be a finite dimensional complex simple Lie superalgebra.
Assume that (a Dynkin diagram of) ${\mathfrak{g}}$
is $A(m,n)$, $B(m,n)$, $C(n)$, $D(m,n)$, $F(4)$, $G(3)$, or $D(2,1;x)$.
Such  ${\mathfrak{g}}$ is called a {\it{basic classical Lie superalgebra}}.
If it is not $B(0,n)$, it has several Dynkin diagrams, which are not transformed to one another
by an action of the Weyl group defined for ${\mathfrak{g}}$.
However different two of those are transformed to each other
by repetition of operations of the odd reflections of ${\mathfrak{g}}$ (\cite[Appendix~II]{LSS85}, see also \cite[Subsection~2.3]{FSS89},  \cite{S96}).
The mathematical object obtained by adding the odd reflections to the Weyl group is not a group but a groupoid, which 
is called a {\it{Weyl groupoid}}.
Similarly to the Coxeter groups, the Weyl groupoids are also presented by 
Coxeter relations \cite{HY08}.
Every Coxeter group has a root system, which is obtained by the Tits geometric representation.
On the other hand,
we have not yet known if
every Weyl groupoid has (an extended version of) a generalized root system \cite{HY08} or not.
At present we need a generalized root system at first in order to define a Weyl groupoid
(having appropriate properties).
If a generalized root system (resp. a Weyl groupoid) is composed of finite elements, it is called
a {\it{finite generalized root system}}
(resp. a {\it{finite Weyl groupoid}}).
It is known that a Weyl groupoid is finite if and only if its generalized root system is so. There are non-isomorphic two finite generalized root systems
for which the corresonding finite Weyl groupoids are isomorphic.
An example is Nr.~13 and Nr.~14 of Subsection~\ref{subsection:rankthree},
see \cite[Fig.~35]{Y22} and \cite[Table~2]{CH12}, \cite[Remark~A.1]{CH15}.
The finite generalized root systems have been classified by \cite{CH15}.
In this paper, we show that the Cayley graph  
of the Weyl groupoid for every finite generalized root system
has a Hamiltonian cycle, see Theorem~\ref{theorem:Main}.

This paper is organized as follows.
Section~\ref{section:Basic} collects basic facts.
Section~\ref{section:LieS} treats 
the finite generalized root systems of simple Lie algebras and superalgebras
of type $A$-$G$,
since for rank $\geq 9$,
those form all
the irreducible finite generalized root systems
(IFGRSs for short).
Section~\ref{section:HamCy}
gives an explicit form of a Hamiltonian cycle of 
the Cayley graph
for
every (resp. some) IFRGS
of rank three (resp. rank four),
and gives the main theorem Theorem~\ref{theorem:Main} of this paper.
In Appendix, we give approximate values of 
the second largest eigenvalue of the adjacency matrix
of  
the Cayley graph of
an IFRGS
of rank three,
collect basic facts of crystallographic hyperplane arrangements,
and treat an IFRGS
for the Lie superalgebra $D(2,1;x)$ as an example.

\section{Basic facts} \label{section:Basic}
\subsection{Cayley graphs of finite groups}

Although a Cayley graph of a finite group is not a main subject of this paper,
we begin with recalling the definition of it.
The facts written in this subsection is rather independent from those in the other parts of this paper.

\begin{definition} \label{definition:defofCayley}
Let $G$ be a finite group.
Let $e$ denote the unit of $G$.
Let $X$ be a subset of $G$
such that $e\notin X$, $\{x^{-1}|x\in X\}=X$,
and $G=\{x_1x_2\cdots x_k|k\in\bN, x_t\in X(t\in\fkJ_{1,k})\}\cup\{e\}$.
Let $E_{G,X}:=\{\{y,yx\}|y\in G, x\in X\}\,(\subset{\mathfrak{P}}_2(G))$.
We call the finite graph $\Gamma(G,E_{G,X})$
the {\it{Cayley graph}} of $G$ and $X$.
\end{definition}

For $n\in\bN$, 
let ${\mathrm{Alt}}(n)$ be a the alternating group of degree $n$.
Notice $|{\mathrm{Alt}}(n)|={\frac {n!} 2}$.
Assume $n\geq 3$.
The set $X_n:=\{x_1:=(123)$, $x_2:=(132)$,
$x_i=(12)(i,i+1)$ $(i\in\fkJ_{3,n-1})\}$ generates ${\mathrm{Alt}}(n)$
as a group, 
and the defining relations for $X_n$ are
\begin{equation*}
\begin{array}{l}
x_1^3=e, x_1^2=x_2, x_i^2=e \, (i\in\fkJ_{3,n-1}),
(x_1x_3)^3=e, 
(x_1x_i)^2=e \, (i\in\fkJ_{4,n-1}), \\
(x_ix_{i+1})^3=e \, (i\in\fkJ_{3,n-2}), 
(x_ix_j)^2=e \, (i,j\in\fkJ_{3,n-1}, i+2\leq j). 
\end{array}
\end{equation*}
Let $\Gamma({\mathrm{Alt}}(n)):=\Gamma({\mathrm{Alt}}(n),E_{{\mathrm{Alt}}(n),X_n})$.

Since ${\mathrm{Alt}}(3)$ is a cyclic group of order three,
$\Gamma({\mathrm{Alt}}(3))$ clearly has a Hamiltonian cycle.

Denote the elements of ${\mathrm{Alt}}(4)$ as follows: 
$a_1:=1, a_2:=(12)(34), a_3:=(13)(24)$, $a_4:=(14)(23)$, $a_5:=(123), 
a_6:=(132), a_7:=(124), 
a_8:=(142), a_9:=(134), a_{10}:=(143), a_{11}:=(234), a_{12}:=(243)$.
A Hamiltonian cycle
$f:\fkJ_{1,13}\to{\mathrm{Alt}}(4)$ of $\Gamma({\mathrm{Alt}}(4))$ is 
defined by
$f(1):=a_1$, $f(2):=a_6$, $f(3):=a_{11}$, $f(4):=a_3$,
$f(5):=a_8$, $f(6):=a_{12}$, $f(7):=a_2$, $f(8):=a_{10}$,
$f(9):=a_7$, $f(10):=a_4$, $f(11):=a_9$, $f(12):=a_5$,
$f(13)=a_1$.
We draw $\Gamma({\mathrm{Alt}}(4))$ in Fig.~1,
where $f$ is indicated by the thick lines. 


\begin{equation*}
\CayleyAltFourDash
\end{equation*}

Let $1\cdot x_{i_1}x_{i_2}\cdots x_{i_r}$ with
$r:={\frac {n!} 2}$ mean a Hamiltonian cycle $h$
of $\Gamma({\mathrm{Alt}}(n))$
defined by $h(1):=e$ and $h(t+1):=x_{i_1}x_{i_2}\cdots x_{i_t}$
$(t\in\fkJ_{2,r})$. 
Then the above $f$ is
$1\cdot x_2x_3x_1x_1x_3x_2x_2x_3x_1x_1x_3x_2$.

Let $n\in\fkJ_{5,\infty}$. 
Since $x_1x_{n-1}=x_{n-1}x_2$, $x_2x_{n-1}=x_{n-1}x_1$, $x_ix_{n-1}=x_{n-1}x_i$
$(i\in\fkJ_{3,n-3})$,
using a standard way for Theorems~\ref{theorem:one} and \ref{theorem:two}
and an induction on $n$,
we see that $\Gamma({\mathrm{Alt}}(n))$ has a Hamitonian cycle.
For example, a Hamiltonian cycle of $\Gamma({\mathrm{Alt}}(5))$ is:
\newline\newline
{\tiny{{\small{$1\cdot$}}
$x_4 x_3 x_1 x_4 x_1 x_3 x_2 x_2 x_3 x_1 
x_4 x_1 x_3 x_2 x_2 x_3 x_1 x_1 x_3 x_2 
x_4 x_3 x_1 x_1 x_3 x_2 x_2 x_3 x_1 x_1 
x_3 x_2 x_4 x_3 x_4 x_3 x_2 x_2 x_3 x_4 
x_3 x_2 x_2 x_3 x_1 x_1 x_3 x_2 x_4 x_3 
x_1 x_1 x_3 x_2 x_2 x_3 x_1 x_1 x_3 x_2$}}\,,
\newline\newline
and that of $\Gamma({\mathrm{Alt}}(6))$ is:
\newline
\begin{spacing}{0.4}
\noindent
{\tiny{{\small{$1\cdot$}} \newline
$x_4x_3x_5x_4x_1x_5x_4x_2x_3x_1x_1x_3x_2x_2x_3x_1x_1x_3x_4x_2x_3x_1x_1x_
3x_2x_2x_3x_4x_3x_5x_2x_3x_1x_4x_1x_5x_4x_2x_3x_1x_1x_3x_2x_2x_3x_1x_
1x_3x_4x_2x_3x_1x_1x_3x_2x_2x_3x_4x_3x_5$ 
\newline 
$\cdot x_2x_3x_1x_4x_1x_3x_2x_2x_3x_1x_1x_3x_2x_4x_3x_1x_1x_3x_2x_2x_3x_1x_1x_
3x_2x_4x_3x_4x_3x_2x_2x_3x_4x_3x_2x_2x_3x_1x_1x_3x_2x_4x_3x_1x_1x_3x_
2x_2x_3x_1x_1x_3x_2x_4x_3x_1x_4x_1x_3x_5$
\newline 
$\cdot x_2x_3x_4x_3x_4x_2x_3x_1x_1x_3x_2x_2x_3x_1x_1x_3x_4x_2x_3x_1x_1x_3x_2x_
2x_3x_1x_4x_1x_3x_2x_2x_3x_1x_4x_1x_5x_2x_2x_3x_1x_1x_3x_2x_4x_3x_1x_
1x_3x_2x_2x_3x_1x_1x_3x_2x_4x_3x_4x_3x_2$
\newline 
$\cdot x_2x_3x_4x_3x_2x_2x_3x_1x_1x_3x_2x_4x_3x_1x_1x_3x_2x_2x_3x_1x_1x_3x_2x_
4x_3x_1x_4x_1x_3x_5x_2x_3x_4x_3x_4x_2x_3x_1x_1x_3x_2x_2x_3x_1x_1x_3x_
4x_2x_3x_1x_1x_3x_2x_2x_3x_1x_4x_1x_3x_2$
\newline 
$\cdot x_2x_3x_1x_4x_1x_5x_2x_2x_3x_1x_4x_1x_3x_2x_2x_3x_1x_1x_3x_2x_4x_3x_1x_
1x_3x_2x_2x_3x_1x_1x_3x_2x_4x_3x_4x_3x_2x_2x_3x_4x_3x_2x_2x_3x_1x_1x_
3x_2x_4x_3x_1x_1x_3x_2x_2x_3x_1x_1x_3x_2$
\newline 
$\cdot x_4x_3x_5x_4x_1x_3x_2x_2x_3x_1x_4x_1x_3x_2x_2x_3x_1x_1x_3x_2x_4x_3x_1x_
1x_3x_2x_2x_3x_1x_1x_3x_2x_4x_3x_4x_3x_2x_2x_3x_4x_3x_2x_2x_3x_1x_1x_
3x_2x_4x_3x_1x_1x_3x_2x_2x_3x_1x_1x_3x_2$
}}\,.
\end{spacing}

\vspace{1cm}

\noindent
These can be found by using Mathematica~13.3 \cite{Mathe23}.

\subsection{Finite generalized root systems}
Let $\fca$ be a free $\bZ$-module of finite rank.
Assume $\fca\ne\{0\}$.
Let $\Pfca$ be the power set of $\fca$.
(the notation $\Pfca$ is compatible with
the one ${\mathfrak{P}}(X)$ of Introduction).
Let $\Bfca$ be the subset of $\Pfca$
formed by all the $\bZ$-bases of $\fca$.

\begin{definition}\label{definition:DefOfFGRS}
Let $\rR\in\Pfca$. 
Let $\rR^+_B:=\rR\cap\rmSpan_{\bZgeqo}B$
and $\rR^-_B:=\rR\cap\rmSpan_{\bZleqo}B$.
Let
\begin{equation*}
\bBR:=\{B\in\Bfca|B\subset\rR, \rR=\rR^+_B\cup \rR^-_B\}.
\end{equation*}
We say that $\rR$ is a {\it{finite generalized root system 
{\rm{(}}FGRS for short{\rm{)}}
over $\fca$}}
if  
the following axioms $(\rR 0)$-$(\rR 3)$ are fulfilled.
\newline\newline
$(\rR 0)$ \quad $\rR$ is a non-empty finite set. \newline
$(\rR 1)$ \quad $\bBR\ne\emptyset$. \newline
$(\rR 2)$ \quad $\forall \al\in\rR$,
$\rR\cap\bZ\al=\{\al,-\al\}$. \newline
$(\rR 3)$ \quad $\forall B\in\bBR$, $\forall  \al\in B$,
$\exists B^{(\al)}\in \bBR$,
$\rR^+_{B^{(\al)}}\cap\rR^-_B=\{-\al\}$.
\end{definition}

\begin{definition}
Let $\rR^\prime$ {\rm{(}}resp. $\rR^{\prime\prime}${\rm{)}}
be an FGRS over $\fca^\prime$
{\rm{(}}resp. $\fca^{\prime\prime}${\rm{)}}.
We say that $\rR^\prime$
is {\it{isomorphic}} to
$\rR^{\prime\prime}$
if there exists a 
$\bZ$-module isomorphism 
$f:\fca^\prime\to\fca^{\prime\prime}$
such that $f(\rR^\prime)=\rR^{\prime\prime}$.
\end{definition}

Let $\rR$ be an FGRS over $\fca$.
Then there exist $N^{\rR,B}_{\al,\beta}\in\bZ$
$(B\in\bBR,\al,\beta\in B)$ such that 
\begin{equation*}
N^{\rR,B}_{\al,\al}=-2,\,\,N^{\rR,B}_{\al,\beta}\in\bZgeqo\,(\al\ne\beta)
\,\,\mbox{and}\,\,B^{(\al)}=\{\,\beta+N^{\rR,B}_{\al,\beta}\al\,|\,\beta\in B\,\}.
\end{equation*}
We have $(B^{(\al)})^{(-\al)}=B$.

\begin{lemma}\label{lemma:bBprime}
{\rm{({\it{see}} \cite{Y16})}}
{\rm{(1)}}
Let $B,B^\prime\in \bBR$.
Then{\rm{:}}
\begin{equation*}
\begin{array}{lcl}
|\rR^-_B\cap\rR^+_{B^\prime}|
& = &\min\{r\in\bZgeqo|
\exists t\in\fkJ_{0,r},
\exists \beta_t\in B_t (t\in\fkJ_{0,r-1}),
\\
& & \quad\quad\quad\quad
B_0=B, B_r=B^\prime, 
B_{t+1}=B_t^{(\beta_t)}
(t\in\fkJ_{0,r-1})\}.
\end{array}
\end{equation*}
\newline
{\rm{(2)}} Let $\bB^\prime$ be a non-empty subset of $\bBR$
such  that the same condition as $(\rR 3)$ with $\bB^\prime$ in place of $\bBR$
is fulfilled. Then $\bB^\prime=\bBR$. 
\newline
{\rm{(3)}}
Let $\fca_\bR:=\fca\otimes_\bZ\bR$, where notice
that $\fca_\bR$ is an $|\fkI|$-dimesional $\bR$-linear space,
$\fca$ can be idetified with the $\bZ$-submodule of
$\fca_\bR$
in a natural way. Let $\fca_\bR^*$ be the dual $\bR$-linear space of $\fca_\bR$.
Let  $f\in\fca_\bR^*\setminus\{0\}$ be such that $0\notin f(\rR)$.
Then there exists $B\in\bBR$
such that $\rR^+_B=\rR\cap f^{-1}(\bRgneqo)$.
\newline
{\rm{(4)}} We have
\begin{equation*}
\rR=\cup_{B\in\bBR}B.
\end{equation*} In particular, we have $\rR^-_B=-\rR^+_B$ $(B\in\bBR)$.
 \newline
{\rm{(5)}} We have $|\bBR|\in 2\bN$.
Moreover there exist two subsets $Y_1$, $Y_2$ of $\bBR$
such that $Y_1\cup Y_2=\bBR$,
$Y_1 \cap Y_2=\emptyset$,
$|Y_1|=|Y_2|={\frac 1 2}|\bBR|$,
and 
$Y_2=\{B^{(\al)}|B\in 
Y_1,\al \in B\}$, 
$Y_1=\{(B^\prime)^{(\beta)}|B^\prime\in 
Y_2,\beta \in B^\prime\}$.
\end{lemma}

Until the end of this section,
let $\rR$ be an FGRS over $\fca$,
and let $\fkI$ be a finite set with $|\fkI|=|B|$ for all $B\in\bBR$.

\begin{proposition}{\rm{({\it{see}} \cite{Y16})}}\label{proposition:ordering}
There exist $\al^B_i\in\rR$
$(B\in\bBR,i\in\fkI)$
such that
$B=\{\al^B_i|i\in\fkI\}$ $(B\in\bBR)$, and 
$\al^{B^\prime}_j=\al^B_j+N^{\rR,B}_{\al^B_i,\al^B_j}\al^B_i$
$(B\in\bB,i, j\in\fkI)$,
where $B^\prime:=B^{(\al^B_i)}$.
\end{proposition}

For $B\in\bBR$, let $B^{(i)}:=B^{(\al^B_i)}$
and $N^{\rR,B}_{ij}:=N^{\rR,B}_{\al^B_i,\al^B_j}$.
We have 
\begin{equation}\label{eqn:NRBij}
N^{\rR,B^{(i)}}_{ij}=N^{\rR,B}_{ij}\quad (i,j\in\fkI).
\end{equation}

For $B\in\bBR$, define the $|\fkI|\times |\fkI|$-matrix $G^B_i=[g^B_{kr}]_{k,r\in\fkI}$
over $\bZ$ by $\al^{B^{(i)}}_r=\sum_{k\in\fkI}g^B_{kr}\al^B_k$,
where precisely we have
\begin{equation*}
\mbox{$g^B_{ii}=-1$, $g^B_{jj}=1$, $g^B_{ij}=N^{\rR,B}_{ij}$, $g^B_{ji}=0$ $(i\ne j)$, 
and $g^B_{xy}=0$ $(i\notin\{x,y\})$}.
\end{equation*}
Notice that $G^B_i$ is the change of basis matrix from $B$ to $B^{(i)}$,
which means that viewing $B$ as the $1\times |\fkI|$-matrix,
we have
\begin{equation*}
B^{(i)}=BG^B_i.
\end{equation*} 
By \eqref{eqn:NRBij}, we have
\begin{equation}
(B^{(i)})^{(i)}=B,\,\,
G^B_i=(G^B_i)^{-1}=G^{B^{(i)}}_i
\quad (B\in\bBR,\,i,j\in\fkI).
\end{equation}

For $B\in\bBR$ and $i,j\in\fkI$,
let $m^B_{ij}:=|(\bZgeqo\al^B_i+\bZgeqo\al^B_j)\cap\rR^+_B|\,
(\in\bN)$, whence $m^B_{ii}=1$.

\begin{lemma}
Let $i,j\in\fkI$ be such that $i\ne j$.
Let $B\in\bBR$, and let $m:=m^B_{ij}$.
Define $B_k\in\bBR$ $(k\in\bZgeqo)$ by
$B_0:=B$, and $B_{2t-1}:=B_{2(t-1)}^{(i)}$,
$B_{2t}:=B_{2t-1}^{(j)}$ $(t\in\bN)$.
Then $B_{2m}=B$,
and $B_x\ne B_y$ for $x,y\in\fkJ_{0,2m-1}$ with $x\ne y$.
\end{lemma}
\begin{proof}
For $r,l\in\bZgeqo$,
let $x(r,l):=|(\bZgeqo\al^{B_r}_i\oplus\bZgeqo\al^{B_r}_j)\cap(\bZgeqo\al^{B_l}_i\oplus\bZgeqo\al^{B_l}_j)\cap\rR^+_{B_r}|$.
Let $r\in\bZgeqo$. We can see $x(r,r+t)=m-t$ $(t\in\fkJ_{0,m})$
and $(\bZgeqo\al^{B_{r+m}}_i\oplus\bZgeqo\al^{B_{r+m}}_j)\cap\rR^+_{B_r}=
(\bZleqo\al^{B_r}_i\oplus\bZleqo\al^{B_r}_j)\cap\rR^+_{B_r}$.
Hence $x(r,r+m+t)=t$ $(t\in\fkJ_{0,m})$ and $B_{r+2m}=B_r$.
\end{proof}

\begin{definition} \label{definition:finiteWG}
Let $\rR$ be an FGRS over $\fca$. 
\newline
{\rm{(1)}} For $B,B^\prime\in\bBR$,
define the $\bZ$-isomorphism $\trxi_{B^\prime,B}:\fca\to\fca$ by 
$\trxi_{B^\prime,B}(\al^B_i)=\al^{B^\prime}_i$ $(i\in\fkI)$.
Let $\sim$ be an equivalence relation on $\bBR$.
We say that $\sim$ is a {\it{groupoid equivalence relation}}
if for $B,B^\prime\in\bBR$ with $B\sim B^\prime$,
we have $\trxi_{B^\prime,B}(\rR^B)=\rR^{B^\prime}$
and $B^{(i)}\sim(B^\prime)^{(i)}$
$(i\in\fkI)$.
\newline
{\rm{(2)}} Let $\sim$ be a groupoid equivalence relation on $\bBR$.
Let $\bBRsim$ be the quotient set of $\bBR$ by $\sim$.
Let $[B]^\sim:=\{B^\prime\in\bBR|B^\prime\sim B\}(\in\Pfca)$ for $B\in\bBR$,
i.e., $\bBRsim=\{[B]^\sim|B\in\bBR\}$.
We call $\sim$ {\it{smallest}} (resp. {\it{largest}}) if 
$[B]^\sim=\{B^\prime\in\bBR|\trxi_{B^\prime,B}(\rR^B)=\rR^{B^\prime}\}$
(resp. $[B]^\sim=\{B\}$)
for all $B\in\bBR$.
Define the $\bZ$-module automorphism $\tausimi:\bBRsim\to\bBRsim$
by $\tausimi([B]^\sim):=[B^{(i)}]^\sim$
$(B\in\bBR)$.
Notice that if $\sim$ is largest, then the canonical map 
$f:\bBR\to\bBRsim$ $(f(B):=[B]^\sim)$ is a bijection.
\end{definition}

\begin{lemma} {\rm{({\it{see}} \cite{Y16})}} \label{lemma:tausimij}
{\rm{(1)}} $(\tausimi)^2=\rmid_{\bBRsim}$.
\newline
{\rm{(2)}} For $B\in\bBR$ and $i,j\in\fkI$ with $i\ne j$,
letting $m:=m^B_{ij}$, we have
$(\tausimi\tausimj)^m([B]^\sim)=[B]^\sim$.
\end{lemma}

\begin{definition} \label{definition:CayleyofFgrt}
Let $R$ be an FGRS over $\fca$. 
\newline
{\rm{(1)}}
Let $V(\rR):=\bBR$, and let
$E(\rR):=\{\{B, B^{(i)}\}|B\in V(\rR), i\in\fkI\}$.
Let $\Gamma(\rR)$ denote the finite graph $\Gamma(V(\rR),E(\rR))$. 
We call $\Gamma(\rR)$ the {\it{Cayley graph}} of $\rR$.
\newline
{\rm{(2)}} Let $B\in V(\rR)$. Let $k\in\bN$ and $i_t\in\fkI$ $(t\in\fkJ_{1,k})$.
Then let the symbol
\begin{equation*}
1^B\cdot s_{i_1}s_{i_2}\cdots s_{i_k}
\end{equation*} mean 
the map $f:\fkJ_{1,k+1}\to V(\rR)$
defined by $f(1):=B$
and $f(t+1):=f(t)^{(i_t)}$
$(t\in\fkJ_{1,k})$.
By Definition~\ref{definition:HamMap}~(3),
we say that $1^B\cdot s_{i_1}s_{i_2}\cdots s_{i_k}$ a {\it{Hamiltonian cycle}} of $\Gamma(\rR)$
if $k=|V(\rR)|$ and the above $f$ satisfies $f(k+1)=B$
and $f(\fkJ_{1,k})=V(\rR)$.
\newline
{\rm{(3)}} We say that $\rR$ is {\it{reducible}} 
if there exist $B\in\bB$
and $\fkI^\prime\subset\fkI$
such that $\fkI^\prime\ne\emptyset$,
$\fkI\setminus\fkI^\prime\ne\emptyset$,
and $m^B_{ij}=2$ for all $i\in\fkI^\prime$
and all $j\in\fkI\setminus\fkI^\prime$.
We say that $\rR$ is {\it{irreducible}} 
if $\rR$ is not reducible.
\end{definition}

The following lemma is clear from Lemma~\ref{lemma:bBprime}~(1).

\begin{lemma}\label{lemma:rkonetwo}
Keep the notation of Definition~{\rm{\ref{definition:CayleyofFgrt}}}. 
\newline
{\rm{(1)}}
Assume ${\mathrm{rank}}_\bZ\fca=1$.
Let $\{\al\}$ be a $\bZ$-base of $\fca$, i.e., $\fca=\bZ\al$.
Then $\rR=\{\al,-\al\}$, and $\Gamma(\rR)=(V(\rR),E(\rR))$
with $V(\rR)=\{\{\al\},\{-\al\}\}$ and $E(\rR)=\{\{\al\},\{-\al\}\}$, i.e.,
$|V(\rR)|=2$ and $|E(\rR)|=1$. 
Then $1^{\{\al\}}\cdot s_1s_1$ is the Hamiltonian cycle of $\Gamma(\rR)$,
where we let $\fkI=\{1\}$.
\newline
{\rm{(2)}}
Assume ${\mathrm{rank}}_\bZ\fca=2$.
Let $k:={\frac 1 2}|V(\rR)|\,(\in\bN)$, where see 
Lemma~{\rm{\ref{lemma:bBprime}~(4)}}.
Let $\fkI=\{1,2\}$, and let $i_{2t-1}:=1$, $i_{2t}:=2$
$(t\in\fkJ_{1,k})$.
Then $1^B\cdot s_{i_1}s_{i_2}\cdots s_{i_{2k}}$ 
is the Hamiltonian cycle of $\Gamma(\rR)$
for every $B\in\bBR$.
\end{lemma}

\begin{remark}
The irreducible FGRSs were classified by \cite{CH15}. 
For the rank three cases, they had been classified by \cite{CH12}.  
For the rank two cases, see \cite{CH09b}.
See also Theorem~\ref{theorem:classGHlist} below, which was obtained by \cite{CH15}. 
\end{remark}

By Lemma~\ref{lemma:bBprime}, we can easily see{\rm{:}}

\begin{lemma}\label{lemma:redHam}
Let $\rR$ be an FGRS over $\fca$.
Assume that $\rR$ is reducible.
Let $n:={\mathrm{rank}}_\bZ\fca$.
Let $\fkI^\prime$, $\fkI^{\prime\prime}$
be non-empty proper subsets of
$\fkI$ such that
$\fkI=\fkI^\prime\cup\fkI^{\prime\prime}$,
$\fkI^\prime\cap\fkI^{\prime\prime}=\emptyset$,
and 
$m^B_{ij}=2$
for all $B\in\bBR$,
all $i\in\fkI^\prime$ and all $j\in\fkI^{\prime\prime}$.
Let $\fca^\prime:=\oplus_{i\in\fkI^\prime}\bZ\al^B_i$
and $\fca^{\prime\prime}:=\oplus_{j\in\fkI^{\prime\prime}}\bZ\al^B_j$
for some $B\in\bBR$. {\rm{(}}We have 
$\fca=\fca^\prime\oplus\fca^{\prime\prime}$.{\rm{)}}
Let $\rR^\prime:=\rR\cap\fca^\prime$
and  $\rR^{\prime\prime}:=\rR\cap\fca^{\prime\prime}$.
\newline
{\rm{(1)}} $\rR^\prime$ and $\rR^{\prime\prime}$
are FGRSs over $\fca^\prime$ and $\fca^{\prime\prime}$
respectively.
Moreover $\rR=\rR^\prime\cup\rR^{\prime\prime}$.
Furthermore
the map $f:\bB(\rR^\prime)\times\bB(\rR^{\prime\prime})\to\bBR$
defined by $f(X,Y):=X\cup Y$ is bijective.
\newline
{\rm{(2)}}
Let $k:=|\bB(\rR^\prime)|$ and
$l:=|\bB(\rR^{\prime\prime})|$.
Let $r:={\frac l 2}$,
where we see $r\in\bN$ by Lemma~{\rm{\ref{lemma:bBprime}~(4)}}. 
Let $z:=kl$, where $z=|\bBR|$ by {\rm{(1)}}.
Let $X\in\bB(\rR^\prime)$,
$Y\in\bB(\rR^{\prime\prime})$,
and $B:=X\cup Y$, where $B\in\bBR$ by {\rm{(1)}}.
Assume that there exist Hamiltonian cycles
$1^X\cdot s_{i_1}s_{i_2}\cdots s_{i_k}$
and $1^Y\cdot s_{j_1}s_{j_2}\cdots s_{j_l}$
of $\Gamma(\rR^\prime)$
and $\Gamma(\rR^{\prime\prime})$ respectively.
Define the map $h:\fkJ_{1,z}\to\fkI$
by
\begin{equation*}
\left\{\begin{array}{lcl}
h(t+2xk)&:=&i_t \quad
(t\in\fkJ_{1,k-1}, x\in\fkJ_{0,r-1}), \\
h(ku)&:=& j_u \quad
(u\in\fkJ_{1,l}), \\
h(t+(2x-1)k)&:=&i_{k-t} \quad
(t\in\fkJ_{1,k-1}, x\in\fkJ_{1,r}).
\end{array}\right.
\end{equation*}
Then $1^B\cdot s_{h(1)}s_{h(2)}\cdots s_{h(z)}$
is a Hamiltonian cycle of $\Gamma(\rR)$.
In particular, $\Gamma(\rR)$ has a Hamiltonian cycle.
\end{lemma}

\subsection{Categorical finite generalized root systems}

The reader can skip this subsection.

\begin{definition} \label{definition:DefCFGRS}
Let $\fkI$ be a non-empty finite set.
\newline
{\rm{(1)}} Let $\hatC=[\hatc_{ij}]_{i,j\in\fkI}$ be an $|\fkI|\times|\fkI|$-matrix over $\bZ$.
We call $\hatC$
a {\it{generalized Cartan matrix}} if 
$\hatc_{ii}=2$ $(i\in\fkI)$
and $\delta_{\hatc_{jk},0}=\delta_{\hatc_{kj},0}$
$(j,k\in\fkI,j\ne k)$,
where
$\delta_{xy}$ means the Kronecker's delta.
\newline
{\rm{(2)}}
Let $\FtwoI$ be the group presented by the generating set $\fkI$
and the relations $i^2 =e$, 
where $e$ is the unit of $\fkI$.
Let $\mcA$ be a non-empty finite set.
Assume that $\FtwoI$ transitively acts on $\mcA$.
Let $\hatC^a=[\hatc^a_{ij}]_{i,j\in\fkI}$
$(a\in\mcA)$ be generarized Cartan matrices
such that 
\begin{equation*}
\forall a\in\mcA, \forall i,\forall j\in\fkI, \hatc^{i\cdot a}_{ij}=\hatc^a_{ij}.
\end{equation*}
We call the date
$\mcC=\mcC(\fkI,\mcA,(\hatC^a)_{a\in\mcA})$
the {\it{Cartan scheme}}.
\newline
{\rm{(3)}} 
Let $\mcC=\mcC(\fkI,\mcA,(\hatC^a)_{a\in\mcA})$ be a Cartan scheme.
For $a\in\mcA$, let $\hatfca^a$ be a free $\bZ$-module 
with a $\bZ$-base ${\hat{B}}_a=\{\hatal^a_i|i\in\fkI\}$,
i.e., $\hatfca^a=\oplus_{i\in\fkI}\bZ\hatal^a_i$.
For $a\in\mcA$, define 
the $\bZ$-module isomorpshism
$\hats^a_i:\hatfca^a\to\hatfca^{i\cdot a}$
by $\hats^a_i(\hatal^a_j):=\hatal^{i\cdot a}_j-\hatc^a_{ij}\hatal^{i\cdot a}_i$.
Notice 
\begin{equation}\label{eqn:sseqid}
\hats^{i\cdot a}_i\hats^a_i=\rmid_{\hatfca^a}.
\end{equation}
For $a\in\mcA$, let $\hatR^a$ be a non-empty subset of $\hatfca^a$,
and let $\hatR^{a,+}:=\hatR^a\cap(\oplus_{i\in\fkI}\bZgeqo\hatal^a_i)$.
$\hatR^{a,-}:=\hatR^a\cap(\oplus_{i\in\fkI}\bZleqo\hatal^a_i)$.
We call the data
\begin{equation*}
\mcR=\mcR(\mcC,(\hatR^a,{\hat{B}}^a)_{a\in\mcA})
\end{equation*} a {\it{categigorical finite 
generalized root system}} (CFGRS for sort) 
if the following conditions $(\hatR0)$-$(\hatR4)$
are fulfilled.
\newline\newline
$(\hatR0)$ \quad $\hatR^a$ is a non-empty finite subset of $\hatfca^a$. \newline
$(\hatR1)$ \quad $\hatR^a=\hatR^{a,+}\cup\hatR^{a,-}$. \newline
$(\hatR2)$ \quad $\hatR^a\cap\bZ\hatal^a_i=\{\hatal^a_i,-\hatal^a_i\}$. \newline
$(\hatR3)$ \quad $\hats^a_i(\hatR^a)=\hatR^{i\cdot a}$. \newline
$(\hatR4)$ \quad Let $a\in\mcA$, and $t\in\bN$, $i_u\in\fkI$
$(u\in\fkJ_{1,t})$. If
\begin{equation*}
\forall j\in\fkI,\,\,
\hats^{(i_2\cdots i_{t-1}i_t)\cdot a}_{i_1}
\cdots
\hats^{i_t\cdot a}_{i_{t-1}}\hats^a_{i_t}(\hatal^a_j)
=\hatal^{(i_1i_2\cdots i_t)\cdot a}_j,
\end{equation*} then $(i_1i_2\cdots i_t)\cdot a=a$.
\end{definition}

\begin{definition}\label{definition:BasicDefCFGRS}
Let $\mcR(\mcC,(\hatR^a,{\hat{B}}^a)_{a\in\mcA})$ be a CFGRS.
\newline
{\rm{(1)}}
For $a\in\mcA$ and $i,j\in\fkI$. 
let
$\hatm^a_{ij}:=|\hatR^{a,+}\cap(\bZgeqo\hatal^a_i
+\bZgeqo\hatal^a_j)|$,
whence $\hatm^a_{ii}=1$.
\newline
{\rm{(2)}}
Let $a$, $b\in\mcA$.
Let ${\mathrm{Iso}}(a,b)$
be the set formed by all the 
$\bZ$-module isomorphisms 
from 
$\hatfca^b$ to $\hatfca^a$.
Define $\hattrxi^{a,b}\in{\mathrm{Iso}}(a,b)$
by $\hattrxi^{a,b}(\hatal^b_i):=\hatal^a_i$ $(i\in\fkI)$.
{\rm{(}}Notice $\rmid_{\hatfca^a}=\hattrxi^{a,a}$
and $\hats^{i\cdot a}_i=\hattrxi^{i\cdot a,a}\hats^a_i\hattrxi^{a,i\cdot a}$.{\rm{)}}
Define the subset $\hatW(a,b)$ of ${\mathrm{Iso}}(a,b)$
by
\begin{equation*}
\hatW(a,b):=\{\hats^{i_1\cdot a}_{i_1}\hats^{i_2i_1\cdot a}_{i_2}
\cdots \hats^{i_k\cdots i_2i_1\cdot a}_{i_k}\, 
|\,k\in\bZgeqo, i_t\in\fkI\,(t\in\fkJ_{1,k}), b=i_k\cdots i_2i_1\cdot a)\}. 
\end{equation*}
By \eqref{eqn:sseqid} and $(\hatR4)$,
we see that
$\hattrxi^{a,b}\in\hatW(a,b)\Leftrightarrow a=b$.

\end{definition}

\begin{lemma}{\rm{(\cite[Lemma~1.5]{AYY15})}} \label{lemma:LemAAY}
Let $\mcR(\mcC,(\hatR^a,{\hat{B}}^a)_{a\in\mcA})$ be a CFGRS.
Let $a\in\mcA$, and $i,j\in\fkI$.
Let $m:=\hatm^a_{ij}$.
Let $i_{2k-1}:=j$ and $i_{2k}:=i$ $(k\in\fkJ_{1,m})$.
Then
\begin{equation*}
\mbox{$(\hatR4)^\prime$}\quad\quad
(ij)^m\cdot a=a\quad\mbox{and}\quad
\hats^{(i_{2m-1}\cdots i_2i_1)\cdot a}_i
\hats^{(i_{2m-2}\cdots i_2i_1)\cdot a}_j
\cdots
\hats^{i_1\cdot a}_i\hats^a_j=\rmid_{{\hatfca^a}}.
\end{equation*}
\end{lemma}
As for the definition of a CFGRS, the condition $(\hatR4)$ can be
replaced by 
$(\hatR4)^\prime$, see \cite[Remark~1.6]{AYY15}.
\newline\par
Let $\FzeroI$ (resp. $\FoneI$) be the free semigroup (resp. the free monoid) with free generating set $\fkI$.
Let $\hate$ denote the unit of $\FoneI$.
We regard  $\FzeroI$ as the subset $\FoneI\setminus\{\hate\}$
of $\FoneI$.
Define the map $\varsigma:\FoneI\to\bZgeqo$
by $\varsigma(\hate):=0$ and $\varsigma(ui):=\varsigma(u)+1$
$(i\in\fkI, u\in\FoneI)$.
Namely $\varsigma(i_1i_2\cdots i_k)=k$
for $k\in\bN$ and $i_t\in\fkI$ $(t\in\fkJ_{1,k})$.
\begin{definition} 
Let $\mcR(\mcC,(\hatR^a,{\hat{B}}^a)_{a\in\mcA})$ be a CFGRS.
Let $a\in\mcA$. \newline
{\rm{(1)}} Define the map $\varpi_a:\FoneI\to\mcA$ by
$\varpi_a(\hate):=a$ and $\varpi_a(iu):=\varpi_{i\cdot a}(u)$
$(i\in\fkI, u\in\FoneI)$. Namely we have
\begin{equation*}
\mbox{$\varpi_a(i_1i_2\cdots i_k)=i_k\cdots i_2i_1\cdot a$
for $k\in\bN$ and $i_t\in\fkI$ $(t\in\fkJ_{1,k})$,}
\end{equation*}
where regard $i_k\cdots i_2i_1$ of the RHS 
as the element of $\FtwoI$.
\newline
{\rm{(2)}}
For $u$, $v\in\FoneI$, write  $u ({\overset{a}{\equiv}})^\prime v$
if there exist $u_1$, $u_2\in\FoneI$ and $i,j\in\fkI$ such that
$u=u_1u_2$ and $v=u_1(ij)^mu_2$, 
where $m:=\hatm^{\varpi_a(u_1)}_{ij}$.
For $u$, $v\in\FoneI$, write  $u ({\overset{a}{\equiv}}) v$
if $u=v$, $u ({\overset{a}{\equiv}})^\prime v$ or $v ({\overset{a}{\equiv}})^\prime u$.
For $u$, $v\in\FoneI$, write  $u {\overset{a}{\equiv}} v$
if there exist $k\in\bN$ and $y_t\in\FoneI$ $(t\in\fkI_{1,k})$
such that $y_1=u$, $y_k=v$
and $y_t({\overset{a}{\equiv}})y_{t+1}$ $(t\in\fkJ_{1,k-1})$.
Then ${\overset{a}{\equiv}}$ is an equivalence relation on $\FoneI$. \newline
{\rm{(3)}} Define $\hats^{\varpi_a(u)}_u\in\hatW(a,\varpi_a(u))$ $(u\in\FoneI)$
by $\hats^a_\hate:=\rmid_{\hatfca^a}$
and $\hats^{\varpi_a(iu)}_{iu}:=\hats^{i\cdot a}_i\hats^{\varpi_{i\cdot a}(u)}_u$
$(i\in\fkI, u\in\FoneI)$.
Let ${^a}\hatW:=\{w\hattrxi^{b,a}|b\in\mcA,w\in\hatW(a,b)\}(\subset{\mathrm{Iso}}(a,a))$.
Define the map 
$\hatell^a:{^a}\hatW\to\bZgeqo$ by
\begin{equation}\label{eqn:hatellaw}
\hatell^a(y):=\min\{\varsigma(u)|u\in\FoneI,
\hats^{\varpi_a(u)}_u\hattrxi^{\varpi_a(u),a}=y\}.
\end{equation} It follows from $(\hatR4)$ and \eqref{eqn:sseqid} that 
\begin{equation} \label{eqn:xiauxaiav}
\forall u, \forall v\in\FoneI, 
\quad 
\hats^{\varpi_a(u)}_u\hattrxi^{\varpi_a(u),a}
=\hats^{\varpi_a(v)}_u\hattrxi^{\varpi_a(v),a}
\,\,\Rightarrow\,\,
\varpi_a(u)=\varpi_a(v).
\end{equation}
\end{definition}

\begin{proposition}{\rm{(c.f. \cite{HY08})}}
Let $\mcR(\mcC,(\hatR^a,{\hat{B}}^a)_{a\in\mcA})$ be a CFGRS. 
\newline
{\rm{(1)}} Let $a\in\mcA$, and $u,v\in\FoneI$. Then{\rm{:}}
\begin{equation}\label{eqn:sxiua}
\hats^{\varpi_a(u)}_u\hattrxi^{\varpi_a(u),a}
=\hats^{\varpi_a(v)}_v\hattrxi^{\varpi_a(v),a}
\quad\Leftrightarrow\quad u{\overset{a}{\equiv}}v.
\end{equation}
If this is the case, we have $\varpi_a(u)=\varpi_a(v)$.
\newline
{\rm{(2)}} Let $a,b\in\mcA$, and $w\in\hatW(a,b)$. Then{\rm{:}}
\begin{equation*}
\hatell^a(w)=|\hatR^{a,+}\cap w(\hatR^{b,-})|.
\end{equation*}
\newline
{\rm{(3)}} Let $a\in\mcA$.
Then $\hatR^a$ is an FGRS over $\hatfca^a$,
and the map ${\hat{\psi}}^a:{^a}\hatW\to\bB(\hatR^a)$
defined by ${\hat{\psi}}^a(y):=y({\hat{B}}^a)$ is bijective.
Moreover, for $u\in\FoneI$ and $i\in\fkI$, we have
${\hat{\psi}}^a(\hats^{\varpi_a(u)}_u\hattrxi^{\varpi_a(u),a})^{(i)}={\hat{\psi}}^a(\hats^{\varpi_a(ui)}_{ui}\hattrxi^{\varpi_a(ui),a})$.
\newline
{\rm{(4)}} Keep the notation of {\rm{(3)}}.
Define the equivalence relation $\sim$ on $\bB(\hatR^a)$ as follows.
For $B_1$, $B_2\in\bB(\hatR^a)$, 
we define $B_1\sim B_2$
if there exists $b\in\mcA$
such that $B_1={\hat{\psi}}^a(w_1\hattrxi^{b,a})$
and $B_2={\hat{\psi}}^a(w_2\hattrxi^{b,a})$
for some $w_1$, $w_2\in\hatW(a,b)$.
Then  $\sim$ is a groupoid equivalence relation on $\bB(\hatR^a)$.
\newline
{\rm{(5)}} 
Keep the notation of Lemma~{\rm{\ref{lemma:LemAAY}}}.
Let $u_t:=i_1i_2\cdots i_t$ $(t\in\fkJ_{1,2m})$.
Then we have $\hats^{\varpi_a(u_t)}_{u_t}\hattrxi^{\varpi_a(u_t),a}({\hat{B}}^a)\ne {\hat{B}}^a$
for $t\in\fkJ_{1,2m-1}$.
Moreover $\varpi_a(u_{2m})=a$ and $\hats^{\varpi_a(u_{2m})}_{u_{2m}}=\rmid_{\hatfca^a}$.
\end{proposition}

We can easily see the following.

\begin{lemma}\label{lemma:RtBt}
Keep the notation of Definition~{\rm{\ref{definition:finiteWG}}}.
Let $k:=|\bBRsim|$. Let $\{B_t|t\in\fkJ_{1,k}\}$
be a complete set of representatives of $\bBRsim$.
Let ${\tilde{\mcA}}:=\fkJ_{1,k}$.
Define the action of $\FtwoI$ on ${\tilde{\mcA}}$
by $[B_{i\cdot t}]^\sim=\tausimi([B_t]^\sim)$ $(t\in{\tilde{\mcA}})$.
For $t\in{\tilde{\mcA}}$, let $\hatC(t):=(-N^{\rR,B_t}_{ij})_{i,j\in\fkI}$
and $\rR_t:=\rR$.
Then ${\tilde{\mcC}}={\tilde{\mcC}}(\fkI,\mcA,(\hatC(t))_{t\in{\tilde{\mcA}}})$ is a Cartan scheme,
and  $\mcR({\tilde{\mcC}},(\rR_t,B_t)_{t\in{\tilde{\mcA}}})$ is a CFGRS.
\end{lemma}

\begin{definition}\label{definition:CayGofW}
Let $\mcR=\mcR(\mcC,(\hatR^a,{\hat{B}}^a)_{a\in\mcA})$ be a CFGRS.
\newline
{\rm{(1)}} We define the semigroup ${\mathcal{W}}(\mcR)$ as follow.
As a set, let  
\begin{equation*}
{\mathcal{W}}(\mcR):=\{0\}\coprod(\coprod_{a,b\in\mcA}\hatW(a,b))),
\end{equation*}
where $\coprod$ means a disjoint union. The multiplication $\cdot$ of ${\mathcal{W}}(\mcR)$
is defined as follows.
If $x\in\hatW(a,b)$, $y\in\hatW(b,c)$
for some $a,b,c\in\mcA$, define $x\cdot y:=xy\,(\in\hatW(a,c))$.
Otherwise let $x\cdot y:=0$.
We call ${\mathcal{W}}(\mcR)$ the {\it{Weyl groupoid}} of $\mcR$.
We call ${\mathcal{W}}(\mcR)$ {\it{irreducible}} if
the FGRS $\hatR^a$ for some $a\in\mcA$ is irreducible.
\newline
{\rm{(2)}} Let $a\in\mcA$. Let $V:={^a}\hatW\,(=\{\hats^{\varpi_a(u)}_u\hattrxi^{\varpi_a(u),a}|
u\in\FoneI\})$. Let 
\begin{equation*}
E:=\{\{\hats^{\varpi_a(u)}_u\hattrxi^{\varpi_a(u),a},\hats^{\varpi_a(ui)}_{ui}\hattrxi^{\varpi_a(ui),a}\}|u\in\FoneI,i\in\fkI\}
\,(\subset{\mathfrak{P}}_2(V)).
\end{equation*}
Let $\Gamma_a({\mathcal{W}}(\mcR)):=\Gamma(V,E).$
We call the finite graph $\Gamma_a({\mathcal{W}}(\mcR))$
is the {\it{Cayley graph at $a$ of  ${\mathcal{W}}(\mcR)$}}.
We see that as a graph, 
$\Gamma_a({\mathcal{W}}(\mcR))$ is isomorphic to 
the Cayley graph $\Gamma(\hatR^{a_1})$ of the FGRS $\hatR^{a_1}$
$(a_1\in\mcA)$.
\end{definition}

We will not use Definition~\ref{definition:DefIsoCFGRS} below in this paper.

\begin{definition}\label{definition:DefIsoCFGRS}
For $t\in\fkJ_{1,2}$, let $\mcC_t=\mcC(\fkI,\mcA_t,(\hatC^{a_t})_{a_t\in\mcA_t})$ $(t\in\fkJ_{1,2})$ be a Cartan scheme.
We say that $\mcC_1$ and $\mcC_2$ are {\it{isomorphic}} if 
there exists a bijection $\varphi:\mcA_1\to\mcA_2$ such that
$\hatC^{\varphi(a_1)}=\hatC^{a_1}$ for all $a_1\in\mcA_1$.
For $t\in\fkJ_{1,2}$, let $\mcR_t=\mcR(\mcC_t,(\hatR^{a_t},{\hat{B}}^{a_t})_{a_t\in\mcA_t})$ be a CFGRS.
We can see that if $\mcC_1$ and $\mcC_2$ are isomorphic, then
we have $f^{a_1}(\hatR^{a_1})=\hatR^{\varphi(a_1)}$ for all $a_1\in\mcA_1$, where 
let $\varphi:\mcA_1\to\mcA_2$ be as above and define 
the $\bZ$-isomorphisms $f^{a_1}:\rmSpan_\bZ({\hat{B}}^{a_1})
\to\rmSpan_\bZ({\hat{B}}^{\varphi(a_1)})$ $(a_1\in\mcA_1)$
by $f^{a_1}(\hatal^{a_1}_i):=\hatal^{\varphi(a_1)}_i$ $(i\in\fkI)$.
We say that $\mcR_1$ and $\mcR_2$ are {\it{quasi-isomorphic}} if 
there exist $a_1\in\mcA_1$ and $a_2\in\mcA_2$
such that $g(\hatR^{a_1})=\hatR^{a_2}$, where 
define 
the $\bZ$-isomorphism $g:\rmSpan_\bZ({\hat{B}}^{a_1})
\to\rmSpan_\bZ({\hat{B}}^{a_2})$ 
by $g(\hatal^{a_1}_i):=\hatal^{a_2}_i$ $(i\in\fkI)$.
\end{definition}

\subsection{FGRSs of generalized quantum groups}
Let $\bK$ be a field. Let $\bKt:=\bK\setminus\{0\}$.
For $m\in\bZgeqo$ and $x$, $y\in\bK$,
let
$(m)_x:=\sum_{t=1}^mx^{t-1}$, $(m)_x!:=\prod_{t=1}^m(t)_x$, 
$(m;x,y):=1-x^{m-1}y$ and $(m;x,y)!:=\prod_{t=1}^m(t;x,y)$.

Let $\fca$ be a free $\bZ$-module of finite rank.
We say that a map $\chi:\fca\times\fca\to\bKt$ is a {\it{bicharacter}} 
if
\begin{equation*}
\chi(\lambda+\mu,\nu)=\chi(\lambda,\nu)\chi(\mu,\nu),\,\,
\chi(\lambda,\mu+\nu)=\chi(\lambda,\mu)\chi(\lambda,\nu)
\quad (\lambda,\mu,\nu\in\fca).
\end{equation*}

Let $\rR$ be an FGRS over $\fca$.
We call $\rR$ {\it{$\chi$-associated}} if
there  exists a bicharacter $\chi:\fca\times\fca\to\bKt$ such that
\begin{equation*}
\begin{array}{l}
\fkJ_{0,N^{\rR,B}_{\al,\beta}}
=\{m\in\bZgeqo|(m)_{\chi(\al,\al)}!(m;\chi(\al,\al),
\chi(\al,\beta)\chi(\beta,\al))!\ne 0\} \\
\quad 
(B\in\bBR,\al,\beta\in B, \al\ne \beta).
\end{array}
\end{equation*}
We call $\rR$ {\it{$\bK$-associated}} if
$\rR$ is $\chi$-associated for some  bicharacter $\chi:\fca\times\fca\to\bKt$.

Let $\rR$ be a $\chi$-associated FGRS over $\fca$.
Let $B\in\bBR$. For such $\chi$ and $B$, let $D_\bK(\chi,B)$ 
be the finite graph defined by $D_\bK(\chi,B)=\Gamma(V,E)$
with $V:=B$ and $E:=\{\{\al,\beta\}|\al,\beta\in B, 
\al\ne\beta, \chi(\al,\beta)\chi(\beta,\al)\ne 1\}$,
respectively and, moreover, each vertex $\al\in V=B$
is labeled by $\chi(\al,\al)$ and 
each edge $\{\al,\beta\}\in E$
is labeled by $\chi(\al,\beta)\chi(\beta,\al)$.
We call $D_\bK(\chi,B)$ the {\it{generalized Dynkin diagram
associated with $\chi$ and $B$}}.
Examples of  $D_\bK(\chi,B)$ are drawn by Fig.~6.

\begin{remark}
Let $\bK$ be a field of characteristic zero.
All the $\bK$-associated irreducibe FGRSs were classified by 
\cite{Hec09}, which means that all the $D_\bK(\chi,B)$s
for all $\chi$-associated irreducible FGRSs $\rR$ with $B\in\bBR$
are listed by \cite{Hec09},
and the concrete forms of such $\rR$ are described by \cite{AA17}. 
See see \cite[Theorem~4.9]{AY15}, \cite[Theorem~4.9]{HY10} 
(see also \cite[Section~6]{BY18}) for example
to know
how the $\bK$-associated FGRSs interact with 
the generalized quantum groups,
or the Drinfeld doubles of
the Nichols algebras of diagonal type.
\end{remark}

\begin{theorem}\label{theorem:YamMain}
{\rm{(\cite[Theorem~6.3]{Y22})}}
Let $\bK$ be a field of characteristic zero.
Let $\rR$ be a $\bK$-associated FGRS.
Then $\Gamma(\rR)$ has a Hamiltonian cycle. 
\end{theorem}

\section{Finite generalized root systems associated with finite dimensional Lie (super)algebras
of type ${\mathrm{A}}$-${\mathrm{G}}$}\label{section:LieS}
Let $\frg$ be a $\bC$-linear space.
Let $\frg(t)$ $(t\in\fkJ_{0,1})$
be two linear $\bC$-subspaces of $\frg$
such that
$\frg=\frg(0)\oplus\frg(1)$.
For $t\in\bZ\setminus\fkJ_{0,1}$,  let $\frg(t)$ be such that $\frg(t)=\frg(t+2)=\frg(t-2)$.
Let $[\,,\,]:\frg\times\frg\to\frg$ be a bilinear map with
$[\frg(t_1),\frg(t_2)]\subset\frg(t_1+t_2)$ $(t_1,t_2\in\bZ)$.
We say that $\frg=(\frg,[\,,\,])$ is a {\it{Lie superalgebra}}
if 
\begin{equation*}
\begin{array}{l}
[y,x]=-(-1)^{t_1t_2}[x,y] \quad (x\in\frg(t_1), y\in\frg(t_2)\,(t_1,t_2\in\fkJ_{0,1})), \\
\mbox{$[[x,y],z]$}=[x,[y,z]]+(-1)^{t_1t_2}[y,[x,z]]
\quad (x\in\frg(t_1), y\in\frg(t_2)\,(t_1,t_2\in\fkJ_{0,1}),\,z\in\frg).
\end{array}
\end{equation*}
If $\frg(1)=\{0\}$, we call $\frg=(\frg,[\,,\,])$ a {\it{Lie algebra}}.
For Lie superalgebras $\frg$ and $\frg^\prime$, 
 we say that a $\bC$-linear isomorphism $f:\frg\to\frg^\prime$
 is a {\it{Lie superalgebra isomorphism}}
 if $f(\frg(t))=\frg^\prime(t)$ $(t\in\fkJ_{0,1})$
 and $f([X,Y])=[f(X),f(Y)]$ $(X,Y\in\frg)$.

Let $\fkI$ be a non-empty finite set. 
Let $\fca$ be the free $\bZ$-module
with a $\bZ$-base $\{\dal_i|i\in\fkI\}$,
i.e., $\fca=\oplus_{i\in\fkI}\bZ\dal_i$.
Let $\Pi:=\{\dal_i|i\in\fkI\}$.
Let $\fkI_{\mathrm{odd}}$ be a subset of $\fkI$.
Let $\fkI_{\mathrm{even}}:=\fkI\setminus \fkI_{\mathrm{odd}}$.
\begin{equation*}
\begin{array}{l}
\mbox{Let $A:=[a_{ij}]_{i,j\in\fkI}$ be an $|\fkI|\times|\fkI|$ matrix over $\bC$.} 
\end{array}
\end{equation*}
Then we have a Lie superalgebra $\frg=\frg(A,\fkI_{\mathrm{odd}})$ 
satisfying the following conditions $(LS1)$-$(LS4)$.
\newline\newline
$(LS1)$ As a Lie superalgebra, $\frg$ is generated by
$H_i$, $E_i$, $F_i$ $(i\in\fkI)$ with
$H_i\in \frg(0)$, 
$E_j$, $F_j\in \frg(0)$ $(j\in\fkI_{\mathrm{even}})$, 
$E_k$, $F_k\in \frg(1)$ $(k\in\fkI_{\mathrm{odd}})$. 
\newline
$(LS2)$ We have:
\begin{equation*}
[H_i,H_j]=0,\,\,[H_i,E_j]=a_{ij}E_j,\,\,[H_i,F_j]=-a_{ij}F_j,\,\,
[E_i,F_j]=\delta_{ij}H_i\quad (i,j\in\fkI). 
\end{equation*}
$(LS3)$ There exist the linear subspaces $\frg_\lambda$
$(\lambda\in\fca)$ such that $\frg=\oplus_{\lambda\in\fca}\frg_\lambda$,
$H_i\in\frg_0$, $E_i\in\frg_{\dal_i}$, $F_i\in\frg_{-\dal_j}$, and
$[\frg_\lambda,\frg_\mu]\subset\frg_{\lambda+\mu}$ $(\lambda, \mu\in\fca)$.
\newline
$(LS4)$ We have $\dim_\bC \frg_0=|\fkI|$ and
$\dim_\bC \frg_{\dal_i}=\dim_\bC \frg_{-\dal_i}=1$ $(i\in\fkI)$.
\newline
$(LS5)$ For $\lambda\in{\mathrm{Span}}_{\bZgeqo}(\Pi)\setminus\{0\}$,
we have $\{X\in \frg_\lambda|[X,F_i]=0\,(i\in\fkI)\}=\{0\}$
and $\{Y\in \frg_{-\lambda}|[E_i,Y]=0\,(i\in\fkI)\}=\{0\}$.
\newline\newline
Define the $\bZ$-module homomorphism $p:\fca\to\bZ$ by
$p(\dal_i):=0$ $(i\in\fkI_{\mathrm{odd}})$
and $p(\dal_j):=1$ $(j\in\fkI_{\mathrm{odd}})$. 
We have the Lie superalgebra automorphism
$\Omega:\frg\to\frg$
defined by $\Omega(H_i):=-H_i$,
$\Omega(E_i):=(-1)^{p(\dal_i)}F_i$,
and $\Omega(F_i):=E_i$. 
In particular, we have
\begin{equation*}
\dim \frg_\lambda=\dim \frg_{-\lambda}.
\end{equation*}
Let ${\dot{\rR}}(A,\fkI_{\mathrm{odd}}):=\{\lambda\in\fca\setminus\{0\}|\dim_\bC\frg_\lambda\geq 1\}$.
Let ${\dot{\rR}}^\pm(A,\fkI_{\mathrm{odd}}):={\dot{\rR}}(A,\fkI_{\mathrm{odd}})\cap{\mathrm{Span}}_{\pm\bZgeqo}\Pi$.
By $(LS2)$, we have 
\begin{equation} \label{eqn:dotRpm}
{\dot{\rR}}(A,\fkI_{\mathrm{odd}})={\dot{\rR}}^+(A,\fkI_{\mathrm{odd}})\cup {\dot{\rR}}^-(A,\fkI_{\mathrm{odd}}).
\end{equation}
We also have ${\dot{\rR}}^-(A,\fkI_{\mathrm{odd}})=-{\dot{\rR}}^+(A,\fkI_{\mathrm{odd}})$.
For $i\in\fkI$, we have
\begin{equation*}
\begin{array}{l}
\mbox{$[[E_i,E_i],F_i]=(-1+(-1)^{p(\dal_i)})a_{ii}E_i$, \,\, $[[E_i,[E_i,E_i]],F_i]=(1+(-1)^{p(\dal_i)})a_{ii}[E_i,E_i]$,} \\
\mbox{$[E_i,[F_i,F_i]]=(-1+(-1)^{p(\dal_i)})a_{ii}E_i$, \,\, $[E_i,[F_i,[F_i,F_i]]]=(1+(-1)^{p(\dal_i)})a_{ii}[F_i,F_i]$,}
\end{array}
\end{equation*}
which implies
\begin{equation*}
{\dot{\rR}}(A,\fkI_{\mathrm{odd}})\cap\bZ{\dal_i}=
\left\{\begin{array}{ll}
\{\dal_i,-\dal_i\} & \quad\mbox{if $a_{ii}=0$ or $i\in\fkI_{\mathrm{even}}$}, \\
\{\dal_i,2\dal_i,-\dal_i,-2\dal_i\} & \quad\mbox{if $a_{ii}\ne 0$ and $i\in\fkI_{\mathrm{odd}}$}. 
\end{array}\right.
\end{equation*}

The following lemma is directly proved.

\begin{lemma}
Let $i,j\in\fkI$ with $i\ne j$. 
Define $b_m\in\bZ$ 
$(m\in\bZgeqo)$ by 
\begingroup
\begin{equation*}
\renewcommand{\arraystretch}{1.1}
\begin{array}{lcl}
b_{2n}&:=& n(((n-1)+n(-1)^{p(\dal_i)})a_{ii}+(1+(-1)^{p(\dal_i)})a_{ij}), \\
b_{2n+1}&:=& ((n+1)+n(-1)^{p(\dal_i)})(na_{ii}+a_{ij}), 
\end{array}
\end{equation*}where $n\in\bZgeqo$.
We have{\rm{:}}
\begin{equation}\label{eqn:recAm}
b_{m+1}=(-1)^{m p(\dal_i)}(m a_{ii}+a_{ij})+b_m\quad (m\in\bZgeqo).
\end{equation}
\endgroup 
\newline
{\rm{(1)}}
Let $m\in\bZgeqo$, and let
$E_{m;i,j}:=({\mathrm{ad}}E_i)^m(E_j)$, $F_{m;i,j}:=({\mathrm{ad}}F_i)^m(F_j)$,
and
${\bar b}_m:=\prod_{k=2}^m b_k$. Then we have the following equations.
\begin{equation}\label{eqn:inpfml}
\renewcommand{\arraystretch}{1.2}
\begin{array}{l}
\mbox{$[E_{m;i,j}, F_i]=(-1)^{p(\dal_i)p(\dal_j)}b_m E_{m-1;i,j}$\quad $(m\geq 1)$}, \\
\mbox{$[E_i, F_{m;i,j}]=(-1)(-1)^{(m-1)p(\dal_i)}b_m F_{m-1;i,j}$\quad $(m\geq 1)$}, \\
\mbox{$[E_{m;i,j}, F_{(m-1);i,j}]=(-1)(-1)^{(m-1)(p(\dal_i)+p(\dal_i)p(\dal_j))}a_{ji}{\bar b}_m E_i$\quad $(m\geq 1)$}, \\
\mbox{$[E_{(m-1);i,j}, F_{m;i,j}]=(-1)^{mp(\dal_i)p(\dal_j)}a_{ji}{\bar b}_m F_i$\quad $(m\geq 1)$}, \\
\mbox{$[E_{m;i,j}, F_{m;i,j}]=
\left\{\begin{array}{l}
H_j \quad (m=0), \\
(-1)^{mp(\dal_i)p(\dal_j)}{\bar b}_m(ma_{ji}H_i+a_{ij}H_j)\quad (m\geq 1).
\end{array}\right. $}
\end{array}
\end{equation} We also have
\begin{equation}\label{eqn:inpfmld}
\begin{array}{l}
\mbox{$[E_{m;i,j},F_j]=-a_{ji}(ad E_i)^{m-1}(E_i)\quad (m\geq 1)$,} \\
\mbox{$[E_j,F_{m;i,j}]=(-1)^{mp(\dal_i)p(\dal_j)}a_{ji}(ad F_i)^{m-1}(F_i)\quad (m\geq 1)$.}
\end{array}
\end{equation}
{\rm{(2)}}  $\dal_i+\dal_j\in{\dot{\rR}}^+(A,\fkI_{\mathrm{odd}})$
if and only if 
$a_{ij}\ne 0$ or $a_{ji}\ne 0$.
\newline
{\rm{(3)}}
Let $m\in\fkJ_{2,\infty}$. Then
$m\dal_i+\dal_j\in{\dot{\rR}}^+(A,\fkI_{\mathrm{odd}})$
if and only if 
$\prod_{k=1}^m b_k\ne 0$.
\end{lemma}

\begin{lemma}
Let $x\in\bCt$.
Then we have the Lie superalgebra isomorphism
$f:\frg(xA,\fkI_{\mathrm{odd}})\to\frg(A,\fkI_{\mathrm{odd}})$ defined by
$f(H_i):=xH_i$, $f(E_i):=xE_i$, and $f(F_i):=F_i$
$(i\in\fkI)$.
In particular, we have ${\dot{\rR}}^+(xA,\fkI_{\mathrm{odd}})
={\dot{\rR}}^+(A,\fkI_{\mathrm{odd}})$.
\end{lemma}

\begin{equation}\label{eqn:symCond}
\begin{array}{l}
\mbox{From now on, until the end of Section~\ref{section:LieS},} \\
\mbox{we assume that $A$ is symmetric {\rm{(}}i.e., $a_{ij}=a_{ji}$ $(i,j\in\fkI)${\rm{)}}.}
\end{array}
\end{equation}

\begin{definition}
(1) Let $\fca_\bC:=\fca\otimes_\bZ\bC$.
Define the $\bC$-bilinear map $(\,,\,)^{(A,\fkI_{\mathrm{odd}})}:\fca_\bC\times\fca_\bC\to\bC$
by $(\dal_i,\dal_j)^{(A,\fkI_{\mathrm{odd}})}:=a_{ij}$ $(i,j\in\fkI)$.
\newline 
(2) For $i,j\in\fkI$ with $i\ne j$,
let $N^{(A,\fkI_{\mathrm{odd}})}_{ij}:=|(\dal_j+\bN\dal_i)\cap
{\dot{\rR}}^+(A,\fkI_{\mathrm{odd}})|\,
(\in\bZgeqo\cup\{\infty\})$.
For $i\in\fkI$, let $N^{(A,\fkI_{\mathrm{odd}})}_{ii}:=-2$.
\newline
(3) Let $i\in\fkI$. 
If $N^{(A,\fkI_{\mathrm{odd}})}_{ij}\ne\infty$
for all $j\in\fkI$,
we say that $(A,\fkI_{\mathrm{odd}})$ is {\it{$i$-good}}.
\newline
(4) Let $i\in\fkI$ be such that $(A,\fkI_{\mathrm{odd}})$ is $i$-good.
Then we let
$\dal^{(i)}_j:=\dal_j+N^{(A,\fkI_{\mathrm{odd}})}_{ij}\dal_i\,(\in\fca)$
$(j\in\fkI)$, 
$\Pi^{(i)}:=\{\dal^{(i)}_j|j\in\fkI\}$
{\rm{(}}a $\bZ$-basis of $\fca${\rm{)}}, 
$a^{(i)}_{xy}:=(\dal^{(i)}_x,\dal^{(i)}_y)^{(A,\fkI_{\mathrm{odd}})}\,(\in\bC)$ 
$(x,y\in\fkI)$,
$A^{(i)}:=[a^{(i)}_{xy}]_{x,y\in\fkI}\,
(\in{\mathrm{M}}_{|\fkI|}(\bC))$,
and
$\fkI_{\mathrm{odd}}^{(i)}:=\{j\in\fkI|p(\dal^{(i)}_j)\in 2\bZ+1\}\,(\subset \fkI)$.
Moreover 
we define:
\begin{equation*}
(A,\fkI_{\mathrm{odd}})^{(i)}:=(A^{(i)},\fkI_{\mathrm{odd}}^{(i)}).
\end{equation*}
Notice that $(A,\fkI_{\mathrm{odd}})^{(i)}$ is also $i$-good, see Lemma~\ref{lemma:iigood}
below.
\newline
(5) Let $A^\prime\in{\mathrm{M}}_{|\fkI|}(\bC)$.
Let $\fkI^\prime_{\mathrm{odd}}$ be a subset 
of $\fkI$.  
We write
\begin{equation*}
(A^\prime,\fkI^\prime_{\mathrm{odd}}) \ddotsim (A,\fkI_{\mathrm{odd}})
\end{equation*}
if $(A^\prime,\fkI^\prime_{\mathrm{odd}})=(A,\fkI_{\mathrm{odd}})$
or there exist
$k\in\bN$ and $i_t\in\fkI$ $(t\in\fkJ_{1,k})$
such that $(A^\prime,\fkI^\prime_{\mathrm{odd}})
=(A[k],\fkI_{\mathrm{odd}}[k])$
and 
$(A[u-1],\fkI_{\mathrm{odd}}[u-1])$ are $i_u$-good for 
all $u\in\fkJ_{1,k}$,
where let us mean $(A[0],\fkI_{\mathrm{odd}}[0]):=(A,\fkI_{\mathrm{odd}})$
and
$(A[u],\fkI_{\mathrm{odd}}[u]):=(A[u-1],\fkI_{\mathrm{odd}}[u-1])^{(i_u)}$ 
$(u\in\fkJ_{1,k})$.
\end{definition}

We can directly prove the following lemmas and proposition.

\begin{lemma}\label{lemma:iigood}
Let $i\in\fkI$, and
assume that  $(A,\fkI_{\mathrm{odd}})$ is $i$-good.
\newline
{\rm{(1)}}
Assume that $p(\dal_i)=0$ and $a_{ii}=0$.
Then $a_{ij}=0$ for all $j\in\fkI\setminus\{i\}$.
In particular, $(A,\fkI_{\mathrm{odd}})^{(i)}=(A,\fkI_{\mathrm{odd}})$.
\newline
{\rm{(2)}}
Assume that $a_{ii}\ne 0$.
Then $(A,\fkI_{\mathrm{odd}})^{(i)}=(A,\fkI_{\mathrm{odd}})$. 
Moreover if $p(\dal_i)=1$,
we have $N^{(A,\fkI_{\mathrm{odd}})}_{ij}\in 2\bZgeqo$
for all $j\in\fkI\setminus\{i\}$.
\newline
{\rm{(3)}}
Assume that $p(\dal_i)=1$ and $a_{ii}=0$.
Then $N^{(A,\fkI_{\mathrm{odd}})}_{ij}\in\fkJ_{0,1}$ for all $j\in\fkI\setminus\{i\}$.
Moreover 
$(A,\fkI_{\mathrm{odd}})^{(i)}$ is $i$-good, and we have
$((A,\fkI_{\mathrm{odd}})^{(i)})^{(i)}=(A,\fkI_{\mathrm{odd}})$.
\end{lemma}

\begin{proposition} \label{theorem:oddref}
Let $i\in\fkI$, and
assume that  $(A,\fkI_{\mathrm{odd}})$ is $i$-good.
For $j\in\fkI\setminus\{i\}$,
let $m_j:=N^{(A,\fkI_{\mathrm{odd}})}_{ij}$,
and let $x_j,y_j\in\bC\setminus\{0\}$ be such that  
$[E_{m_j;i,j}, F_{m_j;i,j}]=x_jy_j(H_j+m_jH_i)$
{\rm{(}}See \eqref{eqn:inpfml}{\rm{)}}.
Then
there exists a unique 
Lie superalgebra isomorphism
\begin{equation*}
L=L^{(A,\fkI_{\mathrm{odd}})^{(i)}}_i:\frg((A,\fkI_{\mathrm{odd}})^{(i)})\to\frg(A,\fkI_{\mathrm{odd}})
\end{equation*}
such that
\begin{equation*}
L(H_j):=
\left\{\begin{array}{ll}
-H_i & \quad (j=i), \\
H_j+m_jH_i & \quad (j\ne i),
\end{array}
\right.
\end{equation*}
\begin{equation*}
L(E_j):=
\left\{\begin{array}{ll}
(-1)^{p(\dal_i)}F_i & \quad (j=i), \\
x_j^{-1}E_{m_j;i,j} & \quad (j\ne i),
\end{array}
\right.
\end{equation*}
\begin{equation*}
L(F_j):=
\left\{\begin{array}{ll}
E_i & \quad (j=i), \\
y_j^{-1}F_{m_j;i,j} & \quad (j\ne i).
\end{array}
\right.
\end{equation*}
In particular, we have the following {\rm{(a)}} and {\rm{(b)}}.
\newline
{\rm{(a)}} Define
the $\bZ$-module isomorophism 
$f:\fca\to\fca$
by $f(\dal_j):=\dal^{(i)}_j$ $(j\in\fkI)$.
Then we have
$f({\dot{\rR}}((A,\fkI_{\mathrm{odd}})^{(i)}))
={\dot{\rR}}(A,\fkI_{\mathrm{odd}})$.
\newline
{\rm{(b)}} We have:
\begin{equation*}
{\dot{\rR}}(A,\fkI_{\mathrm{odd}})
=({\dot{\rR}}(A,\fkI_{\mathrm{odd}})\cap{\mathrm{Span}}_{\bZgeqo}\Pi^{(i)})\cup
({\dot{\rR}}(A,\fkI_{\mathrm{odd}})\cap{\mathrm{Span}}_{\bZleqo}\Pi^{(i)}).
\end{equation*}
\end{proposition}

Let $\rR(A,\fkI_{\mathrm{odd}}):={\dot{\rR}}(A,\fkI_{\mathrm{odd}})
\setminus\{2\beta|\beta\in{\dot{\rR}}(A,\fkI_{\mathrm{odd}})\}$.
Let $\rR^+(A,\fkI_{\mathrm{odd}}):=\rR(A,\fkI_{\mathrm{odd}})\cap{\mathrm{Span}}_{\bZgeqo}\Pi)$.
It is clear that
\begin{equation*}
\dim\frg(A,\fkI_{\mathrm{odd}})<\infty
\Leftrightarrow
|\rR(A,\fkI_{\mathrm{odd}})|<\infty
\Leftrightarrow
|{\dot{\rR}}(A,\fkI_{\mathrm{odd}})|<\infty.
\end{equation*}

We can directly see the following lemma.

\begin{lemma}\label{lemma:LSasschi}
Assume $\dim\frg(A,\fkI_{\mathrm{odd}})<\infty$.
\newline
{\rm{(1)}} $R(A,\fkI_{\mathrm{odd}})$ is an FGRS over $\fca$.
\newline {\rm{(2)}} $\rR(A,\fkI_{\mathrm{odd}})^{(i)}=\rR((A,\fkI_{\mathrm{odd}})^{(i)})$ for $i\in\fkI$.
\newline
{\rm{(3)}} Assume that $a_{ij}\in\bZ$ for all $i,j\in\fkI$,
where recall $A=[a_{ij}]_{i, j,\in\fkI}$.
Let $\bK$ be a field of characteristic zero.
Let $q\in\bKt$ be such that $q^n\ne 1$ for all $n\in\bN$.
Let $\chi:\fca\times\fca\to\bKt$ be the bichbaracter 
defined by $\chi(\lambda,\mu):=(-1)^{p(\lambda)p(\mu)}q^{(\lambda,\mu)^{(A,\fkI_{\mathrm{odd}})}}$.
Then $\rR(A,\fkI_{\mathrm{odd}})$ is the $\chi$-associated FGRS
with $\Pi\in\bB(\rR(A,\fkI_{\mathrm{odd}}))$. 
\end{lemma}

We define the Dynkin diagram $D(A,\fkI_{\mathrm{odd}})$ for the pair $(A,\fkI_{\mathrm{odd}})$
in the following way.
(In fact, we only define it appearing in {\rm{Fig.~1}}~-~{\rm{Fig.~3}} below.)
The vertex of $D(A,\fkI_{\mathrm{odd}})$ corresponding to $i\in\fkI$ is:

\setlength{\unitlength}{1mm}
\begin{picture}(40,25)(0,-5)

\put(2.5,12){${\tiny{i}}$}\put(3,10){\circle{2}}
\put(10,09){$\mbox{(if $i\in\fkI_{\mathrm{even}}$ and $a_{ii}\ne 0$),}$}

\put(62.5,12){${\tiny{i}}$}
\put(63,10){\circle{2}}
\put(62.25,9.3){\rotatebox{45}{\line(1,0){2}}}
\put(62.25,9.3){\rotatebox{45}{\line(0,1){2}}}
\put(70,9){$\mbox{(if $i\in\fkI_{\mathrm{odd}}$ and $a_{ii}=0$),}$}

\put(2.5,2){${\tiny{i}}$}\put(3,0){\circle*{2}}
\put(10,-1){$\mbox{(if $i\in\fkI_{\mathrm{odd}}$ and $a_{ii}\ne 0$),}$}

\put(62.5,02){${\tiny{i}}$}
\put(63,00){\circle{2}}
\put(62,0){{\line(1,0){2}}}
\put(70,-1){$\mbox{(if $i\in\fkI_{\mathrm{even}}$ and $a_{ii}=0$)}$.}

\end{picture}

\noindent
Regarding the vertices of $D(A,\fkI_{\mathrm{odd}})$ corresponding to $i,j\in\fkI$
with $i\ne j$, 
and the edge between them,
we mean (recall \eqref{eqn:symCond}):

\setlength{\unitlength}{1mm}
\begin{picture}(40,80)(0,0)

\put(02.5,72){${\tiny{i}}$}
\put(02.25,69.3){\rotatebox{45}{\line(1,0){2}}}
\put(02.25,69.3){\rotatebox{45}{\line(0,1){2}}}

\put(12.5,72){${\tiny{j}}$}
\put(12.25,69.3){\rotatebox{45}{\line(1,0){2}}}
\put(12.25,69.3){\rotatebox{45}{\line(0,1){2}}}\put(12.00,69.00){\line(1,0){2}}
\put(20,69){$\mbox{(if $a_{ij}=a_{ji}=0$, where}$}

\put(62.25,69.3){\rotatebox{45}{\line(1,0){2}}}
\put(62.25,69.3){\rotatebox{45}{\line(0,1){2}}}

\put(66,69){$\mbox{,}$}

\put(72.25,69.3){\rotatebox{45}{\line(1,0){2}}}
\put(72.25,69.3){\rotatebox{45}{\line(0,1){2}}}\put(72.00,69.00){\line(1,0){2}}
\put(76,69){$\mbox{$\in\{$}$}
\put(83,70){\circle{2}}
\put(85,69){$\mbox{,}$}
\put(88,70){\circle*{2}}
\put(90,69){$\mbox{,}$}
\put(93,70){\circle{2}}
\put(92.25,69.3){\rotatebox{45}{\line(1,0){2}}}
\put(92.25,69.3){\rotatebox{45}{\line(0,1){2}}}

\put(95,69){$\mbox{$\}),$}$}


\put(02.5,62){${\tiny{i}}$}
\put(3,60){\circle{2}}

\put(12.5,63){${\tiny{j}}$}
\put(13,60){\circle{2}}
\put(12.25,59.3){\rotatebox{45}{\line(1,0){2}}}
\put(12.25,59.3){\rotatebox{45}{\line(0,1){2}}}

\put(4,60){\line(1,0){8}}

\put(20,59){$\mbox{(if ${\frac {a_{ii}} 2}=a_{ij}=a_{ji}\ne 0$ and $a_{jj}=0$),}$}


\put(02.5,52){${\tiny{i}}$}
\put(3,50){\circle{2}}

\put(12.5,53){${\tiny{j}}$}
\put(13,50){\circle{2}}
\put(12.25,49.3){\rotatebox{45}{\line(1,0){2}}}
\put(12.25,49.3){\rotatebox{45}{\line(0,1){2}}}

\put(7,51){${\tiny{x}}$}
\put(4,50){\line(1,0){8}}

\put(20,49){$\mbox{(if ${\frac {a_{ii}} 2}=a_{ij}=a_{ji}=x\ne 0$ and $a_{jj}=0$),}$}


\put(02.5,42){${\tiny{i}}$}
\put(3,40){\circle{2}}

\put(12.5,43){${\tiny{j}}$}
\put(13,40){\circle{2}}

\put(4,40){\line(1,0){8}}

\put(20,39){$\mbox{(if $-2a_{ij}=-2a_{ji}=a_{ii}=a_{jj}\ne 0$),}$}


\put(02.5,32){${\tiny{i}}$}
\put(3,30){\circle{2}}\put(3,30){\circle*{1}}

\put(12.5,33){${\tiny{j}}$}
\put(13,30){\circle{2}}

\put(5,30.75){\line(1,0){7}}
\put(5,29.25){\line(1,0){7}}

\put(4,30){\rotatebox{45}{\line(1,0){3}}}
\put(4,28){\rotatebox{45}{\line(0,1){3}}}

\put(20,29){$\mbox{(if $-2a_{ij}=-2a_{ji}=2a_{ii}=a_{jj}\ne 0$, where}$}

\put(93,30){\circle{2}}\put(93,30){\circle*{1}}
\put(96,29){$\mbox{$\in\{$}$}
\put(103,30){\circle{2}}
\put(105,29){$\mbox{,}$}
\put(108,30){\circle*{2}}
\put(110,29){$\mbox{$\}),$}$}


\put(02.5,22){${\tiny{i}}$}
\put(3,20){\circle{2}}

\put(12.5,23){${\tiny{j}}$}
\put(13,20){\circle{2}}
\put(12.25,19.3){\rotatebox{45}{\line(1,0){2}}}
\put(12.25,19.3){\rotatebox{45}{\line(0,1){2}}}

\put(5,20.75){\line(1,0){7}}
\put(5,19.25){\line(1,0){7}}

\put(7,22){${\tiny{x}}$}

\put(4,20){\rotatebox{45}{\line(1,0){3}}}
\put(4,18){\rotatebox{45}{\line(0,1){3}}}

\put(20,19){$\mbox{(if $x=a_{ij}=a_{ji}=-a_{ii}\ne 0$ and $a_{jj}=0$),}$}


\put(02.5,12){${\tiny{i}}$}
\put(3,10){\circle{2}}

\put(4,10){\rotatebox{45}{\line(1,0){3}}}
\put(4,8){\rotatebox{45}{\line(0,1){3}}}

\put(5,11){\line(1,0){7.3}}
\put(4,10){\line(1,0){8}}
\put(5,9){\line(1,0){7.3}}

\put(12.5,13){${\tiny{j}}$}
\put(13,10){\circle{2}}

\put(20,9){$\mbox{(if $a_{jj}=3a_{ii}=-2a_{ij}=-2a_{ji}\ne 0$).}$}

\end{picture}

\begin{picture}(60,20)(3,-5)	
\put(0,0){$\wealsolet$}
\end{picture}

The following theorem is well-known.

\begin{theorem} {\rm{(\cite{Kac77}, \cite[Theorem~5.3.2]{VdL86})}} \label{theorem:superDynkin}
{\rm{(}}Recall \eqref{eqn:symCond}.{\rm{)}}
The following conditions {\rm{(i)}}-{\rm{(ii)}} are equivalent.
\newline
{\rm{(i)}} 
$\dim\frg(A,\fkI_{\mathrm{odd}})<\infty$,
and $\rR(A,\fkI_{\mathrm{odd}})$ is irreducible.
\newline
{\rm{(ii)}}
There exists $(A^\prime, \fkI^\prime_{\mathrm{odd}})$
such that $(A,\fkI_{\mathrm{odd}})\ddotsim(A^\prime, \fkI^\prime_{\mathrm{odd}})$
and $D(A^\prime, \fkI^\prime_{\mathrm{odd}})$ is one
of the Dynkin diagrams in {\rm{Fig.~1}}~-~{\rm{Fig.~4}} below.
{\rm{(}}In {\rm{Fig.~1}}~-~{\rm{Fig.~4}}, we omit to attach $i\in\fkI$ to each vertex.{\rm{)}}
\end{theorem}

\begin{picture}(40,90)(0,-5)	

\put(0,60){$\DynkinAn$} \put(70,60){$\DynkinEn$}
	
\put(0,50){$\DynkinBn$} \put(70,50){$\DynkinFfour$}

\put(0,40){$\DynkinCn$} \put(70,40){$\DynkinGtwo$}

\put(0,25){$\DynkinDn$}

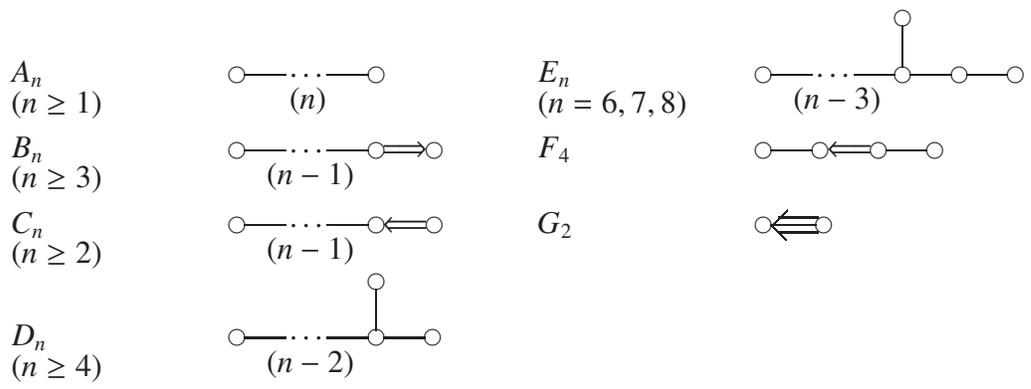
\put(00,25){{\rm{Fig.~2. Dynkin diagrams of finite dimensional complex simple Lie algebras}}}

\end{picture}

\begin{picture}(40,30)(0,-5)	
\put(20,25){$\Dynkinsloneone$}

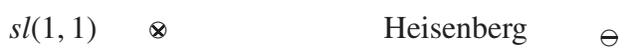
\put(00,25){{\rm{Fig.~3. Dynkin diagrams of Heisenberg Lie algebra and Lie superalgebra $sl(1|1)$}}}

\end{picture}

\begin{picture}(40,150)(0,-5)

\put(0,135){$\DynkinAmzero$}
	
\put(0,125){$\DynkinAmn$}

\put(0,115){$\DynkinBzeron$}

\put(0,105){$\DynkinBonem$}

\put(0,95){$\DynkinBmone$}

\put(0,85){$\DynkinBmn$}

\put(0,75){$\DynkinSCn$}

\put(0,60){$\DynkinDtwon$}

\put(0,45){$\DynkinDmone$}

\put(0,30){$\DynkinDmn$}

\put(0,20){$\DynkinSFfour$}

\put(0,10){$\DynkinSGthree$}

\put(0,0){$\DynkinDtwooneal$}

\put(-6,0){
$\begin{array}{l}
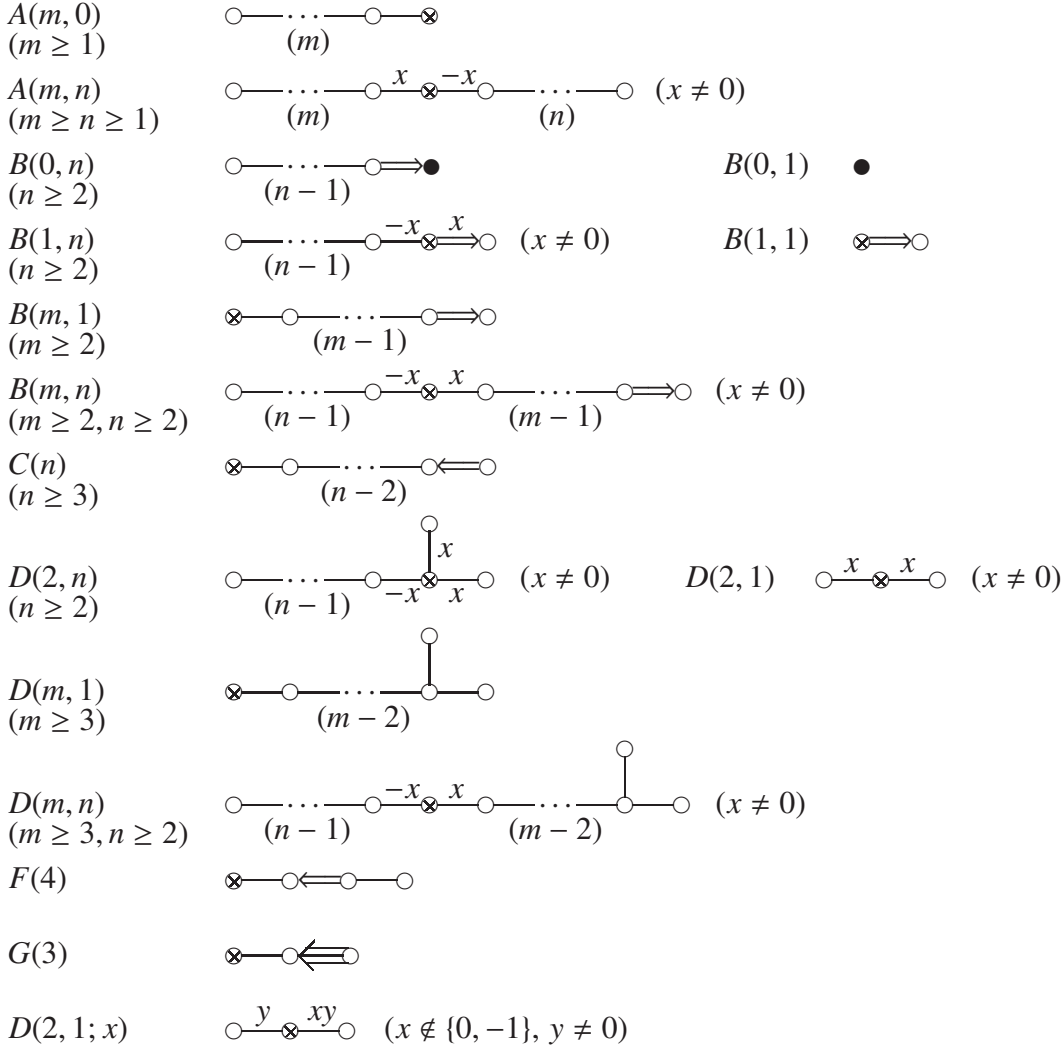
{\mbox{Fig.~4. Dynkin diagrams of finite dimensional complex Lie superalgebras}} \\
\quad\quad\quad {\mbox{of type $A$-$G$}}
\end{array}$
}

\end{picture}

\begin{remark}
Let $D(A,\fkI_{\mathrm{odd}})$ be $A(m,m)$ with $m\geq 1$.
Then $\frg(A,\fkI_{\mathrm{odd}})$ is not simple since 
it has a one dimensional center $Z$.
And $\frg(A,\fkI_{\mathrm{odd}})/Z$ is simple.
\end{remark}

We directly see the following lemma.

\begin{lemma} As a subset of $\fca$, 
$\rR^+(A,\fkI_{\mathrm{odd}})$ with
$D(A,\fkI_{\mathrm{odd}})$ being $A(m,n)$ {\rm{(}}resp. $B(m,n)${\rm{)}}
equals the one being $A_{m+n}$ {\rm{(}}reap. $B_{m+n}${\rm{)}}.
\end{lemma}

\begin{definition} \label{definition:defPhiZ}
Let $r\in\fkJ_{3,\infty}$, and let $Z$ be a subset of $\fkJ_{1,r-1}$
(i.e., $Z\in{\mathfrak{P}}(\fkJ_{1,r-1})$).
Assume $\fkI=\fkJ_{1,r}$.
Define the subset $\Phi_{r,Z}$ (resp. $\Psi_{r,Z}$, resp. $\Psi^\prime_{r,Z}$) of $\fca$ in the way that
defining
$\epsilon_i\in\fca\otimes_\bZ\bQ$ $(i\in\fkI)$ 
by
$\dal_i=\epsilon_i-\epsilon_{i+1}$ $(i\in\fkJ_{1,r-2})$,
$\dal_{r-1}=\epsilon_{r-1}-\epsilon_r$
(resp. $=\epsilon_{r-1}-\epsilon_r$, resp. $=2\epsilon_r$),
$\dal_r=\epsilon_{r-1}+\epsilon_r$
(resp. $=2\epsilon_r$, resp. $=\epsilon_{r-1}-\epsilon_r$),
and letting $X^\prime:=\{\epsilon_i-\epsilon_j, \epsilon_i+\epsilon_j|i,j\in\fkI,i<j\}$ and $X^{\prime\prime}:=\{2\epsilon_j|j\in Z\}$
(resp $=\{2\epsilon_j|j\in Z\cup\{r\}\}$,
resp $=\{2\epsilon_j|j\in Z\cup\{r\}\}$),
it is defined by 
$X^\prime\cup X^{\prime\prime}$.
The sets $\Phi_{r,Z}$, $\Psi_{r,Z}$ and $\Psi^\prime_{r,Z}$ are virtually the same as those of \cite[Definitions~3.9 and 3.17]{CH15}.
\end{definition}

We can directly see the following lemma.

\begin{lemma}
Keep the notation of Definition~{\rm{\ref{definition:defPhiZ}}}.
Then we have{\rm{:}}
\newline
{\rm{(1)}} $\Phi_{r,\emptyset}$ {\rm{(}}resp.  $\Psi_{r,\fkJ_{1,r-1}}$, resp. $\Psi_{r,\fkJ_{2,r-1}}$, 
resp. $\Phi_{r,\fkJ_{1,k}}$ with $k\in\fkJ_{1,r-2}$, 
resp. $\Phi_{3,\fkJ_{1,2}}${\rm{)}}
equals $\rR^+(A,\fkI_{\mathrm{odd}})$ with
$D(A,\fkI_{\mathrm{odd}})$ being $D_r$ {\rm{(}}resp. $C_r$, resp. $C(r)$, resp. $D(r-k,k)$, resp. $D(2,1;x)${\rm{)}} of
{\rm{Fig.~1}}~-~{\rm{Fig.~3}},
  where we let $D_3$ mean $A_3$ if $r=3$. 
\newline
{\rm{(2)}} Let $Z_1, Z_2\in{\mathfrak{P}}(\fkJ_{1,r-1})$.
Then{\rm{:}} 
\begin{equation*}
|Z_1|=|Z_2|\,\Leftrightarrow\,
\Phi_{r,Z_1} \ddotsim \Phi_{r,Z_2}\,\Leftrightarrow\,
\Psi_{r,Z_1} \ddotsim \Psi_{r,Z_2}\,\Leftrightarrow\,
\Psi_{r,Z_1} \ddotsim \Psi^\prime_{r,Z_2},
\end{equation*}
where we write $\rR^+(A,\fkI_{\mathrm{odd}}) \ddotsim\rR^+(A^\prime,\fkI^\prime_{\mathrm{odd}})$
if $(A,\fkI_{\mathrm{odd}}) \ddotsim (A^\prime,\fkI^\prime_{\mathrm{odd}})$.
Moreover $\Phi_{r,Z_1}\ddotsim \Psi_{r,Z_2}$
if and only if $|Z_1|=|Z_2|+1$.
\end{lemma}

\begin{definition}{\rm{(\cite{CH15})}}  \label{definition:DefSporadic}
Let $\fca$ be as above.
Let $n:={\mathrm{rank}}_\bZ\fca$.
Let $\rR$ be an irreducible 
FGRS over $\fca$. We call $\rR$ {\it{sporadic}}
if $\rR$ is not isomorphic to 
any irreducible FGRS
$\rR(A,\fkI_{\mathrm{odd}})$
over $\fca$ with $D(A,\fkI_{\mathrm{odd}})$
being $A_n$, 
$B_n$, $C_n$, $D_n$, $C(n)$ or $D(m,n-m)$. 
\end{definition}

As we write in the following theorem,
the classification of all FGRSs
has been achieved by \cite{CH15} and \cite{CH09b}.

\begin{theorem}{\rm{(\cite{CH15})}} \label{theorem:classGHlist}
There exists no 
sporadic irreducible FGRS
over $\fca$ if ${\mathrm{rank}}_\bZ\fca\geq 9$.
For $n\in\fkJ_{3,8}$,
the table {\rm{\cite[B.2.~Rank n]{CH15}}} 
is the classification list 
of all non-isomorphic sporadic irreducible FGRSs
over $\fca$ with ${\mathrm{rank}}_\bZ\fca=n$. {\rm{(}}See {\rm{\cite{CH09b}}}
for rank two cases.{\rm{)}}
\end{theorem}

\section{Hamiltonian cycles of finite generalized root systems} 
\label{section:HamCy}
\subsection{Rank three cases} \label{subsection:rankthree}
Let ${\hat{\fca}}$ be a free $\bZ$-module of rank three.
Let ${\hat{B}}=\{{\hat{\al}}_1,{\hat{\al}}_2,{\hat{\al}}_3\}$ be a $\bZ$-base of ${\hat{\fca}}$,
i.e., ${\hat{\fca}}=\bZ{\hat{\al}}_1\oplus\bZ{\hat{\al}}_2\oplus\bZ{\hat{\al}}_3$.
In this section, we let the symbol $1^x2^y3^z$ $(x,y,z\in\bN)$ of \cite[Appendix~A]{CH12} mean 
$x{\hat{\al}}_1+y{\hat{\al}}_2+z{\hat{\al}}_3$.
For $k\in\fkJ_{1,55}$, let ${\hat{\rR}}(k)$
be the irreducible FGRS over ${\hat{\fca}}$ defined by  
\cite[Nr.~$k$ of Appendix~A]{CH12}. 
For example, we have
\begin{equation*}
\begin{array}{lcl}
{\hat{\rR}}(1) & = & \{\pm{\hat{\al}}_1, \pm{\hat{\al}}_2, \pm{\hat{\al}}_3, 
\pm({\hat{\al}}_1+{\hat{\al}}_2), \pm({\hat{\al}}_1+{\hat{\al}}_3),
\pm({\hat{\al}}_1+{\hat{\al}}_2+{\hat{\al}}_3)\}, \\
{\hat{\rR}}(2) & = & \{\pm{\hat{\al}}_1, \pm{\hat{\al}}_2, \pm{\hat{\al}}_3, 
\pm({\hat{\al}}_1+{\hat{\al}}_2), \pm({\hat{\al}}_1+{\hat{\al}}_3), \pm({\hat{\al}}_2+{\hat{\al}}_3),
\pm({\hat{\al}}_1+{\hat{\al}}_2+{\hat{\al}}_3)\}.
\end{array}
\end{equation*}
If $k\in\fkJ_{6,55}$, ${\hat{\rR}}(k)$ is the same as the one of 
\cite[Nr.~$k-5$ of B.2.~Rank 3]{CH15},
see Theorem~\ref{theorem:classGHlist}.

In the following, we give one of the Hamltonian cycles of $\Gamma({\hat{\rR}}(k))$ for $k\in\fkJ_{1,55}$,
which we have found by the command \cite[FindHamiltonianCycle]{Mathe23} of Mathematica~13.3 (or its earlier version).


\newpage

\noindent
Nr.1\,(Length=24):
\begin{spacing}{0.5}
\noindent
{\tiny{\small{$1^{\mbox{${\hat{B}}$}}\cdot$}}\\
$s_3 s_1 s_3 s_1 s_2 s_1 s_3 s_2 s_3 s_1 s_3 s_1 s_3 s_2 s_3 s_1 s_3 s_1 s_3 s_2 s_3 s_1 s_2
  s_1$}
\end{spacing}
\vskip\baselineskip
\noindent
Nr.2\,(Length=32):
\begin{spacing}{0.5}
\noindent
{\tiny{\small{$1^{\mbox{${\hat{B}}$}}\cdot$}}\\
$s_3s_1s_2s_1s_3s_1s_2s_1s_2s_1s_3s_1s_2s_1s_3s_1s_3s_1s_2s_1s_3s_1s_2s_1s_2s_1s_3s_1s_2s_1s_3s_1$}
\end{spacing}
\vskip\baselineskip
\noindent
Nr.3\,(Length=40):
\begin{spacing}{0.5}
\noindent
{\tiny{\small{$1^{\mbox{${\hat{B}}$}}\cdot$}}\\
$s_1s_2s_1s_3s_2s_3s_1s_3s_1s_3s_2s_3s_1s_3s_2s_3s_1s_2s_1s_2
s_1s_3s_1s_2s_1s_2s_1s_3s_1s_2s_1s_3s_1s_3s_1s_2s_1s_2s_1s_3$}
\end{spacing}
\vskip\baselineskip
\noindent
Nr.4\,(Length=48):
\begin{spacing}{0.5}
\noindent
{\tiny{\small{$1^{\mbox{${\hat{B}}$}}\cdot$}}\\
$s_2s_1s_2s_1s_3s_1s_2s_3s_2s_1s_2s_1s_2s_1s_2s_3s_2s_1s_3s_1
s_2s_1s_2s_1s_2s_1s_3s_1s_2s_1s_2s_3s_2s_1s_2s_1s_3s_1s_3s_1 
s_2s_1s_2s_1s_2s_1s_3s_1$}
\end{spacing}
\vskip\baselineskip
\noindent
Nr.5\,(Length=48):
\begin{spacing}{0.5}
\noindent
{\tiny{\small{$1^{\mbox{${\hat{B}}$}}\cdot$}}\\
$s_2s_1s_3s_1s_2s_1s_2s_1s_2s_1s_3s_1s_2s_1s_2s_1s_2s_1s_3s_1
s_2s_1s_3s_1s_2s_1s_2s_1s_2s_1s_3s_1s_2s_3s_2s_1s_2s_1s_2s_1 
s_2s_3s_2s_1s_2s_1s_3s_1$}
\end{spacing}
\vskip\baselineskip
\noindent
Nr.6\,(Length=60):
\begin{spacing}{0.5}
\noindent
{\tiny{\small{$1^{\mbox{${\hat{B}}$}}\cdot$}}\\
$s_2s_1s_2s_1s_3s_1s_2s_1s_3s_1s_3s_1s_2s_1s_3s_1s_2s_1s_2s_1
s_3s_1s_3s_1s_2s_1s_2s_1s_3s_1s_3s_1s_3s_1s_2s_1s_3s_1s_2s_1 
s_2s_1s_3s_1s_2s_1s_3s_1s_3s_1$ \newline $s_2s_1s_2s_1s_3s_1s_3s_1s_2s_1$}
\end{spacing}
\vskip\baselineskip
\noindent
Nr.7\,(Length=60):
\begin{spacing}{0.5}
\noindent
{\tiny{\small{$1^{\mbox{${\hat{B}}$}}\cdot$}}\\
$s_3s_1s_2s_1s_3s_1s_3s_1s_2s_1s_2s_1s_3s_1s_2s_1s_2s_1s_3s_1
s_2s_1s_3s_1s_2s_1s_2s_3s_2s_1s_2s_3s_2s_1s_3s_1s_2s_1s_2s_1 
s_3s_1s_2s_1s_2s_1s_3s_2s_3s_1$ \newline $s_2s_3s_2s_1s_3s_2s_3s_1s_2s_1$}
\end{spacing}
\vskip\baselineskip
\noindent
Nr.8\,(Length=72):
\begin{spacing}{0.5}
\noindent
{\tiny{\small{$1^{\mbox{${\hat{B}}$}}\cdot$}}\\
$s_3s_1s_3s_1s_2s_1s_3s_1s_3s_1s_2s_1s_2s_1s_3s_1s_2s_1s_3s_1
s_3s_1s_2s_1s_2s_1s_2s_1s_3s_1s_2s_1s_3s_1s_3s_1s_2s_1s_2s_1 
s_2s_1s_3s_1s_2s_1s_3s_1s_3s_1$ \newline $s_2s_1s_3s_1s_2s_1s_2s_1s_3s_1
s_2s_1s_3s_1s_3s_1s_2s_1s_2s_1s_2s_1$}
\end{spacing}
\vskip\baselineskip
\noindent
Nr.9\,(Length=84):
\begin{spacing}{0.5}
\noindent
{\tiny{\small{$1^{\mbox{${\hat{B}}$}}\cdot$}}\\
$s_2s_1s_2s_1s_2s_1s_3s_1s_2s_1s_3s_1s_3s_1s_3s_1s_2s_1s_2s_1
s_3s_1s_3s_1s_3s_1s_2s_1s_3s_1s_2s_1s_2s_1s_3s_1s_2s_1s_3s_2 
s_3s_1s_3s_1s_3s_2s_3s_1s_3s_2$ \newline $s_3s_1s_3s_2s_3s_1s_3s_1s_3s_2
s_3s_1s_2s_1s_3s_1s_2s_1s_2s_1s_3s_2s_3s_1s_3s_1s_3s_1s_3s_2 
s_3s_1s_2s_1$}
\end{spacing}
\vskip\baselineskip
\noindent
Nr.10\,(Length=84):
\begin{spacing}{0.5}
\noindent
{\tiny{\small{$1^{\mbox{${\hat{B}}$}}\cdot$}}\\
$s_2s_1s_3s_1s_2s_1s_2s_1s_3s_1s_3s_1s_2s_1s_3s_1s_3s_1s_3s_1
s_2s_1s_3s_2s_3s_1s_3s_1s_3s_1s_3s_2s_3s_1s_3s_1s_2s_1s_3s_1 
s_2s_1s_3s_1s_2s_1s_2s_1s_3s_1$ \newline $s_2s_1s_3s_1s_3s_1s_2s_1s_2s_1
s_3s_1s_3s_1s_2s_1s_2s_1s_2s_1s_3s_1s_3s_1s_2s_1s_2s_1s_3s_1 
s_3s_1s_2s_1$}
\end{spacing}
\vskip\baselineskip
\noindent
Nr.11\,(Length=96):
\begin{spacing}{0.5}
\noindent
{\tiny{\small{$1^{\mbox{${\hat{B}}$}}\cdot$}}\\
$s_2s_1s_2s_1s_3s_1s_3s_1s_2s_1s_2s_1s_3s_1s_2s_1s_2s_1s_3s_1
s_2s_1s_2s_1s_3s_1s_2s_1s_3s_1s_2s_1s_2s_3s_2s_1s_3s_1s_2s_1 
s_2s_1s_3s_1s_2s_1s_2s_1s_3s_1$ \newline $s_2s_1s_2s_1s_3s_1s_2s_3s_1s_3
s_2s_3s_1s_3s_2s_3s_2s_1s_3s_1s_3s_1s_2s_3s_1s_3s_2s_1s_3s_1 
s_2s_1s_2s_1s_2s_1s_2s_1s_3s_1s_2s_3s_2s_1s_2s_3$}
\end{spacing}
\vskip\baselineskip
\noindent
Nr.12\,(Length=96):
\begin{spacing}{0.5}
\noindent
{\tiny{\small{$1^{\mbox{${\hat{B}}$}}\cdot$}}\\
$s_1s_2s_1s_3s_1s_3s_1s_2s_1s_2s_1s_3s_1s_2s_1s_2s_1s_2s_1s_3
s_1s_2s_1s_2s_1s_3s_1s_3s_1s_2s_1s_2s_3s_2s_1s_2s_1s_2s_3s_2 
s_1s_2s_1s_3s_1s_2s_1s_2s_1s_2$ \newline $s_1s_2s_1s_3s_1s_3s_1s_2s_1s_3
s_1s_2s_1s_2s_1s_3s_1s_2s_1s_2s_1s_2s_1s_3s_1s_2s_1s_3s_1s_2 
s_3s_2s_1s_2s_1s_2s_3s_2s_1s_2s_1s_3s_1s_2s_1s_2$}
\end{spacing}
\newpage
\noindent
Nr.13\,(Length=96):
\begin{spacing}{0.5}
\noindent
{\tiny{\small{$1^{\mbox{${\hat{B}}$}}\cdot$}}\\
$s_3s_1s_3s_1s_2s_1s_3s_1s_2s_1s_3s_1s_2s_1s_2s_1s_3s_1s_2s_1
s_2s_1s_3s_1s_3s_1s_3s_1s_2s_1s_3s_1s_3s_1s_2s_1s_3s_1s_2s_1 
s_2s_1s_3s_1s_2s_1s_3s_1s_3s_1$ \newline $s_3s_1s_2s_1s_3s_1s_2s_1s_3s_1
s_2s_1s_3s_1s_2s_1s_3s_1s_2s_1s_2s_1s_3s_1s_2s_1s_3s_1s_3s_1 
s_2s_1s_3s_1s_3s_1s_3s_1s_2s_1s_2s_1s_3s_1s_2s_1$}
\end{spacing}
\vskip\baselineskip
\noindent
Nr.14\,(Length=96):
\begin{spacing}{0.5}
\noindent
{\tiny{\small{$1^{\mbox{${\hat{B}}$}}\cdot$}}\\
$s_3s_1s_3s_1s_2s_1s_3s_1s_2s_1s_3s_1s_2s_1s_2s_1s_3s_1s_2s_1
s_2s_1s_3s_1s_3s_1s_3s_1s_2s_1s_3s_1s_3s_1s_2s_1s_3s_1s_2s_1 
s_2s_1s_3s_1s_2s_1s_3s_1s_3s_1$ \newline $s_3s_1s_2s_1s_3s_1s_2s_1s_3s_1
s_2s_1s_3s_1s_2s_1s_3s_1s_2s_1s_2s_1s_3s_1s_2s_1s_3s_1s_3s_1 
s_2s_1s_3s_1s_3s_1s_3s_1s_2s_1s_2s_1s_3s_1s_2s_1$}
\end{spacing}
\vskip\baselineskip
\noindent
Nr.15\,(Length=112):
\begin{spacing}{0.5}
\noindent
{\tiny{\small{$1^{\mbox{${\hat{B}}$}}\cdot$}}\\
$s_2s_1s_2s_1s_3s_1s_2s_1s_3s_1s_3s_1s_3s_1s_2s_1s_3s_1s_2s_1
s_3s_1s_2s_1s_2s_1s_2s_1s_2s_1s_3s_1s_3s_1s_2s_1s_3s_1s_3s_1 
s_2s_3s_2s_1s_2s_1s_2s_3s_2s_1$ \newline $s_2s_3s_1s_3s_2s_3s_1s_3s_2s_1
s_2s_1s_2s_3s_2s_1s_2s_1s_2s_3s_2s_1s_3s_1s_2s_1s_3s_1s_3s_1 
s_2s_1s_3s_1s_3s_1s_2s_1s_2s_1s_3s_1s_2s_3s_1s_3s_2s_1s_2s_1
$ \newline $s_3s_1s_2s_1s_2s_1s_3s_1s_3s_1s_2s_1$}
\end{spacing}
\vskip\baselineskip
\noindent
Nr.16\,(Length=128):
\begin{spacing}{0.5}
\noindent
{\tiny{\small{$1^{\mbox{${\hat{B}}$}}\cdot$}}\\
$s_1s_2s_1s_3s_1s_3s_1s_2s_1s_2s_1s_3s_1s_2s_1s_3s_1s_3s_1s_3
s_1s_2s_1s_2s_1s_3s_1s_2s_3s_2s_1s_2s_1s_2s_3s_2s_1s_2s_1s_3 
s_2s_3s_1s_3s_1s_3s_2s_3s_2s_3$ \newline $s_1s_3s_1s_3s_1s_3s_2s_3s_1s_3
s_2s_3s_1s_2s_1s_3s_1s_3s_1s_2s_1s_2s_1s_3s_1s_2s_1s_3s_1s_3 
s_1s_2s_1s_3s_1s_3s_1s_2s_1s_2s_1s_2s_1s_2s_1s_3s_1s_2s_1s_3
$ \newline $s_1s_2s_1s_3s_1s_3s_1s_3s_1s_2s_1s_2s_1s_3s_1s_2s_1s_3s_1s_2 
s_1s_3s_1s_2s_1s_2s_1s_2$}
\end{spacing}
\vskip\baselineskip
\noindent
Nr.17\,(Length=144):
\begin{spacing}{0.5}
\noindent
{\tiny{\small{$1^{\mbox{${\hat{B}}$}}\cdot$}}\\
$s_2s_3s_1s_3s_1s_3s_2s_3s_1s_3s_2s_1s_3s_1s_3s_1s_2s_1s_3s_1
s_3s_1s_3s_1s_2s_3s_1s_3s_2s_1s_3s_1s_2s_1s_3s_1s_2s_3s_2s_1 
s_2s_1s_2s_3s_2s_3s_1s_3s_2s_3$ \newline $s_1s_3s_2s_3s_2s_3s_1s_3s_1s_3
s_1s_3s_2s_1s_2s_3s_2s_1s_3s_1s_2s_1s_3s_1s_2s_1s_2s_3s_1s_3 
s_2s_1s_2s_1s_2s_3s_2s_1s_3s_1s_2s_3s_2s_1s_2s_1s_2s_3s_2s_1
$ \newline $s_2s_3s_1s_3s_2s_3s_1s_3s_2s_3s_1s_3s_2s_3s_1s_3s_2s_3s_1s_3 
s_1s_3s_1s_3s_2s_3s_2s_3s_1s_3s_2s_1s_2s_3s_2s_1s_2s_3s_1s_3
s_2s_3s_1s_3$}
\end{spacing}
\vskip\baselineskip
\noindent
Nr.18\,(Length=144):
\begin{spacing}{0.5}
\noindent
{\tiny{\small{$1^{\mbox{${\hat{B}}$}}\cdot$}}\\
$s_2s_1s_3s_1s_3s_1s_2s_1s_3s_1s_2s_1s_2s_1s_2s_1s_3s_1s_2s_3
s_1s_3s_2s_1s_3s_1s_2s_1s_2s_1s_3s_1s_3s_1s_2s_1s_2s_3s_1s_3 
s_2s_1s_3s_1s_2s_3s_1s_3s_2s_3$ \newline $s_2s_3s_1s_2s_1s_2s_1s_3s_2s_1
s_2s_3s_1s_2s_1s_2s_1s_3s_2s_1s_2s_3s_1s_2s_1s_2s_1s_3s_2s_3 
s_2s_3s_1s_3s_2s_1s_3s_1s_2s_3s_1s_3s_2s_1s_2s_1s_3s_1s_2s_1
$ \newline $s_2s_1s_2s_1s_2s_1s_3s_1s_3s_1s_2s_1s_3s_1s_2s_3s_2s_1s_3s_1 
s_2s_1s_2s_1s_2s_1s_3s_1s_2s_3s_2s_1s_3s_1s_2s_1s_3s_1s_3s_1
s_2s_1s_2s_1$}
\end{spacing}
\vskip\baselineskip
\noindent
Nr.19\,(Length=160):
\begin{spacing}{0.5}
\noindent
{\tiny{\small{$1^{\mbox{${\hat{B}}$}}\cdot$}}\\
$s_1s_2s_1s_3s_1s_3s_1s_2s_1s_2s_1s_3s_1s_2s_1s_2s_3s_2s_1s_2
s_1s_2s_1s_2s_3s_2s_1s_3s_1s_2s_3s_2s_1s_2s_1s_2s_1s_2s_3s_2 
s_1s_3s_1s_3s_1s_2s_1s_2s_1s_3$ \newline $s_1s_3s_1s_3s_1s_2s_1s_2s_1s_2
s_3s_2s_1s_2s_1s_2s_3s_2s_1s_2s_1s_3s_1s_2s_1s_2s_1s_3s_1s_3 
s_1s_2s_3s_2s_1s_2s_1s_2s_1s_2s_1s_2s_3s_2s_1s_2s_1s_2s_3s_1
$ \newline $s_3s_2s_3s_1s_3s_2s_1s_3s_1s_3s_1s_2s_3s_1s_3s_2s_1s_2s_1s_3 
s_1s_2s_3s_2s_1s_2s_1s_2s_3s_2s_1s_2s_1s_3s_1s_2s_1s_3s_1s_3
s_1s_2s_3s_2s_1s_2s_1s_2s_1s_2$ \newline $s_3s_2s_3s_1s_3s_2s_1s_3s_1s_2$}
\end{spacing}
\vskip\baselineskip
\noindent
Nr.20\,(Length=160):
\begin{spacing}{0.5}
\noindent
{\tiny{\small{$1^{\mbox{${\hat{B}}$}}\cdot$}}\\
$s_2s_1s_2s_1s_3s_1s_3s_1s_2s_1s_2s_1s_3s_1s_2s_1s_2s_1s_2s_1
s_3s_1s_2s_1s_2s_1s_3s_1s_3s_1s_2s_1s_2s_3s_2s_1s_2s_1s_2s_3 
s_2s_1s_3s_1s_2s_1s_2s_1s_3s_1$ \newline $s_3s_1s_2s_1s_2s_1s_3s_1s_3s_1
s_3s_1s_2s_1s_2s_3s_2s_1s_2s_1s_2s_3s_2s_1s_2s_1s_2s_3s_2s_1 
s_2s_1s_2s_1s_3s_1s_2s_1s_3s_1s_2s_3s_2s_1s_2s_1s_2s_3s_1s_3
$ \newline $s_2s_3s_1s_3s_2s_3s_1s_3s_2s_1s_2s_1s_2s_3s_2s_1s_2s_1s_3s_1 
s_3s_1s_3s_1s_2s_1s_2s_1s_3s_1s_3s_1s_2s_1s_3s_1s_2s_1s_2s_1
s_2s_1s_3s_1s_2s_1s_2s_1s_3s_1$ \newline $s_2s_1s_2s_1s_3s_1s_3s_1s_2s_1$}
\end{spacing}
\vskip\baselineskip
\noindent
Nr.21\,(Length=160):
\begin{spacing}{0.5}
\noindent
{\tiny{\small{$1^{\mbox{${\hat{B}}$}}\cdot$}}\\
$s_1s_3s_1s_2s_1s_3s_1s_2s_1s_3s_1s_2s_1s_2s_1s_2s_1s_3s_1s_2
s_1s_2s_1s_3s_1s_3s_1s_3s_1s_2s_1s_3s_1s_2s_1s_2s_1s_3s_1s_2 
s_3s_2s_1s_2s_1s_3s_1s_2s_1s_3$ \newline $s_1s_2s_1s_2s_1s_3s_1s_3s_1s_2
s_1s_2s_1s_3s_1s_2s_1s_3s_1s_3s_1s_2s_1s_3s_1s_2s_1s_3s_1s_2 
s_1s_2s_1s_3s_1s_2s_1s_2s_1s_3s_1s_2s_1s_3s_1s_2s_3s_2s_1s_3
$ \newline $s_1s_2s_1s_2s_1s_3s_1s_2s_1s_2s_1s_2s_1s_3s_1s_2s_1s_3s_1s_2 
s_3s_2s_1s_3s_1s_2s_1s_3s_1s_2s_1s_3s_1s_3s_1s_2s_1s_3s_1s_2
s_1s_2s_1s_3s_1s_3s_1s_2s_1s_2$ \newline $s_1s_3s_1s_2s_1s_3s_1s_2s_1s_2$}
\end{spacing}
\vskip\baselineskip
\noindent
Nr.22\,(Length=180):
\begin{spacing}{0.5}
\noindent
{\tiny{\small{$1^{\mbox{${\hat{B}}$}}\cdot$}}\\
$s_3s_1s_3s_2s_3s_1s_3s_2s_1s_2s_3s_2s_1s_2s_1s_2s_3s_2s_1s_2
s_3s_1s_3s_2s_3s_1s_3s_2s_3s_1s_3s_2s_1s_3s_1s_3s_1s_2s_1s_3 
s_1s_2s_1s_2s_1s_3s_1s_2s_1s_2$ \newline $s_1s_2s_1s_3s_1s_2s_3s_2s_1s_3
s_1s_2s_1s_2s_1s_3s_1s_3s_1s_2s_1s_2s_1s_3s_1s_2s_3s_2s_1s_3 
s_1s_2s_1s_2s_3s_2s_3s_2s_1s_2s_3s_2s_1s_2s_3s_2s_1s_2s_1s_2
$ \newline $s_3s_2s_1s_2s_3s_2s_1s_2s_1s_2s_3s_2s_1s_2s_1s_3s_1s_2s_1s_2 
s_1s_3s_1s_2s_3s_2s_1s_2s_1s_2s_3s_2s_1s_3s_1s_2s_1s_2s_1s_3
s_1s_2s_1s_2s_3s_2s_1s_2s_1s_2$ \newline $s_3s_2s_1s_2s_3s_2s_1s_2s_3s_2 
s_1s_3s_1s_2s_1s_2s_1s_3s_1s_2s_1s_3s_1s_3s_1s_2s_3s_1s_3s_2$}
\end{spacing}
\vskip\baselineskip
\noindent
Nr.23\,(Length=180):
\begin{spacing}{0.5}
\noindent
{\tiny{\small{$1^{\mbox{${\hat{B}}$}}\cdot$}}\\
$s_3s_2s_3s_2s_1s_2s_3s_2s_1s_2s_3s_2s_1s_2s_3s_2s_3s_1s_3s_2
s_3s_1s_3s_2s_1s_2s_3s_2s_1s_2s_1s_3s_1s_2s_1s_3s_1s_3s_2s_1 
s_2s_3s_2s_3s_2s_1s_2s_3s_1s_3$ \newline $s_1s_3s_2s_3s_1s_2s_1s_3s_1s_2
s_1s_3s_1s_3s_1s_3s_1s_2s_1s_3s_1s_2s_1s_3s_1s_3s_1s_3s_1s_3 
s_1s_2s_1s_3s_2s_3s_1s_3s_2s_1s_2s_3s_2s_1s_3s_1s_2s_1s_3s_1
$ \newline $s_2s_3s_1s_3s_2s_1s_3s_1s_3s_1s_2s_1s_3s_1s_2s_3s_1s_3s_1s_3 
s_2s_3s_1s_2s_1s_3s_1s_2s_1s_3s_1s_2s_3s_1s_3s_2s_1s_3s_1s_3
s_1s_2s_1s_3s_1s_2s_1s_3s_1s_3$ \newline $s_2s_1s_2s_3s_1s_2s_1s_3s_1s_2 
s_1s_3s_2s_1s_2s_3s_1s_2s_1s_3s_1s_3s_1s_3s_1s_2s_1s_3s_1s_2$}
\end{spacing}
\vskip\baselineskip
\noindent
Nr.24\,(Length=200):
\begin{spacing}{0.5}
\noindent
{\tiny{\small{$1^{\mbox{${\hat{B}}$}}\cdot$}}\\
$s_3s_2s_1s_3s_1s_2s_1s_3s_1s_3s_1s_2s_1s_3s_1s_3s_1s_3s_2s_3
s_2s_3s_1s_3s_1s_3s_2s_3s_1s_3s_1s_3s_2s_3s_1s_2s_1s_3s_1s_2 
s_1s_3s_1s_2s_3s_1s_3s_2s_1s_3$ \newline $s_1s_3s_1s_2s_1s_3s_2s_3s_1s_3
s_2s_3s_1s_3s_1s_3s_1s_3s_2s_3s_1s_2s_1s_3s_1s_2s_3s_1s_3s_2 
s_3s_1s_3s_2s_1s_2s_3s_2s_1s_2s_1s_3s_1s_2s_1s_3s_1s_2s_1s_3
$ \newline $s_1s_3s_1s_2s_1s_3s_2s_3s_2s_1s_2s_1s_2s_3s_2s_3s_1s_3s_1s_3 
s_1s_3s_2s_3s_1s_2s_1s_3s_1s_2s_1s_2s_1s_3s_1s_2s_1s_3s_1s_3
s_1s_3s_1s_2s_1s_3s_1s_2s_1s_3$ \newline $s_1s_3s_1s_3s_2s_1s_2s_3s_1s_2 
s_1s_3s_2s_1s_2s_1s_2s_3s_2s_1s_2s_3s_1s_3s_2s_3s_2s_3s_1s_2
s_1s_3s_1s_3s_1s_3s_1s_2s_1s_3s_1s_2s_1s_3s_1s_3s_1s_2s_1s_2$}
\end{spacing}
\vskip\baselineskip
\noindent
Nr.25\,(Length=192):
\begin{spacing}{0.5}
\noindent
{\tiny{\small{$1^{\mbox{${\hat{B}}$}}\cdot$}}\\
$s_1s_2s_1s_2s_1s_3s_1s_2s_1s_2s_1s_2s_1s_3s_1s_2s_1s_2s_3s_2
s_1s_3s_1s_2s_1s_2s_1s_2s_1s_3s_1s_3s_1s_2s_1s_2s_1s_3s_1s_2 
s_1s_2s_1s_2s_1s_3s_1s_2s_1s_2$ \newline $s_1s_3s_1s_3s_1s_2s_1s_2s_1s_2
s_1s_3s_1s_2s_1s_2s_3s_2s_1s_2s_1s_2s_1s_2s_3s_2s_1s_2s_1s_2 
s_1s_3s_1s_2s_1s_3s_1s_2s_1s_2s_1s_2s_3s_2s_1s_2s_1s_2s_1s_2
$ \newline $s_3s_2s_3s_2s_1s_2s_1s_2s_1s_3s_1s_2s_1s_3s_1s_2s_1s_2s_1s_2 
s_1s_3s_1s_3s_1s_2s_1s_2s_1s_3s_1s_2s_1s_2s_1s_2s_1s_3s_1s_2
s_1s_2s_1s_3s_1s_3s_1s_2s_1s_2$ \newline $s_1s_2s_1s_3s_1s_2s_1s_3s_1s_2 
s_1s_2s_1s_2s_1s_2s_1s_2s_1s_3s_1s_2s_1s_2s_1s_2s_1s_3s_1s_2
s_1s_2s_1s_2s_1s_3s_1s_2s_1s_3s_1s_2$}
\end{spacing}
\vskip\baselineskip
\noindent
Nr.26\,(Length=200):
\begin{spacing}{0.5}
\noindent
{\tiny{\small{$1^{\mbox{${\hat{B}}$}}\cdot$}}\\
$s_2s_3s_1s_3s_2s_3s_1s_3s_2s_3s_1s_3s_2s_1s_2s_3s_2s_1s_3s_1
s_2s_1s_3s_1s_3s_1s_2s_3s_2s_3s_2s_3s_2s_1s_2s_3s_1s_3s_2s_3 
s_1s_3s_1s_3s_2s_3s_1s_3s_1s_2$ \newline $s_1s_3s_2s_3s_1s_3s_2s_3s_1s_3
s_2s_3s_1s_3s_2s_1s_2s_3s_2s_1s_2s_3s_1s_3s_2s_3s_1s_3s_1s_3 
s_2s_3s_1s_2s_1s_3s_1s_2s_1s_3s_1s_3s_2s_1s_2s_3s_1s_2s_1s_3
$ \newline $s_2s_1s_2s_3s_2s_3s_2s_1s_3s_1s_2s_1s_3s_1s_2s_3s_2s_1s_2s_3 
s_1s_3s_2s_3s_1s_2s_1s_3s_1s_2s_1s_3s_1s_3s_1s_2s_1s_3s_2s_3
s_1s_3s_1s_3s_1s_3s_2s_3s_1s_2$ \newline $s_1s_3s_1s_2s_1s_2s_1s_3s_1s_2 
s_1s_3s_1s_3s_2s_1s_2s_3s_1s_3s_1s_2s_1s_3s_2s_3s_1s_3s_1s_2
s_1s_3s_1s_3s_1s_2s_1s_3s_2s_3s_1s_3s_1s_2s_1s_3s_1s_2s_1s_3$}
\end{spacing}
\vskip\baselineskip
\noindent
Nr.27\,(Length=200):
\begin{spacing}{0.5}
\noindent
{\tiny{\small{$1^{\mbox{${\hat{B}}$}}\cdot$}}\\
$s_2s_3s_2s_3s_2s_1s_2s_3s_2s_1s_2s_3s_2s_1s_2s_3s_2s_3s_1s_3
s_2s_3s_1s_3s_2s_3s_2s_1s_3s_1s_3s_1s_3s_1s_2s_3s_1s_3s_2s_1 
s_3s_1s_3s_1s_2s_1s_3s_1s_2s_1$ \newline $s_3s_2s_1s_2s_3s_1s_3s_1s_2s_1
s_3s_1s_2s_1s_3s_1s_3s_1s_2s_3s_2s_1s_2s_1s_2s_3s_2s_1s_3s_1 
s_2s_1s_3s_1s_2s_1s_3s_2s_3s_1s_3s_1s_3s_2s_1s_2s_3s_2s_1s_2
$ \newline $s_3s_1s_3s_2s_3s_2s_3s_1s_3s_2s_1s_2s_3s_2s_1s_2s_3s_2s_3s_2 
s_1s_2s_3s_2s_3s_1s_2s_1s_3s_1s_3s_1s_2s_3s_2s_1s_2s_3s_2s_1
s_3s_1s_2s_1s_3s_1s_3s_1s_3s_1$ \newline $s_2s_1s_3s_2s_3s_2s_3s_1s_3s_1 
s_3s_2s_3s_2s_3s_1s_3s_2s_3s_1s_2s_1s_3s_1s_2s_1s_2s_1s_3s_1
s_3s_1s_2s_1s_2s_1s_3s_1s_2s_1s_2s_1s_3s_1s_3s_1s_2s_1s_2s_1$}
\end{spacing}
\vskip\baselineskip
\noindent
Nr.28\,(Length=192):
\begin{spacing}{0.5}
\noindent
{\tiny{\small{$1^{\mbox{${\hat{B}}$}}\cdot$}}\\
$s_1s_3s_1s_2s_1s_3s_1s_2s_1s_2s_1s_2s_3s_2s_1s_2s_3s_2s_1s_2
s_1s_2s_1s_2s_1s_2s_3s_2s_1s_2s_3s_2s_1s_3s_1s_2s_1s_2s_1s_2 
s_1s_3s_1s_2s_3s_2s_1s_2s_3s_2$ \newline $s_1s_2s_1s_2s_1s_2s_1s_2s_3s_2
s_1s_2s_3s_2s_1s_3s_1s_2s_1s_2s_3s_2s_1s_2s_3s_2s_1s_2s_3s_2 
s_1s_2s_3s_2s_1s_2s_1s_2s_1s_2s_1s_2s_3s_2s_3s_2s_1s_2s_1s_2
$ \newline $s_1s_2s_1s_2s_3s_2s_1s_2s_3s_2s_1s_2s_3s_2s_1s_2s_3s_2s_1s_2 
s_1s_2s_1s_2s_1s_2s_3s_2s_1s_2s_3s_2s_3s_2s_1s_2s_1s_2s_1s_2
s_3s_2s_1s_2s_3s_2s_3s_2s_1s_2$ \newline $s_1s_2s_3s_2s_1s_2s_3s_2s_1s_2 
s_1s_3s_1s_2s_1s_3s_1s_2s_1s_2s_1s_2s_1s_2s_1s_3s_1s_2s_1s_2
s_1s_2s_1s_2s_1s_3s_1s_2s_1s_2s_1s_2$}
\end{spacing}
\vskip\baselineskip
\noindent
Nr.29\,(Length=216):
\begin{spacing}{0.5}
\noindent
{\tiny{\small{$1^{\mbox{${\hat{B}}$}}\cdot$}}\\
$s_3s_1s_3s_1s_2s_1s_2s_1s_3s_1s_2s_1s_2s_1s_3s_1s_2s_1s_2s_3
s_1s_3s_1s_3s_2s_1s_3s_1s_2s_3s_1s_3s_1s_3s_2s_1s_2s_1s_3s_2 
s_3s_1s_3s_2s_3s_1s_3s_2s_3s_1$ \newline $s_2s_3s_2s_1s_2s_3s_1s_3s_1s_3
s_2s_1s_2s_1s_2s_3s_2s_1s_3s_1s_2s_1s_2s_1s_3s_1s_3s_1s_2s_1 
s_2s_1s_3s_1s_3s_1s_2s_1s_2s_1s_3s_1s_2s_1s_2s_1s_3s_1s_2s_1
$ \newline $s_2s_1s_3s_1s_2s_1s_3s_1s_3s_1s_2s_3s_2s_1s_3s_1s_2s_1s_2s_1 
s_2s_3s_2s_3s_2s_1s_2s_1s_2s_1s_2s_3s_1s_3s_2s_3s_1s_3s_2s_1
s_3s_1s_2s_3s_1s_3s_1s_3s_2s_1$ \newline $s_2s_1s_3s_1s_2s_3s_2s_1s_3s_1 
s_2s_1s_3s_1s_3s_1s_2s_1s_2s_1s_3s_1s_2s_1s_2s_1s_3s_1s_2s_1
s_2s_3s_1s_3s_2s_3s_1s_3s_2s_1s_2s_1s_2s_1s_2s_3s_2s_3s_2s_1 
$ \newline $s_2s_1s_3s_1s_2s_1s_2s_1s_3s_1s_2s_1s_3s_1s_2s_1$}
\end{spacing}
\vskip\baselineskip
\noindent
Nr.30\,(Length=220):
\begin{spacing}{0.5}
\noindent
{\tiny{\small{$1^{\mbox{${\hat{B}}$}}\cdot$}}\\
$s_3s_2s_1s_2s_3s_2s_1s_3s_1s_2s_1s_3s_1s_2s_1s_3s_2s_3s_1s_3
s_1s_3s_2s_3s_1s_2s_1s_3s_1s_2s_1s_3s_1s_2s_3s_1s_3s_2s_1s_3 
s_1s_3s_1s_2s_1s_3s_1s_3s_1s_2$ \newline $s_3s_2s_3s_1s_3s_2s_3s_1s_3s_2
s_1s_2s_3s_2s_1s_2s_1s_2s_1s_3s_1s_2s_1s_3s_1s_2s_1s_3s_1s_2 
s_1s_3s_1s_3s_1s_3s_1s_3s_1s_2s_1s_3s_1s_3s_1s_2s_1s_2s_1s_3
$ \newline $s_1s_2s_1s_3s_1s_2s_1s_3s_1s_3s_1s_3s_1s_2s_1s_3s_1s_2s_1s_3 
s_1s_2s_1s_2s_1s_3s_1s_3s_1s_3s_1s_2s_1s_3s_1s_2s_1s_3s_1s_3
s_1s_3s_1s_3s_1s_2s_1s_3s_2s_3$ \newline $s_1s_3s_2s_3s_1s_3s_2s_3s_1s_3 
s_2s_3s_1s_3s_1s_3s_2s_3s_1s_2s_1s_3s_1s_2s_1s_3s_1s_2s_1s_2
s_1s_3s_1s_3s_1s_2s_1s_3s_2s_3s_1s_3s_1s_3s_1s_3s_2s_3s_1s_2 
$ \newline $s_1s_3s_1s_3s_1s_3s_1s_2s_1s_3s_1s_3s_1s_2s_1s_2s_1s_3s_1s_2$}
\end{spacing}
\newpage
\noindent
Nr.31\,(Length=220):
\begin{spacing}{0.5}
\noindent
{\tiny{\small{$1^{\mbox{${\hat{B}}$}}\cdot$}}\\
$s_1s_2s_1s_3s_1s_3s_1s_2s_1s_2s_1s_3s_1s_2s_1s_3s_1s_3s_1s_2
s_1s_3s_1s_3s_1s_2s_1s_3s_1s_2s_1s_3s_1s_3s_1s_2s_1s_2s_1s_2 
s_1s_2s_1s_3s_1s_2s_1s_2s_1s_2$ \newline $s_1s_3s_1s_2s_1s_3s_1s_3s_1s_2
s_1s_3s_1s_2s_1s_2s_1s_3s_1s_2s_1s_3s_1s_3s_1s_2s_1s_3s_1s_2 
s_1s_2s_1s_2s_1s_3s_1s_3s_1s_2s_1s_2s_1s_2s_1s_3s_1s_3s_1s_2
$ \newline $s_1s_2s_1s_3s_1s_2s_1s_2s_1s_3s_1s_3s_1s_2s_1s_3s_1s_2s_1s_3 
s_2s_3s_1s_3s_2s_3s_1s_3s_1s_3s_1s_3s_2s_3s_2s_3s_1s_3s_2s_3
s_1s_2s_1s_3s_1s_3s_1s_2s_1s_3$ \newline $s_1s_2s_1s_3s_1s_3s_1s_2s_1s_3 
s_1s_2s_1s_3s_1s_3s_1s_2s_1s_3s_1s_2s_1s_3s_1s_3s_1s_2s_1s_3
s_1s_3s_1s_2s_1s_3s_1s_2s_1s_3s_2s_3s_1s_3s_2s_3s_2s_3s_1s_2 
$ \newline $s_1s_3s_1s_3s_1s_2s_1s_3s_1s_2s_1s_2s_1s_3s_1s_3s_1s_2s_1s_2$}
\end{spacing}
\vskip\baselineskip
\noindent
Nr.32\,(Length=240):
\begin{spacing}{0.5}
\noindent
{\tiny{\small{$1^{\mbox{${\hat{B}}$}}\cdot$}}\\
$s_2s_3s_1s_3s_1s_3s_2s_3s_1s_3s_1s_3s_2s_3s_1s_3s_2s_1s_2s_1
s_2s_3s_2s_1s_3s_1s_2s_3s_2s_1s_3s_2s_3s_1s_2s_3s_2s_1s_3s_1 
s_2s_3s_2s_1s_2s_3s_2s_1s_3s_1$ \newline $s_3s_2s_3s_2s_1s_3s_1s_2s_3s_2
s_1s_2s_3s_1s_3s_2s_3s_2s_3s_1s_3s_2s_1s_2s_3s_2s_1s_3s_1s_2 
s_3s_1s_3s_2s_1s_3s_1s_3s_1s_2s_3s_1s_3s_2s_1s_3s_1s_3s_1s_2
$ \newline $s_3s_1s_3s_2s_1s_3s_1s_2s_3s_2s_1s_2s_3s_1s_3s_2s_3s_2s_3s_1 
s_3s_2s_1s_2s_3s_2s_1s_3s_1s_2s_3s_2s_3s_1s_2s_1s_3s_1s_3s_1
s_2s_1s_2s_1s_3s_1s_3s_1s_3s_1$ \newline $s_2s_3s_1s_3s_2s_1s_3s_1s_3s_1 
s_2s_1s_3s_1s_3s_1s_2s_3s_1s_3s_2s_1s_3s_1s_3s_1s_3s_1s_2s_3
s_1s_3s_2s_1s_3s_1s_3s_1s_2s_1s_3s_1s_3s_1s_2s_3s_1s_3s_2s_1 
$ \newline $s_3s_1s_3s_1s_3s_1s_2s_3s_1s_3s_2s_1s_3s_1s_3s_1s_2s_1s_3s_1
s_3s_1s_2s_3s_2s_3s_1s_3s_2s_3s_1s_3s_1s_3s_2s_3s_1s_3s_1s_3$}
\end{spacing}
\vskip\baselineskip
\noindent
Nr.33\,(Length=240):
\begin{spacing}{0.5}
\noindent
{\tiny{\small{$1^{\mbox{${\hat{B}}$}}\cdot$}}\\
$s_1s_2s_1s_2s_1s_3s_1s_2s_1s_2s_1s_3s_1s_2s_3s_2s_1s_3s_1s_2
s_1s_2s_1s_2s_1s_3s_1s_3s_1s_2s_1s_2s_3s_2s_1s_2s_1s_2s_1s_2 
s_3s_2s_1s_2s_1s_3s_1s_3s_1s_2$ \newline $s_1s_2s_1s_2s_1s_3s_1s_2s_1s_2
s_1s_3s_1s_3s_1s_2s_1s_3s_1s_3s_2s_3s_1s_2s_1s_2s_3s_2s_1s_2 
s_1s_2s_3s_2s_1s_2s_1s_3s_2s_3s_1s_3s_1s_2s_1s_3s_1s_3s_1s_2
$ \newline $s_3s_1s_3s_2s_3s_1s_3s_2s_3s_1s_3s_2s_1s_2s_3s_2s_1s_3s_1s_3 
s_1s_2s_1s_3s_1s_3s_2s_3s_1s_2s_1s_2s_1s_3s_1s_3s_1s_2s_1s_2
s_1s_2s_1s_3s_1s_3s_1s_2s_1s_2$ \newline $s_1s_3s_2s_3s_1s_3s_1s_2s_1s_3 
s_1s_3s_1s_2s_3s_2s_1s_2s_3s_2s_1s_2s_1s_2s_3s_2s_1s_2s_1s_2
s_1s_3s_1s_2s_1s_2s_1s_2s_1s_3s_1s_3s_1s_2s_1s_3s_1s_3s_1s_2 
$ \newline $s_1s_2s_1s_3s_1s_3s_1s_2s_1s_3s_1s_3s_1s_2s_1s_2s_1s_2s_1s_3
s_1s_3s_1s_2s_1s_2s_1s_2s_1s_3s_1s_2s_3s_2s_1s_2s_1s_2s_3s_2$}
\end{spacing}
\vskip\baselineskip
\noindent
Nr.34\,(Length=240):
\begin{spacing}{0.5}
\noindent
{\tiny{\small{$1^{\mbox{${\hat{B}}$}}\cdot$}}\\
$s_1s_2s_1s_3s_1s_3s_1s_2s_1s_2s_1s_3s_1s_2s_1s_3s_1s_2s_1s_2
s_1s_2s_1s_2s_1s_2s_1s_3s_1s_3s_1s_2s_1s_2s_1s_3s_1s_3s_1s_2 
s_1s_2s_1s_3s_1s_3s_1s_2s_1s_2$ \newline $s_1s_3s_1s_2s_1s_2s_1s_3s_1s_2
s_1s_3s_1s_2s_1s_2s_1s_3s_1s_2s_1s_2s_1s_3s_1s_3s_1s_2s_1s_2 
s_1s_3s_1s_2s_1s_3s_1s_3s_1s_2s_1s_3s_1s_2s_1s_2s_1s_3s_1s_3
$ \newline $s_1s_3s_1s_2s_1s_2s_1s_3s_1s_3s_1s_2s_1s_2s_1s_3s_1s_2s_1s_3 
s_1s_3s_1s_2s_1s_3s_1s_2s_1s_2s_1s_2s_1s_2s_1s_2s_1s_3s_1s_3
s_1s_2s_1s_2s_1s_3s_1s_3s_1s_2$ \newline $s_1s_2s_1s_3s_1s_3s_1s_2s_1s_2 
s_1s_3s_1s_2s_1s_2s_1s_3s_1s_2s_1s_2s_1s_3s_1s_3s_1s_2s_1s_2
s_1s_2s_1s_3s_1s_2s_1s_2s_1s_3s_1s_3s_1s_2s_1s_2s_1s_3s_1s_3 
$ \newline $s_1s_3s_1s_2s_1s_2s_1s_3s_1s_2s_1s_3s_1s_3s_1s_2s_1s_3s_1s_2
s_1s_2s_1s_3s_1s_3s_1s_2s_1s_2s_1s_3s_1s_2s_1s_2s_1s_3s_1s_3$}
\end{spacing}
\vskip\baselineskip
\noindent
Nr.35\,(Length=264):
\begin{spacing}{0.5}
\noindent
{\tiny{\small{$1^{\mbox{${\hat{B}}$}}\cdot$}}\\
$s_2s_3s_1s_3s_1s_2s_1s_2s_1s_3s_1s_2s_1s_3s_1s_2s_1s_3s_1s_3
s_1s_2s_1s_3s_1s_2s_1s_2s_1s_2s_1s_3s_1s_2s_1s_3s_1s_3s_1s_2 
s_1s_3s_1s_2s_1s_3s_1s_2s_1s_3$ \newline $s_1s_3s_1s_2s_1s_2s_1s_3s_1s_2
s_1s_2s_1s_3s_1s_2s_1s_2s_1s_3s_2s_1s_2s_3s_1s_2s_1s_3s_2s_3 
s_1s_3s_1s_3s_2s_3s_1s_2s_1s_3s_2s_1s_2s_3s_1s_2s_1s_2s_1s_3
$ \newline $s_1s_2s_3s_2s_1s_3s_1s_2s_1s_3s_1s_2s_3s_2s_1s_2s_1s_3s_2s_3 
s_1s_2s_3s_2s_1s_3s_2s_3s_1s_2s_1s_2s_3s_2s_1s_3s_1s_2s_1s_3
s_1s_2s_1s_2s_1s_3s_1s_3s_1s_2$ \newline $s_1s_3s_1s_2s_1s_3s_1s_2s_3s_2 
s_1s_2s_1s_2s_3s_2s_1s_2s_1s_2s_3s_2s_1s_3s_1s_2s_1s_3s_1s_2
s_1s_3s_1s_3s_1s_2s_3s_2s_1s_3s_1s_2s_1s_2s_1s_3s_1s_2s_1s_3 
$ \newline $s_2s_3s_1s_3s_2s_3s_1s_3s_2s_1s_2s_3s_2s_1s_3s_1s_2s_1s_3s_1
s_2s_1s_3s_1s_2s_3s_2s_1s_2s_3s_1s_3s_2s_3s_1s_3s_2s_3s_1s_3 
s_1s_2s_1s_2s_1s_3s_1s_3s_1s_2$ \newline $s_1s_2s_1s_3s_1s_3s_1s_2s_1s_3
s_1s_2s_1s_3$}
\end{spacing}
\vskip\baselineskip
\noindent
Nr.36\,(Length=336):
\begin{spacing}{0.5}
\noindent
{\tiny{\small{$1^{\mbox{${\hat{B}}$}}\cdot$}}\\
$s_2s_3s_1s_3s_2s_1s_3s_1s_2s_1s_2s_1s_3s_1s_2s_1s_2s_1s_2s_3
s_1s_3s_2s_1s_2s_1s_3s_1s_2s_1s_2s_1s_3s_1s_2s_1s_2s_1s_3s_1 
s_2s_3s_2s_1s_2s_1s_3s_1s_2s_1$ \newline $s_3s_1s_3s_1s_2s_1s_3s_1s_2s_3
s_2s_1s_2s_1s_2s_3s_1s_3s_2s_3s_1s_3s_2s_3s_1s_3s_2s_3s_1s_3 
s_2s_3s_1s_3s_2s_1s_2s_3s_2s_1s_3s_1s_2s_3s_2s_1s_2s_3s_2s_1
$ \newline $s_2s_1s_2s_3s_2s_1s_2s_1s_2s_1s_2s_3s_1s_3s_2s_3s_2s_1s_3s_1 
s_3s_1s_2s_3s_1s_3s_2s_1s_2s_1s_2s_1s_2s_1s_2s_1s_2s_3s_2s_1
s_3s_1s_3s_1s_2s_1s_3s_1s_2s_1$ \newline $s_3s_1s_2s_1s_2s_1s_3s_1s_2s_1 
s_2s_1s_3s_1s_3s_1s_2s_1s_2s_3s_2s_1s_2s_1s_2s_3s_1s_3s_2s_3
s_1s_3s_2s_3s_1s_3s_1s_3s_2s_1s_2s_1s_2s_1s_2s_3s_2s_3s_2s_1 
$ \newline $s_2s_1s_3s_1s_2s_1s_2s_1s_3s_1s_2s_1s_2s_3s_1s_3s_2s_1s_2s_3
s_2s_1s_2s_1s_3s_1s_2s_1s_2s_1s_3s_1s_2s_3s_1s_3s_2s_1s_3s_1 
s_2s_1s_2s_1s_2s_3s_2s_3s_2s_1$ \newline $s_2s_1s_2s_3s_2s_1s_3s_1s_2s_1
s_3s_1s_2s_1s_3s_1s_3s_1s_3s_1s_2s_1s_2s_1s_3s_1s_2s_3s_2s_1 
s_2s_3s_2s_1s_2s_1s_2s_3s_2s_3s_1s_3s_2s_3s_1s_3s_1s_3s_2s_3
$ \newline $s_1s_3s_2s_3s_1s_3s_2s_1s_2s_3s_2s_3s_2s_1s_2s_1s_2s_1s_2s_3 
s_2s_1s_2s_1s_2s_3s_2s_1s_2s_3s_1s_3s_2s_3s_1s_3$}
\end{spacing}
\vskip\baselineskip
\noindent
Nr.37\,(Length=336):
\begin{spacing}{0.5}
\noindent
{\tiny{\small{$1^{\mbox{${\hat{B}}$}}\cdot$}}\\
$s_2s_3s_2s_1s_2s_3s_2s_1s_2s_1s_2s_1s_3s_1s_3s_1s_2s_1s_3s_1
s_3s_1s_3s_1s_3s_1s_2s_1s_3s_1s_2s_1s_2s_1s_3s_1s_2s_1s_3s_1 
s_3s_2s_3s_1s_3s_1s_3s_2s_3s_2$ \newline $s_3s_1s_3s_1s_3s_1s_3s_1s_3s_2
s_3s_1s_3s_2s_3s_2s_3s_1s_3s_1s_3s_2s_3s_1s_3s_2s_3s_1s_3s_1 
s_3s_2s_3s_1s_3s_1s_3s_2s_3s_1s_3s_1s_2s_3s_2s_1s_3s_2s_3s_1
$ \newline $s_2s_3s_1s_3s_2s_3s_2s_3s_1s_3s_2s_1s_2s_3s_2s_1s_2s_3s_1s_3 
s_2s_3s_2s_3s_2s_3s_1s_3s_2s_1s_2s_3s_2s_1s_2s_3s_1s_3s_2s_3
s_2s_3s_1s_3s_2s_1s_3s_1s_3s_1$ \newline $s_2s_1s_3s_1s_3s_2s_3s_1s_3s_2 
s_1s_2s_3s_1s_3s_1s_3s_2s_3s_1s_2s_1s_3s_1s_3s_2s_3s_1s_3s_2
s_1s_2s_3s_1s_3s_1s_3s_2s_3s_1s_2s_1s_3s_1s_3s_1s_3s_1s_2s_1 
$ \newline $s_3s_2s_3s_1s_3s_1s_3s_2s_3s_1s_3s_2s_3s_1s_3s_1s_3s_2s_1s_2
s_3s_2s_1s_2s_1s_2s_3s_2s_1s_2s_3s_1s_3s_1s_3s_2s_3s_1s_2s_1 
s_3s_1s_3s_1s_2s_1s_2s_1s_3s_1$ \newline $s_3s_1s_2s_1s_3s_2s_3s_1s_3s_1
s_3s_2s_3s_2s_3s_1s_3s_2s_3s_1s_3s_1s_3s_1s_3s_1s_3s_2s_3s_1 
s_3s_1s_3s_2s_3s_1s_3s_2s_3s_1s_3s_2s_3s_1s_3s_2s_3s_1s_3s_2
$ \newline $s_3s_1s_3s_1s_3s_2s_3s_1s_3s_1s_2s_1s_3s_1s_2s_1s_2s_1s_3s_1 
s_3s_1s_2s_1s_3s_1s_2s_3s_2s_1s_2s_1s_2s_1s_2s_1$}
\end{spacing}
\newpage
\noindent
Nr.38\,(Length=336):
\begin{spacing}{0.5}
\noindent
{\tiny{\small{$1^{\mbox{${\hat{B}}$}}\cdot$}}\\
$s_1s_2s_1s_3s_1s_3s_1s_2s_1s_2s_1s_3s_1s_2s_1s_2s_1s_2s_1s_2
s_1s_3s_1s_2s_1s_3s_1s_3s_1s_2s_1s_3s_1s_2s_1s_2s_1s_3s_1s_3 
s_1s_2s_1s_3s_1s_2s_1s_3s_1s_2$ \newline $s_1s_2s_1s_2s_1s_2s_1s_3s_1s_2
s_1s_3s_1s_2s_1s_2s_1s_3s_1s_3s_1s_3s_1s_2s_1s_3s_1s_2s_1s_2 
s_1s_2s_1s_2s_1s_3s_1s_2s_1s_3s_1s_3s_1s_2s_1s_3s_1s_2s_1s_2
$ \newline $s_1s_3s_1s_3s_1s_2s_1s_3s_1s_3s_1s_2s_3s_1s_3s_2s_1s_2s_1s_2 
s_1s_3s_1s_2s_1s_3s_1s_3s_1s_3s_1s_2s_1s_2s_1s_3s_1s_2s_1s_2
s_1s_3s_1s_2s_1s_2s_1s_3s_1s_2$ \newline $s_1s_3s_1s_2s_1s_2s_1s_3s_1s_2 
s_1s_2s_1s_3s_1s_2s_1s_3s_1s_2s_1s_3s_1s_2s_3s_2s_1s_3s_1s_2
s_1s_2s_1s_2s_1s_3s_1s_2s_1s_3s_1s_2s_1s_2s_1s_2s_1s_3s_1s_2 
$ \newline $s_1s_3s_1s_2s_1s_2s_1s_3s_1s_3s_1s_3s_1s_2s_1s_3s_1s_2s_1s_2
s_1s_3s_1s_2s_1s_2s_1s_2s_1s_2s_1s_3s_1s_3s_1s_2s_1s_2s_1s_3 
s_1s_2s_1s_3s_1s_2s_1s_3s_1s_2$ \newline $s_1s_3s_1s_3s_1s_2s_1s_3s_1s_2
s_1s_2s_3s_2s_1s_2s_1s_3s_1s_2s_1s_2s_1s_2s_1s_3s_1s_2s_1s_2 
s_3s_2s_1s_2s_1s_3s_1s_2s_1s_3s_1s_2s_1s_2s_1s_3s_1s_2s_3s_2
$ \newline $s_1s_2s_1s_2s_1s_2s_3s_1s_3s_1s_3s_2s_3s_1s_3s_2s_3s_1s_3s_2 
s_1s_2s_3s_2s_1s_2s_1s_2s_1s_2s_3s_2s_1s_2s_1s_3$}
\end{spacing}
\vskip\baselineskip
\noindent
Nr.39\,(Length=336):
\begin{spacing}{0.5}
\noindent
{\tiny{\small{$1^{\mbox{${\hat{B}}$}}\cdot$}}\\
$s_1s_2s_1s_3s_1s_3s_1s_2s_1s_2s_1s_3s_1s_2s_1s_2s_1s_2s_1s_2
s_1s_3s_1s_2s_1s_3s_1s_3s_1s_2s_1s_3s_1s_2s_1s_2s_1s_3s_1s_3 
s_1s_2s_1s_3s_1s_2s_1s_3s_1s_2$ \newline $s_1s_2s_1s_2s_1s_2s_1s_3s_1s_2
s_1s_3s_1s_2s_1s_2s_1s_3s_1s_3s_1s_3s_1s_2s_1s_3s_1s_2s_1s_2 
s_1s_2s_1s_2s_1s_3s_1s_2s_1s_3s_1s_3s_1s_2s_1s_3s_1s_2s_1s_2
$ \newline $s_1s_3s_1s_3s_1s_2s_1s_3s_1s_3s_1s_2s_3s_1s_3s_2s_1s_2s_1s_2 
s_1s_3s_1s_2s_1s_3s_1s_3s_1s_3s_1s_2s_1s_2s_1s_3s_1s_2s_1s_2
s_1s_3s_1s_2s_1s_2s_1s_3s_1s_2$ \newline $s_1s_3s_1s_2s_1s_2s_1s_3s_1s_2 
s_1s_2s_1s_3s_1s_2s_1s_3s_1s_2s_1s_3s_1s_2s_3s_2s_1s_3s_1s_2
s_1s_2s_1s_2s_1s_3s_1s_2s_1s_3s_1s_2s_1s_2s_1s_2s_1s_3s_1s_2 
$ \newline $s_1s_3s_1s_2s_1s_2s_1s_3s_1s_3s_1s_3s_1s_2s_1s_3s_1s_2s_1s_2
s_1s_3s_1s_2s_1s_2s_1s_2s_1s_2s_1s_3s_1s_3s_1s_2s_1s_2s_1s_3 
s_1s_2s_1s_3s_1s_2s_1s_3s_1s_2$ \newline $s_1s_3s_1s_3s_1s_2s_1s_3s_1s_2
s_1s_2s_3s_2s_1s_2s_1s_3s_1s_2s_1s_2s_1s_2s_1s_3s_1s_2s_1s_2 
s_3s_2s_1s_2s_1s_3s_1s_2s_1s_3s_1s_2s_1s_2s_1s_3s_1s_2s_3s_2
$ \newline $s_1s_2s_1s_2s_1s_2s_3s_1s_3s_1s_3s_2s_3s_1s_3s_2s_3s_1s_3s_2 
s_1s_2s_3s_2s_1s_2s_1s_2s_1s_2s_3s_2s_1s_2s_1s_3$}
\end{spacing}
\vskip\baselineskip
\noindent
Nr.40\,(Length=364):
\begin{spacing}{0.5}
\noindent
{\tiny{\small{$1^{\mbox{${\hat{B}}$}}\cdot$}}\\
$s_2s_1s_2s_3s_2s_1s_3s_1s_2s_1s_2s_1s_3s_1s_2s_1s_3s_1s_3s_1
s_2s_1s_3s_1s_2s_1s_2s_1s_3s_1s_2s_1s_2s_1s_2s_1s_2s_1s_3s_1 
s_2s_1s_2s_1s_3s_1s_2s_1s_3s_1$ \newline $s_2s_3s_1s_3s_2s_1s_3s_1s_3s_1
s_2s_1s_3s_1s_3s_1s_2s_1s_3s_1s_2s_3s_1s_3s_2s_3s_1s_3s_2s_1 
s_3s_1s_3s_1s_2s_1s_3s_1s_3s_1s_2s_1s_3s_1s_2s_1s_2s_1s_3s_1
$ \newline $s_2s_1s_2s_1s_3s_1s_3s_1s_2s_1s_3s_1s_2s_1s_3s_1s_3s_1s_2s_1 
s_2s_1s_3s_1s_2s_3s_2s_1s_2s_1s_2s_3s_2s_1s_2s_1s_2s_3s_1s_3
s_2s_1s_3s_1s_2s_3s_1s_3s_2s_1$ \newline $s_2s_1s_2s_3s_2s_1s_3s_1s_2s_3 
s_2s_1s_2s_1s_2s_3s_2s_1s_2s_1s_2s_3s_1s_3s_2s_3s_1s_3s_2s_3
s_2s_1s_2s_3s_2s_1s_3s_1s_2s_1s_2s_1s_3s_1s_2s_1s_2s_1s_3s_1 
$ \newline $s_3s_1s_2s_1s_3s_1s_3s_1s_2s_1s_3s_1s_2s_1s_3s_1s_3s_1s_2s_1
s_3s_1s_3s_1s_3s_1s_2s_1s_3s_1s_2s_1s_3s_1s_2s_1s_3s_1s_3s_1 
s_2s_1s_3s_1s_2s_1s_2s_1s_3s_1$ \newline $s_2s_3s_1s_3s_2s_1s_3s_1s_2s_1
s_2s_1s_2s_1s_3s_1s_3s_1s_2s_1s_3s_1s_2s_1s_3s_1s_2s_1s_3s_2 
s_3s_1s_2s_1s_3s_1s_2s_1s_2s_1s_3s_1s_2s_1s_3s_1s_3s_1s_2s_1
$ \newline $s_3s_1s_2s_1s_3s_2s_3s_1s_3s_1s_2s_1s_3s_1s_2s_1s_2s_1s_2s_1 
s_3s_1s_3s_1s_2s_1s_3s_1s_2s_1s_3s_1s_3s_1s_2s_1s_2s_1s_3s_1
s_3s_1s_2s_1s_3s_1s_2s_1s_3s_1$ \newline $s_3s_1s_2s_1s_2s_1s_2s_3s_2s_1 
s_2s_3s_2s_1$}
\end{spacing}
\vskip\baselineskip
\noindent
Nr.41\,(Length=364):
\begin{spacing}{0.5}
\noindent
{\tiny{\small{$1^{\mbox{${\hat{B}}$}}\cdot$}}\\
$s_3s_1s_2s_1s_3s_1s_2s_1s_3s_1s_3s_1s_3s_2s_3s_2s_3s_1s_3s_2
s_1s_2s_3s_2s_1s_3s_1s_2s_1s_3s_2s_3s_1s_3s_2s_1s_2s_3s_1s_3 
s_2s_3s_1s_3s_2s_3s_1s_2s_1s_3$ \newline $s_1s_2s_1s_3s_1s_2s_1s_3s_1s_3
s_1s_2s_3s_2s_1s_2s_1s_2s_3s_2s_3s_1s_3s_2s_3s_2s_3s_1s_3s_2 
s_1s_2s_3s_1s_3s_2s_3s_2s_1s_2s_3s_2s_1s_3s_1s_2s_1s_3s_1s_3
$ \newline $s_1s_2s_3s_1s_3s_2s_1s_3s_1s_3s_1s_2s_1s_3s_2s_3s_1s_3s_2s_3 
s_1s_3s_2s_1s_2s_3s_1s_3s_1s_3s_2s_3s_1s_2s_1s_3s_1s_2s_3s_1
s_3s_1s_3s_2s_3s_1s_2s_1s_3s_1$ \newline $s_2s_1s_3s_2s_3s_1s_3s_1s_3s_2 
s_3s_1s_3s_2s_1s_3s_2s_3s_1s_2s_3s_2s_1s_2s_3s_2s_1s_3s_2s_3
s_1s_2s_3s_1s_3s_2s_3s_1s_3s_1s_3s_2s_3s_1s_2s_1s_3s_1s_3s_1 
$ \newline $s_2s_1s_2s_1s_2s_1s_2s_1s_2s_1s_3s_2s_3s_1s_3s_1s_2s_1s_3s_1
s_3s_1s_3s_2s_1s_2s_3s_1s_2s_1s_3s_2s_1s_2s_3s_1s_2s_1s_3s_1 
s_3s_1s_2s_1s_3s_2s_1s_2s_3s_1$ \newline $s_2s_1s_3s_2s_1s_2s_3s_2s_3s_1
s_2s_1s_3s_2s_3s_1s_3s_2s_1s_2s_3s_2s_1s_3s_1s_2s_1s_3s_1s_2 
s_3s_1s_3s_2s_1s_3s_1s_3s_1s_3s_1s_2s_1s_3s_1s_2s_3s_1s_3s_2
$ \newline $s_1s_2s_3s_2s_1s_2s_3s_1s_3s_2s_3s_1s_3s_2s_1s_3s_1s_2s_1s_3 
s_1s_2s_3s_1s_3s_2s_3s_1s_2s_1s_3s_1s_3s_1s_3s_1s_2s_1s_3s_2
s_1s_2s_3s_1s_2s_1s_3s_1s_3s_1$ \newline $s_2s_1s_3s_1s_2s_1s_3s_2s_3s_1 
s_3s_1s_3s_2$}
\end{spacing}
\vskip\baselineskip
\noindent
Nr.42\,(Length=392):
\begin{spacing}{0.5}
\noindent
{\tiny{\small{$1^{\mbox{${\hat{B}}$}}\cdot$}}\\
$s_2s_3s_1s_3s_1s_3s_2s_3s_1s_3s_1s_3s_2s_3s_1s_3s_2s_1s_3s_1
s_3s_1s_2s_1s_3s_1s_3s_1s_2s_1s_3s_1s_3s_2s_3s_2s_3s_1s_3s_1 
s_3s_2s_3s_1s_2s_1s_3s_1s_3s_1$ \newline $s_2s_1s_2s_1s_3s_1s_3s_1s_2s_1
s_3s_2s_3s_1s_3s_1s_3s_2s_1s_2s_3s_2s_1s_2s_3s_1s_2s_1s_3s_2 
s_3s_1s_3s_2s_3s_1s_3s_1s_3s_2s_3s_1s_2s_1s_3s_2s_3s_1s_2s_1
$ \newline $s_2s_1s_3s_1s_3s_1s_2s_1s_3s_1s_2s_1s_3s_2s_3s_1s_3s_2s_3s_2 
s_3s_1s_3s_2s_3s_1s_3s_1s_3s_2s_3s_1s_2s_1s_3s_1s_2s_1s_3s_1
s_2s_3s_1s_3s_2s_1s_3s_1s_3s_1$ \newline $s_2s_1s_3s_1s_3s_1s_2s_3s_1s_3 
s_2s_1s_3s_1s_2s_1s_3s_1s_2s_1s_2s_1s_3s_1s_3s_1s_2s_1s_2s_1
s_3s_1s_2s_1s_3s_1s_3s_1s_2s_1s_3s_2s_3s_1s_3s_1s_3s_1s_3s_2 
$ \newline $s_3s_1s_3s_1s_3s_2s_3s_1s_3s_1s_3s_2s_3s_1s_2s_1s_3s_1s_2s_1
s_3s_1s_2s_1s_3s_1s_3s_1s_3s_1s_2s_1s_3s_1s_3s_1s_2s_1s_2s_1 
s_3s_1s_2s_1s_3s_1s_2s_3s_1s_3$ \newline $s_2s_1s_3s_1s_3s_1s_2s_1s_3s_1
s_3s_1s_2s_1s_2s_1s_3s_1s_2s_1s_3s_1s_2s_1s_3s_2s_3s_1s_3s_1 
s_3s_2s_3s_1s_3s_2s_3s_2s_1s_2s_3s_1s_3s_1s_2s_1s_3s_1s_3s_1
$ \newline $s_2s_1s_3s_1s_2s_1s_3s_1s_2s_3s_2s_1s_3s_1s_2s_1s_2s_1s_3s_1 
s_2s_3s_2s_1s_2s_1s_2s_3s_1s_3s_2s_3s_1s_3s_1s_3s_2s_3s_1s_3
s_2s_3s_1s_3s_2s_1s_2s_3s_2s_1$ \newline $s_2s_1s_3s_1s_3s_2s_3s_1s_3s_1 
s_3s_2s_3s_1s_2s_1s_3s_1s_3s_1s_2s_1s_3s_1s_2s_1s_3s_1s_3s_1
s_2s_3s_2s_1s_2s_3s_1s_3s_2s_3s_1s_3$}
\end{spacing}
\vskip\baselineskip
\noindent
Nr.43\,(Length=392):
\begin{spacing}{0.5}
\noindent
{\tiny{\small{$1^{\mbox{${\hat{B}}$}}\cdot$}}\\
$s_2s_3s_2s_1s_2s_3s_2s_1s_2s_1s_2s_1s_2s_1s_2s_1s_3s_1s_2s_1
s_3s_1s_2s_1s_3s_1s_2s_1s_2s_1s_3s_1s_3s_1s_3s_1s_2s_1s_2s_1 
s_3s_1s_2s_1s_3s_1s_2s_1s_3s_1$ \newline $s_2s_1s_2s_1s_3s_1s_3s_1s_2s_1
s_3s_1s_2s_1s_3s_1s_3s_1s_2s_1s_2s_1s_2s_1s_3s_1s_3s_1s_2s_1 
s_3s_1s_3s_1s_2s_1s_3s_1s_2s_3s_2s_1s_2s_1s_2s_3s_2s_1s_3s_1
$ \newline $s_2s_1s_2s_1s_3s_1s_2s_1s_3s_1s_2s_3s_2s_1s_2s_1s_2s_3s_2s_1 
s_3s_1s_2s_1s_3s_1s_3s_1s_2s_1s_2s_1s_3s_1s_2s_1s_2s_1s_2s_1
s_3s_1s_2s_1s_3s_1s_2s_1s_2s_1$ \newline $s_3s_1s_3s_1s_2s_1s_3s_1s_3s_1 
s_3s_1s_2s_1s_3s_1s_3s_1s_2s_1s_3s_1s_2s_1s_2s_1s_3s_1s_2s_1
s_3s_1s_2s_1s_2s_1s_3s_2s_3s_1s_2s_1s_2s_1s_3s_1s_2s_1s_3s_1 
$ \newline $s_2s_3s_2s_1s_2s_1s_2s_3s_2s_1s_3s_1s_2s_1s_3s_1s_3s_1s_2s_1
s_3s_2s_3s_1s_2s_1s_3s_1s_3s_1s_2s_1s_2s_1s_3s_1s_2s_1s_3s_1 
s_2s_1s_3s_1s_3s_2s_3s_1s_2s_1$ \newline $s_3s_1s_2s_3s_2s_1s_2s_1s_3s_1
s_3s_1s_2s_1s_2s_1s_3s_2s_3s_1s_3s_1s_2s_1s_3s_1s_2s_1s_3s_1 
s_3s_1s_3s_1s_3s_1s_2s_1s_3s_1s_2s_1s_3s_1s_2s_3s_2s_1s_2s_3
$ \newline $s_1s_3s_2s_1s_2s_3s_2s_1s_3s_1s_3s_1s_2s_1s_3s_1s_3s_1s_3s_1 
s_2s_1s_3s_1s_2s_3s_2s_1s_2s_1s_2s_3s_2s_3s_1s_3s_2s_3s_1s_3
s_2s_1s_2s_3s_2s_1s_2s_1s_2s_1$ \newline $s_3s_1s_2s_1s_3s_1s_2s_1s_2s_1 
s_3s_1s_3s_1s_3s_1s_3s_1s_2s_1s_3s_1s_2s_1s_3s_1s_3s_1s_2s_1
s_3s_1s_2s_1s_2s_1s_3s_1s_3s_1s_2s_1$}
\end{spacing}
\newpage
\noindent
Nr.44\,(Length=392):
\begin{spacing}{0.5}
\noindent
{\tiny{\small{$1^{\mbox{${\hat{B}}$}}\cdot$}}\\
$s_2s_3s_1s_3s_1s_3s_2s_3s_1s_3s_2s_1s_2s_3s_2s_1s_2s_3s_1s_3
s_2s_3s_1s_3s_2s_3s_2s_3s_1s_3s_1s_3s_1s_3s_2s_3s_1s_2s_1s_3 
s_1s_2s_1s_2s_1s_2s_1s_3s_1s_2$ \newline $s_1s_3s_1s_3s_1s_3s_1s_2s_1s_3
s_2s_3s_1s_3s_2s_1s_3s_1s_2s_1s_2s_1s_3s_1s_3s_1s_2s_1s_3s_2 
s_3s_1s_3s_1s_3s_2s_3s_1s_2s_1s_3s_1s_3s_1s_2s_3s_1s_3s_2s_1
$ \newline $s_3s_1s_2s_1s_3s_1s_2s_1s_3s_1s_2s_1s_3s_1s_2s_1s_2s_1s_3s_1 
s_3s_1s_3s_1s_2s_1s_2s_1s_3s_1s_2s_3s_1s_3s_2s_3s_2s_3s_1s_3
s_2s_1s_2s_3s_1s_3s_2s_3s_2s_3$ \newline $s_1s_3s_2s_1s_2s_3s_1s_3s_2s_3 
s_2s_3s_1s_3s_2s_1s_2s_1s_3s_1s_3s_1s_2s_1s_3s_1s_3s_1s_2s_1
s_3s_1s_2s_3s_2s_1s_2s_1s_2s_3s_2s_3s_1s_3s_2s_1s_2s_3s_2s_1 
$ \newline $s_2s_3s_1s_3s_2s_3s_2s_1s_2s_1s_2s_3s_2s_1s_2s_1s_3s_2s_3s_1
s_3s_1s_2s_1s_3s_1s_2s_1s_3s_1s_2s_1s_3s_2s_3s_1s_3s_1s_3s_2 
s_1s_2s_3s_2s_3s_1s_2s_1s_3s_1$ \newline $s_2s_1s_2s_3s_1s_3s_2s_3s_1s_3
s_2s_3s_2s_1s_2s_1s_2s_3s_2s_1s_3s_2s_3s_1s_3s_1s_3s_2s_1s_2 
s_3s_1s_2s_1s_3s_1s_3s_1s_2s_1s_3s_1s_3s_1s_2s_1s_2s_3s_1s_3
$ \newline $s_2s_3s_1s_2s_1s_3s_1s_3s_1s_3s_1s_2s_1s_2s_1s_3s_1s_3s_2s_1 
s_2s_3s_1s_3s_1s_2s_1s_2s_1s_3s_1s_3s_1s_2s_1s_3s_1s_3s_1s_2
s_1s_3s_1s_3s_1s_2s_1s_3s_2s_1$ \newline $s_2s_3s_2s_3s_1s_3s_2s_3s_1s_3 
s_2s_3s_1s_3s_2s_1s_2s_1s_2s_3s_2s_1s_2s_3s_1s_3s_2s_1s_2s_1
s_2s_3s_2s_1s_2s_3s_1s_3s_2s_3s_1s_3$}
\end{spacing}
\vskip\baselineskip
\noindent
Nr.45\,(Length=420):
\begin{spacing}{0.5}
\noindent
{\tiny{\small{$1^{\mbox{${\hat{B}}$}}\cdot$}}\\
$s_2s_1s_2s_3s_2s_1s_2s_3s_2s_1s_2s_1s_2s_1s_2s_3s_2s_3s_1s_3
s_2s_3s_1s_3s_2s_1s_3s_1s_3s_1s_2s_1s_3s_1s_3s_1s_3s_1s_3s_1 
s_2s_3s_1s_3s_2s_1s_3s_1s_3s_1$ \newline $s_2s_1s_2s_1s_3s_1s_3s_1s_3s_1
s_2s_1s_3s_1s_3s_1s_3s_1s_3s_1s_2s_1s_3s_1s_2s_1s_3s_1s_2s_1 
s_2s_1s_3s_1s_3s_1s_2s_1s_3s_2s_3s_1s_2s_1s_3s_1s_3s_1s_2s_1
$ \newline $s_3s_1s_2s_3s_2s_1s_2s_1s_2s_3s_2s_1s_3s_1s_2s_1s_2s_1s_3s_1 
s_3s_1s_3s_1s_2s_1s_3s_1s_3s_1s_2s_1s_2s_1s_2s_1s_3s_2s_3s_1
s_2s_1s_2s_1s_3s_1s_2s_1s_3s_1$ \newline $s_2s_1s_2s_1s_3s_1s_3s_1s_3s_1 
s_2s_1s_2s_1s_3s_1s_2s_1s_3s_1s_2s_1s_3s_1s_2s_1s_3s_1s_2s_3
s_1s_3s_2s_1s_3s_1s_3s_1s_2s_1s_3s_1s_3s_1s_2s_1s_2s_1s_3s_1 
$ \newline $s_2s_1s_3s_1s_2s_1s_3s_2s_3s_1s_3s_1s_3s_2s_3s_1s_2s_1s_3s_1
s_2s_1s_3s_1s_2s_1s_3s_1s_2s_1s_3s_1s_2s_1s_3s_1s_2s_1s_3s_2 
s_3s_1s_3s_2s_3s_1s_3s_2s_3s_2$ \newline $s_3s_1s_3s_2s_3s_1s_2s_1s_3s_1
s_2s_1s_3s_1s_2s_1s_3s_1s_2s_1s_3s_1s_2s_1s_3s_1s_2s_1s_3s_1 
s_3s_1s_2s_1s_3s_1s_2s_1s_3s_2s_3s_1s_3s_1s_3s_2s_3s_1s_2s_1
$ \newline $s_3s_1s_3s_1s_2s_1s_2s_1s_3s_1s_3s_1s_2s_1s_3s_2s_3s_1s_3s_1 
s_3s_2s_3s_1s_3s_2s_3s_2s_3s_1s_3s_2s_3s_1s_3s_1s_3s_1s_2s_1
s_3s_1s_3s_1s_2s_1s_3s_1s_2s_1$ \newline $s_3s_1s_2s_3s_2s_1s_3s_1s_2s_1 
s_2s_1s_3s_1s_2s_3s_2s_1s_2s_1s_2s_3s_1s_3s_2s_3s_1s_3s_1s_3
s_2s_3s_1s_3s_2s_3s_1s_3s_2s_1s_2s_3s_2s_1s_2s_1s_2s_1s_2s_3 
$ \newline $s_2s_1s_2s_3s_2s_1s_2s_1s_2s_3s_2s_1s_2s_3s_2s_1s_2s_1s_2s_3$}
\end{spacing}
\vskip\baselineskip
\noindent
Nr.46\,(Length=420):
\begin{spacing}{0.5}
\noindent
{\tiny{\small{$1^{\mbox{${\hat{B}}$}}\cdot$}}\\
$s_1s_3s_2s_3s_1s_3s_1s_3s_2s_1s_2s_3s_2s_1s_2s_3s_1s_3s_2s_3
s_1s_3s_2s_3s_2s_1s_3s_1s_2s_3s_1s_3s_2s_3s_2s_3s_1s_3s_1s_3 
s_1s_3s_2s_3s_1s_3s_2s_3s_1s_2$ \newline $s_1s_3s_1s_2s_1s_2s_1s_3s_2s_3
s_1s_3s_2s_3s_2s_3s_1s_3s_1s_3s_2s_1s_2s_3s_1s_3s_2s_3s_1s_2 
s_1s_3s_1s_3s_1s_3s_1s_3s_1s_2s_1s_3s_1s_2s_1s_3s_1s_2s_1s_3
$ \newline $s_1s_3s_1s_2s_1s_3s_2s_3s_1s_2s_1s_3s_1s_2s_1s_3s_2s_3s_1s_3 
s_1s_2s_1s_3s_1s_3s_1s_2s_1s_2s_1s_3s_1s_2s_1s_2s_1s_3s_1s_3
s_1s_2s_1s_3s_1s_2s_1s_3s_2s_3$ \newline $s_1s_3s_1s_3s_1s_3s_2s_3s_1s_3 
s_1s_3s_2s_3s_1s_2s_1s_3s_1s_2s_1s_3s_2s_3s_1s_3s_1s_2s_1s_3
s_1s_3s_1s_2s_1s_3s_1s_2s_1s_3s_2s_3s_1s_3s_1s_2s_1s_2s_1s_3 
$ \newline $s_1s_2s_1s_3s_1s_2s_1s_2s_1s_3s_1s_3s_1s_3s_1s_2s_1s_3s_1s_2
s_1s_2s_1s_2s_1s_3s_2s_3s_1s_3s_1s_3s_2s_1s_2s_3s_1s_2s_1s_3 
s_2s_1s_2s_3s_2s_1s_2s_3s_1s_2$ \newline $s_1s_3s_1s_3s_1s_2s_1s_3s_2s_3
s_1s_3s_1s_3s_2s_3s_2s_3s_1s_3s_2s_1s_2s_3s_1s_2s_1s_3s_2s_1 
s_2s_3s_1s_3s_2s_1s_2s_3s_1s_2s_1s_3s_2s_1s_2s_3s_1s_2s_1s_3
$ \newline $s_1s_2s_1s_3s_1s_2s_1s_2s_1s_3s_1s_2s_1s_3s_1s_3s_1s_2s_1s_3 
s_2s_3s_1s_3s_1s_2s_1s_3s_1s_3s_1s_2s_1s_3s_2s_3s_1s_3s_2s_1
s_2s_3s_2s_1s_3s_1s_2s_1s_2s_1$ \newline $s_3s_1s_2s_3s_2s_1s_2s_3s_2s_1 
s_3s_1s_2s_1s_3s_1s_2s_1s_2s_1s_3s_1s_2s_3s_2s_1s_2s_3s_2s_1
s_2s_3s_1s_3s_2s_3s_1s_3s_2s_3s_2s_1s_2s_1s_3s_1s_2s_3s_2s_1 
$ \newline $s_2s_1s_2s_1s_2s_3s_2s_1s_2s_1s_3s_1s_2s_3s_1s_3s_2s_3s_1s_3$}
\end{spacing}
\vskip\baselineskip
\noindent
Nr.47\,(Length=420):
\begin{spacing}{0.5}
\noindent
{\tiny{\small{$1^{\mbox{${\hat{B}}$}}\cdot$}}\\
$s_3s_2s_1s_2s_1s_2s_3s_2s_1s_2s_3s_2s_1s_2s_1s_2s_3s_2s_1s_2
s_3s_2s_3s_2s_1s_3s_2s_3s_1s_3s_1s_3s_2s_3s_1s_2s_1s_3s_2s_3 
s_1s_3s_2s_3s_1s_3s_1s_2s_1s_2$ \newline $s_1s_3s_1s_2s_1s_3s_1s_2s_1s_2
s_1s_3s_1s_2s_1s_3s_1s_2s_1s_2s_1s_2s_1s_3s_1s_2s_1s_2s_1s_3 
s_1s_2s_1s_2s_1s_3s_1s_2s_1s_3s_1s_2s_3s_2s_1s_2s_1s_2s_3s_2
$ \newline $s_1s_2s_3s_1s_3s_2s_3s_1s_3s_2s_3s_1s_3s_2s_1s_2s_3s_2s_1s_2 
s_1s_2s_3s_2s_1s_3s_1s_2s_1s_3s_1s_2s_1s_2s_1s_3s_1s_2s_3s_2
s_1s_3s_1s_2s_1s_2s_1s_3s_1s_3$ \newline $s_1s_3s_1s_2s_1s_3s_1s_3s_1s_3 
s_1s_2s_1s_2s_1s_3s_1s_2s_1s_3s_1s_2s_1s_2s_1s_3s_1s_3s_1s_2
s_1s_3s_1s_2s_1s_3s_1s_2s_3s_2s_1s_2s_3s_2s_1s_3s_1s_3s_1s_2 
$ \newline $s_1s_2s_1s_3s_1s_2s_1s_2s_1s_3s_1s_2s_3s_2s_3s_2s_1s_2s_3s_1
s_3s_2s_1s_3s_1s_2s_3s_1s_3s_2s_1s_3s_1s_3s_1s_2s_1s_3s_1s_2 
s_1s_2s_1s_3s_1s_2s_3s_2s_1s_2$ \newline $s_3s_2s_1s_3s_1s_2s_1s_2s_1s_3
s_2s_3s_1s_3s_1s_3s_2s_3s_1s_2s_1s_3s_1s_2s_1s_3s_1s_2s_1s_3 
s_1s_3s_1s_2s_1s_2s_1s_3s_1s_3s_1s_2s_1s_2s_1s_3s_1s_2s_3s_2
$ \newline $s_3s_2s_1s_2s_3s_2s_1s_2s_3s_2s_1s_2s_3s_2s_1s_2s_3s_2s_3s_2 
s_3s_2s_1s_2s_1s_3s_1s_2s_1s_2s_1s_3s_1s_3s_1s_2s_1s_2s_1s_3
s_1s_3s_1s_2s_3s_2s_1s_2s_1s_3$ \newline $s_1s_2s_1s_3s_1s_2s_1s_2s_1s_2 
s_1s_3s_1s_2s_1s_3s_1s_2s_1s_2s_1s_3s_1s_3s_1s_2s_1s_3s_2s_3
s_1s_3s_1s_3s_2s_3s_1s_2s_3s_2s_3s_2s_1s_2s_3s_2s_3s_2s_1s_2 
$ \newline $s_3s_2s_1s_3s_2s_3s_1s_2s_3s_2s_1s_2s_3s_2s_1s_2s_3s_2s_1s_2$}
\end{spacing}
\vskip\baselineskip
\noindent
Nr.48\,(Length=448):
\begin{spacing}{0.5}
\noindent
{\tiny{\small{$1^{\mbox{${\hat{B}}$}}\cdot$}}\\
$s_3s_1s_3s_2s_3s_1s_3s_1s_3s_2s_3s_1s_3s_2s_1s_3s_1s_2s_1s_2
s_1s_3s_1s_3s_1s_2s_1s_3s_2s_3s_1s_3s_1s_3s_2s_3s_1s_2s_1s_3 
s_1s_2s_3s_2s_1s_2s_1s_3s_1s_3$ \newline $s_2s_3s_1s_2s_1s_3s_1s_3s_1s_2
s_3s_2s_1s_3s_2s_3s_1s_2s_3s_2s_1s_3s_1s_2s_3s_2s_1s_2s_1s_2 
s_3s_1s_3s_2s_3s_1s_3s_2s_3s_1s_3s_2s_3s_2s_1s_2s_3s_2s_1s_3
$ \newline $s_2s_3s_1s_2s_3s_2s_1s_2s_3s_2s_1s_2s_3s_2s_1s_3s_1s_2s_1s_2 
s_1s_3s_1s_2s_1s_2s_3s_2s_1s_3s_1s_2s_3s_2s_1s_2s_1s_2s_3s_1
s_3s_2s_3s_1s_3s_2s_1s_3s_2s_3$ \newline $s_1s_2s_3s_2s_1s_2s_1s_3s_1s_3 
s_1s_2s_1s_3s_2s_3s_1s_3s_1s_3s_2s_3s_1s_3s_1s_3s_2s_3s_1s_2
s_1s_3s_1s_2s_1s_2s_1s_3s_1s_2s_1s_2s_1s_3s_2s_3s_1s_3s_2s_1 
$ \newline $s_2s_3s_2s_1s_2s_3s_2s_1s_2s_3s_1s_3s_2s_3s_1s_3s_1s_2s_1s_3
s_1s_3s_1s_2s_3s_2s_3s_2s_1s_3s_2s_3s_1s_2s_3s_2s_1s_2s_3s_2 
s_1s_2s_3s_2s_1s_2s_3s_2s_1s_3$ \newline $s_1s_2s_1s_3s_2s_3s_1s_3s_1s_3
s_1s_2s_1s_3s_1s_2s_3s_2s_1s_3s_1s_3s_1s_2s_1s_2s_1s_3s_1s_3 
s_1s_2s_1s_3s_2s_3s_1s_3s_1s_3s_2s_3s_1s_2s_1s_3s_2s_3s_1s_3
$ \newline $s_1s_3s_2s_1s_2s_3s_1s_2s_1s_3s_2s_1s_2s_3s_2s_1s_2s_3s_2s_1 
s_2s_3s_2s_3s_1s_3s_2s_1s_2s_3s_1s_3s_2s_3s_2s_3s_1s_3s_2s_3
s_1s_2s_1s_3s_1s_2s_3s_1s_3s_1$ \newline $s_3s_2s_1s_2s_1s_2s_3s_2s_1s_3 
s_1s_2s_1s_2s_1s_3s_1s_2s_1s_2s_3s_2s_1s_3s_1s_2s_1s_3s_2s_1
s_2s_1s_2s_3s_2s_1s_2s_3s_2s_1s_2s_1s_2s_3s_1s_3s_1s_2s_1s_3 
$ \newline $s_2s_3s_1s_3s_1s_3s_2s_3s_1s_3s_1s_3s_1s_3s_2s_3s_1s_2s_1s_3
s_1s_3s_1s_2s_3s_2s_1s_3s_1s_2s_1s_3s_1s_2s_1s_3s_1s_2s_3s_1 
s_3s_2s_3s_2s_3s_1s_3s_2$}
\end{spacing}
\vskip\baselineskip
\noindent
Nr.49\,(Length=448):
\begin{spacing}{0.5}
\noindent
{\tiny{\small{$1^{\mbox{${\hat{B}}$}}\cdot$}}\\
$s_2s_3s_1s_3s_1s_3s_2s_3s_1s_3s_1s_3s_2s_3s_1s_3s_2s_1s_3s_1
s_2s_3s_2s_1s_2s_1s_2s_3s_2s_1s_2s_3s_2s_1s_2s_1s_2s_1s_3s_1 
s_2s_1s_2s_1s_3s_1s_2s_1s_3s_1$ \newline $s_2s_1s_3s_1s_2s_3s_2s_1s_3s_1
s_2s_1s_2s_1s_3s_1s_2s_3s_2s_1s_2s_3s_2s_1s_3s_2s_3s_1s_2s_3 
s_2s_3s_1s_3s_2s_3s_1s_3s_1s_3s_2s_3s_1s_3s_2s_3s_1s_3s_2s_1
$ \newline $s_2s_3s_2s_1s_2s_3s_2s_1s_3s_1s_3s_1s_3s_1s_2s_1s_2s_1s_3s_1 
s_2s_1s_3s_1s_2s_1s_3s_1s_2s_1s_3s_1s_2s_3s_1s_3s_2s_1s_3s_1
s_3s_1s_2s_1s_3s_2s_3s_1s_3s_2$ \newline $s_3s_1s_2s_1s_2s_1s_3s_1s_2s_3 
s_2s_1s_2s_3s_1s_3s_2s_1s_2s_3s_2s_1s_3s_1s_2s_1s_2s_1s_2s_1
s_3s_1s_2s_1s_3s_1s_2s_1s_2s_1s_2s_1s_3s_1s_2s_1s_2s_1s_3s_1 
$ \newline $s_3s_1s_2s_1s_2s_1s_2s_3s_2s_1s_3s_1s_3s_1s_2s_1s_2s_1s_3s_1
s_2s_3s_2s_1s_2s_1s_2s_3s_1s_3s_2s_3s_1s_3s_1s_3s_2s_3s_1s_3 
s_2s_3s_1s_3s_2s_1s_2s_1s_2s_3$ \newline $s_2s_1s_3s_1s_2s_1s_3s_1s_2s_1
s_2s_1s_3s_1s_2s_3s_2s_1s_2s_1s_2s_3s_2s_1s_3s_1s_2s_1s_3s_1 
s_2s_1s_2s_1s_3s_1s_3s_1s_2s_1s_3s_1s_2s_3s_2s_1s_2s_1s_2s_3
$ \newline $s_2s_1s_2s_1s_3s_1s_3s_1s_2s_1s_2s_1s_3s_1s_2s_1s_3s_1s_2s_3 
s_2s_1s_2s_3s_2s_1s_2s_3s_1s_3s_2s_1s_2s_3s_2s_1s_2s_3s_2s_1
s_3s_1s_3s_1s_2s_1s_3s_1s_3s_1$ \newline $s_2s_1s_2s_1s_3s_1s_2s_1s_3s_1 
s_2s_1s_3s_1s_3s_1s_2s_3s_2s_1s_2s_3s_2s_1s_3s_1s_2s_1s_2s_1
s_3s_1s_2s_1s_2s_1s_3s_1s_2s_1s_3s_1s_2s_3s_2s_1s_2s_1s_2s_3 
$ \newline $s_2s_1s_3s_1s_2s_1s_3s_1s_2s_1s_2s_1s_3s_1s_2s_1s_2s_1s_2s_1
s_3s_1s_2s_3s_2s_1s_2s_1s_2s_1s_2s_3s_2s_3s_1s_3s_2s_1s_3s_1 
s_2s_3s_1s_3s_2s_3s_1s_3$}
\end{spacing}
\vskip\baselineskip
\noindent
Nr.50\,(Length=448):
\begin{spacing}{0.5}
\noindent
{\tiny{\small{$1^{\mbox{${\hat{B}}$}}\cdot$}}\\
$s_2s_1s_2s_3s_2s_1s_2s_3s_2s_1s_2s_1s_2s_1s_2s_3s_2s_3s_1s_3
s_2s_3s_1s_3s_2s_1s_3s_1s_3s_1s_2s_1s_3s_1s_2s_1s_2s_1s_3s_1 
s_2s_1s_2s_1s_2s_3s_2s_1s_2s_3$ \newline $s_2s_1s_2s_1s_2s_3s_2s_1s_2s_1
s_3s_1s_2s_3s_2s_1s_3s_1s_2s_1s_2s_1s_3s_1s_2s_1s_2s_3s_2s_1 
s_3s_1s_2s_1s_2s_1s_3s_1s_2s_1s_2s_1s_2s_1s_3s_1s_2s_1s_2s_1
$ \newline $s_3s_1s_3s_1s_2s_1s_3s_1s_3s_1s_2s_1s_2s_1s_3s_1s_3s_1s_2s_1 
s_3s_1s_2s_1s_2s_1s_3s_1s_3s_1s_2s_1s_2s_3s_2s_1s_2s_1s_2s_3
s_2s_1s_3s_1s_2s_1s_3s_1s_3s_1$ \newline $s_2s_1s_2s_1s_3s_1s_2s_1s_2s_1 
s_2s_1s_3s_1s_2s_1s_3s_1s_2s_1s_2s_1s_3s_1s_3s_1s_2s_1s_3s_1
s_3s_1s_2s_1s_2s_1s_3s_1s_2s_3s_2s_1s_2s_1s_2s_3s_2s_1s_3s_1 
$ \newline $s_2s_1s_3s_1s_2s_1s_3s_1s_3s_1s_2s_3s_1s_3s_2s_1s_3s_1s_3s_1
s_2s_1s_2s_1s_3s_1s_3s_1s_3s_1s_2s_1s_2s_1s_3s_1s_3s_1s_2s_1 
s_3s_1s_2s_1s_2s_1s_3s_1s_2s_1$ \newline $s_2s_1s_3s_2s_3s_1s_3s_1s_2s_1
s_3s_1s_3s_1s_2s_1s_3s_1s_3s_2s_3s_1s_2s_1s_3s_1s_2s_3s_2s_1 
s_2s_1s_2s_1s_3s_1s_2s_1s_2s_1s_3s_1s_3s_1s_2s_1s_2s_1s_3s_1
$ \newline $s_3s_1s_3s_1s_2s_1s_2s_1s_3s_1s_3s_1s_2s_1s_2s_1s_3s_1s_2s_1 
s_2s_1s_2s_1s_3s_1s_2s_1s_3s_1s_2s_1s_2s_1s_3s_1s_3s_1s_2s_1
s_3s_1s_3s_1s_2s_1s_2s_1s_3s_1$ \newline $s_2s_1s_3s_1s_2s_1s_2s_1s_3s_1 
s_2s_1s_2s_1s_3s_1s_2s_1s_3s_1s_3s_1s_2s_1s_3s_1s_3s_1s_3s_1
s_2s_1s_3s_1s_2s_1s_2s_1s_3s_1s_2s_1s_3s_1s_3s_1s_2s_1s_3s_1 
$ \newline $s_2s_1s_3s_1s_2s_3s_1s_3s_2s_3s_1s_3s_2s_3s_1s_3s_2s_1s_2s_3
s_2s_1s_2s_1s_2s_1s_2s_3s_2s_1s_2s_3s_2s_1s_2s_1s_2s_3s_2s_1 
s_2s_3s_2s_1s_2s_1s_2s_3$}
\end{spacing}
\vskip\baselineskip
\noindent
Nr.51\,(Length=476):
\begin{spacing}{0.5}
\noindent
{\tiny{\small{$1^{\mbox{${\hat{B}}$}}\cdot$}}\\
$s_2s_3s_2s_1s_2s_1s_2s_3s_1s_3s_2s_3s_1s_3s_1s_3s_2s_3s_1s_3
s_2s_1s_3s_1s_2s_3s_2s_1s_2s_1s_2s_3s_2s_1s_2s_3s_2s_1s_2s_1 
s_2s_1s_3s_1s_2s_1s_2s_1s_3s_1$ \newline $s_2s_1s_3s_1s_2s_1s_3s_1s_2s_3
s_2s_1s_3s_1s_2s_1s_2s_1s_3s_1s_2s_3s_2s_1s_2s_1s_2s_3s_2s_3 
s_2s_1s_3s_1s_2s_1s_3s_1s_3s_1s_2s_1s_2s_1s_3s_1s_2s_1s_3s_1
$ \newline $s_2s_1s_3s_1s_3s_1s_2s_1s_2s_1s_3s_1s_3s_1s_2s_1s_2s_1s_3s_1 
s_3s_1s_2s_1s_2s_1s_2s_1s_3s_1s_3s_1s_2s_1s_3s_1s_3s_1s_3s_1
s_3s_1s_2s_1s_3s_1s_2s_1s_3s_1$ \newline $s_2s_1s_3s_1s_2s_1s_2s_3s_2s_1 
s_2s_1s_2s_3s_2s_3s_2s_1s_2s_1s_2s_3s_2s_1s_3s_1s_2s_1s_2s_1
s_3s_1s_3s_1s_3s_1s_3s_1s_2s_1s_3s_1s_2s_1s_3s_1s_2s_1s_2s_1 
$ \newline $s_2s_1s_2s_1s_3s_1s_2s_1s_3s_1s_2s_1s_2s_1s_2s_1s_3s_1s_3s_1
s_2s_1s_2s_1s_3s_1s_2s_3s_2s_1s_2s_1s_2s_3s_2s_1s_3s_1s_2s_1 
s_3s_1s_2s_3s_2s_3s_2s_1s_2s_3$ \newline $s_1s_3s_2s_1s_2s_3s_2s_1s_2s_1
s_2s_3s_2s_1s_3s_1s_3s_1s_3s_1s_3s_1s_2s_1s_2s_1s_3s_1s_3s_1 
s_2s_1s_3s_1s_2s_3s_2s_1s_3s_1s_2s_1s_2s_1s_2s_1s_3s_1s_2s_3
$ \newline $s_1s_3s_2s_1s_2s_1s_3s_1s_2s_1s_2s_3s_2s_1s_2s_1s_3s_1s_2s_1 
s_3s_1s_2s_1s_3s_1s_2s_1s_2s_1s_3s_1s_2s_3s_2s_1s_3s_1s_2s_1
s_2s_1s_3s_1s_2s_1s_3s_1s_2s_1$ \newline $s_3s_1s_3s_1s_2s_1s_2s_1s_3s_1 
s_2s_1s_2s_1s_2s_1s_3s_1s_2s_1s_2s_1s_3s_1s_3s_1s_2s_1s_3s_1
s_2s_1s_2s_1s_3s_1s_2s_1s_2s_3s_2s_1s_2s_1s_2s_3s_2s_1s_2s_1 
$ \newline $s_2s_3s_2s_1s_2s_1s_2s_1s_3s_1s_3s_1s_3s_1s_2s_1s_2s_1s_3s_1
s_3s_1s_2s_1s_3s_1s_2s_1s_2s_1s_3s_1s_2s_3s_2s_1s_3s_1s_3s_1 
s_2s_1s_3s_1s_3s_1s_2s_3s_2s_1$ \newline $s_2s_1s_2s_3s_2s_1s_3s_1s_3s_1
s_2s_1s_3s_1s_2s_1s_2s_3s_2s_1s_2s_3s_2s_1s_2s_1$}
\end{spacing}
\vskip\baselineskip
\noindent
Nr.52\,(Length=504):
\begin{spacing}{0.5}
\noindent
{\tiny{\small{$1^{\mbox{${\hat{B}}$}}\cdot$}}\\
$s_3s_1s_3s_1s_2s_1s_2s_1s_3s_1s_2s_1s_2s_3s_2s_1s_2s_1s_2s_1
s_2s_3s_2s_1s_3s_1s_2s_1s_3s_1s_3s_1s_3s_1s_2s_1s_2s_1s_3s_1 
s_2s_1s_2s_1s_2s_1s_2s_3s_2s_1$ \newline $s_3s_1s_3s_1s_2s_1s_2s_1s_3s_1
s_2s_3s_2s_1s_2s_3s_2s_1s_2s_1s_2s_1s_2s_3s_2s_1s_2s_3s_2s_3 
s_2s_1s_2s_3s_2s_1s_2s_1s_2s_1s_2s_1s_2s_3s_2s_1s_3s_1s_2s_1
$ \newline $s_2s_1s_2s_1s_2s_1s_2s_1s_3s_1s_2s_3s_2s_1s_2s_1s_2s_1s_2s_1 
s_2s_3s_2s_1s_2s_3s_2s_1s_3s_2s_3s_1s_2s_1s_2s_3s_2s_1s_3s_1
s_2s_3s_2s_1s_2s_1s_2s_3s_2s_1$ \newline $s_3s_1s_2s_1s_3s_2s_3s_1s_3s_1 
s_3s_1s_3s_2s_3s_1s_2s_1s_3s_1s_2s_1s_3s_1s_2s_1s_2s_3s_1s_3
s_2s_1s_3s_1s_3s_1s_2s_3s_1s_3s_2s_3s_1s_3s_2s_1s_2s_1s_2s_1 
$ \newline $s_2s_3s_2s_1s_3s_1s_2s_1s_3s_1s_2s_1s_2s_1s_3s_1s_3s_1s_2s_1
s_3s_1s_2s_1s_2s_1s_2s_1s_3s_1s_2s_3s_2s_1s_2s_3s_1s_3s_2s_1 
s_2s_1s_2s_3s_2s_1s_2s_1s_2s_1$ \newline $s_3s_1s_2s_1s_3s_1s_2s_3s_2s_1
s_2s_1s_3s_1s_2s_1s_2s_1s_2s_3s_2s_1s_2s_1s_2s_3s_2s_1s_2s_1 
s_2s_1s_2s_3s_2s_1s_2s_3s_2s_1s_2s_1s_3s_1s_2s_3s_2s_1s_2s_1
$ \newline $s_2s_3s_2s_1s_2s_1s_3s_2s_3s_1s_2s_3s_2s_1s_2s_3s_2s_1s_2s_1 
s_2s_1s_3s_1s_2s_1s_2s_3s_2s_1s_2s_1s_3s_1s_3s_1s_2s_1s_3s_1
s_2s_1s_2s_3s_2s_1s_2s_1s_2s_1$ \newline $s_2s_3s_2s_1s_3s_1s_3s_1s_2s_1 
s_3s_1s_2s_1s_2s_1s_3s_1s_3s_1s_2s_1s_2s_1s_2s_1s_3s_1s_2s_1
s_2s_1s_3s_2s_3s_1s_3s_1s_3s_2s_3s_1s_3s_1s_3s_1s_3s_2s_3s_1 
$ \newline $s_2s_1s_3s_1s_3s_1s_2s_3s_2s_1s_3s_2s_3s_1s_2s_1s_2s_3s_2s_1
s_2s_1s_3s_1s_2s_1s_2s_1s_3s_1s_3s_1s_3s_1s_2s_1s_2s_1s_3s_1 
s_3s_1s_2s_3s_2s_1s_2s_1s_2s_1$ \newline $s_2s_3s_2s_1s_2s_1s_3s_1s_2s_1
s_3s_1s_3s_1s_2s_1s_2s_3s_2s_1s_2s_1s_2s_1s_2s_1s_2s_3s_2s_1 
s_3s_1s_2s_3s_2s_1s_2s_1s_2s_1s_2s_1s_2s_3s_2s_1s_2s_1s_3s_1
$ \newline $s_2s_1s_3s_1$}
\end{spacing}
\vskip\baselineskip
\noindent
Nr.53\,(Length=504):
\begin{spacing}{0.5}
\noindent
{\tiny{\small{$1^{\mbox{${\hat{B}}$}}\cdot$}}\\
$s_3s_2s_1s_3s_1s_2s_1s_2s_1s_3s_1s_2s_1s_3s_1s_2s_1s_2s_1s_2
s_1s_3s_1s_2s_3s_2s_1s_3s_1s_2s_1s_2s_1s_2s_1s_3s_1s_2s_1s_2 
s_1s_2s_1s_3s_1s_2s_1s_2s_3s_2$ \newline $s_1s_2s_1s_3s_1s_2s_1s_3s_1s_2
s_1s_2s_1s_3s_1s_2s_3s_2s_1s_2s_1s_2s_1s_3s_1s_3s_1s_2s_1s_2 
s_1s_3s_1s_2s_1s_2s_1s_2s_1s_3s_1s_2s_1s_2s_1s_3s_1s_3s_1s_2
$ \newline $s_1s_3s_1s_2s_1s_2s_1s_3s_1s_2s_1s_3s_1s_3s_1s_3s_1s_2s_1s_2 
s_1s_3s_1s_3s_1s_2s_1s_2s_1s_3s_1s_2s_3s_2s_1s_2s_3s_2s_1s_2
s_1s_2s_1s_2s_3s_2s_3s_1s_3s_2$ \newline $s_3s_1s_3s_2s_3s_1s_3s_2s_1s_2 
s_1s_2s_3s_2s_1s_3s_1s_2s_1s_2s_1s_3s_1s_3s_1s_2s_1s_3s_1s_3
s_1s_2s_1s_2s_1s_3s_1s_2s_1s_3s_1s_3s_1s_2s_1s_2s_1s_3s_1s_2 
$ \newline $s_3s_2s_1s_3s_1s_2s_1s_2s_1s_2s_1s_3s_1s_2s_1s_2s_1s_2s_1s_3
s_1s_2s_1s_2s_3s_2s_1s_2s_1s_3s_1s_2s_1s_2s_1s_3s_1s_2s_1s_3 
s_1s_2s_1s_2s_3s_2s_1s_2s_1s_2$ \newline $s_3s_2s_1s_2s_1s_2s_3s_2s_1s_2
s_1s_3s_1s_3s_1s_3s_1s_2s_1s_2s_1s_3s_1s_2s_3s_2s_1s_2s_1s_2 
s_3s_2s_1s_2s_1s_2s_3s_2s_1s_2s_3s_2s_1s_2s_1s_2s_3s_2s_1s_2
$ \newline $s_3s_2s_1s_2s_3s_2s_1s_2s_3s_2s_1s_3s_1s_2s_1s_2s_1s_2s_1s_3 
s_1s_2s_1s_2s_1s_2s_1s_3s_1s_2s_1s_2s_3s_2s_1s_2s_3s_2s_3s_2
s_3s_2s_1s_3s_1s_2s_1s_2s_1s_2$ \newline $s_1s_3s_1s_2s_1s_2s_3s_2s_1s_2 
s_1s_2s_1s_2s_3s_2s_1s_3s_1s_2s_1s_2s_3s_2s_1s_2s_3s_2s_1s_3
s_1s_2s_1s_3s_1s_2s_3s_2s_1s_2s_1s_2s_3s_2s_1s_2s_1s_2s_3s_2 
$ \newline $s_1s_3s_1s_2s_1s_2s_1s_3s_1s_3s_1s_3s_1s_2s_1s_3s_1s_2s_1s_2
s_1s_3s_1s_2s_3s_2s_1s_2s_1s_2s_3s_1s_3s_2s_3s_1s_3s_2s_3s_1 
s_3s_2s_1s_3s_1s_2s_1s_2s_3s_2$ \newline $s_1s_2s_1s_2s_1s_2s_3s_2s_1s_3
s_1s_2s_1s_2s_1s_2s_1s_3s_1s_2s_1s_2s_1s_2s_1s_3s_1s_2s_1s_2 
s_1s_3s_1s_2s_3s_2s_1s_2s_1s_3s_1s_2s_1s_2s_1s_3s_1s_2s_1s_2
$ \newline $s_1s_2s_1s_2$}
\end{spacing}
\vskip\baselineskip
\noindent
Nr.54\,(Length=612):
\begin{spacing}{0.5}
\noindent
{\tiny{\small{$1^{\mbox{${\hat{B}}$}}\cdot$}}\\
$s_2s_3s_1s_2s_3s_2s_3s_1s_3s_1s_3s_2s_1s_2s_3s_2s_1s_2s_3s_1
s_3s_2s_3s_2s_3s_1s_3s_2s_1s_3s_1s_2s_3s_2s_1s_2s_1s_2s_3s_2 
s_1s_2s_1s_2s_1s_2s_1s_2s_1s_2$ \newline $s_3s_2s_1s_3s_1s_2s_1s_3s_1s_3
s_1s_3s_2s_3s_1s_2s_3s_2s_1s_3s_2s_3s_1s_3s_2s_3s_1s_2s_1s_3 
s_1s_2s_1s_2s_1s_3s_1s_3s_1s_2s_1s_3s_1s_2s_1s_2s_1s_3s_1s_3
$ \newline $s_1s_2s_3s_1s_3s_2s_3s_1s_3s_2s_1s_3s_2s_3s_1s_2s_3s_2s_1s_3 
s_2s_3s_1s_2s_1s_2s_1s_2s_3s_2s_1s_3s_1s_2s_1s_3s_1s_3s_1s_2
s_1s_2s_1s_3s_1s_2s_1s_3s_1s_3$ \newline $s_1s_3s_1s_3s_1s_2s_1s_3s_2s_1 
s_2s_3s_2s_1s_3s_1s_2s_3s_1s_3s_2s_1s_3s_1s_2s_3s_2s_1s_2s_3
s_1s_3s_1s_3s_2s_3s_1s_3s_1s_3s_2s_3s_1s_2s_1s_3s_1s_3s_1s_2 
$ \newline $s_1s_3s_1s_3s_1s_2s_1s_2s_1s_3s_1s_3s_2s_3s_1s_3s_1s_3s_2s_3
s_1s_2s_1s_3s_1s_2s_1s_3s_1s_2s_1s_3s_2s_3s_1s_3s_1s_3s_2s_3 
s_1s_2s_1s_3s_1s_3s_1s_2s_1s_2$ \newline $s_1s_3s_1s_3s_1s_2s_1s_3s_1s_2
s_1s_2s_1s_3s_1s_3s_1s_2s_1s_3s_1s_2s_1s_3s_1s_2s_1s_2s_1s_3 
s_1s_3s_1s_2s_1s_2s_1s_3s_1s_2s_1s_3s_1s_2s_3s_2s_1s_2s_1s_2
$ \newline $s_3s_2s_1s_2s_3s_2s_1s_2s_1s_2s_3s_2s_1s_3s_1s_2s_1s_2s_1s_2 
s_1s_3s_1s_2s_1s_2s_1s_3s_1s_3s_1s_2s_1s_3s_1s_2s_1s_2s_1s_3
s_1s_2s_3s_2s_1s_2s_1s_3s_1s_3$ \newline $s_2s_3s_1s_2s_1s_3s_1s_2s_3s_2 
s_1s_2s_1s_2s_3s_2s_3s_2s_1s_2s_3s_2s_1s_3s_2s_3s_1s_2s_3s_2
s_1s_3s_1s_2s_1s_3s_2s_3s_1s_3s_1s_3s_2s_3s_1s_2s_1s_3s_1s_3 
$ \newline $s_1s_2s_1s_2s_1s_2s_1s_3s_1s_2s_1s_3s_1s_2s_1s_2s_3s_2s_1s_2
s_3s_2s_1s_2s_3s_2s_1s_2s_1s_2s_3s_2s_1s_3s_1s_2s_1s_2s_1s_3 
s_1s_2s_1s_2s_3s_2s_1s_3s_1s_2$ \newline $s_1s_2s_1s_3s_1s_2s_1s_3s_1s_2
s_1s_2s_1s_3s_1s_2s_1s_2s_1s_3s_1s_2s_1s_3s_1s_3s_1s_2s_1s_2 
s_1s_3s_1s_2s_1s_3s_1s_2s_1s_3s_1s_2s_3s_2s_1s_2s_1s_2s_3s_2
$ \newline $s_1s_3s_1s_2s_1s_3s_1s_2s_1s_2s_1s_3s_1s_2s_3s_2s_1s_2s_1s_3 
s_1s_2s_1s_2s_1s_3s_1s_2s_1s_2s_1s_3s_1s_3s_1s_2s_1s_2s_1s_3
s_1s_2s_1s_3s_1s_3s_1s_2s_1s_2$ \newline $s_1s_3s_1s_2s_1s_3s_1s_2s_1s_3 
s_1s_2s_1s_3s_2s_3s_1s_3s_1s_3s_2s_3s_1s_2s_1s_3s_1s_2s_1s_3
s_1s_3s_1s_3s_1s_3s_1s_2s_1s_3s_1s_2s_1s_3s_1s_3s_1s_2s_3s_2 
$ \newline $s_1s_2s_1s_2s_1s_2s_1s_2s_3s_2s_1s_3$}
\end{spacing}
\newpage
\noindent
Nr.55\,(Length=720):
\begin{spacing}{0.5}
\noindent
{\tiny{\small{$1^{\mbox{${\hat{B}}$}}\cdot$}}\\
$s_2s_3s_2s_1s_2s_1s_2s_3s_1s_3s_2s_3s_1s_2s_1s_3s_1s_3s_1s_2
s_1s_3s_2s_3s_1s_2s_1s_3s_1s_3s_1s_2s_1s_3s_1s_3s_2s_1s_2s_3 
s_1s_3s_1s_3s_1s_2s_1s_3s_2s_1$ \newline $s_2s_3s_2s_1s_3s_1s_2s_3s_1s_3
s_2s_3s_1s_3s_2s_1s_3s_1s_2s_3s_2s_1s_2s_1s_2s_3s_2s_1s_2s_3 
s_1s_3s_1s_3s_2s_3s_1s_3s_2s_3s_1s_3s_1s_3s_2s_3s_1s_3s_1s_3
$ \newline $s_1s_3s_2s_3s_1s_3s_1s_3s_2s_3s_1s_3s_2s_3s_1s_3s_1s_3s_2s_1 
s_2s_3s_2s_1s_2s_1s_2s_3s_2s_1s_3s_1s_2s_3s_2s_3s_2s_1s_2s_3
s_1s_3s_2s_3s_1s_3s_1s_3s_2s_3$ \newline $s_1s_3s_2s_1s_2s_3s_2s_3s_2s_1 
s_2s_1s_3s_2s_3s_1s_2s_1s_3s_1s_3s_1s_2s_1s_3s_1s_3s_1s_2s_1
s_3s_2s_3s_1s_3s_1s_3s_2s_3s_1s_2s_1s_3s_1s_2s_1s_3s_1s_2s_3 
$ \newline $s_2s_1s_3s_1s_2s_1s_2s_1s_3s_1s_3s_2s_3s_1s_3s_1s_3s_2s_1s_2
s_3s_2s_1s_2s_3s_1s_2s_1s_3s_2s_3s_1s_3s_1s_3s_2s_3s_2s_3s_1 
s_3s_2s_1s_2s_3s_1s_3s_1s_3s_1$ \newline $s_3s_2s_3s_1s_3s_2s_3s_1s_3s_2
s_3s_1s_3s_1s_3s_2s_3s_1s_3s_2s_3s_1s_3s_1s_3s_2s_3s_1s_3s_1 
s_3s_2s_3s_1s_2s_1s_3s_1s_3s_1s_2s_1s_2s_3s_2s_1s_3s_2s_3s_1
$ \newline $s_2s_3s_1s_3s_2s_3s_1s_3s_2s_3s_1s_3s_2s_1s_2s_3s_2s_3s_2s_1 
s_2s_3s_1s_3s_2s_3s_1s_3s_2s_3s_1s_3s_2s_1s_3s_1s_3s_1s_2s_1
s_3s_2s_3s_1s_3s_1s_3s_2s_3s_1$ \newline $s_3s_1s_2s_1s_3s_2s_3s_1s_3s_1 
s_3s_2s_3s_1s_3s_2s_3s_1s_3s_1s_3s_2s_3s_1s_3s_1s_2s_3s_2s_1
s_3s_2s_3s_1s_2s_3s_1s_3s_2s_3s_1s_3s_2s_1s_2s_1s_2s_3s_2s_1 
$ \newline $s_3s_1s_2s_3s_2s_1s_2s_1s_2s_3s_2s_1s_2s_1s_2s_3s_2s_1s_2s_1
s_2s_3s_2s_1s_2s_1s_2s_3s_2s_1s_3s_1s_2s_3s_1s_3s_2s_3s_2s_3 
s_1s_3s_2s_1s_2s_3s_2s_1s_2s_3$ \newline $s_1s_3s_1s_3s_2s_3s_2s_1s_3s_2
s_3s_1s_2s_3s_2s_1s_3s_2s_3s_1s_3s_1s_3s_2s_1s_2s_3s_2s_1s_2 
s_3s_2s_1s_2s_1s_2s_3s_2s_1s_2s_3s_1s_2s_1s_3s_2s_1s_2s_3s_1
$ \newline $s_3s_2s_3s_1s_3s_1s_3s_2s_3s_1s_3s_1s_2s_1s_3s_1s_3s_1s_2s_1 
s_2s_1s_3s_1s_2s_1s_3s_1s_3s_1s_2s_1s_2s_1s_3s_1s_2s_1s_3s_1
s_3s_1s_3s_1s_2s_1s_2s_1s_2s_1$ \newline $s_3s_1s_2s_1s_3s_2s_3s_1s_3s_1 
s_2s_1s_3s_2s_3s_1s_3s_1s_3s_2s_3s_1s_2s_1s_3s_1s_3s_1s_2s_1
s_3s_1s_3s_1s_2s_1s_3s_1s_3s_1s_2s_1s_3s_2s_3s_2s_3s_1s_3s_1 
$ \newline $s_3s_2s_3s_1s_2s_3s_2s_1s_2s_1s_2s_3s_2s_1s_2s_3s_2s_1s_3s_1
s_3s_1s_2s_1s_3s_1s_3s_2s_3s_1s_2s_1s_2s_1s_3s_1s_3s_1s_2s_1 
s_2s_1s_3s_1s_2s_1s_3s_1s_2s_3$ \newline $s_2s_1s_3s_1s_2s_1s_2s_1s_2s_1
s_3s_1s_3s_1s_2s_1s_2s_1s_3s_2s_3s_1s_2s_1s_3s_1s_2s_1s_3s_1 
s_3s_1s_2s_3s_2s_1s_2s_1s_2s_3s_2s_1s_2s_3s_2s_1s_2s_1s_2s_3
$ \newline $s_2s_1s_2s_3s_2s_1s_3s_1s_2s_1s_2s_1s_3s_2s_3s_1s_3s_2s_3s_1$}
\end{spacing}

\subsection{Rank four cases} \label{subsection:rankfour}
Let ${\check{\fca}}$ be a free $\bZ$-module of rank four.
Let ${\check{B}}=\{{\check{\al}}_1,{\check{\al}}_2,{\check{\al}}_3, {\check{\al}}_4\}$ be a $\bZ$-base of 
${\check{\fca}}$,
i.e., ${\check{\fca}}=\oplus_{t=1}^4\bZ{\check{\al}}_t$.
In this section, we let the symbol $1^x2^y3^z4^u$ $(x,y,z,u\in\bZgeqo)$ of \cite[Nr.~$k$ of B.2.~Rank 4]{CH15} mean 
$x{\check{\al}}_1+y{\check{\al}}_2+z{\check{\al}}_3+u{\check{\al}}_4$.
For $k\in\fkJ_{1,11}$, let ${\check{\rR}}(k)$
be the irreducible FGRS of rank four defined by  
\cite[Nr.~$k$ of B.2.~Rank 4]{CH15}.
Let $\Theta(k)$ be the set of $D_\bK(\chi,{\check{B}})$'s of \cite[Table~3]{Hec09} such that
${\check{\rR}}(k)$ is $\chi$-associated and ${\check{B}}\in\bB({\check{\rR}}(k))$; instead of writing up the elements of $\Theta(k)$,
we write Row's of \cite[Table~3]{Hec09}.
\newline\newline
$\Theta(1)=\{14\}$, $\Theta(2)=\{18\}$, $\Theta(3)=\{9\}$ $(F(4))$, $\Theta(4)=\emptyset$, $\Theta(5)=\{20\}$, 
$\Theta(6)=\{4\}$ $(F_4)$, $\Theta(7)=\{17\}$,
$\Theta(8)=\emptyset$ ,$\Theta(9)=\{22\}$, $\Theta(10)=\emptyset$, $\Theta(11)=\{21\}$.
\newline\newline
In the following, we give one of the Hamltonian cycles of $\Gamma({\check{\rR}}(k))$ for $k\in\{4,8,10\}$,
which we have found by applying a program of Mathematica~13.3 \cite{Mathe23} (or its earlier version) 
assembled along with the idea of the proofs
of Theorems~\ref{theorem:one} and \ref{theorem:two}.

\vspace{1cm}

\noindent
Nr.4\,(Length=864):
\begin{spacing}{0.5}
\noindent
{\tiny{{\small{$1^{\mbox{${\check{B}}$}}\cdot$}}\\
$s_3s_1s_2s_1s_3s_1s_3s_1s_2s_1s_4s_3s_2s_3s_1s_4s_1s_2s_1s_4
s_1s_2s_3s_2s_1s_4s_1s_2s_3s_1s_4s_1s_3s_2s_1s_4s_1s_2s_1s_3 
s_1s_2s_1s_2s_4s_2s_4s_2s_3s_2$ \newline $s_4s_3s_1s_2s_4s_2s_1s_2s_1s_2
s_3s_2s_1s_2s_4s_2s_1s_2s_3s_2s_1s_2s_3s_2s_1s_2s_4s_2s_3s_2 
s_1s_3s_1s_2s_1s_4s_1s_4s_1s_3s_2s_3s_1s_3s_2s_1s_2s_3s_2s_3
$ \newline $s_2s_1s_2s_4s_1s_2s_3s_2s_1s_2s_1s_2s_3s_2s_1s_2s_1s_2s_3s_2 
s_1s_3s_1s_2s_1s_2s_1s_4s_2s_1s_2s_4s_1s_2s_1s_2s_1s_3s_1s_2
s_3s_2s_1s_2s_1s_2s_3s_2s_1s_3$ \newline $s_1s_2s_1s_2s_1s_4s_1s_3s_2s_3 
s_1s_3s_2s_3s_2s_3s_1s_4s_1s_2s_1s_3s_1s_2s_1s_4s_1s_3s_2s_3
s_1s_3s_1s_3s_2s_1s_2s_3s_1s_2s_1s_3s_2s_1s_2s_3s_1s_3s_1s_3 
$ \newline $s_2s_4s_1s_3s_2s_3s_1s_4s_1s_2s_3s_1s_3s_2s_1s_3s_1s_2s_4s_2
s_1s_2s_4s_1s_3s_2s_3s_2s_3s_1s_3s_2s_3s_1s_3s_1s_3s_2s_3s_1 
s_3s_2s_3s_2s_3s_1s_3s_2s_3s_1$ \newline $s_3s_1s_3s_2s_4s_2s_3s_2s_1s_2
s_3s_1s_3s_2s_3s_2s_3s_1s_3s_2s_3s_2s_3s_1s_4s_2s_3s_2s_1s_2 
s_3s_1s_3s_2s_3s_1s_3s_2s_1s_2s_3s_2s_3s_2s_1s_4s_1s_3s_2s_3
$ \newline $s_1s_2s_1s_3s_1s_3s_1s_2s_1s_3s_2s_3s_1s_4s_1s_2s_1s_3s_1s_4 
s_3s_2s_1s_2s_3s_2s_1s_3s_1s_2s_1s_3s_1s_2s_1s_2s_1s_3s_1s_2
s_3s_2s_1s_2s_1s_2s_3s_2s_1s_3$ \newline $s_1s_2s_3s_2s_1s_2s_3s_2s_1s_2 
s_3s_2s_1s_3s_1s_2s_1s_2s_1s_4s_1s_3s_1s_2s_3s_2s_1s_2s_1s_2
s_3s_2s_1s_2s_4s_2s_1s_2s_3s_2s_3s_2s_1s_2s_3s_2s_3s_2s_1s_2 
$ \newline $s_1s_2s_3s_2s_1s_2s_1s_3s_1s_2s_1s_3s_1s_2s_1s_2s_1s_3s_1s_2
s_3s_2s_1s_2s_3s_2s_3s_2s_1s_2s_3s_2s_4s_2s_3s_2s_1s_3s_2s_3 
s_1s_2s_1s_2s_3s_2s_1s_2s_3s_2$ \newline $s_1s_2s_1s_3s_2s_3s_1s_2s_1s_2
s_1s_3s_1s_2s_3s_2s_1s_2s_3s_2s_1s_2s_3s_1s_3s_2s_1s_3s_1s_2 
s_3s_1s_4s_2s_1s_2s_3s_1s_3s_2s_3s_1s_3s_2s_3s_2s_3s_1s_3s_2
$ \newline $s_1s_2s_1s_2s_3s_2s_1s_2s_1s_2s_1s_2s_3s_2s_3s_2s_1s_2s_1s_2 
s_3s_2s_1s_2s_3s_2s_1s_3s_1s_2s_1s_3s_1s_2s_1s_2s_1s_2s_4s_2
s_1s_2s_3s_2s_1s_2s_3s_2s_3s_2$ \newline $s_1s_2s_1s_2s_3s_2s_1s_2s_4s_1 
s_3s_2s_3s_1s_3s_2s_1s_3s_1s_2s_3s_1s_3s_1s_3s_2s_3s_1s_3s_2
s_1s_2s_3s_2s_1s_3s_1s_2s_1s_2s_1s_2s_1s_3s_2s_3s_1s_2s_3s_2 
$ \newline $s_1s_2s_3s_2s_1s_3s_2s_3s_1s_2s_1s_2s_3s_2s_1s_2s_4s_3s_1s_3
s_1s_2s_1s_3s_1s_2s_3s_2s_1s_2s_1s_2s_3s_2s_1s_2s_3s_2s_1s_4 
s_1s_2s_3s_1s_3s_2s_1s_3s_1s_2$ \newline $s_3s_1s_3s_2s_1s_3s_1s_2s_4s_2
s_1s_3s_1s_2s_3s_2s_3s_2s_1s_2s_3s_1s_3s_2s_3s_1s_3s_2s_1s_2 
s_3s_2s_3s_2s_1s_4s_1s_2s_1s_2s_1s_3s_1s_2s_3s_2s_1s_2s_1s_2
$ \newline $s_3s_2s_1s_4s_1s_3s_1s_2s_1s_3s_1s_2s_3s_2s_1s_2s_1s_3s_1s_2 
s_1s_2s_1s_2s_1s_3s_1s_2s_3s_2s_1s_2s_3s_2s_1s_2s_1s_2s_3s_2
s_1s_2s_1s_3s_1s_2s_3s_2s_1s_2$ \newline $s_3s_2s_1s_3s_1s_2s_1s_2s_1s_4 
s_2s_1s_2s_3s_2s_3s_1s_2s_1s_3s_2s_1s_2s_3s_1s_2s_1s_3s_1s_3
s_1s_2s_3s_2s_1s_3s_2s_3s_1s_2s_1s_2s_3s_2s_3s_2s_1s_2s_1s_3 
$ \newline $s_2s_3s_1s_2s_3s_2s_1s_3s_1s_3s_1s_2s_1s_3s_4s_1s_3s_1s_3s_1
s_2s_1s_3s_1s_2s_3s_2s_1s_2s_1s_2s_3s_1s_3s_2s_1s_2s_1s_3s_1 
s_2s_1s_3s_1s_2s_1s_2s_3s_1s_3$ \newline $s_2s_3s_2s_1s_2s_1s_2s_3s_2s_1s_2s_3s_2s_1$
}}
\end{spacing}


\newpage
\noindent
Nr.~8\,(Length=1920):
\begin{spacing}{0.4}
\noindent
{\tiny{{\small{$1^{\mbox{${\check{B}}$}}\cdot$}}\newline
$s_1s_2s_1s_3s_2s_1s_2s_3s_1s_2s_1s_2s_1s_3s_2s_1s_2s_3s_1s_2
s_1s_3s_2s_3s_1s_3s_2s_1s_2s_3s_2s_3s_2s_3s_2s_1s_2s_1s_2s_3 
s_2s_1s_3s_1s_3s_1s_2s_1s_4s_1$ \newline $s_3s_2s_3s_1s_3s_1s_3s_2s_3s_1
s_3s_1s_3s_2s_3s_2s_3s_1s_3s_1s_3s_2s_3s_1s_2s_1s_3s_1s_2s_1 
s_3s_1s_3s_1s_2s_3s_2s_1s_2s_1s_3s_1s_3s_2s_3s_1s_3s_2s_1s_2
$ \newline $s_3s_2s_3s_2s_1s_2s_3s_2s_1s_2s_3s_1s_2s_4s_3s_2s_3s_2s_3s_1 
s_3s_2s_3s_1s_3s_1s_3s_2s_3s_2s_3s_1s_3s_4s_3s_1s_3s_2s_3s_1
s_3s_1s_2s_1s_3s_2s_1s_2s_3s_1$ \newline $s_2s_1s_3s_2s_3s_1s_3s_2s_1s_2 
s_3s_2s_1s_2s_3s_2s_3s_2s_1s_2s_1s_2s_1s_2s_3s_2s_1s_2s_3s_2
s_1s_3s_1s_2s_1s_2s_1s_3s_1s_2s_3s_1s_3s_2s_3s_2s_3s_1s_3s_1 
$ \newline $s_3s_4s_3s_1s_3s_1s_3s_2s_3s_2s_3s_1s_3s_2s_3s_1s_3s_1s_3s_2
s_3s_2s_3s_4s_3s_1s_2s_1s_3s_1s_3s_1s_2s_3s_2s_1s_2s_1s_2s_3 
s_2s_1s_2s_3s_2s_1s_3s_1s_3s_1$ \newline $s_2s_1s_3s_2s_3s_1s_3s_1s_3s_2
s_4s_2s_3s_1s_3s_2s_3s_1s_3s_2s_3s_1s_2s_1s_3s_2s_3s_1s_3s_1 
s_3s_2s_3s_2s_3s_2s_3s_1s_3s_1s_3s_1s_3s_2s_3s_1s_3s_2s_3s_1
$ \newline $s_2s_1s_3s_2s_3s_1s_3s_2s_3s_1s_3s_1s_3s_2s_3s_1s_3s_2s_3s_1 
s_2s_1s_3s_1s_3s_2s_3s_1s_3s_1s_3s_2s_3s_1s_2s_1s_3s_2s_3s_1
s_3s_2s_3s_2s_3s_1s_3s_2s_3s_1$ \newline $s_3s_1s_3s_2s_3s_1s_4s_1s_2s_4 
s_3s_4s_3s_2s_3s_4s_3s_1s_3s_4s_3s_1s_3s_1s_2s_3s_4s_3s_1s_3
s_2s_1s_3s_4s_2s_1s_2s_3s_4s_3s_1s_2s_1s_2s_3s_4s_3s_1s_3s_2 
$ \newline $s_3s_1s_3s_1s_3s_2s_3s_2s_3s_4s_3s_2s_3s_2s_3s_1s_3s_2s_4s_3
s_4s_3s_2s_3s_2s_3s_4s_3s_1s_2s_1s_3s_2s_3s_1s_3s_4s_3s_2s_3 
s_1s_3s_1s_2s_1s_3s_4s_3s_1s_3$ \newline $s_2s_4s_2s_4s_2s_3s_1s_2s_3s_2
s_1s_3s_2s_3s_2s_3s_1s_2s_3s_2s_1s_3s_2s_3s_1s_2s_4s_2s_3s_1 
s_2s_1s_3s_2s_1s_2s_3s_1s_2s_1s_3s_2s_1s_2s_3s_4s_3s_2s_1s_3
$ \newline $s_4s_2s_1s_3s_4s_2s_1s_3s_4s_3s_1s_3s_2s_3s_2s_3s_1s_3s_4s_3 
s_1s_2s_3s_4s_3s_1s_2s_3s_2s_1s_3s_4s_3s_1s_2s_3s_4s_3s_1s_2
s_3s_2s_1s_2s_3s_2s_1s_3s_1s_3$ \newline $s_4s_3s_2s_3s_2s_3s_1s_3s_1s_3 
s_2s_3s_1s_3s_2s_3s_2s_3s_1s_3s_1s_3s_2s_3s_4s_2s_1s_3s_1s_3
s_1s_2s_1s_3s_2s_3s_4s_2s_3s_2s_3s_2s_1s_2s_3s_4s_1s_3s_1s_3 
$ \newline $s_1s_2s_3s_2s_1s_3s_4s_1s_2s_1s_3s_4s_2s_1s_2s_3s_2s_1s_3s_4
s_2s_1s_3s_4s_3s_2s_3s_2s_3s_1s_3s_1s_3s_2s_3s_2s_3s_1s_3s_1 
s_3s_2s_3s_2s_3s_1s_3s_4s_3s_1$ \newline $s_2s_1s_3s_1s_2s_3s_2s_1s_2s_1
s_2s_3s_2s_1s_2s_3s_2s_1s_3s_1s_3s_1s_2s_1s_3s_1s_3s_1s_2s_4 
s_3s_1s_3s_2s_3s_2s_3s_1s_3s_4s_3s_1s_2s_1s_2s_1s_3s_4s_2s_3
$ \newline $s_1s_3s_2s_3s_2s_3s_1s_3s_2s_1s_2s_3s_2s_3s_2s_1s_2s_3s_2s_3 
s_2s_4s_2s_1s_3s_1s_3s_1s_3s_1s_2s_1s_3s_1s_3s_1s_2s_1s_2s_1
s_3s_1s_2s_1s_3s_4s_3s_2s_3s_2$ \newline $s_3s_1s_3s_1s_3s_2s_3s_2s_3s_1 
s_3s_1s_3s_2s_3s_2s_3s_1s_3s_4s_3s_1s_2s_1s_2s_1s_3s_4s_3s_2
s_3s_2s_3s_1s_3s_1s_3s_2s_3s_2s_3s_4s_3s_1s_2s_4s_3s_1s_2s_1 
$ \newline $s_2s_3s_1s_3s_2s_1s_3s_1s_2s_1s_3s_1s_2s_3s_1s_3s_2s_3s_4s_3
s_1s_3s_1s_3s_2s_3s_2s_3s_1s_3s_2s_3s_1s_3s_1s_3s_2s_3s_2s_3 
s_1s_3s_4s_3s_1s_3s_4s_3s_1s_3$ \newline $s_1s_3s_2s_3s_4s_3s_2s_3s_2s_3
s_1s_3s_1s_3s_2s_3s_1s_3s_2s_3s_2s_3s_1s_3s_1s_3s_2s_3s_4s_3 
s_1s_3s_2s_3s_1s_2s_1s_3s_1s_2s_1s_3s_1s_3s_1s_3s_1s_2s_1s_3
$ \newline $s_1s_3s_1s_2s_1s_2s_1s_3s_1s_3s_4s_3s_1s_2s_1s_3s_2s_1s_2s_3 
s_1s_2s_1s_2s_1s_3s_2s_1s_2s_3s_1s_2s_1s_3s_2s_3s_1s_3s_2s_1
s_2s_3s_2s_3s_2s_3s_2s_4s_2s_3$ \newline $s_2s_3s_2s_1s_2s_3s_2s_1s_2s_3 
s_1s_3s_2s_3s_2s_3s_1s_3s_2s_1s_2s_1s_2s_3s_2s_1s_2s_3s_2s_1
s_3s_1s_4s_3s_1s_3s_2s_3s_1s_2s_3s_2s_1s_2s_1s_3s_1s_3s_1s_3 
$ \newline $s_1s_2s_1s_3s_1s_2s_1s_3s_2s_3s_1s_2s_1s_3s_1s_3s_1s_2s_1s_3
s_1s_2s_1s_3s_2s_3s_1s_3s_2s_3s_1s_3s_2s_3s_1s_3s_1s_3s_2s_3 
s_2s_3s_1s_4s_3s_2s_1s_2s_3s_1$ \newline $s_3s_2s_3s_2s_3s_1s_3s_2s_4s_3
s_1s_3s_2s_3s_1s_3s_1s_2s_3s_2s_1s_3s_4s_1s_3s_2s_3s_1s_2s_3 
s_2s_3s_2s_1s_3s_2s_3s_1s_2s_3s_2s_1s_2s_1s_2s_1s_2s_3s_2s_3
$ \newline $s_2s_1s_2s_1s_3s_1s_3s_1s_2s_1s_3s_1s_2s_1s_3s_2s_3s_1s_3s_1 
s_3s_2s_1s_2s_3s_1s_2s_1s_3s_2s_1s_2s_3s_1s_3s_1s_3s_2s_3s_4
s_3s_2s_3s_1s_3s_1s_3s_2s_3s_1$ \newline $s_2s_1s_3s_1s_3s_1s_2s_1s_3s_1 
s_3s_1s_2s_3s_2s_1s_2s_3s_2s_1s_3s_4s_1s_2s_3s_2s_1s_2s_3s_1
s_3s_2s_3s_1s_3s_2s_3s_1s_3s_2s_1s_2s_3s_2s_1s_2s_3s_2s_1s_2 
$ \newline $s_1s_3s_1s_2s_3s_2s_1s_2s_3s_2s_1s_2s_1s_2s_3s_2s_1s_3s_1s_3
s_1s_2s_1s_2s_3s_1s_3s_2s_1s_3s_1s_2s_1s_3s_1s_3s_1s_2s_4s_3 
s_2s_3s_1s_3s_2s_3s_2s_3s_1s_3$ \newline $s_1s_3s_4s_1s_2s_3s_1s_2s_1s_2
s_1s_3s_2s_1s_2s_3s_1s_2s_1s_3s_1s_2s_1s_3s_2s_3s_2s_3s_2s_3 
s_1s_3s_1s_3s_2s_3s_2s_3s_1s_3s_2s_1s_2s_1s_2s_3s_2s_1s_2s_3
$ \newline $s_2s_1s_3s_1s_2s_1s_2s_1s_3s_1s_2s_3s_1s_3s_2s_3s_2s_3s_1s_2 
s_4s_3s_1s_2s_3s_2s_1s_2s_1s_2s_1s_2s_3s_2s_3s_2s_1s_2s_3s_4
s_3s_2s_3s_2s_1s_2s_3s_2s_3s_2$ \newline $s_1s_2s_3s_1s_3s_2s_3s_2s_3s_1 
s_3s_2s_1s_4s_1s_2s_3s_2s_1s_3s_2s_3s_1s_2s_3s_2s_1s_2s_3s_2
s_1s_3s_2s_3s_1s_2s_3s_2s_3s_2s_1s_3s_1s_3s_1s_2s_1s_3s_2s_3 
$ \newline $s_1s_3s_2s_3s_1s_2s_1s_3s_1s_3s_1s_2s_4s_2s_3s_2s_3s_2s_1s_2
s_3s_1s_3s_2s_1s_2s_3s_2s_3s_2s_1s_2s_1s_2s_3s_2s_1s_3s_1s_2 
s_3s_2s_1s_2s_1s_2s_3s_4s_2s_1$ \newline $s_3s_2s_4s_3s_1s_2s_3s_2s_1s_2
s_1s_2s_3s_2s_1s_3s_1s_2s_1s_4s_1s_3s_1s_3s_2s_3s_2s_3s_2s_3 
s_1s_3s_1s_3s_2s_3s_2s_3s_2s_3s_1s_3s_2s_3s_1s_3s_1s_3s_2s_3
$ \newline $s_1s_3s_2s_3s_1s_3s_2s_3s_1s_3s_2s_3s_1s_3s_2s_3s_2s_3s_1s_3 
s_2s_3s_1s_3s_2s_3s_2s_3s_1s_3s_2s_3s_2s_3s_2s_3s_1s_3s_1s_3
s_2s_3s_1s_3s_1s_3s_2s_3s_1s_3$ \newline $s_2s_3s_1s_3s_2s_3s_1s_3s_2s_3 
s_1s_3s_2s_4s_2s_3s_1s_3s_1s_3s_2s_3s_1s_3s_2s_3s_2s_3s_1s_3
s_1s_3s_4s_1s_2s_3s_1s_2s_1s_3s_2s_3s_1s_3s_2s_1s_2s_3s_2s_1 
$ \newline $s_2s_3s_2s_3s_2s_1s_2s_1s_2s_1s_2s_3s_2s_1s_2s_3s_2s_1s_3s_1
s_2s_1s_2s_1s_3s_1s_2s_3s_1s_3s_2s_3s_2s_3s_1s_3s_1s_3s_2s_3 
s_2s_3s_2s_3s_4s_3s_2s_3s_1s_3$ \newline $s_1s_2s_1s_3s_1s_2s_1s_2s_1s_3
s_1s_3s_1s_2s_1s_3s_1s_3s_1s_2s_1s_2s_1s_3s_4s_3s_1s_3s_2s_3 
s_1s_3s_2s_3s_1s_3s_2s_3s_1s_2s_1s_3s_1s_2s_1s_3s_1s_3s_1s_2
$ \newline $s_1s_3s_2s_3s_1s_2s_1s_3s_1s_2s_1s_3s_1s_3s_1s_3s_1s_2s_1s_2 
s_3s_2s_1s_3s_2s_3s_1s_3s_2s_3s_1s_2s_3s_2s_1s_3s_1s_3s_1s_2
s_4s_3s_2s_3s_1s_3s_2s_3s_2s_3$ \newline $s_1s_3s_1s_3s_2s_3s_1s_3s_2s_3 
s_2s_3s_4s_2s_3s_2s_1s_2s_3s_2s_3s_2s_1s_3s_1s_2s_1s_2s_1s_3
s_1s_2s_3s_2s_1s_2s_3s_2s_1s_2s_1s_2s_3s_4s_3s_2s_3s_1s_3s_1 
$ \newline $s_3s_2s_3s_4s_3s_2s_1s_2s_1s_3s_1s_3s_1s_2s_1s_3s_2s_1s_2s_3
s_2s_1s_2s_1s_2s_3s_1s_2s_1s_3s_2s_1s_2s_3s_2s_1s_2s_3s_1s_3 
s_2s_3s_1s_3s_1s_2s_3s_2s_1s_3$ \newline $s_2s_3s_1s_2s_3s_2s_3s_2s_1s_3
s_2s_3s_1s_2s_3s_2s_1s_3s_4s_1s_2s_1s_2s_3s_2s_3s_2s_1s_2s_3 
s_2s_3s_2s_1s_2s_1s_2s_3s_2s_1s_2s_1s_2s_3s_2s_3s_2s_1s_2s_3
$ \newline $s_2s_3s_2s_4s_2s_1s_3s_1s_2s_3s_1s_3s_2s_3s_2s_3s_1s_3s_2s_3$
}}
\end{spacing}
\vskip\baselineskip
\noindent
Nr.~10\,(Length=2688):
\begin{spacing}{0.4}
\noindent
{\tiny{{\small{$1^{\mbox{${\check{B}}$}}\cdot$}} \newline
$s_1s_3s_1s_3s_1s_2s_1s_3s_1s_3s_1s_2s_1s_2s_1s_3s_2s_3s_1s_3
s_2s_3s_2s_3s_1s_2s_1s_3s_1s_2s_1s_3s_1s_2s_3s_2s_1s_2s_1s_2 
s_1s_3s_2s_3s_1s_2s_3s_2s_1s_2$ \newline $s_1s_2s_3s_2s_1s_2s_1s_3s_1s_3
s_1s_2s_1s_3s_1s_3s_1s_2s_1s_4s_1s_2s_3s_4s_3s_2s_1s_2s_1s_2 
s_3s_4s_3s_1s_3s_1s_3s_2s_3s_4s_3s_2s_3s_1s_3s_4s_2s_1s_3s_4
$ \newline $s_3s_1s_3s_4s_1s_3s_4s_3s_1s_2s_3s_2s_1s_3s_2s_3s_1s_2s_3s_4 
s_1s_2s_1s_3s_4s_2s_3s_4s_1s_3s_4s_2s_1s_3s_1s_3s_4s_3s_1s_2
s_3s_4s_1s_2s_3s_4s_3s_2s_3s_4$ \newline $s_3s_2s_1s_2s_3s_4s_1s_3s_4s_3 
s_1s_3s_2s_3s_1s_3s_4s_3s_2s_1s_3s_4s_3s_2s_1s_2s_3s_2s_1s_2
s_3s_4s_2s_1s_3s_4s_3s_1s_3s_4s_3s_1s_3s_1s_3s_2s_3s_1s_3s_4 
$ \newline $s_2s_3s_2s_1s_2s_1s_2s_3s_4s_2s_3s_4s_3s_1s_3s_4s_3s_1s_2s_1
s_2s_1s_2s_1s_2s_1s_3s_1s_3s_4s_3s_2s_1s_3s_4s_2s_1s_3s_1s_2 
s_3s_2s_1s_2s_3s_1s_3s_2s_3s_1$ \newline $s_3s_4s_3s_2s_3s_2s_3s_1s_3s_4
s_3s_1s_3s_1s_3s_2s_3s_1s_3s_4s_3s_1s_2s_1s_3s_4s_3s_2s_3s_1 
s_3s_2s_3s_4s_3s_2s_1s_2s_1s_2s_3s_2s_1s_3s_1s_2s_3s_2s_1s_2
$ \newline $s_1s_2s_3s_1s_3s_2s_1s_3s_1s_2s_3s_2s_3s_4s_1s_3s_1s_3s_1s_2 
s_1s_3s_4s_2s_3s_2s_1s_2s_3s_2s_1s_3s_4s_2s_3s_2s_1s_2s_1s_2
s_3s_2s_1s_3s_1s_2s_1s_3s_1s_3$ \newline $s_4s_2s_1s_2s_1s_3s_1s_3s_4s_3 
s_1s_2s_1s_3s_1s_2s_1s_3s_1s_3s_4s_3s_1s_2s_3s_2s_1s_3s_2s_3
s_1s_2s_3s_2s_1s_3s_2s_3s_1s_2s_1s_2s_1s_2s_1s_3s_1s_2s_3s_2 
$ \newline $s_1s_2s_3s_2s_1s_3s_1s_2s_1s_2s_1s_2s_1s_3s_1s_3s_1s_2s_1s_2
s_1s_3s_2s_3s_2s_3s_1s_3s_2s_3s_4s_3s_1s_2s_1s_3s_1s_3s_1s_2 
s_1s_2s_1s_3s_1s_3s_2s_3s_1s_3$ \newline $s_1s_3s_2s_3s_1s_3s_1s_3s_4s_3
s_1s_2s_1s_3s_1s_3s_1s_2s_1s_3s_2s_3s_2s_3s_1s_3s_2s_1s_2s_3 
s_2s_1s_2s_3s_2s_1s_2s_3s_1s_3s_1s_3s_2s_3s_1s_3s_1s_3s_2s_3
$ \newline $s_1s_3s_1s_3s_2s_3s_1s_3s_1s_3s_2s_3s_1s_3s_2s_3s_1s_3s_2s_3 
s_1s_2s_1s_3s_1s_2s_1s_3s_1s_2s_1s_2s_1s_2s_1s_3s_1s_3s_1s_2
s_1s_3s_1s_2s_1s_3s_4s_2s_1s_2$ \newline $s_1s_3s_1s_2s_1s_3s_1s_3s_1s_2 
s_1s_3s_1s_3s_1s_2s_3s_2s_4s_2s_1s_3s_1s_2s_1s_2s_1s_2s_1s_3
s_1s_2s_3s_2s_1s_2s_3s_1s_3s_2s_3s_1s_3s_1s_3s_2s_3s_2s_3s_1 
$ \newline $s_3s_2s_3s_1s_3s_2s_1s_2s_1s_2s_3s_2s_1s_2s_1s_2s_3s_2s_4s_2
s_1s_3s_1s_2s_1s_3s_1s_3s_1s_2s_1s_3s_1s_2s_1s_2s_1s_3s_1s_2 
s_1s_3s_2s_3s_1s_3s_1s_3s_2s_3$ \newline $s_1s_3s_2s_3s_1s_3s_1s_3s_1s_3
s_2s_3s_2s_3s_1s_3s_2s_3s_1s_4s_1s_2s_3s_1s_3s_2s_1s_2s_3s_4 
s_3s_1s_3s_1s_3s_1s_3s_2s_3s_2s_3s_1s_3s_2s_1s_2s_1s_2s_3s_2
$ \newline $s_1s_2s_3s_2s_1s_2s_1s_2s_3s_2s_3s_2s_1s_2s_3s_1s_3s_1s_3s_2 
s_3s_1s_3s_1s_3s_1s_3s_2s_3s_2s_3s_4s_2s_1s_3s_2s_3s_1s_3s_1
s_3s_2s_3s_2s_3s_1s_3s_1s_3s_2$ \newline $s_3s_1s_2s_1s_3s_1s_3s_1s_2s_1 
s_2s_1s_3s_1s_3s_4s_3s_1s_3s_2s_3s_1s_3s_2s_3s_1s_3s_1s_3s_2
s_3s_1s_3s_4s_3s_1s_3s_1s_3s_2s_3s_1s_3s_1s_3s_2s_3s_1s_2s_1 
$ \newline $s_3s_1s_2s_1s_3s_1s_3s_1s_3s_1s_2s_1s_3s_1s_3s_1s_2s_1s_2s_1
s_3s_1s_3s_4s_3s_1s_2s_1s_2s_1s_3s_1s_3s_1s_2s_1s_3s_1s_2s_1 
s_2s_1s_3s_1s_2s_1s_2s_1s_2s_1$ \newline $s_3s_1s_2s_1s_3s_1s_3s_1s_2s_1
s_3s_1s_2s_1s_3s_1s_2s_1s_2s_1s_3s_2s_3s_1s_2s_1s_3s_1s_3s_1 
s_2s_1s_3s_2s_3s_2s_3s_1s_3s_2s_3s_4s_1s_2s_1s_2s_3s_2s_1s_2
$ \newline $s_3s_2s_1s_3s_1s_2s_1s_2s_1s_3s_1s_2s_3s_2s_3s_2s_1s_2s_3s_2 
s_1s_2s_3s_4s_3s_2s_1s_3s_1s_2s_3s_1s_3s_2s_3s_1s_3s_2s_1s_2
s_1s_2s_3s_2s_1s_2s_1s_3s_2s_3$ \newline $s_1s_2s_3s_2s_1s_2s_3s_2s_1s_3 
s_2s_3s_1s_2s_1s_2s_4s_2s_1s_3s_1s_2s_1s_2s_1s_3s_1s_2s_1s_3
s_1s_2s_1s_2s_1s_3s_2s_3s_1s_3s_2s_3s_1s_3s_1s_3s_1s_3s_2s_3 
$ \newline $s_1s_3s_1s_3s_2s_3s_2s_3s_1s_3s_2s_3s_1s_2s_3s_2s_1s_3s_2s_3
s_1s_2s_1s_2s_1s_2s_3s_2s_1s_2s_1s_2s_3s_2s_1s_3s_1s_2s_3s_2 
s_1s_2s_1s_2s_1s_2s_3s_2s_1s_2$ \newline $s_1s_3s_1s_2s_1s_2s_1s_4s_3s_1
s_3s_2s_3s_1s_3s_2s_1s_2s_1s_2s_3s_2s_1s_3s_1s_2s_1s_2s_3s_2 
s_1s_2s_3s_1s_3s_1s_3s_2s_3s_1s_3s_2s_3s_1s_3s_1s_3s_2s_1s_2
$ \newline $s_3s_2s_1s_2s_1s_3s_1s_2s_3s_2s_1s_2s_1s_2s_3s_1s_3s_2s_3s_1 
s_3s_2s_3s_1s_2s_1s_3s_1s_2s_1s_3s_1s_2s_1s_3s_1s_2s_1s_3s_4
s_3s_2s_3s_1s_2s_1s_3s_1s_3s_1$ \newline $s_2s_3s_2s_1s_2s_3s_2s_1s_2s_1 
s_2s_3s_2s_1s_3s_1s_3s_1s_2s_1s_3s_1s_2s_1s_3s_4s_1s_2s_3s_2
s_3s_2s_1s_2s_1s_2s_3s_2s_3s_2s_1s_2s_3s_2s_1s_3s_1s_2s_1s_2 
$ \newline $s_1s_3s_2s_3s_1s_2s_1s_2s_3s_2s_1s_2s_3s_2s_1s_2s_1s_3s_2s_3
s_1s_3s_1s_2s_1s_2s_1s_3s_1s_2s_1s_3s_4s_3s_2s_3s_2s_3s_1s_3 
s_2s_3s_1s_3s_1s_3s_2s_3s_1s_3$ \newline $s_2s_3s_1s_3s_1s_3s_2s_1s_3s_1
s_2s_3s_1s_3s_2s_3s_1s_3s_1s_3s_2s_3s_1s_3s_1s_3s_2s_1s_2s_3 
s_2s_1s_2s_1s_2s_1s_2s_3s_2s_1s_2s_1s_3s_2s_3s_1s_2s_3s_2s_1
$ \newline $s_3s_2s_3s_1s_2s_3s_2s_1s_2s_3s_2s_1s_3s_2s_3s_1s_2s_1s_2s_4 
s_3s_2s_3s_1s_3s_2s_3s_1s_3s_1s_3s_2s_1s_2s_1s_2s_3s_2s_1s_2
s_3s_2s_1s_2s_1s_2s_3s_2s_1s_2$ \newline $s_1s_3s_1s_2s_3s_2s_1s_2s_1s_2 
s_1s_2s_3s_1s_3s_2s_3s_1s_3s_1s_3s_2s_3s_1s_3s_1s_3s_2s_1s_2
s_3s_2s_1s_2s_1s_3s_1s_2s_3s_2s_1s_2s_1s_2s_3s_1s_3s_2s_3s_1 
$ \newline $s_3s_2s_3s_1s_3s_4s_3s_1s_3s_1s_2s_1s_3s_1s_3s_1s_2s_1s_3s_2
s_3s_1s_3s_2s_3s_1s_3s_2s_1s_2s_1s_2s_3s_2s_1s_2s_1s_2s_3s_2 
s_1s_2s_3s_2s_1s_2s_1s_2s_1s_2$ \newline $s_1s_2s_3s_2s_1s_2s_1s_2s_3s_2
s_3s_2s_1s_2s_3s_2s_1s_2s_1s_2s_3s_2s_1s_2s_1s_2s_3s_1s_3s_2 
s_3s_1s_3s_1s_3s_2s_3s_1s_3s_2s_3s_1s_3s_1s_3s_2s_3s_4s_3s_1
$ \newline $s_3s_2s_1s_2s_3s_2s_1s_2s_3s_1s_3s_2s_3s_1s_2s_1s_3s_1s_3s_1 
s_3s_1s_2s_3s_2s_1s_3s_2s_3s_1s_3s_2s_3s_1s_2s_3s_2s_1s_3s_1
s_2s_1s_3s_1s_3s_2s_1s_2s_3s_1$ \newline $s_2s_1s_3s_4s_3s_2s_3s_2s_3s_1 
s_3s_2s_1s_2s_1s_2s_3s_2s_1s_2s_3s_2s_1s_3s_1s_2s_1s_2s_1s_3
s_1s_2s_3s_4s_3s_1s_3s_2s_3s_1s_2s_1s_2s_1s_3s_1s_2s_3s_2s_1 
$ \newline $s_2s_3s_2s_1s_3s_1s_2s_1s_2s_1s_3s_2s_3s_1s_3s_1s_3s_2s_3s_1
s_3s_2s_3s_2s_3s_1s_2s_1s_3s_1s_3s_1s_2s_3s_2s_1s_2s_1s_4s_3 
s_1s_3s_1s_2s_1s_3s_1s_3s_1s_3$ \newline $s_1s_2s_1s_3s_2s_3s_1s_3s_1s_3
s_1s_3s_2s_3s_1s_3s_2s_3s_1s_3s_1s_3s_1s_3s_2s_3s_1s_3s_4s_3 
s_1s_3s_1s_2s_1s_3s_2s_3s_1s_3s_2s_1s_2s_3s_2s_1s_2s_3s_1s_3
$ \newline $s_2s_3s_2s_3s_1s_2s_1s_3s_2s_1s_2s_3s_1s_3s_1s_2s_1s_3s_1s_3 
s_2s_3s_1s_3s_2s_3s_1s_3s_2s_3s_1s_2s_1s_3s_1s_4s_3s_2s_1s_2
s_1s_2s_1s_2s_3s_2s_3s_2s_1s_2$ \newline $s_3s_1s_3s_1s_3s_2s_3s_1s_3s_2 
s_3s_1s_3s_1s_3s_1s_3s_2s_3s_2s_3s_1s_3s_2s_3s_1s_3s_2s_1s_2
s_1s_2s_3s_2s_1s_2s_1s_2s_3s_4s_3s_1s_3s_2s_3s_1s_2s_1s_3s_1 
$ \newline $s_2s_1s_3s_1s_2s_3s_2s_1s_2s_1s_2s_1s_2s_3s_1s_3s_2s_1s_3s_1
s_2s_3s_1s_3s_2s_1s_2s_1s_2s_3s_2s_1s_3s_2s_3s_1s_2s_3s_2s_1 
s_3s_2s_3s_1s_3s_4s_3s_2s_3s_1$ \newline $s_3s_1s_3s_2s_3s_1s_3s_2s_3s_1
s_2s_1s_3s_2s_3s_1s_3s_1s_3s_2s_1s_2s_3s_2s_1s_2s_1s_2s_1s_2 
s_3s_2s_1s_2s_1s_2s_1s_2s_3s_1s_3s_2s_3s_1s_2s_1s_3s_2s_1s_2
$ \newline $s_1s_2s_3s_2s_1s_3s_1s_2s_1s_3s_1s_2s_1s_2s_1s_3s_1s_2s_1s_2 
s_3s_1s_3s_2s_1s_3s_1s_2s_1s_3s_1s_3s_1s_2s_3s_2s_1s_4s_1s_3
s_1s_3s_2s_3s_1s_3s_1s_3s_2s_3$ \newline $s_1s_2s_1s_3s_1s_3s_1s_2s_1s_3 
s_1s_3s_1s_2s_3s_2s_1s_2s_3s_2s_1s_3s_4s_3s_1s_3s_2s_3s_1s_3
s_2s_3s_1s_3s_1s_3s_2s_3s_1s_3s_1s_3s_2s_3s_1s_2s_1s_3s_1s_3 
$ \newline $s_1s_2s_1s_3s_2s_3s_1s_3s_1s_3s_2s_1s_2s_3s_2s_1s_2s_3s_2s_1
s_2s_3s_1s_3s_2s_3s_2s_3s_1s_3s_1s_3s_2s_3s_1s_2s_1s_2s_1s_2 
s_1s_3s_1s_3s_1s_2s_1s_3s_1s_2$ \newline $s_1s_2s_1s_3s_1s_2s_1s_3s_1s_3
s_1s_2s_4s_2s_1s_3s_1s_2s_3s_2s_3s_2s_1s_2s_3s_2s_1s_2s_3s_2 
s_1s_2s_3s_2s_3s_2s_1s_2s_1s_2s_3s_2s_1s_3s_1s_3s_1s_2s_1s_4
$ \newline $s_3s_2s_1s_2s_1s_3s_2s_3s_1s_2s_3s_2s_1s_3s_2s_3s_1s_3s_2s_3 
s_1s_3s_2s_3s_1s_2s_3s_2s_1s_2s_3s_2s_1s_3s_2s_3s_1s_2s_1s_3
s_1s_3s_1s_2s_1s_3s_2s_3s_2s_3$ \newline $s_1s_3s_2s_1s_2s_3s_2s_1s_2s_3 
s_2s_1s_2s_3s_1s_3s_1s_3s_2s_3s_1s_3s_1s_3s_2s_3s_1s_2s_1s_3
s_1s_2s_1s_3s_1s_2s_3s_2s_1s_2s_1s_2s_4s_2s_3s_2s_1s_2s_3s_2 
$ \newline $s_1s_2s_3s_2s_1s_2s_3s_2s_1s_3s_1s_2s_1s_3s_1s_3s_1s_2s_1s_2
s_1s_3s_1s_2s_1s_3s_1s_3s_2s_3s_1s_2s_1s_2s_3s_2s_1s_2s_3s_2 
s_1s_2s_1s_3s_2s_3s_1s_4s_1s_3$ \newline $s_2s_3s_1s_2s_1s_2s_1s_2s_1s_3
s_1s_2s_3s_2s_1s_2s_3s_2s_1s_3s_1s_2s_1s_2s_1s_2s_1s_3s_1s_3 
s_1s_2s_1s_2s_1s_3s_2s_3s_2s_3s_1s_3s_2s_3s_4s_2s_1s_3s_1s_2
$ \newline $s_1s_3s_1s_2s_3s_2s_1s_2s_3s_2s_1s_2s_1s_2s_3s_2s_1s_2s_3s_2 
s_3s_2s_1s_2s_3s_2s_1s_3s_1s_3s_1s_4s_3s_1s_2s_1s_2s_1s_2s_1
s_3s_1s_2s_1s_3s_1s_2s_1s_3s_1$ \newline $s_2s_1s_3s_1s_2s_3s_2s_1s_2s_3 
s_2s_1s_3s_1s_2s_1s_2s_1s_3s_1s_2s_1s_3s_1s_2s_1s_2s_3s_2s_1
s_2s_1s_3s_2s_3s_1s_3s_4s_3s_2s_3s_1s_3s_1s_3s_1s_3s_2s_3s_1 
$ \newline $s_3s_1s_3s_1s_3s_2s_3s_1s_2s_1s_3s_1s_2s_1s_3s_1s_3s_1s_3s_1
s_2s_1s_2s_1s_3s_1s_3s_1s_3s_1s_2s_4s_2s_1s_3s_1s_3s_2s_1s_2 
s_3s_1s_2s_1s_3s_2s_3s_2s_3s_1$ \newline $s_3s_2s_1s_2s_3s_2s_1s_2s_3s_1
s_3s_2s_3s_1s_2s_1s_3s_1s_3s_1s_3s_1s_2s_1s_3s_1s_2s_1s_3s_2 
s_3s_1s_3s_2s_3s_4s_2s_1s_3s_1s_3s_1s_2s_1s_3s_2s_3s_1s_3s_2
$ \newline $s_3s_1s_2s_1s_3s_1s_3s_1s_2s_1s_3s_2s_3s_2s_3s_1s_3s_2s_1s_2 
s_3s_2s_1s_2s_3s_2s_1s_2s_3s_1s_3s_1s_3s_2s_3s_1s_3s_1s_3s_2
s_3s_1s_2s_1s_3s_1s_3s_1s_2s_1$ \newline $s_3s_1s_2s_1s_3s_2s_3s_1s_3s_1 
s_3s_2s_3s_1s_3s_2s_3s_1s_3s_2s_3s_1s_3s_4s_1s_2s_1s_2s_3s_2
s_3s_2s_1s_2s_3s_2s_3s_2s_1s_2s_1s_2s_3s_2s_1s_2s_1s_2s_3s_2 
$ \newline $s_3s_2s_1s_2s_3s_4s_3s_2s_1s_2s_3s_2s_1s_2s_3s_1s_3s_1s_3s_2
s_3s_1s_3s_1s_3s_2s_3s_1s_3s_2s_3s_1s_3s_2s_3s_1s_2s_1s_3s_1 
s_2s_1s_3s_1s_3s_1s_3s_1s_2s_1$ \newline $s_3s_2s_3s_1s_3s_2s_1s_2s_3s_2
s_1s_4s_1s_3s_1s_2s_1s_3s_2s_3s_1s_3s_2s_3s_1s_2s_1s_3s_1s_3 
s_1s_2s_1s_3s_2s_3s_1s_2$
}}
\end{spacing}

\vspace{1cm}

Let ${\check{\rR}}(0)$ denote the FGRS 
$R(A,\fkI_{\mathrm{odd}})$ of the simple Lie superalgebra $\frg(A,\fkI_{\mathrm{odd}})$
with $D(A,\fkI_{\mathrm{odd}})$
being $D(3,1)$.
The following lemma will be used in the proof of Lemma~\ref{lemma:HamRFiveFive}.

\begin{lemma}\label{lemma:SpPrpOfRfour}
Let $\rR$ be ${\check{\rR}}(0)$, ${\check{\rR}}(1)$ or ${\check{\rR}}(4)$.
Let $k:=|\bBR|$ {\rm{(}}recall $\bBR=V(\rR)${\rm{)}}.
Then for every $B\in\bBR$ and every $j\in\fkJ_{1,4}\,(=\fkI)$,
 there exists a Hamiltonian cycle $1^B\cdot s_{i_1}s_{i_2}\cdots s_{i_k}$
 of $\Gamma(\rR)$ such that $i_k=j$.
\end{lemma}
\begin{proof}
We see that ${\check{\rR}}(0)$ and ${\check{\rR}}(1)$
are the $\chi$-associated FGRSs with
$\chi$ of \cite[Table~4, Row~13]{Hec09} and \cite[Table~4, Row~14]{Hec09}
respectively.
Then the claim 
for ${\check{\rR}}(0)$ and ${\check{\rR}}(1)$ follows from \cite[Fig.~52 and Fig.~64]{Y22}.

Assume $R={\check{\rR}}(4)$.
(We use \cite{Mathe23} to obtain $(\star1)$ and $(\star2)$ below.)
Let $\sim$ be the smallest groupoid equivalence relation on $\bBR$.
Then we have $k=864$ and $|\bBRsim|=36$ (see \cite[Table~2]{CH15}).
In the following, we again write the Hamiltonian cycle $(\star1)$ of $\Gamma({\check{\rR}}(4))$
which is the same as the one of 
Nr.4 above, and in $(\star1)$ we also write the elements ${\overline{c}}$ $(c\in\fkJ_{1,36})$ of $\bBRsim$ which are defined
as follows. Let ${\overline{1}}:=[{\check{B}}]^\sim\,(\in\bBRsim)$.
Letting $i(t)\,(\in\fkJ_{1,4})$ and  ${\overline{c(t)}}\,(\in\bBRsim)$ be the $t$-th 
subscript and superscript of the Hamiltonian cycle $(\star1)$ respectively
($t\in\fkJ_{1,864}$),
we let them mean ${\overline{1}}=\tausim_{i(1)}\tausim_{i(2)}\cdots\tausim_{i(t)}({\overline{c(t)}})$.

\vspace{1cm}

\noindent
$(\star1)$ The Hamiltonian cycle in Nr.4 with the elements ${\overline{c}}\in\bBRsim$ $(c\in\fkJ_{1,36})$:
\begin{spacing}{0.5}
\noindent
{\tiny{{\small{$1^{\mbox{${\check{B}}$}}\cdot$}}\\
$
\bsuai{3}{2}\bsuai{1}{3}\bsuai{2}{3}\bsuai{1}{2}\bsuai{3}{1}\bsuai{1}{4}\bsuai{3}{5}\bsuai{1}{6}\bsuai{2}{7}\bsuai{1}{7}\bsuai{4}{7}\bsuai{3}{8}\bsuai{2}{8}\bsuai{3}{7}\suai{1}{7}\suai{4}{7}\suai{1}{7}\bsuai{2}{6}\bsuai{1}{5}\bsuai{4}{9}\bsuai{1}{9}\bsuai{2}{9}\bsuai{3}{10}\bsuai{2}{10}\bsuai{1}{11}\bsuai{4}{12}\bsuai{1}{13}\bsuai{2}{14}\bsuai{3}{14}\bsuai{1}{15}\bsuai{4}{16}\bsuai{1}{17}\bsuai{3}{17}\bsuai{2}{17}\bsuai{1}{16}\bsuai{4}{15}\bsuai{1}{14}\bsuai{2}{13}\bsuai{1}{12}\bsuai{3}{12}\suai{1}{13}\suai{2}{14}\suai{1}{15}{\underline{\bsuai{2}{18}\bsuai{4}{19}}}\bsuai{2}{20}\bsuai{4}{21}\bsuai{2}{16}\bsuai{3}{16}\bsuai{2}{21}
$ \newline $ 
\bsuai{4}{20}\bsuai{3}{22}\bsuai{1}{22}\bsuai{2}{23}\bsuai{4}{23}\bsuai{2}{22}\suai{1}{22}\suai{2}{23}\bsuai{1}{24}\bsuai{2}{24}\bsuai{3}{25}\bsuai{2}{25}\bsuai{1}{25}\suai{2}{25}\bsuai{4}{25}\suai{2}{25}\suai{1}{25}\suai{2}{25}\bsuai{3}{24}\suai{2}{24}\bsuai{1}{23}\suai{2}{22}\bsuai{3}{20}\bsuai{2}{19}\bsuai{1}{26}\bsuai{2}{27}\bsuai{4}{28}\bsuai{2}{11}\bsuai{3}{11}\bsuai{2}{28}\bsuai{1}{28}\bsuai{3}{28}\suai{1}{28}\suai{2}{11}\bsuai{1}{10}\bsuai{4}{4}{\underline{\bsuai{1}{1}\bsuai{4}{13}}}\suai{1}{12}\suai{3}{12}\bsuai{2}{29}\bsuai{3}{29}\bsuai{1}{18}\bsuai{3}{30}\bsuai{2}{30}\bsuai{1}{30}\suai{2}{30}\bsuai{3}{18}\bsuai{2}{15}\bsuai{3}{15}
$ \newline $ 
\suai{2}{18}\bsuai{1}{29}\bsuai{2}{12}\bsuai{4}{11}\suai{1}{10}\suai{2}{10}\bsuai{3}{9}\suai{2}{9}\suai{1}{9}\suai{2}{9}\suai{1}{9}\suai{2}{9}\suai{3}{10}\suai{2}{10}\suai{1}{11}\suai{2}{28}\suai{1}{28}\suai{2}{11}\suai{3}{11}\suai{2}{28}\suai{1}{28}\suai{3}{28}\suai{1}{28}\suai{2}{11}\suai{1}{10}\suai{2}{10}\suai{1}{11}\suai{4}{12}\suai{2}{29}\suai{1}{18}\suai{2}{15}\suai{4}{16}\suai{1}{17}\suai{2}{17}\suai{1}{16}\suai{2}{21}\bsuai{1}{21}\bsuai{3}{31}\bsuai{1}{31}\bsuai{2}{31}\bsuai{3}{21}\suai{2}{16}\suai{1}{17}\suai{2}{17}\suai{1}{16}\suai{2}{21}\suai{3}{31}\suai{2}{31}\suai{1}{31}\suai{3}{21}
$ \newline $ 
\suai{1}{21}\suai{2}{16}\suai{1}{17}\suai{2}{17}\suai{1}{16}\suai{4}{15}\suai{1}{14}\suai{3}{14}\suai{2}{13}\bsuai{3}{32}\bsuai{1}{32}\bsuai{3}{13}\suai{2}{14}\suai{3}{14}\suai{2}{13}\suai{3}{32}{\underline{\suai{1}{32}\bsuai{4}{2}}}\suai{1}{3}\suai{2}{3}\suai{1}{2}\suai{3}{1}\suai{1}{4}\bsuai{2}{4}\suai{1}{1}\suai{4}{13}\suai{1}{12}\suai{3}{12}\suai{2}{29}\suai{3}{29}\suai{1}{18}\suai{3}{30}\suai{1}{30}\suai{3}{18}\suai{2}{15}\suai{1}{14}\suai{2}{13}\suai{3}{32}\suai{1}{32}\bsuai{2}{32}\suai{1}{32}\suai{3}{13}\suai{2}{14}\suai{1}{15}\suai{2}{18}\suai{3}{30}\suai{1}{30}\suai{3}{18}\suai{1}{29}\suai{3}{29}
$ \newline $ 
\suai{2}{12}\suai{4}{11}\suai{1}{10}\suai{3}{9}\suai{2}{9}\suai{3}{10}\suai{1}{11}\suai{4}{12}\suai{1}{13}\suai{2}{14}\suai{3}{14}\suai{1}{15}\suai{3}{15}\suai{2}{18}\suai{1}{29}\suai{3}{29}\suai{1}{18}\suai{2}{15}\suai{4}{16}\suai{2}{21}\suai{1}{21}\suai{2}{16}\suai{4}{15}\suai{1}{14}\suai{3}{14}\suai{2}{13}\suai{3}{32}\suai{2}{32}\suai{3}{13}\suai{1}{12}\suai{3}{12}\suai{2}{29}\suai{3}{29}\suai{1}{18}\suai{3}{30}\suai{1}{30}\suai{3}{18}\suai{2}{15}\suai{3}{15}\suai{1}{14}\suai{3}{14}\suai{2}{13}\suai{3}{32}\suai{2}{32}\suai{3}{13}\suai{1}{12}\suai{3}{12}\suai{2}{29}\suai{3}{29}\suai{1}{18}
$ \newline $ 
\suai{3}{30}\suai{1}{30}\suai{3}{18}\suai{2}{15}\suai{4}{16}\suai{2}{21}\suai{3}{31}\suai{2}{31}\suai{1}{31}\suai{2}{31}\suai{3}{21}\suai{1}{21}\suai{3}{31}\suai{2}{31}\suai{3}{21}\suai{2}{16}\suai{3}{16}\suai{1}{17}\suai{3}{17}\suai{2}{17}\suai{3}{17}\suai{2}{17}\suai{3}{17}\suai{1}{16}\suai{4}{15}\suai{2}{18}\suai{3}{30}\suai{2}{30}\suai{1}{30}\suai{2}{30}\suai{3}{18}\suai{1}{29}\suai{3}{29}\suai{2}{12}\suai{3}{12}\suai{1}{13}\suai{3}{32}\suai{2}{32}\suai{1}{32}\suai{2}{32}\suai{3}{13}\suai{2}{14}\suai{3}{14}\suai{2}{13}\suai{1}{12}\suai{4}{11}\suai{1}{10}\suai{3}{9}\suai{2}{9}\suai{3}{10}
$ \newline $ 
\suai{1}{11}\suai{2}{28}\suai{1}{28}\suai{3}{28}\suai{1}{28}\suai{3}{28}\suai{1}{28}\suai{2}{11}\suai{1}{10}\suai{3}{9}\suai{2}{9}\suai{3}{10}\suai{1}{11}\suai{4}{12}\suai{1}{13}\suai{2}{14}\suai{1}{15}\suai{3}{15}{\underline{\suai{1}{14}\bsuai{4}{33}}}\bsuai{3}{33}\bsuai{2}{1}\suai{1}{4}\suai{2}{4}\suai{3}{5}\bsuai{2}{5}\suai{1}{6}\bsuai{3}{3}\suai{1}{2}\bsuai{2}{2}\suai{1}{3}\bsuai{3}{6}\suai{1}{5}\suai{2}{5}\suai{1}{6}\suai{2}{7}\suai{1}{7}\suai{3}{8}\bsuai{1}{8}\suai{2}{8}\suai{3}{7}\suai{2}{6}\suai{1}{5}\suai{2}{5}\suai{1}{6}\suai{2}{7}\suai{3}{8}\suai{2}{8}\suai{1}{8}\suai{3}{7}
$ \newline $ 
\suai{1}{7}\suai{2}{6}\suai{3}{3}\suai{2}{3}\suai{1}{2}\suai{2}{2}\suai{3}{1}\bsuai{2}{33}\bsuai{1}{34}\bsuai{2}{34}\bsuai{3}{34}\suai{2}{34}\bsuai{1}{33}\suai{3}{33}\suai{1}{34}\suai{2}{34}\suai{1}{33}\suai{2}{1}\suai{1}{4}\bsuai{4}{10}\suai{1}{11}\suai{3}{11}\suai{1}{10}\suai{2}{10}\suai{3}{9}\suai{2}{9}\suai{1}{9}\suai{2}{9}\suai{1}{9}\suai{2}{9}\suai{3}{10}\suai{2}{10}\suai{1}{11}\suai{2}{28}\bsuai{4}{27}\bsuai{2}{26}\bsuai{1}{19}\suai{2}{20}\suai{3}{22}\suai{2}{23}\bsuai{3}{35}\bsuai{2}{36}\bsuai{1}{36}\bsuai{2}{35}\bsuai{3}{23}\suai{2}{22}\suai{3}{20}\suai{2}{19}\suai{1}{26}\suai{2}{27}
$ \newline $ 
\bsuai{1}{27}\suai{2}{26}\bsuai{3}{26}\suai{2}{27}\suai{1}{27}\suai{2}{26}\suai{1}{19}\bsuai{3}{36}\suai{1}{36}\suai{2}{35}\bsuai{1}{35}\suai{3}{23}\suai{1}{24}\suai{2}{24}\suai{1}{23}\suai{2}{22}\suai{1}{22}\suai{3}{20}\bsuai{1}{20}\suai{2}{19}\suai{3}{36}\suai{2}{35}\suai{1}{35}\suai{2}{36}\bsuai{3}{19}\suai{2}{20}\suai{3}{22}\suai{2}{23}\suai{1}{24}\suai{2}{24}\suai{3}{25}\suai{2}{25}\suai{4}{25}\suai{2}{25}\suai{3}{24}\suai{2}{24}\suai{1}{23}\suai{3}{35}\suai{2}{36}\suai{3}{19}\suai{1}{26}\suai{2}{27}\suai{1}{27}\suai{2}{26}\suai{3}{26}\suai{2}{27}\suai{1}{27}\suai{2}{26}\suai{3}{26}\suai{2}{27}
$ \newline $ 
\suai{1}{27}\suai{2}{26}\suai{1}{19}\suai{3}{36}\suai{2}{35}\suai{3}{23}\suai{1}{24}\suai{2}{24}\suai{1}{23}\suai{2}{22}\suai{1}{22}\suai{3}{20}\suai{1}{20}\suai{2}{19}\suai{3}{36}\suai{2}{35}\suai{1}{35}\suai{2}{36}\suai{3}{19}\suai{2}{20}\suai{1}{20}\suai{2}{19}\suai{3}{36}\suai{1}{36}\suai{3}{19}\suai{2}{20}\suai{1}{20}\suai{3}{22}\suai{1}{22}\suai{2}{23}\suai{3}{35}\suai{1}{35}\bsuai{4}{35}\suai{2}{36}\suai{1}{36}\suai{2}{35}\suai{3}{23}\suai{1}{24}\suai{3}{25}\suai{2}{25}\suai{3}{24}\suai{1}{23}\suai{3}{35}\suai{2}{36}\suai{3}{19}\suai{2}{20}\suai{3}{22}\suai{1}{22}\suai{3}{20}\suai{2}{19}
$ \newline $ 
\suai{1}{26}\suai{2}{27}\suai{1}{27}\suai{2}{26}\suai{3}{26}\suai{2}{27}\suai{1}{27}\suai{2}{26}\suai{1}{19}\suai{2}{20}\suai{1}{20}\suai{2}{19}\suai{3}{36}\suai{2}{35}\suai{3}{23}\suai{2}{22}\suai{1}{22}\suai{2}{23}\suai{1}{24}\suai{2}{24}\suai{3}{25}\suai{2}{25}\suai{1}{25}\suai{2}{25}\suai{3}{24}\suai{2}{24}\suai{1}{23}\suai{3}{35}\suai{1}{35}\suai{2}{36}\suai{1}{36}\suai{3}{19}\suai{1}{26}\suai{2}{27}\suai{1}{27}\suai{2}{26}\suai{1}{19}\suai{2}{20}\suai{4}{21}\suai{2}{16}\suai{1}{17}\suai{2}{17}\suai{3}{17}\suai{2}{17}\suai{1}{16}\suai{2}{21}\suai{3}{31}\suai{2}{31}\suai{3}{21}\suai{2}{16}
$ \newline $ 
\suai{1}{17}\suai{2}{17}\suai{1}{16}\suai{2}{21}\suai{3}{31}\suai{2}{31}\suai{1}{31}{\underline{\suai{2}{31}\bsuai{4}{22}}}\suai{1}{22}\suai{3}{20}\suai{2}{19}\suai{3}{36}\suai{1}{36}\suai{3}{19}\suai{2}{20}\suai{1}{20}\suai{3}{22}\suai{1}{22}\suai{2}{23}\suai{3}{35}\suai{1}{35}\suai{3}{23}\suai{1}{24}\suai{3}{25}\suai{2}{25}\suai{3}{24}\suai{1}{23}\suai{3}{35}\suai{2}{36}\suai{1}{36}\suai{2}{35}\suai{3}{23}\suai{2}{22}\suai{1}{22}\suai{3}{20}\suai{1}{20}\suai{2}{19}\suai{1}{26}\suai{2}{27}\suai{1}{27}\suai{2}{26}\suai{1}{19}\suai{3}{36}\suai{2}{35}\suai{3}{23}\suai{1}{24}\suai{2}{24}\suai{3}{25}\suai{2}{25}
$ \newline $ 
\suai{1}{25}\suai{2}{25}\suai{3}{24}\suai{2}{24}\suai{1}{23}\suai{3}{35}\suai{2}{36}\suai{3}{19}\suai{1}{26}\suai{2}{27}\suai{1}{27}\suai{2}{26}\suai{3}{26}\suai{2}{27}\suai{1}{27}{\underline{\suai{2}{26}\bsuai{4}{29}}}\suai{3}{29}\suai{1}{18}\suai{3}{30}\suai{1}{30}\suai{2}{30}\suai{1}{30}\suai{3}{18}\suai{1}{29}\suai{2}{12}\suai{3}{12}\suai{2}{29}\suai{1}{18}\suai{2}{15}\suai{1}{14}\suai{2}{13}\suai{3}{32}\suai{2}{32}\suai{1}{32}\suai{2}{32}\suai{3}{13}\suai{2}{14}\suai{1}{15}\suai{4}{16}\suai{1}{17}\suai{2}{17}\suai{3}{17}\suai{1}{16}\suai{3}{16}\suai{2}{21}\suai{1}{21}\suai{3}{31}\suai{1}{31}\suai{2}{31}
$ \newline $ 
\suai{3}{21}\suai{1}{21}\suai{3}{31}\suai{2}{31}\suai{1}{31}\suai{3}{21}\suai{1}{21}\suai{2}{16}\suai{4}{15}\suai{2}{18}\suai{1}{29}\suai{3}{29}\suai{1}{18}\suai{2}{15}\suai{3}{15}\suai{2}{18}\suai{3}{30}\suai{2}{30}\suai{1}{30}\suai{2}{30}\suai{3}{18}\suai{1}{29}\suai{3}{29}\suai{2}{12}\suai{3}{12}\suai{1}{13}\suai{3}{32}\suai{2}{32}\suai{1}{32}\suai{2}{32}\suai{3}{13}\suai{2}{14}\suai{3}{14}\suai{2}{13}\suai{1}{12}\suai{4}{11}\suai{1}{10}\suai{2}{10}\suai{1}{11}\suai{2}{28}\suai{1}{28}\suai{3}{28}\suai{1}{28}\suai{2}{11}\suai{3}{11}\suai{2}{28}\suai{1}{28}\suai{2}{11}\suai{1}{10}\suai{2}{10}
$ \newline $ 
\suai{3}{9}\suai{2}{9}\suai{1}{9}\bsuai{4}{5}\suai{1}{6}\suai{3}{3}\suai{1}{2}\suai{2}{2}\suai{1}{3}\suai{3}{6}\suai{1}{5}\suai{2}{5}\bsuai{3}{4}\suai{2}{4}\suai{1}{1}\suai{2}{33}\suai{1}{34}\suai{3}{34}\suai{1}{33}\suai{2}{1}\suai{1}{4}\suai{2}{4}\suai{1}{1}\suai{2}{33}\suai{1}{34}\suai{3}{34}\suai{1}{33}\suai{2}{1}\suai{3}{2}\suai{2}{2}\suai{1}{3}\suai{2}{3}\suai{3}{6}\suai{2}{7}\suai{1}{7}\suai{2}{6}\suai{1}{5}\suai{2}{5}\suai{3}{4}\suai{2}{4}\suai{1}{1}\suai{2}{33}\suai{1}{34}\suai{3}{34}\suai{1}{33}\suai{2}{1}\suai{3}{2}\suai{2}{2}\suai{1}{3}\suai{2}{3}
$ \newline $ 
\suai{3}{6}\suai{2}{7}\suai{1}{7}\suai{3}{8}\suai{1}{8}\suai{2}{8}\suai{1}{8}\suai{2}{8}\suai{1}{8}\bsuai{4}{8}\suai{2}{8}\suai{1}{8}\suai{2}{8}\suai{3}{7}\suai{2}{6}\suai{3}{3}\suai{1}{2}\suai{2}{2}\suai{1}{3}\suai{3}{6}\suai{2}{7}\suai{1}{7}\suai{2}{6}\suai{3}{3}\suai{1}{2}\suai{2}{2}\suai{1}{3}\suai{3}{6}\suai{1}{5}\suai{3}{4}\suai{1}{1}\suai{2}{33}\suai{3}{33}\suai{2}{1}\suai{1}{4}\suai{3}{5}\suai{2}{5}\suai{3}{4}\suai{1}{1}\suai{2}{33}\suai{1}{34}\suai{2}{34}\suai{3}{34}\suai{2}{34}\suai{3}{34}\suai{2}{34}\suai{1}{33}\suai{2}{1}\suai{1}{4}\suai{3}{5}
$ \newline $ 
\suai{2}{5}\suai{3}{4}\suai{1}{1}\suai{2}{33}\suai{3}{33}\suai{2}{1}\suai{1}{4}\suai{3}{5}\suai{1}{6}\suai{3}{3}\suai{1}{2}\suai{2}{2}\suai{1}{3}\suai{3}{6}\bsuai{4}{6}\suai{1}{5}\suai{3}{4}\suai{1}{1}\suai{3}{2}\suai{1}{3}\suai{2}{3}\suai{1}{2}\suai{3}{1}\suai{1}{4}\suai{2}{4}\suai{3}{5}\suai{2}{5}\suai{1}{6}\suai{2}{7}\suai{1}{7}\suai{2}{6}\suai{3}{3}\suai{1}{2}\suai{3}{1}\suai{2}{33}\suai{1}{34}\suai{2}{34}\suai{1}{33}\suai{3}{33}\suai{1}{34}\suai{2}{34}\suai{1}{33}\suai{3}{33}\suai{1}{34}\suai{2}{34}\suai{1}{33}\suai{2}{1}\suai{3}{2}\suai{1}{3}\suai{3}{6}
$ \newline $ 
\suai{2}{7}\suai{3}{8}\suai{2}{8}\suai{1}{8}\suai{2}{8}\suai{1}{8}\suai{2}{8}\suai{3}{7}\suai{2}{6}\suai{1}{5}\suai{2}{5}\suai{3}{4}\suai{2}{4}\suai{1}{1}
$}}\,\,
\end{spacing}

\vspace{1cm}
\noindent
In the Hamiltonian cycle $(\star1)$,
$\suai{i}{c}$ really exists
unless it is $\suai{3}{27}$ or \newline
$\suai{4}{a}$ with $a\in\{1,3,14,17,18,24,26,30,31,32,34,36\}$
(see the bold letters).
We also see that
$\suai{4}{b}$ with  $b\in\{1,14,18,26,31,32\}$
exist in the reverse cycle of $(\star1)$
(see the underlined letters).

In the following, we write the Hamiltonian cycle $(\star2)$ of $\Gamma({\check{\rR}}(4))$ 
different from $(\star1)$, where ${\overline{c}}$ $(c\in\fkJ_{1,36})$
are the same as those of $(\star1)$.

\vspace{1cm}

\noindent
$(\star2)$
\begin{spacing}{0.5}
\noindent
{\tiny{{\small{$1^{\mbox{${\check{B}}$}}\cdot$}}\\
$
\suai{1}{4}\suai{4}{10}\suai{1}{11}\suai{3}{11}\suai{1}{10}\suai{4}{4}\suai{3}{5}\suai{2}{5}\suai{1}{6}\suai{4}{6}\suai{1}{5}\suai{2}{5}\suai{1}{6}\suai{2}{7}\suai{1}{7}\suai{3}{8}\suai{1}{8}\suai{2}{8}\suai{1}{8}\suai{4}{8}\suai{1}{8}\suai{2}{8}\suai{1}{8}\suai{3}{7}\suai{1}{7}\suai{2}{6}\suai{1}{5}\suai{3}{4}\suai{1}{1}\suai{2}{33}\suai{1}{34}\bsuai{4}{17}\suai{1}{16}\suai{3}{16}\suai{1}{17}\suai{2}{17}\suai{3}{17}\suai{2}{17}\suai{1}{16}\suai{2}{21}\suai{4}{20}\suai{2}{19}\suai{4}{18}\suai{2}{15}\suai{3}{15}\suai{1}{14}\suai{3}{14}\suai{2}{13}\suai{3}{32}\suai{1}{32} 
$ \newline $ 
\suai{2}{32}\suai{1}{32}\suai{3}{13}\suai{1}{12}\suai{3}{12}\suai{1}{13}\suai{2}{14}\suai{1}{15}\suai{4}{16}\suai{2}{21}\suai{1}{21}\suai{3}{31}\suai{1}{31}\suai{2}{31}\suai{1}{31}\suai{2}{31}\suai{1}{31}\suai{3}{21}\suai{1}{21}\suai{2}{16}\suai{1}{17}\suai{2}{17}\suai{1}{16}\suai{3}{16}\suai{1}{17}\suai{2}{17}\suai{3}{17}\suai{2}{17}\suai{1}{16}\suai{2}{21}\suai{1}{21}\suai{2}{16}\suai{4}{15}\suai{1}{14}\suai{2}{13}\suai{1}{12}\suai{4}{11}\suai{2}{28}\suai{1}{28}\suai{2}{11}\suai{1}{10}\suai{2}{10}\suai{3}{9}\suai{2}{9}\suai{1}{9}\suai{3}{10}\suai{1}{11}\suai{2}{28}\suai{1}{28}\suai{2}{11} 
$ \newline $ 
\suai{1}{10}\suai{3}{9}\suai{1}{9}\suai{2}{9}\suai{3}{10}\suai{2}{10}\suai{1}{11}\suai{2}{28}\suai{1}{28}\suai{2}{11}\suai{4}{12}\suai{2}{29}\suai{3}{29}\suai{1}{18}\suai{3}{30}\suai{2}{30}\suai{3}{18}\suai{1}{29}\suai{3}{29}\suai{1}{18}\suai{3}{30}{\underline{\suai{2}{30}\bsuai{4}{36}}}\suai{2}{35}\suai{1}{35}\suai{2}{36}\suai{3}{19}\suai{2}{20}\suai{1}{20}\suai{2}{19}\suai{4}{18}\suai{2}{15}\suai{3}{15}\suai{1}{14}\suai{3}{14}\suai{2}{13}\suai{3}{32}\suai{2}{32}\suai{3}{13}\suai{1}{12}\suai{2}{29}\suai{1}{18}\suai{3}{30}\suai{2}{30}\suai{1}{30}\suai{2}{30}\suai{3}{18}\suai{1}{29}\suai{2}{12}\suai{1}{13} 
$ \newline $ 
\suai{3}{32}\suai{2}{32}\suai{3}{13}\suai{2}{14}\suai{3}{14}\suai{1}{15}\suai{4}{16}\suai{2}{21}\suai{3}{31}\suai{1}{31}\suai{3}{21}\suai{2}{16}\suai{4}{15}\suai{2}{18}\suai{1}{29}\suai{3}{29}\suai{2}{12}\suai{3}{12}\suai{1}{13}\suai{2}{14}\suai{3}{14}\suai{2}{13}\suai{1}{12}\suai{4}{11}\suai{1}{10}\suai{2}{10}\suai{1}{11}\suai{4}{12}\suai{2}{29}\suai{3}{29}\suai{1}{18}\suai{3}{30}\suai{1}{30}\suai{3}{18}\suai{2}{15}\suai{3}{15}\suai{1}{14}\suai{3}{14}\suai{2}{13}\suai{3}{32}\suai{2}{32}\suai{3}{13}\suai{1}{12}\suai{3}{12}\suai{2}{29}\suai{3}{29}\suai{1}{18}\suai{3}{30}\suai{1}{30}\suai{3}{18} 
$ \newline $ 
\suai{2}{15}\suai{3}{15}\suai{1}{14}\suai{3}{14}\suai{2}{13}\suai{3}{32}\suai{2}{32}\suai{3}{13}\suai{1}{12}\suai{4}{11}\suai{1}{10}\suai{3}{9}\suai{1}{9}\suai{2}{9}\suai{1}{9}\suai{3}{10}\suai{2}{10}\suai{3}{9}\suai{1}{9}\suai{3}{10}\suai{1}{11}\suai{3}{11}\suai{2}{28}\suai{3}{28}\suai{1}{28}\suai{3}{28}\suai{1}{28}\suai{3}{28}\suai{2}{11}\suai{4}{12}\suai{1}{13}\suai{3}{32}\suai{1}{32}\suai{2}{32}\suai{1}{32}\suai{3}{13}\suai{2}{14}\suai{3}{14}\suai{1}{15}\suai{3}{15}\suai{2}{18}\suai{3}{30}\suai{1}{30}\suai{2}{30}\suai{1}{30}\suai{3}{18}\suai{1}{29}\suai{3}{29}\suai{1}{18}\suai{2}{15} 
$ \newline $ 
\suai{4}{16}\suai{2}{21}\suai{3}{31}\suai{1}{31}\suai{3}{21}\suai{2}{16}\suai{1}{17}\suai{2}{17}\suai{3}{17}\suai{2}{17}\suai{3}{17}\suai{2}{17}\suai{1}{16}\suai{2}{21}\suai{3}{31}\suai{1}{31}\suai{3}{21}\suai{2}{16}\suai{4}{15}\suai{2}{18}\suai{1}{29}\suai{2}{12}\suai{3}{12}\suai{2}{29}\suai{4}{26}\suai{3}{26}\suai{1}{19}\suai{2}{20}\suai{1}{20}\suai{3}{22}\suai{1}{22}\suai{2}{23}\suai{3}{35}\suai{2}{36}\suai{1}{36}\suai{2}{35}\suai{3}{23}\suai{2}{22}\suai{1}{22}\suai{2}{23}\suai{1}{24}\suai{2}{24}\suai{3}{25}\suai{2}{25}\suai{1}{25}\suai{3}{24}\suai{1}{23}\suai{2}{22}\suai{1}{22}\suai{2}{23} 
$ \newline $ 
\suai{1}{24}\suai{3}{25}\suai{1}{25}\suai{2}{25}\suai{3}{24}\suai{2}{24}\suai{1}{23}\suai{3}{35}\suai{1}{35}\suai{2}{36}\suai{1}{36}\suai{3}{19}\suai{1}{26}\suai{2}{27}\suai{1}{27}\bsuai{3}{27}\suai{1}{27}\suai{2}{26}\suai{3}{26}\suai{2}{27}\suai{1}{27}\suai{2}{26}\suai{1}{19}\suai{2}{20}\suai{4}{21}\suai{2}{16}\suai{3}{16}\suai{2}{21}\suai{1}{21}\suai{3}{31}\suai{1}{31}\suai{2}{31}\suai{1}{31}\suai{2}{31}\suai{1}{31}\suai{3}{21}\suai{1}{21}\suai{2}{16}\suai{1}{17}\bsuai{4}{34}\suai{1}{33}\suai{2}{1}\suai{1}{4}\suai{3}{5}\suai{1}{6}\suai{3}{3}\suai{1}{2}\suai{2}{2}\suai{1}{3}\suai{3}{6} 
$ \newline $ 
\suai{1}{5}\suai{3}{4}\suai{1}{1}\suai{2}{33}\suai{1}{34}\suai{2}{34}\suai{1}{33}\suai{3}{33}\suai{1}{34}\suai{2}{34}\suai{1}{33}\suai{2}{1}\suai{3}{2}\suai{2}{2}\suai{1}{3}\suai{2}{3}\suai{3}{6}\suai{2}{7}\suai{1}{7}\suai{2}{6}\suai{1}{5}\suai{2}{5}\suai{3}{4}\suai{2}{4}\suai{1}{1}\suai{3}{2}\suai{1}{3}\suai{2}{3}\suai{1}{2}\suai{3}{1}\suai{1}{4}\suai{3}{5}\suai{1}{6}\suai{2}{7}\suai{1}{7}\suai{3}{8}\suai{1}{8}\suai{4}{8}\suai{1}{8}\suai{3}{7}\suai{1}{7}\suai{2}{6}\suai{3}{3}\suai{1}{2}\suai{3}{1}\suai{2}{33}\suai{1}{34}\suai{2}{34}\suai{1}{33}\suai{3}{33} 
$ \newline $ 
\suai{1}{34}\suai{2}{34}\suai{1}{33}\suai{3}{33}\suai{1}{34}\suai{2}{34}\suai{1}{33}\suai{2}{1}\suai{3}{2}\suai{1}{3}\suai{3}{6}\suai{2}{7}\suai{1}{7}\suai{2}{6}\suai{1}{5}\suai{2}{5}\suai{3}{4}\suai{2}{4}\suai{1}{1}\suai{3}{2}\suai{1}{3}\suai{2}{3}\suai{1}{2}\suai{3}{1}\suai{1}{4}\suai{2}{4}\suai{1}{1}\suai{3}{2}\suai{2}{2}\suai{3}{1}\suai{1}{4}\suai{2}{4}\suai{3}{5}\suai{2}{5}\suai{1}{6}\suai{3}{3}\suai{2}{3}\bsuai{4}{3}\suai{1}{2}\suai{2}{2}\suai{1}{3}\suai{3}{6}\suai{2}{7}\suai{3}{8}\suai{1}{8}\suai{3}{7}\suai{2}{6}\suai{3}{3}\suai{1}{2}\suai{3}{1} 
$ \newline $ 
\suai{1}{4}\suai{3}{5}\suai{2}{5}\suai{3}{4}\suai{1}{1}\suai{2}{33}\suai{1}{34}\suai{2}{34}\suai{1}{33}\suai{3}{33}\suai{1}{34}\suai{2}{34}\suai{1}{33}\suai{2}{1}\suai{1}{4}\suai{2}{4}\suai{1}{1}\suai{3}{2}\suai{1}{3}\suai{3}{6}\suai{1}{5}\suai{2}{5}\suai{1}{6}\suai{2}{7}\suai{1}{7}\suai{3}{8}\suai{1}{8}\suai{2}{8}\suai{1}{8}\suai{3}{7}\suai{1}{7}\suai{2}{6}\suai{3}{3}\suai{2}{3}\suai{1}{2}\suai{2}{2}\suai{3}{1}\suai{2}{33}\suai{1}{34}\suai{2}{34}\suai{1}{33}\suai{2}{1}\suai{1}{4}\suai{4}{10}\suai{1}{11}\suai{2}{28}\suai{1}{28}\suai{3}{28}\suai{1}{28}\suai{2}{11} 
$ \newline $ 
\suai{1}{10}\suai{3}{9}\suai{1}{9}\suai{3}{10}\suai{1}{11}\suai{2}{28}\suai{1}{28}\suai{2}{11}\suai{1}{10}\suai{3}{9}\suai{1}{9}\suai{2}{9}\suai{1}{9}\suai{4}{5}\suai{2}{5}\suai{3}{4}\suai{1}{1}\suai{3}{2}\suai{2}{2}\suai{3}{1}\suai{1}{4}\suai{2}{4}\suai{3}{5}\suai{2}{5}\suai{1}{6}\suai{3}{3}\suai{2}{3}\suai{3}{6}\suai{2}{7}\suai{3}{8}\suai{1}{8}\suai{3}{7}\suai{2}{6}\suai{3}{3}\suai{1}{2}\suai{2}{2}\suai{1}{3}\suai{3}{6}\suai{1}{5}\suai{2}{5}\suai{3}{4}\suai{2}{4}\suai{1}{1}\suai{2}{33}\suai{1}{34}\suai{2}{34}\suai{1}{33}\suai{2}{1}\suai{3}{2}\suai{1}{3} 
$ \newline $ 
\suai{3}{6}\suai{2}{7}\suai{1}{7}\suai{3}{8}\suai{1}{8}\suai{2}{8}\suai{1}{8}\suai{3}{7}\suai{1}{7}\suai{2}{6}\suai{3}{3}\suai{1}{2}\suai{3}{1}\suai{2}{33}\suai{1}{34}\suai{2}{34}\suai{1}{33}\suai{3}{33}\suai{1}{34}\suai{2}{34}\suai{1}{33}\suai{4}{14}\suai{3}{14}\suai{2}{13}\suai{3}{32}\suai{2}{32}\suai{1}{32}\suai{2}{32}\suai{3}{13}\suai{2}{14}\suai{1}{15}\suai{3}{15}\suai{1}{14}\suai{2}{13}\suai{1}{12}\suai{2}{29}\suai{1}{18}\suai{3}{30}\suai{1}{30}\suai{2}{30}\suai{1}{30}\suai{3}{18}\suai{1}{29}\suai{2}{12}\suai{4}{11}\suai{2}{28}\suai{1}{28}\suai{3}{28}\suai{2}{11}\suai{3}{11} 
$ \newline $ 
\suai{1}{10}\suai{2}{10}\suai{3}{9}\suai{2}{9}\suai{1}{9}\suai{3}{10}\suai{2}{10}\suai{3}{9}\suai{1}{9}\suai{2}{9}\suai{3}{10}\suai{2}{10}\suai{1}{11}\suai{4}{12}\suai{1}{13}\suai{2}{14}\suai{3}{14}\suai{2}{13}\suai{1}{12}\suai{3}{12}\suai{1}{13}\suai{3}{32}\suai{1}{32}\suai{2}{32}\suai{1}{32}\suai{3}{13}\suai{2}{14}\suai{3}{14}\suai{1}{15}\suai{3}{15}\suai{2}{18}\suai{3}{30}\suai{1}{30}\suai{2}{30}\suai{1}{30}\suai{3}{18}\suai{1}{29}\suai{3}{29}\suai{1}{18}\suai{2}{15}\suai{4}{16}\suai{2}{21}\suai{1}{21}\suai{2}{16}\suai{1}{17}\suai{2}{17}\suai{3}{17}\suai{2}{17}\suai{1}{16}\suai{3}{16} 
$ \newline $ 
\suai{1}{17}\suai{2}{17}\suai{1}{16}\suai{2}{21}\suai{1}{21}\suai{3}{31}\suai{1}{31}\suai{2}{31}\suai{4}{22}\suai{2}{23}\suai{3}{35}\suai{2}{36}\suai{1}{36}\suai{2}{35}\suai{3}{23}\suai{2}{22}\suai{1}{22}\suai{3}{20}\suai{1}{20}\suai{2}{19}\suai{1}{26}\suai{2}{27}\suai{3}{27}\suai{2}{26}\suai{1}{19}\suai{2}{20}\suai{1}{20}\suai{2}{19}\suai{1}{26}\suai{2}{27}\suai{3}{27}\suai{2}{26}\suai{1}{19}\suai{3}{36}\suai{1}{36}\suai{2}{35}\suai{1}{35}\suai{3}{23}\suai{1}{24}\suai{2}{24}\suai{1}{23}\suai{2}{22}\suai{1}{22}\suai{3}{20}\suai{1}{20}\suai{2}{19}\suai{1}{26}\suai{2}{27}\suai{3}{27}\suai{2}{26} 
$ \newline $ 
\suai{1}{19}\suai{3}{36}\suai{1}{36}\suai{2}{35}\suai{1}{35}\suai{3}{23}\suai{1}{24}\suai{2}{24}\suai{3}{25}\suai{2}{25}\suai{1}{25}\suai{2}{25}\suai{1}{25}\suai{2}{25}\suai{4}{25}\suai{1}{25}\suai{2}{25}\suai{1}{25}\suai{3}{24}\suai{1}{23}\suai{3}{35}\suai{2}{36}\suai{1}{36}\suai{2}{35}\suai{3}{23}\suai{1}{24}\suai{2}{24}\suai{1}{23}\suai{3}{35}\suai{2}{36}\suai{1}{36}\suai{2}{35}\suai{3}{23}\suai{2}{22}\suai{3}{20}\suai{2}{19}\suai{1}{26}\suai{3}{26}\suai{1}{19}\suai{2}{20}\suai{3}{22}\suai{1}{22}\suai{3}{20}\suai{2}{19}\suai{1}{26}\suai{2}{27}\suai{1}{27}\suai{3}{27}\suai{1}{27}\suai{3}{27} 
$ \newline $ 
\suai{1}{27}\suai{2}{26}\suai{1}{19}\suai{2}{20}\suai{3}{22}\suai{1}{22}\suai{3}{20}\suai{2}{19}\suai{1}{26}\suai{3}{26}\suai{1}{19}\suai{2}{20}\suai{3}{22}\suai{2}{23}\suai{3}{35}\suai{2}{36}\suai{1}{36}\suai{2}{35}\suai{3}{23}\suai{4}{23}\suai{2}{22}\suai{3}{20}\suai{2}{19}\suai{3}{36}\suai{2}{35}\suai{1}{35}\suai{2}{36}\suai{3}{19}\suai{2}{20}\suai{1}{20}\suai{3}{22}\suai{1}{22}\suai{2}{23}\suai{1}{24}\suai{2}{24}\suai{1}{23}\suai{3}{35}\suai{2}{36}\suai{3}{19}\suai{1}{26}\suai{2}{27}\suai{1}{27}\suai{2}{26}\suai{3}{26}\suai{2}{27}\suai{1}{27}\suai{2}{26}\suai{3}{26}\suai{2}{27}\suai{1}{27} 
$ \newline $ 
\suai{2}{26}\suai{1}{19}\suai{3}{36}\suai{2}{35}\suai{3}{23}\suai{1}{24}\suai{3}{25}\suai{1}{25}\suai{2}{25}\suai{1}{25}\suai{2}{25}\suai{1}{25}\suai{3}{24}\suai{1}{23}\suai{2}{22}\suai{1}{22}\suai{3}{20}\suai{1}{20}\suai{2}{19}\suai{3}{36}\suai{2}{35}\suai{1}{35}\suai{2}{36}\suai{3}{19}\suai{2}{20}\suai{3}{22}\suai{2}{23}\suai{1}{24}\suai{2}{24}\bsuai{4}{24}\suai{3}{25}\suai{1}{25}\suai{3}{24}\suai{2}{24}\suai{4}{24}\suai{2}{24}\suai{1}{23}\suai{2}{22}\suai{4}{31}\suai{2}{31}\suai{1}{31}\suai{3}{21}\suai{1}{21}\suai{2}{16}\suai{4}{15}\suai{2}{18}\suai{1}{29}\suai{3}{29}\suai{2}{12}\suai{4}{11} 
$ \newline $ 
\suai{2}{28}\suai{3}{28}\suai{1}{28}\suai{2}{11}\suai{4}{12}\suai{2}{29}\suai{1}{18}\suai{2}{15}\suai{3}{15}\suai{2}{18}\suai{1}{29}\suai{2}{12}\suai{1}{13}\suai{4}{1}
$}}\,\,
\end{spacing}
\vspace{1cm}
\noindent
In $(\star2)$, $\suai{3}{27}$ and
$\suai{4}{a}$ with $a\in\{3,17,24,26,34,36\}$
exist (see the bold letters). We see that 
$\suai{4}{30}$ exists
in the reverse cycle of $(\star2)$ (see the underlined pair of factors).
We emphasis that for every ${\overline{c}}\in\bBRsim$ and every $i\in\fkJ_{1,4}$,
$\suai{i}{c}$ exists as a factor of one of $(\star1)$ and $(\star2)$
and their reverse cycles, which completes the proof.
\end{proof}
\subsection{Rank $\geq 5$ cases} \label{subsection:rankFiveME}
Let $\hacun(5):=6$, $\hacun(6):=4$,  $\hacun(7):=2$, and $\hacun(8):=1$. 

Let $r\in\fkJ_{5,8}$. 
Then $\hacun(r)$ is the cardinality of the irreducible finite generalized root systems
listed in \cite[Nr.~$k$ of B.$r-2$.~Rank r]{CH15}.
Let ${\acute{\fca}}$ be a free $\bZ$-module of rank~$r$.
Let $\fkI:=\fkJ_{1,r}$.
Let ${\acute{B}}=\{{\acute{\al}}_i|i\in\fkI\}$ be a $\bZ$-base of 
${\acute{\fca}}$,
i.e., ${\acute{\fca}}=\oplus_{t=1}^r\bZ{\acute{\al}}_t$.
In this section, we let the symbol $1^{x_1}2^{x_2}\cdots r^{x_r}$ $(x_t\in\bZgeqo
\,(t\in\fkJ_{1,r}))$ of \cite[Nr.~$k$ of B.$r-2$.~Rank r]{CH15} mean 
$\sum_{t=1}^r x{\acute{\al}}_t$.
For $k\in\fkJ_{1,\hacun(r)}$, 
let ${\acute{\rR}}(r;k)$ be the irreducible FGRS
of rank $r$ defined by  
\cite[Nr.~$k$ of B.$r-2$.~Rank r]{CH15}.
Let $\Theta(r;k)$ be the set of $D_\bK(\chi,{\acute{B}})$'s of \cite[Table~4]{Hec09} such that
${\acute{\rR}}(r;k)$ is $\chi$-associated and ${\acute{B}}\in\bB({\acute{\rR}}(r;k))$; instead of writing up the elements of $\Theta(r;k)$,
we write Row's of \cite[Table~4]{Hec09}.
Then we have:
\newline\newline
$\Theta(5;1)=\{11\}$, $\Theta(5;2)=\{14\}$, $\Theta(5;3)=\{12\}$, 
$\Theta(5;4)=\{15\}$, $\Theta(5;5)=\emptyset$, $\Theta(5;6)=\{13\}$, \newline
$\Theta(6;1)=\{16\}$ $(E_6)$, $\Theta(6;2)=\{17\}$, $\Theta(6;3)=\{19\}$, 
$\Theta(6;4)=\{18\}$, \newline
$\Theta(7;1)=\{20\}$ $(E_7)$, $\Theta(7;2)=\{21\}$, \newline
$\Theta(8;1)=\{22\}$ $(E_8)$.

\begin{lemma}\label{lemma:HamRFiveFive}
$\Gamma({\acute{\rR}}(5;5))$ has a Hamiltonian cycle.
\end{lemma} 
\begin{proof}
Let $\rR:={\acute{\rR}}(5;5)$.
Recall that $I$ is $\fkJ_{1,5}$ for this $\rR$.
Let $\bBR=\{B=\{\al^B_i|i\in\fkI\}|B\in\bBR\}$
mean those of
Proposition~\ref{proposition:ordering} for $\rR$.
For $i\in\fkI$, define the map ${\tilde{\tau}}_i:\bBR\to\bBR$
by ${\tilde{\tau}}_i(B):=B^{(i)}$ $(B\in\bBR)$.
Notice $({\tilde{\tau}}_i)^2=\rmid_{\bBR}$.
let ${\mathfrak{S}}(\bBR)$ be 
the group of all bijections from $\bBR$ to $\bBR$.
Let $Z$ be the subgroup of ${\mathfrak{S}}(\bBR)$
generated by ${\tilde{\tau}}_i$ $(i\in\fkI)$.
Let $e$ denote the unit of ${\mathfrak{S}}(\bBR)$.
Then $({\tilde{\tau}}_i)^2=e$.
Let $Z$ act on $\bBR$ in a natural way,
i.e., $z\cdot B:=z(B)$ $(B\in\bBR,z\in Z)$.

For a non-empty subset $V$ of $\bBR$, $B\in\bBR$
and a non-empty subset $\fkI^\prime$ of $\fkI$,
we say that a map $v:\fkJ_{1,|V|}\to\fkI^\prime$ is {\it{an H-map with respect to $(V,B, \fkI^\prime)$}}
if  
$V=\{{\tilde{\tau}}_{v(t)}\cdots{\tilde{\tau}}_{v(2)}{\tilde{\tau}}_{v(1)}\cdot B
|t\in\fkJ_{1,|V|}\}$ 
and ${\tilde{\tau}}_{v(|V|)}\cdots{\tilde{\tau}}_{v(2)}{\tilde{\tau}}_{v(1)}\cdot B=B$.

By Lemma~\ref{lemma:bBprime}~(1), we have 
$Z\cdot B=\bBR$ for every $B\in\bBR$.
For $B\in\bBR$, let ${\bar \rR}^B:=\rR\cap(\oplus_{i=1}^4\bZ\al^B_i)$.
Then ${\bar \rR}^B$ is an FGRS of rank-4.
Let ${\bar{Z}}$ be the subgroup of $Z$
generated by ${\tilde{\tau}}_j$ $(i\in\fkJ_{1,4})$.
We can see:
\begin{equation*}
\begin{array}{l}
\mbox{For $B\in\bBR$, we have the bijection $f_B:{\bar{Z}}\cdot B\to \bB({\bar \rR}^B)$} \\
\mbox{defined by $f_B(B^\prime):=B^\prime\cap{\bar \rR}^B$
$(B^\prime\in {\bar{Z}}\cdot B)$.}
\end{array}
\end{equation*}
Since we have known that a Hamiltonian cycle of $\Gamma({\check{R}})$ exists
for every FGRS ${\check{R}}$ of rank-4,
we have an H-map with respect to $({\bar{Z}}\cdot B,B, \fkJ_{1,4})$ .
for every $B\in\bBR$.

Let $\sim$ be a smallest groupoid equivalence relation on $\bBR$.
Then $|\bBRsim|=56$. 
Applying \cite{Mathe23} to each element of $\bBRsim$ in a careful way, we see the following.
\newline\newline
$({\rm{fact}}1)$  
For every $B\in\bBR$ and every $i\in\fkJ_{1,3}$,
we have $(\bZgeqo\al^B_i\oplus\bZgeqo\al^B_5)\cap\rR
=\{\al^B_i,\al^B_5\}$. \newline
$({\rm{fact}}2)$  For every $B\in\bBR$, 
 ${\bar{\rR}}^B$
is isomorphic to ${\check{\rR}}(r)$
for some $r\in\{0,1,4\}$.
\newline\newline By $({\rm{fact}}1)$, we have:
\newline\newline
$({\rm{fact}}3)$  
For every $B\in\bBR$ and every $i\in\fkJ_{1,3}$,
we have ${\tilde{\tau}}_i{\tilde{\tau}}_5\cdot B={\tilde{\tau}}_5{\tilde{\tau}}_i\cdot B$. 
\newline\par
Let $d\in\bN$ and let $B_t\in\bBR$ $(t\in\fkJ_{1,d})$
be such that ${\bar{Z}}\cdot B_{t_1}\cap{\bar{Z}}\cdot B_{t_2}=\emptyset$
$(t_1\ne t_2)$. Let $X:=\cup_{t=1}^dZ\cdot B_t$.
Let $k:=|X|$.
Assume that there exists an H-map $g:\fkJ_{1,k}\to\fkI$ 
with respect to $(X,B_1,\fkI)$.
Define $B_{g,r}\in\bBR$ $(r\in\fkJ_{0,k})$ by
$B_{g,0}:=B_1$ and $B_{g,r}:={\tilde{\tau}}_{g(t)}(B_{g,r-1})$
$(r\in\fkJ_{1,k})$.
Then $\{B_{g,r}|r\in\fkJ_{1,k}\}=X$ and $B_{g,k}=B_1$.
Assume $X\subsetneq\bBR$.
Then ${\tilde{\tau}}_5\cdot X\ne X$.
Let $u\in\fkJ_{0,k-1}$ be such that ${\tilde{\tau}}_5\cdot B_{g,u}\notin X$.
Notice $g(u+1)\ne 5$.
Let ${\dot{B}}:={\tilde{\tau}}_5\cdot B_{g,u}$.
Then ${\bar{Z}}\cdot{\dot{B}}\cap X=\emptyset$.
Let $lb:=|{\bar{Z}}\cdot{\dot{B}}|$.
By $({\rm{fact}}2)$ and Lemma~\ref{lemma:SpPrpOfRfour},
we have a H-map $y:\fkJ_{1,l}\to\fkJ_{1,4}$
with respect to $({\bar{Z}}\cdot{\dot{B}},{\dot{B}},\fkJ_{1,4})$
such that $y(l)=g(u+1)$.
Define the map $x:\fkJ_{k+l}\to\bBR$
by 
\begin{equation*}
x(t)
:=\left\{\begin{array}{ll}
g(t) & \quad (t\in\fkJ_{1,u}), \\
5 & \quad (t=u+1), \\ 
y(t-(u+1)) & \quad (t\in\fkJ_{u+2,u+l}), \\
5 & \quad (t=u+l+1), \\
g(t-l) & \quad (t\in\fkJ_{u+l+2,k+l}). \\
\end{array}\right.
\end{equation*}
By $({\rm{fact3}})$, we see that $x$ is an H-map with respect to $({\bar{Z}}\cdot{\dot{B}}\cup X,
B_1,\fkI)$.
Repeating this argument, we complete the proof.
\end{proof}

\subsection{Main result}

By Lemmas~\ref{lemma:rkonetwo}, \ref{lemma:redHam}, 
\ref{lemma:LSasschi}~(3),
\ref{lemma:HamRFiveFive}, 
Theorems~\ref{theorem:YamMain}, \ref{theorem:classGHlist},
and Subsections~\ref{subsection:rankthree},
\ref{subsection:rankfour}, \ref{subsection:rankFiveME},
we have our main result of this paper:

\begin{theorem}\label{theorem:Main}
For every FGRS $\rR$,
$\Gamma(\rR)$ has a Hamiltonian cycle. 
\end{theorem}

\section{Appendix} 
\subsection{Bipartite Ramanujan Graph} 
In this section,
let $\Gamma=\Gamma(V,E)$ be a finite graph with $|V|\geq 2$, and  let $n:=|V|$.

Let $f:\fkJ\to V$ be a bijection.
Define the real symmetric $n\times n$-matrix $A^\Gamma=(a_{ij})_{i,j\in\fkJ_{1,n}}$
by $a_{ij}:=1$ (resp. $a_{ij}:=0$) for $i$, 
$j\in\fkJ_{1,n}$ with $\{f(i),f(j)\}\in E$
(resp. $\{f(i),f(j)\}\notin E$).
The $A^\Gamma$ is called the {\it{adjacency matrix}} of $\Gamma$.
Let $\lambda^\Gamma_i$ $(i\in\fkJ_{1,n})$ be the eigenvalues
of $A^\Gamma$ with $\lambda_j\geq\lambda_{j+1}$ $(j\in\fkJ_{1,n-1})$.
That is $\det(x E-A)=\prod_{i=1}^n(x-\lambda^\Gamma_i)$ $(x\in\bR)$,
where $E$ is the identity $n\times n$-matrix.
We have $\lambda^\Gamma_1>0$ and 
$\lambda^\Gamma_1\geq \lambda^\Gamma_i\geq -\lambda^\Gamma_1$ $(i\in\fkJ_{2,n})$
(see \cite[Proposition~3.1.1]{BH12}).

We say that $\Gamma$ is {\it{bipartite}} if there exist $m\in\fkJ_{1,n-1}$ and a bijection
$f:\fkJ_{1,n}\to V$ such that $E\subset\{\{f(u),f(v)\}|u\in\fkJ_{1,m},v\in\fkJ_{m+1,n}\}$.
For $d\in\bN$, we say that $\Gamma$ is {\it{$d$-regular}} if
for every $v\in V$, $|\{v^\prime\in V|\{v,v^\prime\}\in E\}|=d$.

If $\Gamma$ is connected, then $\lambda^\Gamma_1>\lambda^\Gamma_i$ $(i\in\fkJ_{2,n})$ (see \cite[Proposition~3.1.1]{BH12}).
If $\Gamma$ is bipartite, then $\lambda^\Gamma_{n+1-i}=-\lambda^\Gamma_i$ $(i\in\fkJ_{1,n})$ 
(see \cite[Proposition~3.4.1]{BH12}).
If $\Gamma$ is $d$-regular, then $\lambda^\Gamma_1=d$ (see \cite[Proposition~3.1.2]{BH12}).
If $\Gamma$ is a connected $d$-regular finite graph,
then we say that $\Gamma$ is a {\it{Ramanuian graph}} if 
$|\lambda^\Gamma_i|\leq 2\sqrt{d-1}$ for $i\in\fkJ_{2,n}$ (see \cite[\S 4.1]{BH12}).
If $\Gamma$ is a connected $d$-regular bipartite finite graph,
then we say that $\Gamma$ is a {\it{bipartite Ramanuian graph}} if 
$|\lambda^\Gamma_i|\leq 2\sqrt{d-1}$ for $i\in\fkJ_{2,n-1}$ (see \cite[\S 1.1]{MSS15}).
See \cite[Definition~2.6, Proposition~7.2, Theorem~7.4]{A11} to know
how bipartite Ramanuian graphs are important for Ihara zeta functions.

It is clear that $\Gamma(\rR)$ for a finite generalized root system $\rR$ is
a connected $|\fkI|$-regular bipartite finite graph, where $\fkI$ is the one of 
Definition~\ref{definition:CayleyofFgrt} for $\rR$.

In the following, we write approximate values of 
$\lambda^{\Gamma({\hat{\rR}}(k))}_2$ for $k\in\fkJ_{1,55}$ 
(see Subsection~\ref{subsection:rankthree} for ${\hat{\rR}}(k))$ and recall $|\fkI|=3$),
which we have found by a program of Mathematica~13.3 \cite{Mathe23} (or its earlier version).
Note $2\sqrt{3-1}=2\sqrt{2}\approx 2.8284271$.

\newpage

\noindent
$\lambda^{\Gamma({\hat{\rR}}(1))}_2\approx 2.4142136$,
$\lambda^{\Gamma({\hat{\rR}}(2))}_2\approx 2.5615528$,
$\lambda^{\Gamma({\hat{\rR}}(3))}_2\approx 2.6818990$,
$\lambda^{\Gamma({\hat{\rR}}(4))}_2\approx 2.7320508$,
$\lambda^{\Gamma({\hat{\rR}}(5))}_2\approx 2.7320508$,
$\lambda^{\Gamma({\hat{\rR}}(6))}_2\approx 2.8051267$,
$\lambda^{\Gamma({\hat{\rR}}(7))}_2\approx 2.7855773$,
$\lambda^{\Gamma({\hat{\rR}}(8))}_2\approx 2.8252517$,
$\lambda^{\Gamma({\hat{\rR}}(9))}_2\approx 2.8565004$,
$\lambda^{\Gamma({\hat{\rR}}(10))}_2\approx 2.8507383$,
$\lambda^{\Gamma({\hat{\rR}}(11))}_2\approx 2.8832262$,
$\lambda^{\Gamma({\hat{\rR}}(12))}_2\approx 2.8823225$,
$\lambda^{\Gamma({\hat{\rR}}(13))}_2\approx 2.8693683$,
$\lambda^{\Gamma({\hat{\rR}}(14))}_2\approx 2.8693683$,
$\lambda^{\Gamma({\hat{\rR}}(15))}_2\approx 2.8951130$,
$\lambda^{\Gamma({\hat{\rR}}(16))}_2\approx 2.9079202$,
$\lambda^{\Gamma({\hat{\rR}}(17))}_2\approx 2.9229562$,
$\lambda^{\Gamma({\hat{\rR}}(18))}_2\approx 2.9175898$,
$\lambda^{\Gamma({\hat{\rR}}(19))}_2\approx 2.9325880$,
$\lambda^{\Gamma({\hat{\rR}}(20))}_2\approx 2.9325880$,
$\lambda^{\Gamma({\hat{\rR}}(21))}_2\approx 2.9284749$,
$\lambda^{\Gamma({\hat{\rR}}(22))}_2\approx 2.9352136$,
$\lambda^{\Gamma({\hat{\rR}}(23))}_2\approx 2.9337466$,
$\lambda^{\Gamma({\hat{\rR}}(24))}_2\approx 2.9418783$,
$\lambda^{\Gamma({\hat{\rR}}(25))}_2\approx 2.9423305$,
$\lambda^{\Gamma({\hat{\rR}}(26))}_2\approx 2.9409603$,
$\lambda^{\Gamma({\hat{\rR}}(27))}_2\approx 2.9425892$,
$\lambda^{\Gamma({\hat{\rR}}(28))}_2\approx 2.9406753$,
$\lambda^{\Gamma({\hat{\rR}}(29))}_2\approx 2.9509914$,
$\lambda^{\Gamma({\hat{\rR}}(30))}_2\approx 2.9492878$,
$\lambda^{\Gamma({\hat{\rR}}(31))}_2\approx 2.9474230$,
$\lambda^{\Gamma({\hat{\rR}}(32))}_2\approx 2.9549913$,
$\lambda^{\Gamma({\hat{\rR}}(33))}_2\approx 2.9514269$,
$\lambda^{\Gamma({\hat{\rR}}(34))}_2\approx 2.9514100$,
$\lambda^{\Gamma({\hat{\rR}}(35))}_2\approx 2.9563145$,
$\lambda^{\Gamma({\hat{\rR}}(36))}_2\approx 2.9678233$,
$\lambda^{\Gamma({\hat{\rR}}(37))}_2\approx 2.9657875$,
$\lambda^{\Gamma({\hat{\rR}}(38))}_2\approx 2.9662090$,
$\lambda^{\Gamma({\hat{\rR}}(39))}_2\approx 2.9658510$,
$\lambda^{\Gamma({\hat{\rR}}(40))}_2\approx 2.9699384$,
$\lambda^{\Gamma({\hat{\rR}}(41))}_2\approx 2.9690946$,
$\lambda^{\Gamma({\hat{\rR}}(42))}_2\approx 2.9719942$,
$\lambda^{\Gamma({\hat{\rR}}(43))}_2\approx 2.9717423$,
$\lambda^{\Gamma({\hat{\rR}}(44))}_2\approx 2.9718242$,
$\lambda^{\Gamma({\hat{\rR}}(45))}_2\approx 2.9743345$,
$\lambda^{\Gamma({\hat{\rR}}(46))}_2\approx 2.9740856$,
$\lambda^{\Gamma({\hat{\rR}}(47))}_2\approx 2.9741680$,
$\lambda^{\Gamma({\hat{\rR}}(48))}_2\approx 2.9764455$,
$\lambda^{\Gamma({\hat{\rR}}(49))}_2\approx 2.9762292$,
$\lambda^{\Gamma({\hat{\rR}}(50))}_2\approx 2.9762813$,
$\lambda^{\Gamma({\hat{\rR}}(51))}_2\approx 2.9780480$,
$\lambda^{\Gamma({\hat{\rR}}(52))}_2\approx 2.9793112$,
$\lambda^{\Gamma({\hat{\rR}}(53))}_2\approx 2.9795653$,
$\lambda^{\Gamma({\hat{\rR}}(54))}_2\approx 2.9818629$,
$\lambda^{\Gamma({\hat{\rR}}(55))}_2\approx 2.9847069$.
\newline\newline
Therefore $\Gamma({\hat{\rR}}(1))$-$\Gamma({\hat{\rR}}(8))$ are bipartite Ramanujan graphs.
Using also the notation of \cite{AA17}, we see that
for $k=1$ (resp. $2$, resp. $3$, resp. $4$, resp. $5$, resp. $6$,
resp. $7$, resp. $8$),
${\hat{\rR}}(k)$ is $A_3$, $A(2,0)$, and $A(1,1)$
(resp. $D(2,1;x)$ with $x\ne 0,-1$, and $D(2,1)$, 
resp. $C(3)$, 
resp. $C_3$, 
resp. $B_3$, $B(0,3)$, $B(2,1)$, 
$B(1,2)$, and ${\mathrm{ufo}}(1)$, 
resp. ${\mathrm{g}}(2,3)$,
resp. ${\mathrm{ufo}}(3)$,
resp. ${\mathrm{ufo}}(4)$).
${\hat{\rR}}(2)$ is also the one of \cite[Table~2, row~11]{Hec09}.

\subsection{Hyperplane arrangements}
\label{subsection:Hypaa}
Here we see how hyperplane arrangements relate to FGRSs.
Theorem~\ref{theorem:EqvHpyAndFGRS} (see below) has originally been achieved
by M.~Cuntz~\cite{Cu11}.
We reprove it
using the terminology of FGRSs.

Let $\fkI$ be a non empty finite set, and let
${\breve{V}}$ be an $|\fkI|$-dimensional $\bR$-linear space.
Let ${\breve{V}}^*$ be the dual $\bR$-linear space of ${\breve{V}}$.
We call an $|\fkI|-1$-dimensional linear subspace of ${\breve{V}}$
a {\it{linear hypeplane}}.
We call a subset $X$ of ${\breve{V}}$
a {\it{hyperplane}}
if $X=x+X^\prime$
for some  $x\in{\breve{V}}$
and a linear hypeplane $X^\prime$ of ${\breve{V}}$.

We say that a non-empty finite set of some hyperplanes (resp. some linear hyperplanes) of ${\breve{V}}$
is a {\it{hyperplane arrangement}} 
(resp. a {\it{linear hyperplane arrangement}}) of ${\breve{V}}$.

For a hyperplane arrangement ${\breve{A}}$ of ${\breve{V}}$,
let ${\mathcal{K}}({\breve{A}})$ 
be the set of all the connected components of 
${\breve{V}}\setminus(\cup_{X\in{\breve{A}}}X)$.
We call an element of ${\mathcal{K}}({\breve{A}})$ 
a {\it{chamber}} of ${\breve{A}}$,
and moreover, for a hyperplane $X$ of ${\breve{V}}$,
call the hyperplane arrangement  $\{Y\cap X|Y\in{\breve{A}}, \emptyset\ne Y\cap X\ne X\}$ of  $X$ 
the {\it{restriction of ${\breve{A}}$ to $X$}}.

We call a subset $K$ of ${\breve{V}}$
an {\it{open simplicial cone}} of ${\breve{V}}$
if there exists an $\bR$-basis $\{v_i|i\in\fkI\}$
of ${\breve{V}}$
with $K=\oplus_{i\in\fkI}\bRgneqo v_i$,
where notice ${\bar{K}}=\oplus_{i\in\fkI}\bRgeqo v_i$.
We say that a linear hyperplane arrangement 
${\breve{A}}$ is a {\it{simplicial arrangement}}
if every element of ${\mathcal{K}}({\breve{A}})$
is an open simplicial cone of ${\breve{V}}$.

Let $K$ be an open simplicial cone of ${\breve{V}}$.
Let
\begin{equation*}
{\breve{\bB}}(K):=\{{\breve{B}}\in{\mathfrak{P}}_{|\fkI|}({\breve{V}}^*)|
\{v\in{\breve{V}}|\forall \lambda\in{\breve{B}}, \lambda(v)>0\}=K\}.
\end{equation*} For a subset $Y$ of ${\breve{V}}$, 
Let ${\breve{\bB}}(K;Y):=\{{\breve{B}}\in{\breve{\bB}}(K)|{\breve{B}}\subset Y\}$.
Let $\{v_i|i\in\fkI\}$ be the above one for $K$.
Then its dual basis $\{v^*_i|i\in\fkI\}$ belongs to ${\breve{\bB}}(K)$.

The following lemma is essential.
\begin{lemma}\label{lemma:BasiHy}
Let $K$ be an open simplicial cone of ${\breve{V}}$.
Let $\{v_i|i\in\fkI\}$ be an $\bR$-basis
of ${\breve{V}}$
with $K=\oplus_{i\in\fkI}\bRgneqo v_i$.
Let ${\breve{B}}_0=\{{\breve{\al}}_i|i\in\fkI\}\,(\in{\mathfrak{P}}_{|\fkI|}({\breve{V}}^*))$
be the dual $\bR$-basis of $\{v_i|i\in\fkI\}$, i.e.,
 ${\breve{\al}}_x(v_y)=\delta_{xy}$ $(x,y\in\fkI)$.
\newline
{\rm{(1)}} We have ${\breve{B}}_0\in{\breve{\bB}}(K)$.
\newline
{\rm{(2)}}
Let ${\breve{\beta}}\in{\breve{V}}^*\setminus\{0\}$
be such that ${\breve{\beta}}(K)\subset\bRgneqo$.
Then  
${\breve{\beta}}\in(\oplus_{i\in\fkI}\bRgeqo{\breve{\al}}_i)
\setminus\{0\}$.
\newline
{\rm{(3)}}
Let ${\breve{B}}^\prime\in{\breve{\bB}}(K)$.
Then
there exist $x_i\in\bRgneqo$ $(i\in\fkI)$
with ${\breve{B}}^\prime=\{x_i{\breve{\al}}_i|i\in\fkI\}$. 
In particular, any element of ${\breve{\bB}}(K)$
is an $\bR$-basis of ${\breve{V}}$.
\newline
{\rm{(4)}} Let $Z$
be an open simplicial cone of ${\breve{V}}$
such that $\rmSpan_\bR({\bar {K}}\cap{\bar {Z}})$
is a linear hyperplane of ${\breve{V}}$.
Then there exist $i\in\fkI$
and $z_j\in\bR$ $(j\in\fkI\setminus\{i\})$ such that
$\{-{\breve{\al}}_i\}\cup\{{\breve{\al}}_j+z_i{\breve{\al}}_j|j\in\fkI\setminus\{i\}\}\in{\breve{\bB}}(Z)$.
Moreover for $j\in\fkI\setminus\{i\}$,
if $\ker({\breve{\al}}_j+z_i{\breve{\al}}_i)\cap K=\emptyset$,
then $z_j\geq 0$.
\end{lemma}
\begin{proof}
{\rm{(1)}} The claim is clear.
\newline
{\rm{(2)}} 
It suffices to show ${\breve{\beta}}(v_k)\geq 0$
for all $k\in\fkI$.
Assume that 
${\breve{\beta}}(v_t)<0$ for some $t\in\fkI$.
This is contradiction 
since ${\breve{\beta}}(xv_t+\sum_{k\in\fkI\setminus\{t\}}v_k)<0$
for some $x\in\bRgneqo$.
\newline
{\rm{(3)}} Let ${\breve{B}}^\prime=\{{\breve{\al}}^\prime_i |i\in\fkI\}\in{\breve{\bB}}(K)$.
Then ${\breve{\al}}^\prime_j=\sum_{i\in\fkI}y_{ij}{\breve{\al}}_i$
with $y_{ij}\in\bR$. 
By {\rm{(2)}}, we have $y_{ij}\in\bRgeqo$ for all $i$, $j\in\fkI$.
For every $i\in\fkI$, there exists $j\in\fkI$
with $y_{ij}\in\bRgneqo$.
Assume that there exists $i\in\fkI$ with ${\breve{B}}^\prime\cap\bRgneqo{\breve{\al}}_i=\emptyset$.
Then there exists $z\in\bRgneqo$ such that ${\breve{\al}}^\prime_k(-v_i+\sum_{j\in\fkI\setminus\{i\}}zv_j)\in\bRgneqo$
for all $k\in\fkI$.
This is contradiction. Hence the claim is true.
\newline
{\rm{(4)}}
Let $i\in\fkI$ be such that ${\bar {K}}\cap{\bar {Z}}
=\oplus_{k\in\fkI\setminus\{i\}}\bRgeqo v_k$.
Let $u=\sum_{k\in\fkI}z_kv_k\in{\breve{V}}\setminus\{0\}$ 
$(z_k\in\bR)$ be such that
${\bar {Z}}=\bRgeqo u\oplus(\oplus_{k\in\fkI\setminus\{i\}}\bRgeqo v_k)$.
Let $Y:=\{u\}\cup\{v_k|k\in\fkI\setminus\{i\}\}$.
Then $Y$ is an
$\bR$-basis of ${\breve{V}}$.
We have $z_i<0$.
We may assume $z_i=-1$.
Then $\{-{\breve{\al}}_j\}\cup\{{\breve{\al}}_k+z_k{\breve{\al}}_i|
k\in\fkI\setminus\{i\}\}$ is a dual basis of $Y$.
The second claim is clear.
\end{proof}

\begin{definition}\label{definition:defCry}
Let ${\breve{A}}$ be a simplicial arrangement of ${\breve{V}}$.
Let ${\breve{R}}$ be a non-empty finite subset of ${\breve{V}}^*$
such that 
\begin{equation*}
{\breve{A}}=\{\ker{\breve{\beta}}|{\breve{\beta}}\in{\breve{R}}\}
\quad\mbox{and}\quad
\bR{\breve{\beta}}\cap{\breve{R}}=\{{\breve{\beta}},-{\breve{\beta}}\}\,
({\breve{\beta}}\in{\breve{R}}).
\end{equation*}
Let $K\in{\mathcal{K}}({\breve{A}})$.
Let ${\breve{B}}_K$ be a unique element of ${\breve{\bB}}(K;{\breve{R}})$
(i.e., ${\breve{\bB}}(K;{\breve{R}})=\{{\breve{B}}_K\}$),
where the uniqueness 
follows from Lemma~\ref{lemma:BasiHy}~(3). 
Let ${\breve{R}}_K^+:={\breve{R}}\cap
(\oplus_{{\breve{\al}}\in{\breve{B}}_K}\bRgeqo{\breve{\al}})$,
where by Lemma~\ref{lemma:BasiHy}~(2), we have 
\begin{equation*}
{\breve{R}}={\breve{R}}_K^+\cup(-{\breve{R}}_K^+),
\quad {\breve{R}}_K^+\cap(-{\breve{R}}_K^+)=\emptyset.
\end{equation*}
We say that $({\breve{A}},{\breve{R}})$ is a 
{\it{crystallographic arrangement}} (or {\it{crystallographic hyperplane arrangement}})
if
\begin{equation}\label{eqn:exsCryB}
\exists Z\in{\mathcal{K}}({\breve{A}}),\,\,
{\breve{R}}\subset\oplus_{{\breve{\al}}\in{\breve{B}}_Z}\bZ{\breve{\al}}.
\end{equation}
\end{definition}

\begin{lemma}\label{lemma:stepbs}
Let $({\breve{A}},{\breve{R}})$ be a crystallographic arrangement
of ${\breve{V}}$.
Then every $K\in{\mathcal{K}}({\breve{A}})$
satisfies the same property as \eqref{eqn:exsCryB}
with $K$ in place of $Z$.
\end{lemma}
\begin{proof} 
Let $Z\in{\mathcal{K}}({\breve{A}})$ be as in \eqref{eqn:exsCryB} and 
$\{{\breve{\al}}_i|i\in\fkI\}:={\breve{B}}_Z$.
Let $K\in{\mathcal{K}}({\breve{A}})$.
We use an induction on $l(K):=|{\breve{R}}_K^+\cap(-{\breve{R}}_Z^+)|$.
If $l(K)=0$, then ${\breve{R}}_K^+={\breve{R}}_Z^+$,
whence $K=Z$ since ${\breve{B}}_K={\breve{B}}_Z$.
Assume  $l(K)>0$.
Then we have $i\in\fkI$ with
$-{\breve{\al}}_i\in{\breve{R}}_K^+\cap(-{\breve{R}}_Z^+)$.
By 
Lemma~\ref{lemma:BasiHy}~(4)
and \eqref{eqn:exsCryB},
there exists $X\in{\mathcal{K}}({\breve{A}})$ such that 
${\breve{B}}_X=\{-{\breve{\al}}_i\}\cup\{{\breve{\al}}_j+n_j{\breve{\al}}_i
|j\in\fkI\setminus\{i\}\}$ for some $n_j\in\bZgeqo$.
We see that $X$ satisfies the same property as \eqref{eqn:exsCryB}
with $X$ in place of $Z$.
We have 
${\breve{R}}_X^+\setminus\{-{\breve{\al}}_i\}
={\breve{R}}_Z^+\setminus\{{\breve{\al}}_i\}$.
Hence $l(X)=l(K)-1$.
\end{proof}

Using an argument similar to that of the proof of Lemma~\ref{lemma:stepbs},
we re-prove the following theorem of \cite{Cu11} by M.~Cuntz.
\begin{theorem} {\rm{(}}{\rm{\cite[Theorem~5.4]{Cu11}}}{\rm{)}}
\label{theorem:EqvHpyAndFGRS}
{\rm{(1)}}
Let $({\breve{A}},{\breve{R}})$ be a crystallographic arrangement
of ${\breve{V}}$. Let ${\breve{\fca}}:=\rmSpan_\bZ({\breve{R}})$.
Then ${\breve{R}}$ is an FGRS over ${\breve{\fca}}$.
\newline
{\rm{(2)}}
Let $\rR$ be an FGRS
over $\fca$.
Let
$\fca_\bR$ and 
$\fca_\bR^*$ be as in Lemma~{\rm{\ref{lemma:bBprime}~(3)}}.
For $\beta\in R$, let $X_\beta:=\{\lambda\in\fca_\bR^*|\lambda(\beta)=0\}$.
Let ${\breve{A}}[\rR]:=\{X_\beta|\beta\in\rR\}$.
Then $({\breve{A}}[\rR],\rR)$ is a crystallographic arrangement.
{\rm{(}}We identify $(\fca_\bR^*)^*$ with $\fca_\bR$
in a natural way.{\rm{)}}
\newline
{\rm{(3)}}
Keep the notation of {\rm{(2)}}.
Let ${\mathfrak{F}}$ be the set of all the FGRSs over $\fca$.
Let ${\mathfrak{C}}$ be the set of all the 
crystallographic arrangements $({\breve{A}},{\breve{R}})$ of $\fca_\bR^*$
with $\rmSpan_\bZ{\breve{R}}=\fca$.
Define the map $\Omega:{\mathfrak{F}}\to{\mathfrak{C}}$
by $\Omega(\rR):=({\breve{A}}[\rR],\rR)$.
Then $\Omega$ is a bijection.
\end{theorem}

\subsection{FGRS of $D(2,1;x)$}
Keep the notation of Subsection~\ref{subsection:rankthree}.
Fig.~5 is
the Cayley graph $\Gamma({\hat{\rR}}(2))$, 
where the bold lines is 
its Hamiltonian cycle
given in Nr.~2 in Subsection~\ref{subsection:rankthree}. 
We know that ${\hat{\rR}}(2)$ is tha same as $\rR(A,\fkI_{\mathrm{odd}})$
with $D(A,\fkI_{\mathrm{odd}})$ being $D(2,1;x)$.
Let
$\{a,b,c,d\}$ of Fig.~5 mean $\bB({\hat{\rR}}(2))/\!\!\sim$ for the smallest
groupoid equivalence $\sim$ of the FGRS ${\hat{\rR}}(2)$,
where $a:=[\{{\hat{\al}}_1,{\hat{\al}}_2,{\hat{\al}}_3\}]^\sim$,
$b:=[\{{\hat{\al}}_1+{\hat{\al}}_2,-{\hat{\al}}_2,{\hat{\al}}_2+{\hat{\al}}_3\}]^\sim$,
$c:=[\{-{\hat{\al}}_1,{\hat{\al}}_1+{\hat{\al}}_2,{\hat{\al}}_1+{\hat{\al}}_3\}]^\sim$,
and
$d:=[\{{\hat{\al}}_1+{\hat{\al}}_3,{\hat{\al}}_2+{\hat{\al}}_3,-{\hat{\al}}_3\}]^\sim$.
\begin{center}
$\CayleyNine$ \\
$\begin{array}{l}
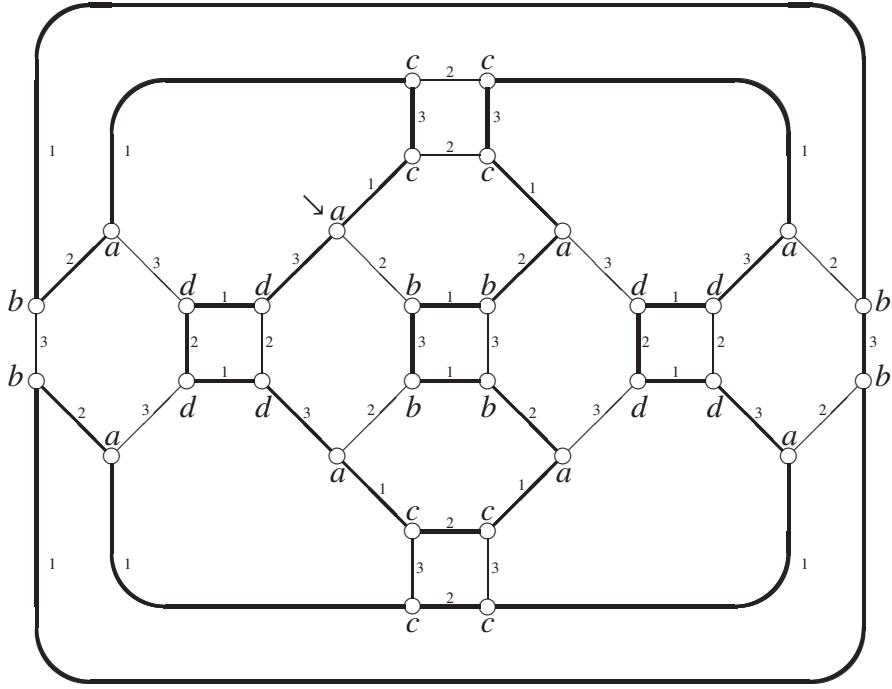
$\mbox{Fig.~5. The Cayley graph $\Gamma({\hat{\rR}}(2))$ and one of its Hamiltonian cycles}$ \\
$\quad\quad\quad
\mbox{with the dot pointed by $\rightarrow$ meaning its start and end place}$
\end{array}$
\end{center}

\begin{center}
$\quantumDynkinDtwooneal$ \\
$\begin{array}{l}
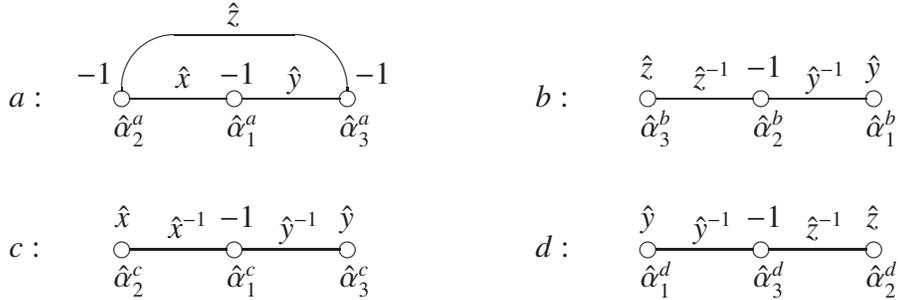
$\mbox{Fig.~6. $D_\bK(\chi,B)$s for ${\hat{\rR}}(2)$ with $[B]^\sim\in\{a,b,c,d\}$,}$
\\
$\quad\quad\quad
\mbox{where $\bK$ is a field of characteristic zero,}$
\\
$\quad\quad\quad
\mbox{and ${\hat x}$, ${\hat y}$, ${\hat z}$, ${\hat x}{\hat y}$,
${\hat x}{\hat z}$, ${\hat y}{\hat z}\in\bKt\setminus\{1\}$
and ${\hat x}{\hat y}{\hat z}=1$.}$
\\
$\quad\quad\quad
\mbox{See also \cite[Table~2, Rows~9,10,11]{Hec09}.}$
\end{array}$
\end{center}

Here, motivated by Theorem~\ref{theorem:EqvHpyAndFGRS} by M.~Cuntz,
we study ${\hat{\rR}}(2)$ in a geometrical way.
We regard ${\hat{\fca}}$ as the $\bZ$-submodule of $\bR^3$ 
with identifying ${\hat{B}}=\{{\hat{\al}}_1,{\hat{\al}}_2,{\hat{\al}}_3\}$ with the standard basis of $\bR^3$.
For $\lambda=\sum_{t=1}^3u_t{\hat{\al}}_t\in\bR\setminus\{0\}$ $(u_t\in\bR)$,
let $\lambda^{\perp,\prime}:=\{\sum_{t=1}^3v_t{\hat{\al}}_t\,(v_t\in\bR)\,|\,
\sum_{t=1}^3u_tv_t=0\}$.
Let ${\mathfrak{H}}:=\{x{\hat{\al}}_1+y{\hat{\al}}_2+z{\hat{\al}}_3|x,y,z\in\bR, x+y+z=1\}$ of  $\bR^3$.
Then Fig.~7 means the restriction of the crystallographic arrangement $\{\beta^{\perp,\prime}|\beta\in{\hat{\rR}}(2)\}$ to ${\mathfrak{H}}$,
where 
let $\beta^\perp:=\beta^{\perp,\prime}\cap{\mathfrak{H}}$.
Let ${\mathfrak{C}}_t$ $(t\in\fkJ_{1,16})$ mean
${\mathfrak{H}}\cap\rmSpan_{\bRgeqo}(B)$
for some $B\in\bB({\hat{\rR}}(2))$
with ${\hat{\al}}_1+{\hat{\al}}_2+{\hat{\al}}_2\in{\hat{\rR}}(2)^+_B$.
In particular, let ${\mathfrak{C}}_8$ mean ${\mathfrak{H}}\cap\rmSpan_{\bRgeqo}({\hat{B}})$.
Let $1^{\mbox{${\hat{B}}$}}\cdot s_{i_1}\cdots s_{i_{32}}$
be the Hamiltonian cycle of $\Gamma({\hat{\rR}}(2))$ given as Nr.~2 in Subsection~\ref{subsection:rankthree}.
Then the following sequence \eqref{eqn:CSeq} equals
$\rmSpan_{\bRgeqo}({\hat{B}}G_{i_1}\cdots G_{i_t})\cap({\mathfrak{H}}\cup(-{\mathfrak{H}}))$
$(t\in\fkJ_{0,32})$, where ${\hat{B}}G_{i_1}\cdots G_{i_t}$ means ${\hat{B}}$ if $t=0$.
\begin{equation}\label{eqn:CSeq}
\begin{array}{l}
{\mathfrak{C}}_8, {\mathfrak{C}}_{13}, {\mathfrak{C}}_{14}, {\mathfrak{C}}_{15},
{\mathfrak{C}}_{16}, {\mathfrak{C}}_{12},  {\mathfrak{C}}_{11}, -{\mathfrak{C}}_5, 
-{\mathfrak{C}}_{10}, -{\mathfrak{C}}_{15}, {\mathfrak{C}}_1, {\mathfrak{C}}_6,
{\mathfrak{C}}_7, {\mathfrak{C}}_2, {\mathfrak{C}}_3, {\mathfrak{C}}_4, \\
{\mathfrak{C}}_5, -{\mathfrak{C}}_{11}, -{\mathfrak{C}}_{12}, -{\mathfrak{C}}_{16},
-{\mathfrak{C}}_{14}, -{\mathfrak{C}}_{13}, -{\mathfrak{C}}_8, -{\mathfrak{C}}_9,
-{\mathfrak{C}}_3, -{\mathfrak{C}}_2, -{\mathfrak{C}}_7, -{\mathfrak{C}}_6,
-{\mathfrak{C}}_1, {\mathfrak{C}}_{15}, {\mathfrak{C}}_{10}, {\mathfrak{C}}_9, \\
{\mathfrak{C}}_8.
\end{array}
\end{equation}

\begin{center}
\setlength{\unitlength}{1mm}
\begin{picture}(120,140)(-13,-35)
\put(-10,64.64){\line(1,0){120}}
\put(-10,30){\line(1,0){120}}

\put(47,64.64){\circle*{1}}\put(45,56.64){$\al_3$}
\put(27,30){\circle*{1}}\put(25,22){$\al_1$}
\put(67,30){\circle*{1}}\put(65,22){$\al_2$}

\put(27,30){\rotatebox{60}{\line(1,0){70}}}\put(27,30){\rotatebox{60}{\line(-1,0){60}}}
\put(27,30){\rotatebox{-60}{\line(1,0){60}}}\put(27,30){\rotatebox{-60}{\line(-1,0){70}}}

\put(67,30){\rotatebox{60}{\line(1,0){70}}}\put(67,30){\rotatebox{60}{\line(-1,0){60}}}
\put(67,30){\rotatebox{-60}{\line(1,0){60}}}\put(67,30){\rotatebox{-60}{\line(-1,0){70}}}

\put(47,30){\circle*{1}}\put(37,24){${\frac 1 2}(\al_1+\al_2)$}
\put(37,47.32){\circle*{1}}\put(18,48){${\frac 1 2}(\al_1+\al_3)$}
\put(57,47.32){\circle*{1}}\put(58,48){${\frac 1 2}(\al_2+\al_3)$}

\put(47,41.55){\circle*{1}}

\put(34,37){$\mbox{\small{${\frac 1 3}(\al_1+\al_2+\al_3)$}}$}
\put(78,16){\rotatebox{-60}{$\al_1^\perp$}}

\put(8,32){$\al_3^\perp$}

\put(18,67){$(\al_1+\al_2)^\perp$}
\put(28,16){\rotatebox{-60}{$(\al_2+\al_3)^\perp$}}
\put(55,4){\rotatebox{60}{$(\al_1+\al_3)^\perp$}}

\put(15,4){\rotatebox{60}{$\al_2^\perp$}}

\put(-5,68){${\mathfrak{C}}_1$}
\put(20,74){${\mathfrak{C}}_2$}
\put(45,74){${\mathfrak{C}}_3$}
\put(65,74){${\mathfrak{C}}_4$}
\put(95,68){${\mathfrak{C}}_5$}

\put(00,54){${\mathfrak{C}}_6$}
\put(25,54){${\mathfrak{C}}_7$}
\put(45,47){${\mathfrak{C}}_8$}
\put(65,54){${\mathfrak{C}}_9$}
\put(90,54){${\mathfrak{C}}_{10}$}

\put(10,22){${\mathfrak{C}}_{11}$}
\put(25,5){${\mathfrak{C}}_{12}$}
\put(45,12){${\mathfrak{C}}_{13}$}
\put(65,5){${\mathfrak{C}}_{14}$}
\put(80,22){${\mathfrak{C}}_{15}$}

\put(45,-15){${\mathfrak{C}}_{16}$}

\end{picture} \\
Fig.~7. The restriction of the crystallographic arrangement $\{\beta^{\perp,\prime}|\beta\in{\hat{\rR}}(2)\}$
to ${\mathfrak{H}}$
  \end{center}

\subsection{Proof of Theorem~\ref{theorem:one}} \label{subsection:wellknownpf}
Here we explain the proofs of Theorems~\ref{theorem:one}.and \ref{theorem:two}.
\newline\newline
{\it{Proof of Theorem~{\rm{\ref{theorem:one}}.}}}
For $x\in\fkJ_{2,\infty}$,
let $\theta_x:\bZ\to\bZ/x\bZ$ be the canonical map.
Let $H$ be the subgroup of $G$ generated by $b$ and $c$.
Let $p:=|H|$, $r:=|G|$ and $u:={\frac r p}$. Note that $p\in 2\bN$
and ${\frac p 2}$ is the order of $bc$.
Let $\psi:\bZ/p\bZ\to H$ be a bijection such that
$\psi(\theta_p(2z))=\psi(\theta_p(2z-1))c$, 
$\psi(\theta_p(2z+1))=\psi(\theta_p(2z))b$ $(z\in\bZ)$. 
Let $g_t\in G$ $(t\in\fkJ_{1,u})$ 
be such that $g_1=e$ and $G=\coprod_{t=1}^ug_tH$ (disjoint union).
Then we can have a bijection $s:\fkJ_{1,u}\to\fkJ_{1,u}$
and bijections $\varphi_v:\bZ/pv\bZ\to\coprod_{t=1}^vg_{s(t)}H$
$(v\in\fkJ_{1,u})$ having the following properties ${\rm{(x)}}$-${\rm{(y)}}$.
\newline\newline
${\rm{(x)}}$ $s(1)=1$ and $\varphi_1=\psi$. \newline
${\rm{(y)}}$ Let $v\in\fkJ_{2,u}$. 
There exist $i,j\in\bZ$ such that 
$\varphi_{v-1}(\theta_{p(v-1)}(i))b=\varphi_{v-1}(\theta_{p(v-1)}(i+1))$,
$g_{s(v)}\psi(\theta_p(j))b=g_{s(v)}\psi(\theta_p(j+1))$,
$\varphi_{v-1}(\theta_{p(v-1)}(i))a=g_{s(v)}\psi(\theta_p(j))$
and
$\varphi_{v-1}(\theta_{p(v-1)}(i+1))a=g_{s(v)}\psi(\theta_p(j+1))$.
We define $\varphi_v$ by $\varphi_v(\theta_{pv}(k)):=\varphi_{v-1}(\theta_{p(v-1)}(i+k))$
$(k\in\fkJ_{1,p(v-1)})$
and $\varphi_v(\theta_{pv}(k^\prime)):=g_{s(v)}\psi(\theta_p(j+p(v-1)+1-k^\prime))$
$(k^\prime\in\fkJ_{p(v-1)+1,pv})$.
\newline\newline
Thus we have a Hamiltonian cycle
$\varphi$ of $\Gamma(G,Z)$ defined by $\varphi(i):=\varphi_u(\theta_r(i))$
$(i\in\fkJ_{1,r})$. \newline \hfill $\Box$
\newline\newline
{\it{Rough proof of Theorem~{\rm{\ref{theorem:two}}}.}}
(See also \cite[Theorem~2.3]{Y22}.)
Notice that there exist $i,j\in\fkI$ with $i\ne j$ 
and $m(i,k)=2$ for all $k\in\fkI\setminus\{i,j\}$.
Hence we can prove Theorem~\ref{theorem:two} in an induction on $|\fkI|$  
by an argument similar to that for 
Theorem~\ref{theorem:one}. 
\hfill $\Box$ 
\vspace{1cm}

\noindent
Acknowledgements.
Regarding this work,
H.~Yamane was partially supported by JSPS Grant-in-Aid
for Scientific Research (C) 22K03225.
We essentially used Mathematica~13.3 \cite{Mathe23} and its earlier versions.

\noindent
Takato Inoue \newline
Mathematics and Informatics Program, \newline
Graduate School of Science and Engineering, \newline
University of Toyama \newline
3190 Gofuku, Toyama 930-8555, Japan \newline
E-mail: m22c1005@ems.u-toyama.ac.jp
\newline\newline
Hiroyuki Yamane \newline
Faculty of Science, Academic Assembly, \newline
University of Toyama \newline
3190 Gofuku, Toyama 930-8555, Japan \newline
E-mail: hiroyuki@sci.u-toyama.ac.jp


\begin{thebibliography}{99}

\bibitem[1]{AA17} N.~Andruskiewitsch, I.~Angiono, 
On finite dimensional Nichols algebras of diagonal type, 
Bull. Math. Sci. 7 (2017), no. 3, 353-573.

\bibitem[2]{AY15} I. Angiono, H. Yamane, The R-matrix of quantum doubles of
Nichols algebras of diagonal type, J. Math. Phys., 56, no. 2 (2015),
021702, 19 pp.

\bibitem[3]{A11} T.~Audrey, Zeta functions of graphs, Cambridge Stud. Adv. Math., 128, Cambridge University Press, Cambridge, 2011.

\bibitem[4]{AYY15} S.~Azam, H.~Yamane, M.~Yousofzadeh, Classification of finite dimensional irreducible representations of generalized quantum groups via Weyl groupoids, Publ. Res. Inst. Math. Sci. 51 (2015), 59-130.




\bibitem[5]{BY18} P.~Batra, H.~Yamane, Centers of generalized quantum groups, J. Pure Appl. Algebra 222 (2018), no. 5, 1203-1241

\bibitem[6]{BH12}  A.E. Brouwer, W.H. Haemers, Spectra of graphs, Springer, New York, etc., 2012. ISBN 978-1-4614-1938-9.

\bibitem[7]{CSW89} J.H.~Conway, N.J.A.~Sloane, Allan~R.~Wilks,
Gray codes for reflection groups, 
Graphs and Combinatorics 5 (1989), 315-325. 

\bibitem[8]{Cu11} M.~Cuntz, Crystallographic arrangements: Weyl groupoids and simplicial arrangements,
Bull. London Math. Soc. 43 (2011), 734-744.





\bibitem[9]{CH09b} M.~Cuntz, I.~Heckenberger, Weyl groupoids of rank two and continued fractions, Algebra \&
Number Theory 3 (2009), 317–340.

\bibitem[10]{CH12}  M.~Cuntz, I.~Heckenberger, Finite Weyl groupoids of rank three, Trans. Amer. Math. Soc. 364 (2012), no. 3, 1369-1393.

\bibitem[11]{CH15}  M.~Cuntz, I.~Heckenberger, Finite Weyl groupoids, J. Reine Angew. Math. 702 (2015), 77-108.

\bibitem[12]{FSS89} L.~Frappat, A.~Sciarrino, P.~Sorba, Structure of basic Lie superalgebras and of their affine extensions, Comm. Math. Phys. 121 (1989), no.~3, 457-500.


\bibitem[13]{Hec09} I.~Heckenberger, Classification of arithmetic root systems, Adv. Math. 220 
(2009), no.~1, 59-124.


\bibitem[14]{HS20} I.~Heckenberger, H.J.~Schneider,
Hopf algebras and root systems.
Mathematical Surveys and Monographs, 247. American Mathematical Society, Providence, RI, 2020,  ISBN: 978-1-4704-5232-2

\bibitem[15]{HY08} I.~Heckenberger, H.~Yamane, A generalization of Coxeter groups, root systems, and Matsumoto's theorem, Math. Z. 259 (2008), no. 2, 255-276,

\bibitem[16]{HY10} I.~Heckenberger, H.~Yamane, Drinfel'd doubles and Shapovalov determinants, Rev. Un. Mat. Argentina 51 (2010), no. 2, 107-146.

\bibitem[17]{Kac77} V.G.~Kac, Lie superalgebras, Adv. Math., 26 (1977), 8-96.



\bibitem[18]{LSS85} D.A.~Leites, M. V.~Saveliev, V. V.~Serganova, Embedding of 
$\rm{OSP}(N/2)$ and the associated nonlinear
supersymmetric equations, Group theoretical methods in physics, Vol. I (Yurmala, 1985), 255--297.


\bibitem[19]{MSS15} A.W.~Marucus, D.A.~Spielman, N.~Srivastava, Interlacing Familes~I: 
Bipartite Ramanujan graphs of all degrees, Ann. of Math. (2) 182 (2015),  no.1, 307--325.

\bibitem[20]{PR06} I.~Pak, R.~Radoicic, Hamiltonian paths in Cayley graphs,
Discrete Mathematics 309 (2009), 5501-5508

\bibitem[21]{S96} V.V.~Serganova, On generalizations of root systems, Comm. Algebra 24 (1996), no. 13, 4281--4299.

\bibitem[22]{SR59} E.~Strasser-Rapaport, Cayley color groups and Hamilton lines, 
Scripta Math. 24 (1959), 51-58.

\bibitem[23]{VdL86} J. W.~Van De Leur, Contragredient Lie superalgebras of finite growth, Thesis, University
of Utrecht, 1986.

\bibitem[24]{Mathe23} Wolfram Research, Inc., Mathematica~13.3, Wolfram Research, Inc., Champaign, Illinois (2023), 
https://www.wolfram.com/mathematics

\bibitem[25]{Y16} H.~Yamane, Generalized root systems and the affine Lie superalgebra 
${\mathrm{G}}^{(1)}(3)$, Sao Paulo J. Math. Sci. 10 (2016), no. 1, 9-19.

\bibitem[26]{Y22} H.~Yamane, Hamilton circuits of Cayley graphs of Weyl groupoids of generalized quantum groups, Toyama Math. J. 43 (2022), 1-76


\end{thebibliography}
\end{document}